\documentclass{tac}
\usepackage{amssymb, amsmath}
\usepackage[all]{xy}
\usepackage{enumitem}

\usepackage{graphics}
\usepackage{quiver}

\usepackage[title]{appendix}

\newtheorem{thm}{Theorem}
\newtheorem{cor}{Corollary}
\newtheorem{prop}{Proposition}
\newtheorem{lma}{Lemma}

\newtheoremrm{dfn}{Definition}
\newtheoremrm{defn}{Definition}
\newtheoremrm{egs}{Examples}
\newtheoremrm{no}{Notation}
\newtheoremrm{rmks}{Remarks}
\newtheoremrm{rmk}{Remark}

\newcommand{\frakA}{\mathfrak{A}}
\newcommand{\frakC}{\mathfrak{C}}
\newcommand{\frakS}{\mathfrak{S}}
\newcommand{\frakW}{\mathfrak{W}}
\newcommand{\frakV}{\mathfrak{V}}
\newcommand{\calB}{\mathcal{B}}
\newcommand{\calC}{\mathcal{C}}
\newcommand{\calD}{\mathcal{D}}
\newcommand{\calG}{\mathcal{G}}
\newcommand{\calP}{\mathcal{P}}
\newcommand{\calU}{\mathcal{U}}
\newcommand{\calV}{\mathcal{V}}
\newcommand{\calW}{\mathcal{W}}
\newcommand{\Hom}{\mbox{\rm Hom}}

\title[Bicategories of Fractions Revisited]{Bicategories of fractions revisited:  towards small homs and canonical 2-cells}
\author{Dorette Pronk, Laura Scull}

\eaddress{Dorette.Pronk@Dal.Ca \CR scull\_l@fortlewis.edu}


\thanks{The first author's research is funded by an NSERC Discovery Grant.}

\begin{document}

\maketitle
\begin{abstract}
 This paper adresses two issues in dealing with bicategories of fractions.  The first is to introduce a set of conditions on a class of arrows in a bicategory which is  weaker than the one given in \cite{Pr-comp} but still allows a bicalculus of fractions.  These conditions allow us to invert a smaller  collection of arrows so that in some cases we may obtain a bicategory of fractions with small hom-categories.   We adapt the construction of the bicategory of fractions to work with the weaker conditions.   The second issue is the difficulty in dealing with 2-cells, which are defined by equivalence classes.  We discuss conditions under which there are canonical representatives for 2-cells, and  how pasting of 2-cells can be simplified in the presence of certain pseudo pullbacks.  We also discuss how both of these improvements apply in the category of orbispaces.  
\end{abstract}


\section{Introduction}  The purpose of this paper is to study some aspects of the structure of bicategories of fractions in more detail.   We focus specifically on two goals.  The first is to develop a weaker version of the calculus of fractions conditions of \cite{Pr-comp} that is still strong enough to allow us to create a bicategory of fractions where arrows are given by spans rather than zig-zags. (We will show that the conditions in \cite{Pr-comp} are not necessary in order to use fractions, although they make the construction slightly easier.) The second goal is to develop conditions under which we have canonical representatives for 2-cells, thus clarifying the structure of the category and its composition operations.  Although this second goal could be considered independently from the first, we will in fact give our proofs in  the context of the weaker conditions; since these imply the conditions of \cite{Pr-comp},  our 2-cell results will apply in both contexts.   For both of these goals, we will discuss how it applies to the example of orbispaces, defined as the   bicategory of fractions of proper \'etale groupoids of suitable topological spaces with respect to the class of essential equivalences as in described in \cite{MP,WIT}.   

For our first goal, we introduce a set of conditions on a class of arrows in a bicategory which is  weaker than the one given in \cite{Pr-comp} but still allows us to form the localization as  a bicalculus of fractions. 
One  potential issue with  localizations which are constructed as categories, or bicategories, of fractions is that the hom-sets, or  hom-categories,  may not be small,    as there is no guarantee in general that the fractions with a given domain and codomain form a set.   To ensure that we do get a set in the classical bicategory of fractions construction,  we need  the class of arrows  $\mathfrak{W}$ to be inverted be small over each object; i.e., for any given object $C$ 
there is only a {\em set} of arrows in $\mathfrak{W}$ with codomain $C$. 
We  may try to find a locally small subclass of the arrows to be inverted which generates the larger class in the sense that  it induces an equivalent category (or bicategory) of fractions.   
This subclass may not satisfy all of the conditions for forming a (bi)category of fractions, so we consider whether any of the conditions can  be weakened.  When an arrow can be factored as a composite of arrows that are to be inverted,  this arrow will receive an inverse in any localization that adds inverses for the arrows in the factorization.  This observation leads us to consider the second condition of \cite{Pr-comp}, the requirement that the  class of arrows to be inverted is closed under composition, as an axiom that could potentially be weakened. We cannot completely omit it: some version of this axiom is needed to be able to define horizontal composition in the bicategory of fractions.
However,  we can replace it by the following condition:\\

[{\bf WB2}] For each pair of composable arrows $\xymatrix@1{B\ar[r]^v&C\ar[r]^w&D}$ in $\frakW$, there is an arrow $\xymatrix@1{A\ar[r]^u&B}$ such that $\xymatrix@1{A\ar[r]^{wvu}&D}$ is in $\frakW$.\\
   
When a class of arrows satisfies this condition together with the other conditions for a bicalculus of fractions given in \cite{Pr-comp},   it generates (through composition and closure under 2-isomorphisms) a larger class of arrows that satisfies all the bicalculus of fractions conditions.  
 In this paper we will carefully consider all the conditions for the bicalculus of fractions and give more optimal versions of these conditions, 
and then provide an adjusted construction of the bicategory of fractions.  This construction is still given with arrows that are single spans rather than zig-zags.  
This also provides us with a slightly weaker set of conditions for the classical construction of the category of fractions as given by Gabriel and Zisman in \cite{GZ}, spelled   out in Corollary~\ref{classicalGZ}.

Our motivating example for this is the bicategory of orbispaces \cite{Sa56, MP,WIT}. A priori, the hom-categories in this category are not small  unless one requires all spaces to be second countable topological manifolds. We can work with a larger class of spaces, however, by observing that the class of essential equivalences has a subclass of essential covering maps that is small over each object, and satisfies the weakened conditions for a bicategory of fractions.

Related results and conditions have been presented in \cite{Roberts-TAC}. Roberts shows that for the case where $\frakW$ is a singleton pretopology satisfying the {\it WISC} condition that each object have a {\em set} of covers that is weakly initial among all covers, the bicategory of fractions will be locally essentially small: each hom-category is equivalent to a small one. By weakening the conditions to obtain a right calculus of fractions we are able to restrict ourselves to only invert the sets of covers when constructing the localization and obtain a locally small bicategory of fractions.

A different construction, of so called {\em faithful fractions}, was introduced in \cite{AV}. The result of this construction has small hom-categories as well. Different additional conditions need to be met to use this construction.

Another issue when working with a  (bi)category of fractions is that the homs are defined by equivalence classes.   For categories, arrows are given by equivalence classes;   for bicategories the same is true for 2-cells. This makes the hom-categories in the bicategory of fractions a priori very large and somewhat mysterious and hard to work with. Horizontal composition of 2-cells for instance is rather cumbersome to describe and calculate.   Our second goal in this paper is to address this issue by providing conditions under which there are canonical representatives for 2-cells and under which the horizontal composition operation is significantly simplified.  A partial simplification of the presentation of 2-cells was provided in the appendix of \cite{Matteo1} under additional hypotheses, but this was not integrated with the operations of horizontal and vertical composition.  In our motivating example of orbispaces, essential equivalences have several nice cancellation properties that allow for a simplification of the 2-cell structure and allow us to use canonical representatives for 2-cells when this is convenient. These cancellation properties were identified as being (representably) fully faithful (ff) and co-ff in \cite{AV} and \cite{Roberts-fractions, Roberts-admitting} and used there to obtain related results about 2-cells in their representations of specific cases of 2-localizations.  

In this work, we prove two types of results about the 2-cell structure: about  the choice of representatives for 2-cells, and about  conditions that allow us to simplify the pasting of 2-cells.   Each representative diagram for a 2-cell in the bicategory of fractions, as in diagram (\ref{first2-cell}) in Section~\ref{newBF2}, is given by two 2-cells in the original bicategory. The `left-hand' 2-cell $\alpha$ is invertible,  and we think of this as the cell that  allows the `right-hand' 2-cell $\beta$ to be defined. 
We focus on the role of the left-hand 2-cell. 
Tommasini indirectly addresses the question of when a 2-cell can be represented by a diagram with a given left-hand 2-cell in \cite{Matteo1}.   In general this is not always possible, and moreover,  two diagrams with the same left-hand 2-cell but different right-hand 2-cells may still represent the same 2-cell in the bicategory of fractions, so the universal homomorphism mapping a bicategory to its bicategory of fractions is in general neither 2-faithful nor 2-full.   However, if the arrows to be inverted satisfy suitable subsets of the  fully faithful or co fully faithful conditions, the situation simplifies and for each pair of spans we may choose any left-hand 2-cell and we show that  each 2-cell in the bicategory of fractions can then be uniquely represented by a diagram involving the given left-hand 2-cell. 

Additionally, for the case when the bicategory has certain pseudo pullbacks, we develop  results to simplify the horizontal composition of 2-cells in the bicategory of fractions.  Overall, our goal is to make the role of 2-cells in the bicategory of fractions more transparent.  In our motivating example of orbispaces these conditions are satisfied;  this will be explored further in \cite{next_paper}.  

Note that in  \cite{AV} the authors use the ff and co-ff cancellation properties of essential equivalences between internal categories in a regular category to describe the localization with respect to essential equivalences as a faithful bicategory of fractions. One of its notable properties is that 2-cells in the fractions bicategory correspond to suitable 2-cells in the original bicategory without needing to take equivalence classes. Similar results are obtained by Roberts for the bicategory of fractions of a pretopology consisting of ff and co-ff arrows. He also gives a canonical presentation for the 2-cells that corresponds to taking the strict pullback as the left-hand cell in the 2-cell diagrams of the bicategory of fractions. The work in this paper sheds further light on why this can be done: if the arrows in $\frakW$ are ff and co-ff one may choose any class of left-hand 2-cells to obtain canonical representations of 2-cells and avoid the need for equivalence classes.

This paper is structured as follows. In Section~\ref{S:background}, we introduce the new, weakened, conditions on a class $\frakW$ to give rise to a bicalculus of fractions,  and develop some theory on liftings of 2-cells related to the fourth condition on $\frakW$, and on relating squares required by the third condition. 
 In Section~\ref{newBF2} we give the new bicategory of fractions construction $\calB(\frakW^{-1})$,  a generalization of the one given in 
\cite{Pr-comp}, with horizontal composition of arrows and 2-cells adjusted to account for the weaker assumption. 
In Section~\ref{comparisons} we investigate the connection between our new construction and the original construction of \cite{Pr-comp}, and show that if $\mathfrak{W}$ 
satisfies the weaker conditions of Section~\ref{S:background}, then the class of arrows obtained by taking the closure of  $\mathfrak{W}$ under composition and 2-isomorphism satisfies the original conditions from \cite{Pr-comp} and gives a bi-equivalent bicategory of fractions.  Additionally, we introduce the notion of weakly initial subclasses of arrows, designed to allow us to pass to an even smaller subclass of arrows to obtain a subclass of a given class of arrows that is small over each object.  
Sections~\ref{simplifications} and \ref{pullbacks} develop our results about simplifying 2-cells.  In Section~\ref{simplifications} we introduce  conditions that allow us to simplify the form of the 2-cells in the bicategory of fractions and obtain canonical representatives for the equivalence classes, and 
in Section~\ref{pullbacks} we investigate the case when the original bicategory has certain pseudo pullbacks and show how this can be used to simplify horizontal composition of 2-cells in the bicategory of fractions.
In Section~\ref{conclusions} we indicate how this work applies to orbispaces, to be further explored in \cite{next_paper}.  
The last sections are appendices containing technical proofs. The first one gives the  associativity 2-cells for composition. The second appendix proves associativity coherence. The third appendix proves that horizontal and vertical composition are well-defined on equivalence classes of 2-cell diagrams. And the fourth appendix gives the proof for a result about the horizontal composition of 2-cell diagrams when the left-hand 2-cells are pseudo pullbacks.

{\bf Acknowledgements} The authors would like to thank Matteo Tommasini for contributing Lemmas \ref{Matteo1} and  \ref{Matteo2} as a way to strengthen the result in Theorem \ref{main5}, Michael Johnson for his helpful conversations and suggestions related to this work, Martin Szyld for helpful conversations in regard to the universal property of the bicategory of fractions, David Roberts for pointing us to related work by him and others, and the referee for an extremely careful reading of an earlier version of this manuscript, leading us to tighten some of the statements and the proofs.


\section{Weaker Conditions for a Bicalculus of Fractions}\label{S:background}

In the first part of this section we introduce the new conditions on a class of arrows in a bicategory that will  give rise to a bicalculus of fractions. These are a weakening of the conditions {\bf BF1--BF5} given in \cite{Pr-comp}. In the second part of this section we develop general results about the structure of the  2-cells in a bicategory with a class of arrows satisfying our new conditions. 

\subsection{The New Conditions} \label{WB-conditions} We list our new conditions on a class of arrows.  
In Section~\ref{newBF2} we will show that these are sufficient for the existence of the bicategory of fractions,  although the specific construction of this bicategory needs to be changed.    
\begin{itemize}
\item {\bf [WB1]} All identities are in $\frakW$.
\item  {\bf [WB2]} For each pair of composable arrows $\xymatrix@1{B\ar[r]^v&C\ar[r]^w& D}$ 
in $\frakW$, there is an arrow $\xymatrix@1{A\ar[r]^u&B}$ such that $\xymatrix@1{A\ar[r]^{wvu}&D}$ is in $\frakW$.
\item  {\bf [WB3]} For every pair $w\colon A\,\to\,B$, $f\colon C\,\to\,B$ with $w\,\in\,\frakW$, 
there exist maps $h, v$,   where $v \in \frakW$, and  an invertible 2-cell $\alpha$ as in the following diagram.
$$
\xymatrix{
D \ar[r]^h \ar[d]_{v} \ar@{}[dr]|{\stackrel{\alpha}{\Longleftarrow}}&A\ar[d]^w
\\
C \ar[r]_f & B
}
$$  
\item  {\bf [WB4]} For any 2-cell $$\alpha\colon w\,\circ\,f\,\Rightarrow\,w\,\circ\,g$$ with $w\,\in\, \frakW$, there exists an arrow  $u\,\in\, \frakW$ and a 2-cell
$$\beta\colon f\,\circ\,u\,\Rightarrow\,g\,\circ\,u$$ such that 
$\alpha\,\circ\,u\,=\,w\,\circ\,\beta$. Furthermore, the collection of such pairs  $(u,\beta)$ has the following property:  
when $(u_1,\beta_1)$ and $(u_2,\beta_2)$ are two such pairs, 
there exist arrows $s,t$, such that $u_1\circ s$ and $u_2\circ t$ are in $\frakW$,  and there is
an invertible 2-cell $\varepsilon\colon u_1\,\circ\,s\,\Rightarrow\,u_2\,\circ\,t$ such that the following diagram commutes:
$$
\xymatrix{
f\circ u_1\circ s\ar[rr]^{\beta_1\circ s} \ar[d]_{f\circ\varepsilon}&&g\circ u_1\circ s\ar[d]^{g\circ\varepsilon}\\
f\circ u_2\circ t\ar[rr]_{\beta_2\circ t}&&g\circ u_2\circ t.
}
$$
\item  {\bf [WB5]} When $w\in \frakW$ and there is an invertible 2-cell $\alpha\colon v\Rightarrow w$, then $v\in \frakW$.
\end{itemize}

\begin{rmks} \label{r:bvswb}
\begin{enumerate}
\item The original condition {\bf  BF1} stated that all equivalences were in the class $\frakW$.  It is well-known that it is sufficient to replace this with the given {\bf [WB1]}; see for instance, \cite{Matteo1}.
\item Condition {\bf [WB2]} is a significantly weaker version of the original condition {\bf BF2}, which required that 
$\frakW$ be closed under composition. 

\item Conditions {\bf [WB3]} and {\bf [WB5]} are the same as the old conditions {\bf BF3} and {\bf BF5} respectively.
\item
When $\alpha$ and $\beta$ are 2-cells as in condition {\bf [WB4]}, 
we will refer to $\beta$ as a {\em lifting of $\alpha$ with respect to $w$}.
In \cite{Pr-comp}, condition {\bf BF4} additionally required that if $\alpha$ is invertible, it has a lifting $\beta$ that is invertible.   We will show in Proposition~\ref{isolifting}  that  this assumption is not needed, as  it can be derived from the other assumptions.  
\end{enumerate}
\end{rmks}

\subsection{Properties of Liftings of 2-Cells}\label{D:add_2_cells} 
In this section we  prove that our condition {\bf [WB4]}, together with the conditions {\bf [WB1]--[WB3]} and {\bf [WB5]}, imply the original condition {\bf BF4}.  To do this, we develop some properties of the 2-cell liftings that {\bf [WB4]} requires,  and show that they can be chosen to respect composition.

We assume throughout this section that $\frakW$ is a class of arrows satisfying conditions {\bf [WB1]}-{\bf [WB5]}.  
We begin by showing that for fixed $w \in \frakW$, the collection of the liftings of cells given by {\bf [WB4]} inherits the vertical composition structure in the sense that the vertical composition of two liftings gives a lifting for the vertical composition of the original cells.

\begin{lma}\label{lifting1}
Let $\frakW$ satisfy {\bf [WB1]}--{\bf [WB5]}.
Suppose that we have arrows
$$
\xymatrix{
B\ar@<1ex>[r]^f \ar[r]|g \ar@<-1ex>[r]_h & C\ar[r]^w&D
}
$$
with $w\in \frakW$, and let $\alpha_1\colon wf\Rightarrow wg$ and $\alpha_2\colon wg\Rightarrow wh$
be 2-cells. 
Then there exists an arrow $u\colon A\to B$  in $\frakW$ 
with 2-cells $\beta_1 \colon fu\Rightarrow gu$ and $\beta_2 \colon gu\Rightarrow hu$
such that $w\beta_1=\alpha_1 u$ and $w\beta_2=\alpha_2 u$.
It follows that $w(\beta_2\cdot\beta_1)=(\alpha_2\cdot\alpha_1)u$.
\end{lma}

\begin{proof}
We begin by choosing two arbitrary arrows and cells as in condition {\bf [WB4]}:
let $u_1\colon A_1\to B$ and $u_2\colon A_2\to B$ be two arrows in $\frakW$ 
with 2-cells $\gamma_1\colon fu_1\Rightarrow gu_1$ and $\gamma_2\colon gu_2\Rightarrow hu_2$
such that $w\gamma_1=\alpha_1u_1$ and $w\gamma_2=\alpha_2u_2$.

Since $u_1$ and $u_2$ are in $\frakW$,  condition {\bf [WB3]} gives us a square 
$$
\xymatrix{
A_3\ar[r]^{s}\ar[d]_{t} \ar@{}[dr]|{\stackrel{\stackrel{\sim}{\Longleftarrow}}{\zeta}} & A_1\ar[d]^{u_1}
\\
A_2\ar[r]_{u_2} & B,
}
$$
with $t\in \frakW$ and $\zeta$ an invertible 2-cell. By Condition {\bf [WB2]}, there is an arrow $v\colon A\to A_3$ such that the composition $u_2tv$ is in $\frakW$, and hence by Condition {\bf [WB5]} and the invertibility of $\zeta$,  $u_1sv\in\frakW$ also.

We claim that the following arrow and 2-cells satisfy the conditions of this lemma:
$u=u_1sv$, $\beta_1=\gamma_1 sv$ and $\beta_2=((h\zeta^{-1})\cdot (\gamma_2 t)\cdot (g\zeta))\circ v$,
as in the diagram,
$$
\xymatrix{&&A\ar[d]^v
\\
 && A_3\ar[dll]_{s}\ar[d]^t\ar[drr]^s
\\
A_1\ar[dr]_{u_1} & \zeta\stackrel{\sim}{\Rightarrow} & A_2\ar[dl]_{u_2}\ar[dr]^{u_2} 
    & \zeta^{-1}\stackrel{\sim}{\Rightarrow} & A_1\ar[dl]^{u_1}
\\
&B\ar[dr]_g&\stackrel{\gamma_2}{\Rightarrow} & B\ar[dl]^h
\\
&&C\rlap{\quad.}
}$$

To prove this claim, first note that since $\gamma_1$ was chosen to satisfy {\bf [WB4]},  $w\beta_1=w\gamma_1 sv=\alpha_1 u_1sv=\alpha_1 u$.
Now using the fact that $\gamma_2$ was also chosen so that  $ w \gamma_2 = \alpha_2 u_2$, we calculate $w\beta_2$ in the following diagrams:
$$
\xymatrix@C=1.5em{
&&A\ar[d]^v &&&&&A\ar[d]^v
\\
 && A_3\ar[dll]_{s}\ar[d]^t\ar[drr]^s && & && A_3\ar[dll]_{s}\ar[d]^t\ar[drr]^s&&
\\
A_1\ar[dr]_{u_1} & \zeta\stackrel{\sim}{\Rightarrow} & A_2\ar[dl]_{u_2}\ar[dr]^{u_2} 
    & \zeta^{-1}\stackrel{\sim}{\Rightarrow} & A_1\ar[dl]^{u_1} & A_1\ar[drr]_{u_1} & \zeta\stackrel{\sim}{\Rightarrow} 
    & A_2\ar[d]^{u_2} & \zeta^{-1}\stackrel{\sim}{\Rightarrow} & A_1 \ar[dll]^{u_1}
\\
&B\ar[dr]_g&\stackrel{\gamma_2}{\Rightarrow} & B\ar[dl]^h & \ar@{}[r]|=&  && B\ar[dl]_g\ar[dr]^h
\\
&&C\ar[d]^w &&&&C\ar[dr]_w & \stackrel{\alpha_2}{\Rightarrow} & C\ar[dl]^w
\\
&&D && & && D
}$$
and this is clearly equal to $\alpha_2 u_1sv=\alpha_2u$, as required. 
\end{proof}

We now use this lemma to  prove that whenever the 2-cell $\alpha\colon wf\Rightarrow wg$ is invertible, there is at least one choice of a pair $(u,\beta)$ for   {\bf [WB4]} such that $\beta$ is also invertible. 

\begin{prop}\label{isolifting}
Let $\frakW$ satisfy the conditions {\bf [WB1]}--{\bf[WB5]}.
If $w\in \frakW$ and $\alpha\colon wf\Rightarrow wg$ is an invertible 2-cell, then there is an arrow $u\in \frakW$ with 
an invertible 2-cell $\beta\colon fu\Rightarrow gu$ such that $w\beta=\alpha u$.
\end{prop}

\begin{proof}
We begin by applying Lemma~\ref{lifting1} to the case where $h=f$, $\alpha_1=\alpha$ and $\alpha_2=\alpha^{-1}$.
This gives us an arrow $v\in \frakW$ and 2-cells $\gamma\colon fv\Rightarrow gv$ and $\gamma'\colon gv\Rightarrow fv$
such that $w\gamma=\alpha v$ and $w\gamma'=\alpha^{-1} v$.
So $w(\gamma'\cdot\gamma)=(\alpha^{-1}\cdot\alpha)v=\mbox{id}_{wf}v$.  This does not guarantee that  $\gamma$ and $ \gamma'$ are inverses, but  we will show that there is a further lifting $v'$ such that $vv' \in \frakW$ and $\gamma v'$ and $\gamma' v'$ are inverses.   

We create $v'$ in two stages.  First we find $u_1$ such that $(\gamma' u_1)(\gamma u_1) = \mbox{id}_{fu_1}$, and then we find $w_1$ such that $(\gamma u_1 w_1) (\gamma' u_1 w_1) = \mbox{id}_{fu_1w_1}$.  
To find $u_1$, we observe that both $w(\gamma' \gamma) = \mbox{id}_{wf} v$ and $w \circ \mbox{id}_{fv} = \mbox{id}_{wf} v$.  Thus,  $(v,\gamma'\cdot \gamma)$ and $(v,\mbox{id}_{fv})$ are both pairs of liftings of $\mbox{id}_{wf}$ with respect to $w$ as in {\bf [WB4]}.  The second half of {\bf [WB4]} gives a relationship between any two such pairs, so applying that here gives two maps, $u_1$ and $u_2$, and an 
 invertible 2-cell, 
$$
\xymatrix{
\ar[r]^{u_1}\ar@{}[dr]|{\delta\Downarrow\cong} \ar[d]_{u_2} & \ar[d]^v
\\
\ar[r]_v & 
}$$
with $vu_i\in\mathfrak{W}$ and  
such that
$$
\xymatrix{
& \ar[dl]_{u_1}\ar[d]^{u_2}\ar[dr]^{u_1} &&&& \ar[dl]_{u_1}\ar[d]^{u_1}\ar[dr]^{u_1} &
\\
\ar[d]_v\ar@{}[r]|{\stackrel{\delta}{\Rightarrow}} & \ar[dl]^v\ar[dr]_v\ar@{}[r]|{\stackrel{\delta^{-1}}{\Rightarrow}} & \ar[d]^v\ar@{}[drr]|= &&\ar[d]_v\ar@{}[dr]|{\stackrel{\gamma u_1}{\Rightarrow}} &\ar[d]|v \ar@{}[dr]|{\stackrel{\gamma' u_1}{\Rightarrow}}  &\ar[d]^v
\\
\ar[dr]_f\ar@{}[rr]|{\mbox{\scriptsize id}_{fv}} &&\ar[dl]^f&&\ar[dr]_f & \ar[d]^g & \ar[dl]^f
\\
&&&&&&
}$$
The left-hand side of this equation is equal to the identity 2-cell, $\mbox{id}_{fvu_1}$, so $\gamma'u_1\cdot\gamma u_1=\mbox{id}_{fvu_1}$.

Now we create $w_1$ via the same argument applied to the 2-cells  $\gamma u_1\cdot\gamma' u_1$ and $\mbox{id}_{gvu_1}$.
We know that  $w(\gamma u_1\cdot\gamma' u_1)=(\alpha\cdot\alpha^{-1})vu_1=\mbox{id}_{wg}vu_1=\mbox{id}_{wgvu_1}=w \mbox{id}_{gvu_1}$.
So both $(vu_1,\gamma u_1\cdot\gamma' u_1)$ and $(vu_1,\mbox{id}_{gvu_1})$ 
are liftings of $\mbox{id}_{wg}$ with respect to $w$, and applying the second half of {\bf [WB4]} as above gives us $w_1, w_2$ and an invertible 2-cell $\epsilon$ such that $vu_1w_i\in\mathfrak{W}$ and
$\gamma u_1w_1\cdot\gamma' u_1w_1=\mbox{id}_{gvu_1w_1}$.  
We conclude that $\gamma' u_1w_1=(\gamma u_1w_1)^{-1}$.  Therefore setting $v' = u_1 w_1$, $u = v v' = vu_1w_1$ and $\beta = \gamma u_1w_1$ satisfies the requirements of the proposition.  
\end{proof}

\begin{rmk}
Combining the proofs for Proposition~\ref{isolifting} and Lemma~\ref{lifting1} shows that if $\alpha$ in 
Proposition~\ref{isolifting}  is invertible, for any arrow $u\in\frakW$ with 2-cell
$\beta\colon fu\Rightarrow gu$ such that $w\beta=\alpha u$,
there is an arrow $s$ such that $\beta\circ s$ is invertible. 
\end{rmk}

The following result concerning cancellability of arrows in $\frakW$ was communicated to us by Matteo Tommasini \cite{Matteo0}.

\begin{lma}\label{Matteo1}
Let $\frakW$ satisfy the conditions {\bf[WB1]}--{\bf[WB4]}.
For any diagram $$\xymatrix@C=5em{C\ar@<1ex>[r]^f\ar@<-1ex>[r]_g\ar@{}[r]|{\Downarrow\beta_1\quad\Downarrow\beta_2} &B\ar[r]^w &A}$$ with $w\in\frakW$, if 
$w\beta_1=w\beta_2$ then there exists an arrow $v\colon D\to C$ in $\frakW$ such that $\beta_1v=\beta_2v$.
\end{lma}

\begin{proof}
Apply the second part of {\bf [WB4]} to $\alpha:=w\beta_1=w\beta_2$, $u_1=u_2:=\mbox{id}_C$ and the 2-cells $\beta_1$ and $\beta_2$ as given (for simplicity we omit the structure cells from the bicategory in this calculation).
This gives us the existence of arrows $v,v'\colon D\rightrightarrows C$ such that $\mbox{id}_Cv,\mbox{id}_Cv'\in \frakW$, and hence $v,v'\in \frakW$ by {\bf [WB5]}, with an invertible 2-cell $\varepsilon\colon v\Rightarrow v'$ such that $\beta_1\circ \varepsilon=\beta_2\circ\varepsilon$.
Composing with $g \varepsilon^{-1}$ gives us that $\beta_1v=\beta_2v$ with $v\in\frakW$ as required.
\end{proof}

\subsection{Squares as in Condition {\bf [WB3]}}
In this section we address a question related to condition {\bf [WB3]}: 
if there are  two squares as in {\bf [WB3]} for the same cospan, how are these squares related to each other?
This question was answered in the proof of Lemma A.1.1 in \cite{Pr-comp} for cospans where  both arrows are in $\frakW$.
Here, we prove a more general result, for cospans with just one arrow in $\frakW$ and  assuming only the weaker condition {\bf [WB2]}.  This result will play a crucial role in the constructions of whiskering of 2-cells with arrows in the bicategory of fractions and in the construction of the associativity isomorphisms.
It will also be used in the study of the equivalence relation on the 2-cells diagrams.

\begin{prop}\label{anytwosquares}
 For $w\colon A\to B$ in $\frakW$ and $f\colon C\to B$ any  arrow in $\calB$, 
and any two squares,
$$
\xymatrix{
D_1\ar[d]_{v_1}\ar[r]^{g_1} \ar@{}[dr]|{\alpha_1\stackrel{\sim}{\Leftarrow}} 
    & A\ar[d]^w & D_2\ar[d]_{v_2}\ar[r]^{g_2} \ar@{}[dr]|{\alpha_2\stackrel{\sim}{\Leftarrow}} & A\ar[d]^w
\\
C\ar[d]_u\ar[r]_f &B & C\ar[d]_u\ar[r]_f &B
\\
X&&X
}$$
where $u, uv_1$ and $uv_2$ are all in $\frakW$,
then there are arrows $s_1$ and $s_2$ and invertible 2-cells $\beta$ and $\gamma$ as in
$$
\xymatrix{
& D_1\ar[dl]_{v_1}\ar[dr]^{g_1} & 
\\
C\ar@{}[r]|{\wr\Downarrow\beta} &E \ar[u]_{s_1}\ar[d]^{s_2}\ar@{}[r]|{\wr\Downarrow\gamma} & A
\\
&D_2\ar[ul]^{v_2}\ar[ur]_{g_2}
}
$$
such that $uv_1s_1\in \frakW$, and the composites $(f\beta)\cdot(\alpha_1s_1)$ and $(\alpha_2 s_2)\cdot(w\gamma)$ are equal:
$$
\xymatrix{\ar[r]^{s_1}\ar[d]_{s_2}\ar@{}[dr]|{\stackrel{\beta}{\Leftarrow}} & \ar[r]^{g_1}\ar[d]|{v_1}\ar@{}[dr]|{\stackrel{\alpha_1}{\Leftarrow}} & \ar[d]^w\ar@{}[drr]|{\textstyle\equiv} && \ar[d]_{s_2}\ar[r]^{s_1}\ar@{}[dr]|{\stackrel{\gamma}{\Leftarrow}} & \ar[d]^{g_1}
\\
\ar[r]_{v_2} & \ar[r]_{f} & && \ar[r]|{g_2}\ar@{}[dr]|{\stackrel{\alpha_2}{\Leftarrow}}\ar[d]_{v_2} & \ar[d]^w
\\
&&&&\ar[r]_f&}
$$
\end{prop}

\begin{proof}
 Since $uv_1$ is in $\frakW$,  condition {\bf [WB3]} gives us a square
$$
\xymatrix{
F\ar[d]_{\overline{v}_1}\ar[r]^{\overline{v}_2}\ar@{}[dr]|{\stackrel{\sim}{\Longleftarrow}\beta'} & D_1\ar[d]^{uv_1}
\\
D_2\ar[r]_{uv_2} & X
}$$
with $\overline{v}_1\in\frakW$. Applying Proposition~\ref{isolifting} to the 2-cell $\beta'\colon  u v_1 \overline{v}_2 \Rightarrow uv_2 \overline{v}_1$, we get an  arrow $\tilde{u}\colon F'\to F$ in $\frakW$ and an invertible 2-cell $\tilde{\beta}'\colon v_1 \overline{v}_2 \tilde{u} \Rightarrow  v_2 \overline{v}_1 \tilde{u}  $.   

Then we have the following invertible 2-cell from $wg_1\overline{v}_2\tilde{u}$ to $wg_2\overline{v}_1\tilde{u}$.
$$
\xymatrix{
&D_1\ar[r]^{g_1}\ar[dr]_{v_1} &A \ar@{}[d]|{\alpha_1\, \Downarrow}\ar[dr]^w
\\
F'\ar[ur]^{\overline{v}_2\tilde{u}}\ar@{}[rr]|{\Downarrow \, \tilde{\beta}'} \ar[dr]_{\overline{v}_1\tilde{u}} && C\ar[r]^f\ar@{}[d]|{\alpha_2^{-1}\, \Downarrow} &B
\\
&D_2\ar[ur]^{v_2}\ar[r]_{g_2}&A\ar[ur]_w&
}
$$
By applying Proposition~\ref{isolifting} with respect to $w$, there is an arrow $\tilde{w}\colon F''\to F'$ in $\frakW$ with an invertible 2-cell $\gamma'\colon g_1\overline{v}_2\tilde{u}\tilde{w}\Rightarrow g_2\overline{v}_1\tilde{u}\tilde{w}$
such that $w\gamma'$ is equal to the pasting of this last diagram composed with $\tilde{w}$.
Finally, by repeatedly applying condition {\bf [WB2]} to the string of composable $\frakW$ arrows  $uv_2, \overline{v}_1,\tilde{u},\tilde{w}$, there is an arrow $t\colon E\to F''$ such that $uv_2\overline{v}_1\tilde{u}\tilde{w}t\in\frakW$. By condition {\bf [WB5]} it follows that  $uv_1\overline{v}_2\tilde{u}\tilde{w}t\in\frakW$ as well.
The reader may verify that $s_1=\overline{v}_2\tilde{u}\tilde{w}t$, $s_2=\overline{v}_1\tilde{u}\tilde{w}t$, $\beta=\tilde{\beta}'\tilde{w}t$ and $\gamma=\gamma't$  
satisfy the conditions of this proposition.
\end{proof}

\begin{rmk} An extension of the result of Proposition~\ref{anytwosquares}, discussing how any two solutions to the problem of this proposition are related, can be found in Appendix~\ref{assoc1}, Proposition~\ref{unique-2-cell}.\end{rmk}


\section{The New Bicategory of Fractions Construction}\label{newBF2}
We will now show that the conditions  introduced in Section~\ref{WB-conditions}  are  sufficient to construct a bicategory of fractions $\calB(\frakW^{-1})$.  Given a bicategory $\calB$ and a class of arrows $\frakW$ which satisfies the conditions {{\bf [WB1]}}--{{\bf [WB5]}}, 
we first  describe the new bicategory $\calB(\frakW^{-1})$, and then show that it has the universal property of the bicategory of fractions.
The objects, arrows and 2-cells of $\calB(\frakW^{-1})$ are defined just as in \cite{Pr-comp}, but we will need to adjust the definition of composition and pasting.
We begin by reminding the reader of the definition as given in \cite{Pr-comp}.

\begin{itemize}
\item
{\it Objects} are the objects of $\calB$.

\item
{\it Arrows} are spans of the form $\xymatrix@1{&\ar[l]_w\ar[r]^f&}$ with $w\in\frakW$ and $f$ an arbitrary arrow in $\calB$.

\item
{\it 2-Cells} are equivalence classes of diagrams of the form 
\begin{equation}\label{first2-cell}
\xymatrix@C=4em{&C\ar[dl]_w\ar[dr]^f
\\
A\ar@{}[r]|{\Downarrow\alpha\cong} & D\ar[u]_u\ar[d]^{u'}\ar@{}[r]|{\Downarrow\beta} & B
\\
&C'\ar[ul]^{w'}\ar[ur]_{f'}\rlap{\quad,}
}
\end{equation}
where $wu$ is in $\frakW$ (and hence $w'u'$ is).
Such a diagram (\ref{first2-cell}) is \emph{equivalent} to another such diagram 
$$
\xymatrix@C=4em{&C\ar[dl]_w\ar[dr]^f
\\
A\ar@{}[r]|{\Downarrow\gamma\cong} & E\ar[u]_v\ar[d]^{v'}\ar@{}[r]|{\Downarrow\delta} & B
\\
&C'\ar[ul]^{w'}\ar[ur]_{f'}
}
$$
(with $wv$ in $\frakW$)
 if and only if there exists a diagram of the form
 $$
 \xymatrix{
 &C
 \ar@{}[d]|(.4){\stackrel{\sim}{\Rightarrow}}\ar@{}[d]|(.6){{\varepsilon}}
 \\
 D\ar[ur]^{u}\ar[dr]_{u'}&F\ar[l]_s\ar[r]^t \ar@{}[d]|(.4){\stackrel{\sim}{\Rightarrow}}\ar@{}[d]|(.6){{\varepsilon'}}&E\ar[dl]^{v'}\ar[ul]_{v}
 \\
 &C'
 }
 $$
with $wus\in \frakW$, such that
$$
\xymatrix{
&& C\ar@/^1ex/[dr]^f\ar@{}[dl]|{\stackrel{\varepsilon}{\Rightarrow}} &&&&&C\ar@/^1ex/[dr]^f
\\
D \ar@/^1.5ex/[urr]^u &\ar[l]^s F \ar[r]_t & E\ar[u]_v\ar[d]^{v'}\ar@{}[r]|{\Downarrow\delta} 
    &B &\equiv & E\ar@/_1.5ex/[drr]_{v'} &F\ar[l]_t\ar[r]^s\ar@{}[dr]|{\stackrel{\varepsilon'}{\Leftarrow}} 
    & D\ar[u]^u\ar[d]^{u'}\ar@{}[r]|{\Downarrow\beta} & B
\\
&&C'\ar@/_1ex/[ur]_{f'}&&&&&C'\ar@/_1ex/[ur]_{f'}
}
$$ 
and 
$$
\xymatrix{
&C\ar@/_1ex/[dl]_w && && & C\ar@/_1ex/[dl]_w \ar@{}[dr]|{\stackrel{\varepsilon}{\Leftarrow}}
\\
A\ar@{}[r]|{\Downarrow\alpha} &D\ar[u]_u\ar[d]_{u'} & F\ar[l]_s\ar[r]^t\ar@{}[dl]|{\stackrel{\varepsilon'}{\Rightarrow}} & \ar@/^1.5ex/[dll]^{v'} 
	&\equiv & A\ar@{}[r]|{\Downarrow\gamma}& \ar[u]^v\ar[d]^{v'}E &\ar[l]^t\ar[r]_s F & D\ar@/_1.5ex/[ull]_u
	\\
&  C'\ar@/^1ex/[ul]^{w'} &&&&&C'\ar@/^1ex/[ul]^{w'}\rlap{\quad.}
}
$$
  \end{itemize}

\begin{rmk}  In the description above, we consistently only require half of our arrow compositions 
to be in $\frakW$.  For example, we require only that $wv \in \frakW$, and not the corresponding $w'v'$;  similarly we only require $wus \in \frakW$.   However, since the $2$-cells are invertible and $\frakW$ satisfies {\bf [WB5]}, the other half  follows automatically. 
\end{rmk}  

The original condition {\bf BF2} was used in \cite{Pr-comp} in the construction of composition of arrows and horizontal and vertical  
composition of 2-cells in the bicategory of fractions.  In constructing these compositions  under our weaker conditions, we need to adjust for the fact that $\frakW$ is no longer closed under composition.  Instead, we have the condition {\bf [WB2]} that allows us to pre-compose with an additional arrow to get a composition in $\frakW$.  The description of the compositions in \cite{Pr-comp} relies heavily on the choices of squares as in
condition {\bf [WB3]} and liftings as in condition {\bf [WB4]} (although,  in fact, the construction only depends on the choices of the squares when they are used to compose the spans, as Tommasini \cite{Matteo1} has shown that different choices made in the composition of 2-cells  give equivalent representatives).   In describing the compositions in the new bicategory of fractions, we use a collection of choices for arrows for composites as in {\bf [WB2]} to augment the choices of squares and liftings to make sure that the necessary arrows are in $\frakW$.  
We list and label these choices here before beginning the constructions so we can refer back to them.

\begin{no}  \label{choices}
The following choices of arrows and 2-cells will be used in the construction of the bicategory of fractions $\calB(\frakW^{-1})$. The first three choices really determine the construction. The last four are just short-cuts for frequently used combinations of the first three.

\begin{itemize}
    \item[{\bf[C1]}]   For each pair of composable arrows $\xymatrix{\ar[r]^v&\ar[r]^u&}$ in $\frakW$ use {\bf [WB2]} to choose an arrow $w_{u,v}$ such that $uvw_{u,v}\in\frakW$.  
When $v$ is an identity arrow, choose  $w_{u,v}$ to be an identity as well.
    \item[{\bf [C2]}]  For every pair $\xymatrix{\ar[r]^{f}&&\ar[l]_u}$ with $u \in \frakW$ use {\bf [WB3]} to choose a square 
$$
\xymatrix{
R\ar[d]_{u'} \ar[r]^{f'}\ar@{}[dr]|{\stackrel{\alpha}{\Longleftarrow}} &T\ar[d]^{u}\\
S \ar[r]_{f} &B
}
$$
with $u'\in\frakW$ and $\alpha$ invertible.
When we want to stress the dependence of $\alpha$ on $f$ and $u$, we denote this cell by $\alpha_{f,u}$. Furthermore, require that when $u=1_B$, we choose the square,
$$
\xymatrix{
A\ar[rr]^f\ar[drr]|f\ar[d]_{1_A} && B\ar@{}[dll]|(.3){\stackrel{\lambda_f}{\Longleftarrow}}|(.7){\stackrel{\rho_f^{-1}}{\Longleftarrow}}\ar[d]^{1_B}
\\
A\ar[rr]_f&&B}
$$
where $\lambda_f$ and $\rho_f$ are the left and right unitor 2-cell respectively.

\item[{\bf [C3]}] Given $\alpha\colon w\,\circ\,f\,\Rightarrow\,w\,\circ\,g$, a 2-cell
with $w\,\in\, \frakW$, choose  a 1-cell $\tilde{w}\,\in\, \frakW$  and a 2-cell
$$\tilde\alpha\colon f\,\circ\,u\,\Rightarrow\,g\,\circ\,u$$ such that 
$\alpha\,\circ\,\tilde w\,=\,w\,\circ\,\tilde\alpha$.   Using Proposition~\ref{isolifting}, we   choose $\tilde\alpha$ to be invertible whenever $\alpha$ is. 

\item[{\bf [C4]}] \label{triplezz} For each  zig-zag, $\xymatrix@1{&\ar[l]_w\ar[r]^f &&\ar[l]_v}$ with $v$ and $w$ in $\frakW$,  {\bf [C2]} determines arrows $f'$ and $v'$ and an invertible 2-cell $\alpha_{f, v}\colon vf'\Rightarrow fv'$.   Compose this with the choice  $w_{w,v'}$ from {\bf [C1]} to get $w v' w_{w,v'} \in \frakW$, to obtain the diagram
$$
\xymatrix{
&&\ar[d]^{w_{w,v'}}
\\
&&\ar[dl]_{v'}\ar[dr]^{f'}\ar@{}[dd]|{\stackrel{\alpha_{f,v}}{\Leftarrow}}
\\
&\ar[dl]_{w}\ar[dr]_f && \ar[dl]^v
\\
&&&\rlap{\qquad.}
}
$$   Defining $\overline{v} = v' w_{w,v'}, \overline{f} = f' w_{w,v'}$ and $\alpha^{w}_{f, v} = \alpha_{f, v }   w_{w,v'}$ gives the chosen diagram 
\begin{equation*}
\xymatrix{
&&\ar[dl]_{\overline{v}}\ar[dr]^{\overline{f}}\ar@{}[dd]|{\stackrel{\alpha_{f,v}^w}{\Leftarrow}}
\\
&\ar[dl]_w\ar[dr]_f &&\ar[dl]^v\\
&&&
}
\end{equation*}
with $w\overline{v}\in\frakW$.  Note that $\overline{v}$ is not guaranteed to be in $\frakW$, but $w\overline{v}$ {\em is} always in $\frakW$ by construction.

\item[{\bf [C5]}] \label{compinW} For each cospan $\xymatrix@1{\ar[r]^w&&\ar[l]_v}$ with both arrows   $w, v \in \frakW$, apply {\bf [C2]} to obtain a square with an invertible 2-cell $\alpha_{w,v}$.  Then compose with $w_{w,v'}$ from {\bf [C1]} to get $v'w_{w,v'} \in \frakW$.  Define $\hat{v} = v'w_{w,v'}, \hat{w} = w'w_{w,v'}$ and $\hat{\alpha}_{w, v} =  \alpha_{w, v}w_{w,v'}$  to obtain the chosen square 

\begin{equation*}
\xymatrix{
\ar[d]_{\hat{v}}\ar[r]^{\hat{w}}\ar@{}[dr]|{\stackrel{\hat\alpha_{w,v}}{\Longleftarrow}}&\ar[d]^v
\\
\ar[r]_w&
}
\end{equation*}
where $w\hat{v}\in\frakW$ and the 2-cell $\hat\alpha_{w,v}$ is invertible. 
 
 \item[{\bf [C6]}] \label{longlifting} For each invertible 2-cell $\alpha\colon w\circ s_1\Rightarrow w\circ s_2$ with $w, ws_1,ws_2\in\frakW$,  apply {\bf [C3]} to obtain  $\tilde{w} \in \frakW$ and $\tilde{\alpha} \colon  s_1 \tilde{w} \Rightarrow s_2 \tilde{w}$, with $\tilde{\alpha}$ invertible.  Then $ws_1$ and $\tilde{w}$ are in $\frakW$, so apply {\bf [C1]} to obtain an arrow $u$ such that $ws_1\tilde{w} u \in \frakW$.  Since $\tilde{\alpha}$ in {\bf [C3]} is invertible, we conclude that $w s_2 \tilde{w} u$ is also in $\frakW$.  Setting $\overline{\tilde{w}} = \tilde{w} u$, we get the chosen lifting  
\begin{equation*}\overline{\tilde\alpha}\colon s_1\overline{\tilde{w}}\Rightarrow s_2\overline{\tilde{w}}\end{equation*} 
 such that $ws_1\overline{\tilde{w}}\in\frakW$ and $\overline{\tilde\alpha}$ is invertible.

\item[{\bf [C7]}] \label{cleverlift}
For each configuration,
$$
\xymatrix@C=3.5em{
&\ar[dl]_u\ar[dr]^{wf}
\\
&\ar[u]^v\ar[d]_{v'}\ar@{}[r]|{\beta\Downarrow}&
\\
&\ar[ur]_{wf'}
}
$$
with $uv$ and $w$ in $\frakW$, {\bf [C3]} determines $\tilde{w}\in \frakW$ and $\tilde\beta\colon fv\tilde{w}\Rightarrow f'v' \tilde{w}$, and {\bf [C1]} determines an arrow $w_{\tilde{w},uv}$ with $uv\tilde{w}w_{\tilde{w},uv} \in \frakW$.  Now write $\overline{\tilde{w}}:= \tilde{w}w_{\tilde{w},uv}$ and precomposing $\tilde{\beta}$ by $w_{\tilde{w},uv}$ gives the chosen 2-cell $\overline{\tilde{\beta}}_u$ with  $uv\overline{\tilde{w}}\in\frakW$.
\begin{equation*}
\xymatrix@C=3.5em{
&\ar[dl]_u\ar[dr]^{f}
\\
&\ar[u]^{v\overline{\tilde{w}}}\ar[d]_{v'\overline{\tilde{w}}}\ar@{}[r]|{\overline{\tilde\beta}_u\Downarrow}&
\\
&\ar[ur]_{f'}
}
\end{equation*}

\end{itemize}
\end{no}

With these choices determined, we will now define the bicategory of fractions.

\noindent {\bf Composition of 1-Cells}
We define the composition of spans $\xymatrix@1{A&S\ar[l]_{u_1}\ar[r]^{f_1} &B}$ and $\xymatrix{B&T\ar[l]_{u_2}\ar[r]^{f_2} & C}$ in $\calB(\frakW^{-1})$ using the chosen square in {\bf [C4]} of Notation~\ref{choices},
$$
\xymatrix{
&&\ar[dl]_{\overline{u}_2}\ar[dr]^{\overline{f}_1}\ar@{}[dd]|(.4){\alpha_{f_1,u_2}^{u_1}}|(.6){\Leftarrow}
\\
&\ar[dl]_{u_1}\ar[dr]_{f_1} &&\ar[dl]^{u_2}\\
&&&
}
$$
so that $u_1\overline{u}_2\in\frakW$.
Then the composition of spans is given by
$$
\xymatrix{A&\ar[l]_{u_1\overline{u}_2}\ar[r]^{f_2\overline{f}_1} & C.}
$$

\begin{rmks}\label{2-cell-coh}
\begin{enumerate}
\item\label{canonical}
Proposition~\ref{anytwosquares} implies that any other choice of a square to define the composition results in an isomorphic arrow in $\calB(\frakW^{-1})$: Proposition~\ref{anytwosquares} gives a 2-cell between the two arrows in $\calB(\frakW^{-1})$ that is observed to be invertible in Remark \ref{vertcomprmks} Part \ref{rmk2}. Proposition~\ref{unique-2-cell} below further shows that the isomorphism is unique when certain properties with respect to the defining squares are required. So given the squares used to define the two ways to compose, there is a canonical invertible 2-cell between the two resulting compositions.
\item
Horizontal composition of 1-cells is clearly not associative in general. In Appendix~\ref{assoc1}, Proposition~\ref{assoc-2-cells} we introduce the family of associativity 2-cells and in Appendix~\ref{assoc2}, Proposition~\ref{P:coherence}, we show that this family satisfies the associativity coherence conditions.
The definition of the associativity cells is a direct generalization of the ones given in \cite{Pr-comp}, but the proof of coherence is a bit more involved. The appendices highlight the technical results that lead to coherence in separate propositions.
\end{enumerate}
\end{rmks}

\noindent {\bf Vertical Composition of 2-Cells}   
We define the vertical composition of 2-cell diagrams,
$$
\xymatrix@C=4em{
&\ar[dl]_{u_1}\ar[dr]^{f_1} & && &  \ar[dl]_{u_2}\ar[dr]^{f_2}&
\\
\ar@{}[r]|{\Downarrow\alpha_1} & \ar[u]|{v_1}\ar[d]|{v_2}\ar@{}[r]|{\Downarrow\beta_1} &\ar@{}[rr]|{ \mbox{and}}
		&& \ar@{}[r]|{\Downarrow\alpha_2} & \ar[u]|{v_3}\ar[d]|{v_4}\ar@{}[r]|{\Downarrow\beta_2} & 
\\
&\ar[ul]^{u_2}\ar[ur]_{f_2} &&&&\ar[ul]^{u_3}\ar[ur]_{f_3} &.
}
$$
First, since $u_2v_3$ and $u_2v_2$ are both in $\frakW$, let 
$$
\xymatrix@C=4em{
\ar[r]^{v_3'}\ar[d]_{v_2'}\ar@{}[dr]|{\stackrel{\delta}{\Leftarrow}} & \ar[d]^{u_2v_2}
\\
\ar[r]_{u_2v_3} & 
}
$$
be the chosen square in {\bf [C5]} of Notation~\ref{choices}: $\delta=\hat{\alpha}_{u_2v_3,u_2v_2}$ and  $u_2v_3v_2'\in\frakW$.  Since   $\delta$ is invertible,  $u_2v_2v_3'\in\frakW$ also.  

Next, apply {\bf [C6]} to  $\delta\colon  u_2v_2v_3' \Rightarrow u_2v_3v_2'$ and obtain an arrow $\overline{\tilde{u}}_2 \in \frakW $ and an invertible  2-cell $\overline{\tilde{\delta}}\colon  v_2v_3'\overline{\tilde{u}}_2 \Rightarrow v_3v_2'\overline{\tilde{u}}_2$.
Note that  $u_2 v_2v_3'\overline{\tilde{u}}_2\in\frakW$, as indicated in {\bf [C6]}. 

This gives us the following representative for the vertical composition,
\begin{equation}\label{verticalcomp}
\xymatrix@C=5em@R=3em{
&&\ar[ddll]_{u_1}\ar[ddrr]^{f_1}&&
\\
&&\ar[u]_{v_1} \ar@{}[dll]|{\alpha_1\Downarrow} \ar@{}[drr]|{\beta_1\Downarrow} \ar[dl]_{v_2}\ar[dr]^{v_2} &&
\\
&\ar[l]_{u_2}\ar@{}[r]|{\overline{\tilde{\delta}}\Downarrow} & \ar[u]|{v_3'\overline{\tilde{u}}_2}\ar@{}[r]|{\overline{\tilde{\delta}}\Downarrow} \ar[d]|{v_2'\overline{\tilde{u}}_2} & \ar[r]^{f_2} &
\\
&& \ar[ul]^{v_3}\ar@{}[ull]|{\alpha_2\Downarrow} \ar[d]^{v_4}\ar@{}[urr]|{\beta_2\Downarrow} \ar[ur]_{v_3} &&
\\
&&\ar[uull]^{u_3}\ar[uurr]_{f_3}
}
\end{equation}
Observe that $u_2v_2v_3'\overline{\tilde{u}}_2\in\frakW$ by construction,  and $u_1v_1v_3'\overline{\tilde{u}}_2 $ and $u_3v_4v_2'\overline{\tilde{u}}_2$ are in $\frakW$ since they are isomorphic to $u_2v_2v_3'\overline{\tilde{u}}_2$.  So this diagram represents a 2-cell from $\xymatrix@1{&\ar[l]_{u_1}\ar[r]^{f_1} &}$ to  $\xymatrix@1{&\ar[l]_{u_3}\ar[r]^{f_3} &}$.

\begin{rmks}\label{vertcomprmks}
\begin{enumerate}
    \item In Appendix C, Proposition~\ref{vertical-well-defined} we show that vertical composition is well-defined on equivalence classes of 2-cell diagrams and in Appendix A, Proposition~\ref{vert_associative} we show that it is strictly associative on equivalence classes of 2-cell diagrams.
    \item \label{rmk2} It is straightforward to check that when both the left- and the right-hand 2-cells in a 2-cell diagram 
    $$
    \xymatrix@C=4em{
    &\ar[dl]_{u_1}\ar[dr]^{f_1} &
\\
\ar@{}[r]|{\wr\Downarrow\alpha} & \ar[u]|{v_1}\ar[d]|{v_2}\ar@{}[r]|{\wr\Downarrow\beta} &
\\
&\ar[ul]^{u_2}\ar[ur]_{f_2} &}
    $$
    are vertically invertible in the original bicategory $\calB$ then the 2-cell in $\calB(\frakW^{-1})$ represented by this diagram is vertically invertible with inverse represented by
    $$
    \xymatrix@C=4em{
    &\ar[dl]_{u_2}\ar[dr]^{f_2} &
\\
\ar@{}[r]|{\wr\Downarrow\alpha^{-1}} & \ar[u]|{v_2}\ar[d]|{v_1}\ar@{}[r]|{\wr\Downarrow\beta^{-1}}&
\\
&\ar[ul]^{u_1}\ar[ur]_{f_1} &.}
    $$
\end{enumerate}

\end{rmks}

\noindent  {\bf Horizontal Composition of 2-Cells} The construction for horizontal composition in \cite{Pr-comp} is given in terms of whiskering on the left and the right. We will address the two cases in the following two subsections.

\subsection{Left Whiskering}\label{left_whiskering}
Suppose we have 
$$
\xymatrix@C=4em{
&\ar[dl]_{u_1}\ar[dr]^{f_1} &
\\
\ar@{}[r]|{\alpha\Downarrow} & \ar[u]_{s_1}\ar[d]^{s_2} \ar@{}[r]|{\beta\Downarrow} & &\ar[l]_v\ar[r]^g & 
\\
&\ar[ul]^{u_2}\ar[ur]_{f_2} &
}
$$ with $u_is_i \in \frakW$ and $\alpha$ invertible, so that the left side represents a 2-cell.  
We begin by constructing the composites of the arrows involved.   
This gives us the cells in the following diagram,
$$
\xymatrix@C=4em{
&\ar[dl]_{u_1}\ar[dr]^{f_1} &\ar[l]_{\overline{v}_1}\ar@{}[d]|{\stackrel{\gamma_1}{\Leftarrow}}\ar[dr]^{\overline{f}_1}
\\
\ar@{}[r]|{\alpha\Downarrow} & \ar[u]_{s_1}\ar[d]^{s_2} \ar@{}[r]|{\beta\Downarrow} & &\ar[l]_v\ar[r]^g & 
\\
&\ar[ul]^{u_2}\ar[ur]_{f_2} &\ar[l]^{\overline{v}_2}\ar@{}[u]|{\stackrel{\gamma_2}{\Leftarrow}}\ar[ur]_{\overline{f}_2}
}
$$
where $\gamma_1=\alpha_{f_1,v}^{u_1}$ and $\gamma_2=\alpha_{f_2,v}^{u_2}$ are the chosen squares of {\bf [C4]} of Notation~\ref{choices}.  (Note that this is not a pasting diagram.)
The next step is to construct squares that complete the cospans $\xymatrix@1{\ar[r]^{s_1} &&\ar[l]_{\overline{v}_1}}$ and $\xymatrix@1{\ar[r]^{s_2} &&\ar[l]_{\overline{v}_2}}$.
Neither $s_i$ nor $\overline{v}_i$ (where $i=1,2$) are necessarily in $\frakW$, but the $u_is_i$ are by assumption, and  the $u_i\overline{v}_i$ are by {\bf [C4]}.   
Now take the squares chosen in {\bf [C5]} for $i=1,2$, 
$$
\xymatrix@C=6em{
\ar[r]^{{s}_i'}\ar[d]_{{v}_i'} \ar@{}[dr]|{\stackrel{\hat{\alpha}_{u_is_i,u_i\overline{v}_i}}{\Leftarrow}} 
	& \ar[d]^{u_i\overline{v}_i}
\\
\ar[r]_{u_is_i} & 
}
$$
where the  composites $u_is_i{v}_i'$ are in $\frakW$ and the 2-cells $\hat{\alpha}_{u_is_i,u_i\overline{v}_i}$ are invertible.  
Now we have $\hat{\alpha}_{u_is_i,u_i\overline{v}_i}\colon  u_i \overline{v}_i s_i'  \Rightarrow u_i s_i v_i'$ where $u_i\in \frakW$, and hence  {\bf [C6]} determines arrows $\tilde{u_i} $ and 2-cells 
$\delta_i\colon \overline{v}_i s_i' \tilde{u}_i  \Rightarrow s_i v_i' \tilde{u}_i $.  If we write $v_i' \tilde{u}_i = \tilde{v}_i$ then we have $u_is_i{\tilde{v}}_i\in\frakW$ for $i=1,2$.

Finally, we want to construct a square to complete the cospan 
$\xymatrix@1{\ar[r]^{\tilde{v}_1}&&\ar[l]_{\tilde{v}_2}}$. 
Neither of the $\tilde{v}_i$ is necessarily in $\frakW$, but the
$u_is_i{\tilde{v}}_i$ are.
Also, since $\alpha\colon u_1s_1\Rightarrow u_2s_2$ is invertible, it follows that 
$u_1s_1\tilde{v}_2\in\frakW$.  Using a sequence of chosen squares and lifts as above, we construct a square
$$
\xymatrix@C=4em{
\ar[r]^{t_2}\ar[d]_{t_1}\ar@{}[dr]|{\stackrel{\delta_3}{\Leftarrow}} & \ar[d]^{{\tilde{v}}_2}
\\
\ar[r]_{{\tilde{v}}_1} &
}
$$
with $\delta_3$ invertible and  $u_1s_1{\tilde{v}}_1t_1\in\frakW$.

To find the right-hand 2-cell in the diagram representing the left whiskering, 
we want to apply a choice of lifting as in condition {\bf [WB4]} to 
 the following diagram, 
$$
\xymatrix@C=5em{
\ar@/^3ex/[2,3]^{\overline{f}_1} \ar[dr]^{\overline{v}_1} &&&&
\\
\ar[u]^{{s}_1'{\tilde{u}}_1} \ar@{}[r]|{{\delta}_1\Downarrow} \ar[dr]|{{\tilde{v}}_1} & \ar[dr]^{f_1} \ar@{}[drr]|{\stackrel{\gamma_1}{\Leftarrow}} &&
\\
\ar[u]^{t_1}\ar[d]_{t_2}\ar@{}[r]|{\delta_3^{-1}\Downarrow} & \ar[u]_{s_1}\ar[d]^{s_2}
\ar@{}[r]|{\beta\Downarrow} & &\ar[l]_v\ar[r]^g &
\\
\ar[ur]|{{\tilde{v}}_2}\ar@{}[r]|{\delta_2^{-1}\Downarrow} \ar[d]_{{s}_2'{\tilde{u}}_2} & \ar[ur]_{f_2} \ar@{}[urr]|{\stackrel{\gamma_2^{-1}}{\Rightarrow}}
\\
\ar[ur]_{\overline{v}_2}\ar@/_3ex/[-2,3]_{\overline{f}_2}
}
$$
and lift with respect to $v$.
However, we need to do this in such a way that we obtain a valid  2-cell diagram.  By construction, $\tilde{v}_1  = v_1' \tilde{u}_1$, and hence the 2-isomoprhism $v_1\delta_1^{-1}t_1:  u_1s_1\tilde{v}_1t_2 \Rightarrow u_1 \overline{v}_1 s_1' \tilde{u}_1 t_1$ ensures that $u_1\overline{v}_1 s_1' \tilde{u}_1 t_1 \in \frakW$.   This allows us to apply  
 {\bf [C7]} to get a lifing  for the diagram
$$
\xymatrix@C=5em{
&&\ar@/_2ex/[2,-2]_{u_1\overline{v}_1}\ar@/^3ex/[2,3]^{\overline{f}_1} \ar[dr]^{\overline{v}_1} &&&&
\\
&&\ar[u]^{{s}_1'{\tilde{u}}_1} \ar@{}[r]|{{\delta}_1\Downarrow} \ar[dr]|{{\tilde{v}}_1} & \ar[dr]^{f_1} \ar@{}[drr]|{\stackrel{\gamma_1}{\Leftarrow}} &&
\\
&&\ar[u]^{t_1}\ar[d]_{t_2}\ar@{}[r]|{\delta_3^{-1}\Downarrow} & \ar[u]_{s_1}\ar[d]^{s_2}\ar@{}[r]|{\beta\Downarrow} & &\ar[l]_v &
\\
&&\ar[ur]|{{\tilde{v}}_2}\ar@{}[r]|{\delta_2^{-1}\Downarrow} \ar[d]_{{s}_2'{\tilde{u}}_2} & \ar[ur]_{f_2} \ar@{}[urr]|{\stackrel{\gamma_2^{-1}}{\Rightarrow}}
\\
&&\ar[ur]_{\overline{v}_2}\ar@/_3ex/[-2,3]_{\overline{f}_2}
}
$$
This gives us an arrow $\tilde{v}$ and a 2-cell 
$\tilde\beta\colon \overline{f}_1{s}_1'{\tilde{u}}_1t_1\tilde{v}\Rightarrow \overline{f}_2{s}_2{\tilde{u}}_2t_2\tilde{v}$
such that $v\tilde\beta$ is equal to the pasting of the previous diagram composed with $\tilde{v}$,  and $u_1\overline{v}_1s_1'{\tilde{u}}_1t_1\tilde{v}\in\frakW$.

The resulting representative for the horizontal composition can be described by
\begin{equation}\label{leftwhiskeringgeneral}
\xymatrix@C=3.3em@R=1.5em{
&&&& \ar[dll]_{\overline{v}_1}\ar[2,4]^{\overline{f}_1} &&&&
\\
&&\ar[dll]_{u_1}\ar@{}[rr]|{\delta_1\Downarrow} && \ar[dl]_{\tilde{v}_1}\ar[u]_{{s}_1'{\tilde{u}}_1} &&&&
\\
&\ar@{}[rr]|{\alpha\Downarrow} &&\ar[ul]_{s_1}\ar[dl]^{s_2} \ar@{}[r]|(.55){\delta_3^{-1}\tilde{v}\Downarrow} & \ar[u]_{t_1\tilde{v}}\ar[d]^{t_2\tilde{v}}\ar@{}[rrrr]|{\tilde\beta\Downarrow} &&&&\ar[r]_g &
\\
&&\ar[ull]^{u_2} \ar@{}[rr]|{\delta_2^{-1}\Downarrow} && \ar[ul]^{\tilde{v}_2} \ar[d]^{{s}_2'{\tilde{u}}_2} &&&&
\\
&&&&\ar[ull]^{\overline{v}_2}\ar[-2,4]_{\overline{f}_2} 
}
\end{equation}

\subsection{Right Whiskering}\label{right_whiskering}
Consider a diagram 
$$
\xymatrix@C=4em{
&&&\ar[dl]_{v_1}\ar[dr]^{g_1}
\\
&\ar[l]_u\ar[r]^f & \ar@{}[r]|{\alpha\Downarrow} & \ar[u]_{s_1}\ar[d]^{s_2}\ar@{}[r]|{\beta\Downarrow} & 
\\
&&&\ar[ul]^{v_2}\ar[ur]_{g_2}
}
$$ with  $ v_1s_1$ and $v_2s_2$ in $\frakW$, and $\alpha$ invertible, so the right side represents a 2-cell.  
Again, we begin  by constructing the horizontal compositions of the arrows involved using the squares of {\bf [C4]} in Notation~\ref{choices} as in the following diagram, 
$$
\xymatrix@C=5em{
&&\ar[dl]_{\overline{v}_1}\ar[r]^{\overline{f}_1} \ar@{}[d]|{\stackrel{\gamma_1}{\Leftarrow}} & \ar[dl]|{v_1}\ar[dr]^{g_1}
\\
&\ar[l]_u \ar[r]^f & \ar@{}[r]|{\alpha\Downarrow} &\ar[u]_{s_1}\ar[d]^{s_2}\ar@{}[r]|{\beta\Downarrow} &
\\
&&\ar[ul]^{\overline{v}_2}\ar[r]_{\overline{f}_2} \ar@{}[u]|{\stackrel{\gamma_2}{\Leftarrow}} &\ar[ul]|{v_2}\ar[ur]_{g_2}
}
$$
where $\gamma_i=\alpha_{f,v_i}^{u}$ and $u\overline{v}_i\in\frakW$ for $i=1,2$. (Note that this is not a pasting diagram.)

Since $v_is_i\in\frakW$ for $i=1,2$ and $u\in\frakW$, we have chosen squares from {\bf [C4]} giving 
\begin{equation}\label{deltas}
\xymatrix{
\ar[r]^{f_i'}\ar[d]_{s'_i}\ar@{}[dr]|{\stackrel{\delta_i}{\Leftarrow}} &\ar[d]^{v_is_i}
\\
\ar[r]_f &
}
\end{equation}
with $us_i'\in\frakW$.
Now apply Proposition~\ref{anytwosquares} to the pairs of squares for $i=1,2$,
$$
\xymatrix@C=3.5em{
\ar[r]^{\overline{f}_i}\ar[d]_{\overline{v}_i}\ar@{}[dr]|{\stackrel{\gamma_i}{\Leftarrow}} &\ar[d]^{v_i}\ar@{}[drr]|{\mbox{and}} && \ar[r]^{s_if_i'}\ar[d]_{s_i'}\ar@{}[dr]|{\stackrel{\delta_i}{\Leftarrow}} & \ar[d]^{v_i}
\\
\ar[r]_f \ar[d]_u& &&\ar[r]_f\ar[d]_u &
\\
&&&&
}
$$
We obtain arrows and invertible 2-cells,
$$
\xymatrix@C=3em{
\ar[r]^{r_i}\ar[d]_{t_i}\ar@{}[dr]|{\stackrel{\varepsilon_i}{\Leftarrow}}& \ar[d]^{s'_i}&\ar@{}[d]|{\mbox{and}}&\ar[d]_{t_i}\ar[r]^{r_i}\ar@{}[drr]|{\stackrel{\varphi_i}{\Leftarrow}} & \ar[r]^{f_i'} & \ar[d]^{s_i}
\\
\ar[r]_{\overline{v}_i} &&&\ar[rr]_{\overline{f}_i}&&
}
$$
such that $u\overline{v}_it_i\in\frakW$ for $i=1,2$ and the composites of the following two pasting diagrams are equal:
$$
\xymatrix@C=3em{
\ar[r]^{r_i}\ar[dd]_{t_i}\ar@{}[ddr]|{\stackrel{\varepsilon_i}{\Leftarrow}} &\ar[dd]^{s'_i}\ar[r]^{f'_i}\ar@{}[ddr]|{\stackrel{\delta_i}{\Leftarrow}} & \ar[d]^{s_i} & &\ar[r]^{r_i}\ar[d]_{t_i} \ar@{}[drr]|{\stackrel{\varphi_i}{\Leftarrow}} & \ar[r]^{{f}_i'} & \ar[d]^{s_i}
\\
&&\ar[d]^{v_i}&\equiv&\ar[rr]_{\overline{f}_i}\ar[d]_{\overline{v}_i}\ar@{}[drr]|{\stackrel{\gamma_i}{\Leftarrow}} &&\ar[d]^{v_i}
\\
\ar[r]_{\overline{v}_i}&\ar[r]_{f}&&&\ar[rr]_f &&
}
$$

Now apply Proposition~\ref{anytwosquares} to the following two squares, where  $v_1s_1, u, us_1', us_2'  \in\frakW$:
$$\xymatrix{
\ar[rr]^{f_1'}\ar[dd]_{s_1'}\ar@{}[ddrr]|{\stackrel{\delta_1}{\Leftarrow}} &&\ar[d]^{s_1} &&\ar[rr]^{f_2'}\ar[dd]_{s_2'}\ar@{}[ddr]|{\stackrel{\delta_2}{\Leftarrow}} &&\ar[dl]_{s_2}\ar[d]^{s_1}
\\
&&\ar[d]^{v_1}&\mbox{and}&&\ar[dr]_{v_2}\ar@{}[r]|{\stackrel{\alpha}{\Leftarrow}} & \ar[d]^{v_1} 
\\
\ar[rr]_{f} \ar[d]^u &&&&\ar[rr]_{f}  \ar[d]^u &&
 \\ &&&&&& \\ 
 }
$$
This gives us arrows and invertible 2-cells
$$
\xymatrix{
\ar[r]^{q}\ar[d]_{p}\ar@{}[dr]|{\stackrel{\tilde\alpha}{\Leftarrow}} & \ar[d]^{s_2'}&\ar@{}[d]|{\mbox{and}} 
& \ar[r]^{q}\ar[d]_{p}\ar@{}[dr]|{\stackrel{\tau}{\Leftarrow}} &\ar[d]^{f_2'}
\\
\ar[r]_{s_1'} & && \ar[r]_{f_1'} &
}
$$
such that u$s_1'p\in\frakW$ and the following two pasting diagrams give the same composite:
$$
\xymatrix@C=4em{
\ar[r]^q\ar[d]_p\ar@{}[dr]|{\stackrel{\tilde\alpha}{\Leftarrow}} & \ar[d]|{s_2'}\ar@{}[dr]|{ \stackrel{\delta_2 \cdot(\alpha f_2')}{\Leftarrow}}\ar[r]^{f_2'}&\ar[d]^{v_1s_1}\ar@{}[drr]|\equiv && \ar[d]_p\ar[r]^q\ar@{}[dr]|{\stackrel{\tau}{\Leftarrow}} &\ar[d]^{f_2'}
\\
\ar[r]_{s_1'}& \ar[r]_f&&&\ar[d]_{s_1'}\ar@{}[dr]|{\stackrel{\delta_1}{\Leftarrow}} \ar[r]|{f_1'} &\ar[d]^{v_1s_1}
\\
&&&&\ar[r]_f&&
}
$$
Thus far we have constructed the following part of the left-hand cell of the whiskered 2-cell diagram,
$$
\xymatrix@C=3.5em@R=1.5em{
&&&\ar@/_3ex/[3,-2]_{\overline{v}_1}
\\
&&&\ar[u]_{t_1}\ar[dl]^{r_1}
\\
&&\ar[dl]|{s_1'}\ar@{}[u]|{\varepsilon_1^{-1}\Downarrow}
\\
&\ar[l]_u\ar@{}[r]|{\tilde\alpha^{-1}\Downarrow} &\ar[u]_p\ar[d]^q
\\
&&\ar[ul]|{s_2'}\ar@{}[d]|{\varepsilon_2\Downarrow}
\\
&&&\ar[ul]_{r_2}\ar[d]^{t_2}
\\
&&&\ar@/^3ex/[-3,-2]^{\overline{v}_2}
}$$
We fill in the gap in the middle by chosen liftings of chosen squares  according to conditions {\bf [WB3]} and {\bf [WB4]}. First note that the $u\overline{v}_it_i$ are in $\frakW$ for $i=1,2$, and hence since $\varepsilon_i$ is invertible,  $us_i'r_i\in\frakW$.
So we have squares from {\bf [C2]},
$$
\xymatrix@C=6em{
\ar[r]^{p'}\ar@{}[dr]|{\stackrel{\rho_1'}{\Leftarrow}}\ar[d]_{r_1'}&\ar[d]^{us_1'r_1}&\ar@{}[d]|{\mbox{and}}&\ar[r]^{q'}\ar@{}[dr]|{\stackrel{\rho_2'}{\Leftarrow}}\ar[d]_{r_2'}&\ar[d]^{us_2'r_2}
\\
\ar[r]_{us_1'p}&&&\ar[r]_{us_2'q}&
}
$$
and we lift with respect to $us_1'$ and $us_2'$  respectively (as in {\bf [C3]}) and add additional arrows $w_1$ and $w_2$ to obtain
arrows $\overline{r}_1=r_1'\tilde{u}_1w_1$ and $\overline{r}_2=r_2'\tilde{u}_2w_2$ both in $\frakW$. If we denote $\overline{p}=p'\tilde{u}_1w_1$  and $\overline{q}=q'\tilde{u}_2w_2$,  we obtain invertible 2-cells
$$
\xymatrix{
\ar[r]^{\overline{p}}\ar@{}[dr]|{\stackrel{\rho_1}{\Leftarrow}}\ar[d]_{\overline{r}_1}&\ar[d]^{r_1}&\ar@{}[d]|{\mbox{and}}&\ar[r]^{\overline{q}}\ar@{}[dr]|{\stackrel{\rho_2}{\Leftarrow}}\ar[d]_{\overline{r}_2}&\ar[d]^{r_2}
\\
\ar[r]_{p}&&&\ar[r]_{q}&
}
$$
Finally, we take a chosen square according to {\bf [C2]},
$$
\xymatrix{
\ar[r]^{\tilde{r}_1}\ar@{}[dr]|{\stackrel{\rho_3}{\Leftarrow}}\ar[d]_{\tilde{r}_2} & \ar[d]^{\overline{r}_2}
\\
\ar[r]_{\overline{r}_1}&
}
$$
with $\tilde{r}_2\in\frakW$.
Since $us_1'p,\overline{r}_1,\tilde{r}_2\in\frakW$, let $x$ be a chosen arrow such that $us_1'p\overline{r}_1\tilde{r}_2x\in\frakW$.
Then the result of the whiskering becomes:
\begin{equation}\label{rightwhiskeringgeneral}
\xymatrix@C=4.5em@R=2em{
&&&\ar@/_3ex/[3,-2]_{\overline{v}_1}\ar@/^1.5ex/[2,2]^{\overline{f}_1}&&&
\\
&&&\ar[u]_{t_1}\ar[dl]|{r_1}\ar[dr]|{r_1}&&
\\
&&\ar[dl]|{s_1'}\ar@{}[u]|{\varepsilon_1^{-1}\Downarrow}&\ar[u]^{\overline{p}}\ar@{}[l]|{\rho_1\Downarrow}\ar@{}[r]|{\rho_1\Downarrow}\ar[dr]|{\overline{r}_1}\ar[dl]|{\overline{r}_1}&\ar[dr]^{f_1'}\ar@{}[u]|{\varphi_1^{-1}\Downarrow}&\ar[dr]^{g_1}
\\
&\ar[l]_u\ar@{}[r]|{\tilde\alpha^{-1}\Downarrow} &\ar[u]_p\ar[d]^q&\ar@{}[l]|{\rho_3^{-1}x\Downarrow}\ar@{}[r]|{\rho_3^{-1}x\Downarrow}\ar[u]|{\tilde{r}_2x}\ar[d]|{\tilde{r}_1x}&\ar[u]_p\ar[d]^q\ar@{}[r]|{\tau^{-1}\Downarrow} &\ar[u]_{s_1}\ar[d]^{s_2}\ar@{}[r]|{\beta\Downarrow} &
\\
&&\ar[ul]|{s_2'}\ar@{}[d]|{\varepsilon_2\Downarrow}&\ar[d]^{\overline{q}}\ar@{}[l]|{\rho_2^{-1}\Downarrow}\ar@{}[r]|{\rho_2^{-1}\Downarrow}\ar[ur]|{\overline{r}_2}\ar[ul]|{\overline{r}_2}&\ar[ur]_{f_2'}\ar@{}[d]|{\varphi_2\Downarrow}&\ar[ur]_{g_2}
\\
&&&\ar[ul]|{r_2}\ar[d]^{t_2}\ar[ur]|{r_2}&&
\\
&&&\ar@/^3ex/[-3,-2]^{\overline{v}_2}\ar@/_1.5ex/[-2,2]_{\overline{f}_2}&&&
}
\end{equation}

\begin{rmks}\label{first-comparison}
\begin{enumerate}
\item \label{fc1}
When the class $\frakW$ of arrows to be inverted satisfies the traditional {\bf BF1--BF5} conditions from \cite{Pr-comp}, this construction 
reduces to the construction given in that paper when one takes the identity arrow whenever a choice of an arrow based on condition {\bf [WB2]} is needed.
The definition of horizontal whiskering here is not exactly the same as the one given in \cite{Pr-comp}, but the 2-cell diagrams obtained are equivalent. This is shown in   \cite{Matteo1}, where it is proved that  various choices to fill the 2-cell diagrams for whiskering
all result in equivalent 2-cell diagrams.
\item
The fact that the horizontal whiskering operations described here are well-defined on equivalence classes of 2-cell diagrams is established in Appendix \ref{well-defined}, Propositions \ref{wd-left-whiskering} and \ref{wd-right-whiskering}.
\end{enumerate}
\end{rmks}

With these definitions, we get the following:   
\begin{thm}\label{weakerfracns}
For any bicategory $\calB$ with a class $\frakW$ of arrows that satisfies conditions {\bf [WB1]}--{\bf[WB5]},  there is a bicategory of fractions $\calB(\frakW^{-1})$ with a homomorphism 
$$J_{\frakW}\colon\calB\to\calB(\frakW^{-1})$$ which sends arrows in $\frakW$ to internal equivalences.  Moreover, this bicategory satisfies the following universal property:   for any bicategory $\calD$, composition with $J_{\frakW}$ 
induces an equivalence of categories $$\Hom(\calB(\frakW^{-1}),\calD)\simeq\Hom_{\frakW}(\calB,\calD),$$  where $\Hom(\calB(\frakW^{-1}),\calD)$ denotes the category of homomorphisms and pseudo, resp. lax, resp. oplax, transformations and $\Hom_{\frakW}(\calB,\calD)$ denotes the subcategory of homomorphisms and pseudo, resp. lax, resp. oplax, transformations  that send arrows in $\frakW$ to internal equivalences.  \end{thm}

\begin{rmks}
\begin{enumerate}
\item
We can speak of transformations sending arrows to internal equivalences by representing them through a pseudo functor into an appropriate bicategory of cylinders on $\calD$ (depending on the type of transformations). For pseudo transformations, the calculus of mates shows that $\Hom_{\frakW}(\calB,\calD)$  is a full subcategory of $\Hom(\calB,\calD)$,
but for lax and oplax transformations this is not the case in general.
\item The universal property phrased in terms of the pseudo transformations determines the bicategory of fractions up to equivalence of bicategories. The other two universal properties are invariant under equivalence of bicategories. Hence we may view this result as saying that whenever a class of arrows admits a calculus of fractions, its bicategorical localization will also have these other two universal properties. 
\item 
The description of the bicategory of fractions given here depends on the choices made for arrows, squares and liftings used in composition. However, the universal property implies that any other choice would give a biequivalent bicategory of fractions. We actually have a stronger result here: we can give explicit pseudofunctors going back and forth that are the identity in all dimensions (objects, arrows and 2-cells), but don't preserve horizontal composition strictly: composition in the domain bicategory may have been defined using a different square from the one used in the codomain bicategory. In Proposition~\ref{unique-2-cell} we show that there is a canonical 2-cell between these two compositions. The property established in Proposition~\ref{unique-2-cell} implies that they satisfy the coherence conditions to form the structure cells of a pseudo functor. Furthermore,  these functors form commutative triangles with the $J_{\frakW}$ functors from $\calB$ into the bicategories of fractions. We will also see in the next section that these bicategories of fractions are biequivalent to a bicategory of fractions as defined in \cite{Pr-comp}.
\end{enumerate}
\end{rmks}

\begin{proof}
Analogous to the situation in \cite{Pr-comp}, we define $J_{\frakW}$ as follows: on objects $J_{\frakW}(A)=A$; on arrows $J_{\frakW}$ sends $A\stackrel{f}{\longrightarrow}B$ to $A\stackrel{1_A}{\longleftarrow}A\stackrel{f}{\longrightarrow}B$; on 2-cells, $J_{\frakW}$ sends 
$\xymatrix@1{A\ar@/^1.5ex/[r]^f\ar@/_1.5ex/[r]_g\ar@{}[r]|{\Downarrow\alpha}&B}$
to
$$
\xymatrix@C=4em{
&A\ar[dl]_{1_A}\ar[dr]^f
\\
A\ar@{}[r]|{\rho_A^{-1}\lambda_A\Downarrow}&A\ar[u]_{1_A}\ar[d]^{1_A}\ar@{}[r]|{\alpha\Downarrow}&B
\\
&A\ar[ul]^{1_A}\ar[ur]_g
}
$$
where $\rho_A$ and $\lambda_A$ are the right and left unitors respectively for $1_A$.
By the way we chose squares involving identity arrows, this gives a pseudo functor $\calB\to\calB(\frakW^{-1})$ with structure cells as defined in \cite{Pr-comp}. The remainder of the proof goes as in \cite{Pr-comp}.
We have given definitions for all of the composition operations in $\calB(\frakW^{-1})$ and shown them to be well-defined and suitably associative, sending arrows in $\frakW$ to internal equivalences.  
There are no coherence requirements on the choices of squares or liftings, so this gives 
a valid construction of a bicategory with all necessary properties.   

The resulting homomorphism of bicategories has the same universal properties as the one for the original bicategory of fractions, 
since the proof of \cite[Theorem 21]{Pr-comp} does not depend on any specific properties of the choices made.
\end{proof} 
 
A different way to derive this result will be given in Theorem~\ref{comparison-thm}.


\section{Equivalences of Bicategories of Fractions}\label{comparisons}

The first  goal of this paper was to provide conditions under which we can take smaller classes of arrows to invert, while still obtaining an equivalent bicategory of fractions.
   In this section we develop a condition to allow us to restrict to a smaller subclass of arrows, namely when a subclass is weakly initial in the original class of arrows. This is related to the  condition {\em WISC}, where we have    weakly initial subsets of the class of arrows to be inverted.   This was  considered in \cite{Roberts-fractions} to obtain a locally essentially small bicategory of fractions.   

We show that if we start with a class of arrows satisfying {\bf [WB1]}--{\bf[WB5]}, and we have a weakly initial subclass which satisfies {\bf [WB1]} and {\bf [WB5]}, then in fact the subclass will satisfy all the conditions{\bf [WB1]}--{\bf[WB5]} and the bicategory of fractions for the subclass is equivalent to the one for the original class of arrows.  
We will then apply this result to a class $\frakW$ of arrows satisfying {\bf [WB1]}--{\bf[WB5]}, and consider its closure under composition and invertible 2-cells, $\widehat{\frakW}$.
We show that $\widehat{\frakW}$ satisfies the conditions {\bf BF1}--{\bf BF5} of \cite{Pr-comp}, and that $\frakW$ is weakly initial in  $\widehat{\frakW}$.  This  gives an equivalence of bicategories
$$
\calB(\frakW^{-1})\simeq\calB(\widehat{\frakW}^{-1}),
$$
giving another proof that the newly constructed bicategories of fractions of Section~\ref{newBF2} are indeed equivalent to the ones introduced in \cite{Pr-comp}.
 
\subsection{Weakly Initial Subclasses}\label{D:covexam}
We begin by reminding the reader of the notion of a weakly initial subclass of arrows.  We will show that the new calculus of fractions conditions descend from a class to a weakly initial subclass.

\begin{dfn}\label{coveringclass}
Let $\frakW\subseteq \frakV$ be two classes of arrows in a bicategory $\calB$.
Then  $\frakW$  is {\em weakly initial} in $\frakV$
if for each arrow $v\in\frakV$, 
there is an arrow $u$ such that $vu\in\frakW$.
\end{dfn}

\begin{prop} \label{P:cover}
Let $\calB$ be a bicategory with a class of arrows $\frakV$ satisfying all the conditions {\bf [WB1]}--{\bf [WB5]},  and a subclass $\frakW\subseteq \frakV$ which is weakly initial in  $\frakV$ and satisfies conditions {\bf [WB1]} and {\bf [WB5]}.   Then $\frakW$ also satisfies conditions {\bf [WB2]}--{\bf [WB4]}.
\end{prop}

\begin{proof}
{\bf {\bf [WB2]}} \ Let $\xymatrix@1{A\ar[r]^{w_1}&B}$ and $\xymatrix@1{B\ar[r]^{w_2}&C}$ be a pair of composable arrows
in $\frakW$. Since $\frakW\subseteq \frakV$ and $\frakV$ satisfies condition {\bf [WB2]}, there is an arrow $u_1$ such that $w_2 w_1u_1\in\frakV$.  
Since $\frakW$ is weakly initial in  $\frakV$, there is an arrow $u_2$ such that $w_2w_1u_1u_2\in\frakW$. So $\frakW$ satisfies condition {\bf [WB2]}.

{\bf {\bf [WB3]}} Consider a cospan of arrows $\xymatrix@1{A\ar[r]^f & C & B\ar[l]_w}$ with $w\in\frakW$.
Since $\frakV$ satisfies {\bf [WB3]},  there is a square with an invertible 2-cell $\alpha$, 
$$
\xymatrix{
D\ar[d]_v\ar[r]^g \ar@{}[dr]|{\stackrel{\alpha}{\Longleftarrow}}& B\ar[d]^w
\\
A\ar[r]_f & C
}
$$
with $v\in\frakV$.
Since $\frakW$ is weakly initial in  $\frakV$, there is an arrow $\left(\xymatrix@1{E\ar[r]^u&D}\right)$ such that $vu\in\frakW$.   Then the square
$$
\xymatrix{
E\ar[d]_{vu}\ar[r]^{gu} \ar@{}[dr]|{\stackrel{\alpha u}{\Longleftarrow}}& B\ar[d]^w
\\
A\ar[r]_f & C}
$$
shows that $\frakW$ satisfies condition {\bf [WB3]}.

{\bf {\bf [WB4]}}  Let $\alpha\colon wf\Rightarrow wg$ be a 2-cell with $w\in\frakW$. Since $w\in\frakV$ and $\frakV$ satisfies {\bf [WB4]}, there is an arrow $v\in\frakV$ with a 2-cell $\beta\colon fv\Rightarrow gv$ such that $\alpha v=w\beta$.
And  since $\frakW$ is weakly initial in  $\frakV$, there is an arrow  $u$ such that $vu\in\frakW$.
Now take $w'=vu\in\frakW$ and $\beta'=\beta u$. Then $w\beta'=\alpha w'$.

To check that $\frakW$ also satisfies the second part of {\bf [WB4]}, let $(w_1', \beta_1) $ and $(w_2', \beta_2) $ be pairs such that $w_1', w_2' \in\frakW$, and $\beta_1\colon w_1'f\Rightarrow w_1'g$, $\beta_2\colon w_2'f\Rightarrow w_2'g$ such that $\alpha w_1'=w\beta_1$ and $\alpha w_2'=w\beta_2$. Since $w,w_1',w_2'\in\frakV$ and we assume that $\frakV$ satisfies {\bf [WB4]}, there are arrows $s,t$ such that $w_1's,w_2't\in\frakV$,  and an invertible 2-cell $\varepsilon\colon  w_1's\Rightarrow w_2't$ such that 
$$
\xymatrix{
fw_1's\ar[d]_{f\varepsilon}\ar[r]^{\beta_1s} &gw_1's\ar[d]^{g\varepsilon}
\\
fw_2't\ar[r]_{\beta_2t} & gw_2't
}
$$
commutes.
Since $w_1's\in\frakV$, there is an arrow $u$ such that $w_1'su\in\frakW$.
Then $w_2'tu\in\frakW$ as well, since $\varepsilon u\colon w_1'su \Rightarrow w_2'tu$ is an invertible 2-cell and   $\frakW$ is closed under invertible 2-cells by  condition {\bf [WB5]}.
So define $s'=su$, $t'=tu$, and $\varepsilon'=\varepsilon u\colon w_1's'\stackrel{\sim}{\Rightarrow}w_2't'$
to obtain a commutative diagram
$$
\xymatrix{
fw_1's'\ar[d]_{f\varepsilon}\ar[r]^{\beta_1s'} &gw_1's'\ar[d]^{g\varepsilon'}
\\
fw_2't'\ar[r]_{\beta_2t'} & gw_2't'
}
$$
as required.
\end{proof}

\begin{thm}\label{d:equivalence}
Let $\calB$ be a bicategory with a class of arrows $\frakV$ satisfying the conditions {\bf [WB1]}-{\bf [WB5]} and a class $\frakW\subseteq \frakV$ which is initial in $\frakV$ and satisfies {\bf [WB1]} and {\bf [WB5]}.   Then there is an equivalence of bicategories $J\colon \calB(\frakW^{-1})\to \calB(\frakV^{-1})$ that makes the following diagram commutative,
$$\xymatrix{
&\calB(\frakW^{-1})\ar[dd]^J_{\simeq}
\\
\calB\ar[ur]^{J_{\frakW}}\ar[dr]_{J_{\frakV}}&
\\
&\calB(\frakV^{-1}).}$$
\end{thm}

\begin{proof}
By the universal property of $\calB(\frakV^{-1})$ there is a canonical pseudo functor $$J\colon \calB(\frakW^{-1})\to\calB(\frakV^{-1}),$$ 
which is the identity on objects, sends the span $(w,f)$ in $\calB(\frakW^{-1})$
to the span $(w,f)$ in $\calB(\frakV^{-1})$ and maps the 2-cell represented by the diagram
$$
\xymatrix@C=3em{
&\ar[dl]_{w_1}\ar[dr]^{f_1}
\\
\ar@{}[r]|{\alpha\Downarrow}&\ar[u]_{u_1}\ar[d]^{u_2}\ar@{}[r]|{\beta\Downarrow}&
\\
&\ar[ul]^{w_2}\ar[ur]_{f_2}
}
$$ in $\calB(\frakW^{-1})$ to the 2-cell represented by this same diagram in $\calB(\frakV^{-1})$.
Note that $J$ sends identity arrows to identity arrows and the comparison cells for compositions of arrows are the canonical 2-cells related to the choices of squares for composition in 
$\calB(\frakW^{-1})$ and $\calB(\frakV^{-1})$, as described in Remark~\ref{2-cell-coh}.(\ref{canonical}). It is clear that $J\circ J_{\frakW}=J_{\frakV}$ as required.

It is obvious that $J$ is an isomorphism on objects. To show that it is essentially surjective on arrows, let
$$\xymatrix{A&C\ar[l]_v\ar[r]^f&B}$$
be an arrow in $\calB(\frakV^{-1})$.
Since $\frakW$ is weakly initial in $\frakV$, there is an arrow $\left(\xymatrix@1{D\ar[r]^u&C}\right)$ such that $vu\in\frakW$.
So the span
$$\xymatrix{A&D\ar[l]_{vu}\ar[r]^{fu}&B}$$
is in the image of $J$.
Furthermore, there is an invertible 2-cell
$$
\xymatrix{
&D\ar[dl]_{vu}\ar[dr]^{fu}
\\
A\ar@{}[r]|{\cong}&D\ar[u]_{1_D}\ar[d]^{u}\ar@{}[r]|{\cong} & B
\\
&C\ar[ul]^v\ar[ur]_f
}
$$ 
showing that $J$ is essentially surjective on arrows.

It remains to show that $J$ is fully faithful on 2-cells.
To show that it is full on 2-cells,
consider the 2-cell represented by the diagram,
\begin{equation}\label{full_on_2-cells}
\xymatrix@C=3em{
&\ar[dl]_{w_1}\ar[dr]^{f_1}
\\
\ar@{}[r]|{\alpha\Downarrow} & \ar[u]^{v_1}\ar[d]_{v_2}\ar@{}[r]|{\beta\Downarrow} &
\\
& \ar[ul]^{w_2}\ar[ur]_{f_2}
}
\end{equation}
with $w_1,w_2\in \frakW$ and $w_1v_1,w_2v_2\in\frakV$.
Since $\frakW$ is weakly initial in  $\frakV$, there is an arrow $u$ such that $w_1v_1u\in\frakW$.
Hence, the 2-cell represented by 
$$
\xymatrix@C=5em{
&\ar[dl]_{w_1}\ar[dr]^{f_1}
\\
\ar@{}[r]|{\alpha u\Downarrow} & \ar[u]^{v_1u}\ar[d]_{v_2u}\ar@{}[r]|{\beta u\Downarrow}&
\\
& \ar[ul]^{w_2}\ar[ur]_{f_2}
}
$$
is in the image of $J$.
This  diagram represents the same 2-cell as (\ref{full_on_2-cells}), since the following diagram with unitor 2-cells gives an equivalence between them:
$$
\xymatrix@C=5em{
&
\\
\ar[ur]^{v_1u}\ar[dr]_{v_2u} & \ar[l]_1\ar[r]^{u} \ar@{}[u]|{\cong}\ar@{}[d]|{\cong} & \ar[ul]_{v_1}\ar[dl]^{v_2}
\\
&
}
$$
Hence (\ref{full_on_2-cells}) is in the image of $J$ and we conclude that $J$ is full on 2-cells.

To verify that $J$ is faithful on 2-cells, consider two 2-cells between the same spans of arrows 
\begin{equation}\label{faithful_on_2-cells}
\xymatrix@C=4em{
&\ar[dl]_{w_1}\ar[dr]^{f_1} &&& & \ar[dl]_{w_1}\ar[dr]^{f_1}
\\
\ar@{}[r]|{\alpha\Downarrow} & \ar[u]^{v_1}\ar[d]_{v_2}\ar@{}[r]|{\beta\Downarrow} &&\mbox{and}&\ar@{}[r]|{\alpha'\Downarrow} & \ar[u]^{v'_1}\ar[d]_{v'_2}\ar@{}[r]|{\beta'\Downarrow} &
\\
& \ar[ul]^{w_2}\ar[ur]_{f_2}&&&&\ar[ul]^{w_2}\ar[ur]_{f_2}
}
\end{equation}
and suppose that these diagrams represent the same 2-cell in $\calB(\frakV^{-1})$.
This means that there is an equivalence given by  arrows $s$ and $t$ with 2-cells $\gamma_1$ and $\gamma_2$
as in
$$
\xymatrix@C=3em{
&
\\
\ar[ur]^{v_1}\ar[dr]_{v_2} & \ar[l]_s\ar[r]^{t} \ar@{}[u]|{\stackrel{\gamma_1}{\Rightarrow}}\ar@{}[d]|{\stackrel{\gamma_2}{\Rightarrow}} & \ar[ul]_{v'_1}\ar[dl]^{v'_2}
\\
&
}
$$
such that the appropriate diagrams of 2-cells commute and $w_1v_1s\in\frakV$.
Since $\frakW$ is weakly initial in  $\frakV$, there is an arrow $u$ such that $w_1v_1su\in\frakW$.
So the diagram 
 $$
\xymatrix@C=4em{
&
\\
\ar[ur]^{v_1}\ar[dr]_{v_2} & \ar[l]_{su}\ar[r]^{tu} \ar@{}[u]|{\stackrel{\gamma_1u}{\Longrightarrow}}\ar@{}[d]|{\stackrel{\gamma_2u}{\Longrightarrow}} & \ar[ul]_{v'_1}\ar[dl]^{v'_2}
\\
&
}
$$
represents an equivalence of the  diagrams in (\ref{faithful_on_2-cells}) 
in $\calB(\frakW^{-1})$.
We conclude that $J$ is fully faithful on 2-cells, and hence is a biequivalence of bicategories.
\end{proof}

\begin{rmk} \label{r:eq}
This theorem implies that the choices made in constructing the bicategory of fractions in Section~\ref{newBF2} do not matter, since $\frakW$ is weakly initial in itself, and Theorem~\ref{d:equivalence} provides an equivalence of bicategories created with different choices.
\end{rmk}

This result can be combined with the condition {\em WISC}  given in \cite{Roberts-TAC} to obtain the following. 

\begin{cor} If $\frakV$ has a weakly initial subset $\frakS_X$ over each object $X$, and these subsets contain identities and are closed under 2-isomorphism (conditions {\bf [WB1]} and {\bf [WB5]}), then the arrows in the weakly initial subsets define a locally small bicategory of fractions $\calB(\frakS^{-1})$, equivalent to $\calB(\frakV^{-1})$. 
\end{cor}

This strengthens the result in \cite{Roberts-TAC} where one would only get a locally essentially small bicategory of fractions.

\begin{rmk}
Our notion of a weakly initial class of arrows is a dual notion to that of the right saturation of a class of arrows defined in \cite{Matteo2}. The right saturation enlarges the class of arrows to be inverted, rather than restricting to a smaller subclass.   
  
The right saturation of a class $\calW$ of arrows consists of those arrows $f\colon C\to D$ for which there exist arrows $g\colon B\to C$ and $h\colon A\to B$
such that $gh$ and $fg$ are both in $\calW$. If $\calW$ satisfies the conditions {\bf BF1}-{\bf BF5}, then so does its saturation,  and the saturation gives rise to an equivalent bicategory of fractions.
It is not difficult to use {\bf [WB3]} to show  that if $\calW\subseteq \calV$ is weakly initial in  $\calV$, then 
$\calV$ is a subset of the saturation of $\calW$.
This does not immediately imply the equivalence of the induced bicategories of fractions, because $\calW$ may  not satisfy {\bf BF2}. However, Theorem~\ref{d:equivalence}
implies that the equivalences of bicategories of fractions in \cite{Matteo2}  apply when we replace {\bf BF2} with {\bf [WB2]}.  
\end{rmk}

\begin{rmk} In the case where one is only interested in obtaining a smaller version of $\calB(\frakV^{-1})(X,Y)$ for a particular object  $X$ (or for a particular class of objects) in the bicategory $\calB$, there is a local version of  Theorem~\ref{d:equivalence}.    Given an object $X$ in $\calB$ and a class of arrows $\frakV$ in $\calB$, we say that a subclass $\frakA\subseteq \frakV$ {\em is weakly initial in  $\frakV$ at $X$}
when the class $\frakA/X$ of arrows in $\frakA$ with codomain $X$ is weakly initial in the class $\frakV/X$ of arrows in $\frakV$ with codomain $X$.  We write $\calB_\frakA(X,Y)$ for the category for spans from $X$ to $Y$ with reverse arrows in $\frakA$ and 2-cells defined using 2-cell diagrams with the appropriate composites in $\frakA$.
Now, if $\frakV$ satisfies conditions {\bf [WB1]}--{\bf [WB5]},  $\frakA\subseteq \frakV$ satisfies condition {\bf [WB1]} and {\bf [WB5]},  and $\frakA$ is weakly initial in $\frakV$ at $X$, there is an equivalence of categories
$$\calB_\frakA(X,Y)\stackrel{\sim}{\rightarrow}\calB(\frakV^{-1})(X,Y),
$$
for any object $Y$ in $\calB$.
\end{rmk}

\subsection{Closure Under Composition}
Given a class of arrows $\frakW$ in a bicategory $\calB$, let  $\widehat{\frakW}$ denote the class obtained from $\frakW$ by closure under composition and   invertible 2-cells.  So  $\widehat{\frakW}$ is the smallest class of arrows in $\calB$ such that
\begin{itemize}
\item
$\frakW\subseteq \widehat{\frakW}$;
\item
If $f_1,f_2\in\widehat{\frakW}$, and $f_2\circ f_1$ is defined, then $f_2\circ f_1\in\widehat{\frakW}$;
\item 
If $f\in\widehat{\frakW}$ and $\alpha\colon f\stackrel{\sim}{\Rightarrow}g$ is an invertible 2-cell in $\calB$, then $g\in\widehat{\frakW}$.
\end{itemize}
Then $\widehat{\frakW}$ has the stronger property {\bf BF2}, and each arrow $w\in\widehat{\frakW}$ will have an invertible 2-cell $\alpha\colon w \stackrel{\sim}{\Rightarrow} w_n\circ\cdots\circ w_1$ with codomain a finite composite of arrows  $w_1,\ldots,w_n\in\frakW$.  

\begin{lma}
If $\frakW$ satisfies the conditions {\bf [WB1]}--{\bf[WB5]}, then $\widehat{\frakW}$ defines a wide subcategory which satisfies the conditions from  \cite{Pr-comp} for constructing a bicategory of fractions.
\end{lma}

\begin{proof}
Since $\frakW$ contains all identities, so does $\widehat{\frakW}$, so $\widehat{\frakW}$ satisfies condition {\bf BF1} from \cite{Pr-comp}.    And $\widehat{\frakW}$ has been created to be closed under composition, verifying {\bf BF2}.   Conditions {\bf BF3}--{\bf BF5} are equivalent to conditions {\bf [WB3]}--{\bf[WB5]} (and {\bf BF3} and {\bf BF5} are identical to their weaker versions);  see Remark~\ref{r:bvswb}.  So it suffices to check conditions {\bf [WB3]}--{\bf [WB5]} for $\widehat{\frakW}$.

Since every arrow $v$ in $\widehat{\frakW}$ is isomorphic to a composition $w_1\circ\ldots\circ w_n$ of finitely many arrows in $\frakW$, repeated application of {\bf [WB3]} for $\frakW$ 
gives us {\bf [WB3]} for $\widehat{\frakW}$:
$$
\xymatrix{
\ar[r]^{f_n}\ar[d]_{w_n'}\ar@{}[dr]|{\alpha_n} & \ar[d]^{w_n}\ar@/^4ex/[4,0]^v_{\cong}
\\
\ar@{..}[d]\ar[r]_{f_{n-1}}\ar@{}[d]|\cdots&\ar@{..}[d]
\\
\ar[d]_{w_2'}\ar[r]^{f_2}\ar@{}[dr]|{\alpha_2}&\ar[d]^{w_2}
\\
\ar[d]_{w_1'}\ar@{}[dr]|{\alpha_1}\ar[r]^{f_1}&\ar[d]^{w_1}
\\
\ar[r]_f &}
$$
Note that $w_n'\circ\cdots\circ w_1'\in\widehat{\frakW}$ by definition.

To verify condition {\bf [WB4]},  suppose that $\alpha\colon wf\Rightarrow wg$ and $\gamma\colon w_n\cdots w_1\stackrel{\sim}{\Rightarrow}w$ with $w_1,\ldots,w_n\in\frakW$.
Repeatedly applying {\bf [WB4]} for $\frakW$ gives us arrows $w'_{n-k}$ and 2-cells $\beta_{n-k}\colon w_{n-k-1}\cdots w_1fw_n'\cdots w_{n-k}'\Rightarrow  w_{n-k-1}\cdots w_1gw_n'\cdots w_{n-k}'$ for $k=0,\ldots, n-1$ such that $w_{n-k}\cdots w_{n-1}w_n\beta_{n-k}=((\gamma^{-1}g)\cdot\alpha\cdot(\gamma f))w_n'w_{n-1}'\cdots w_{n-k}'$. So $\beta_1$ with $w_n'w_{n-1}'\cdots w_1'$ is the required lifting.

To check the compatibility condition in {\bf [WB4]}, consider $\alpha\colon wf\Rightarrow wg$ with liftings
$\alpha'\colon fw'\Rightarrow gw'$ and $\alpha''\colon fw''\Rightarrow gw''$.
Since $w',w''\in\widehat{\frakW}$, there are arrows  $w'_1,\ldots,w'_k$ and $w''_1,\ldots,w''_\ell $ in $\frakW$ with invertible 2-cells, $\delta\colon w'_k\cdots w'_1\Rightarrow w'$ and $\gamma\colon w''_\ell\cdots w''_1\Rightarrow w''$. By repeatedly applying condition {\bf [WB2]} for $\frakW$ there are arrows $u',u''$ such that $w'u'\in\frakW$ and $w''u''\in\frakW$. Hence we can apply {\bf [WB4]} for $\frakW$ to the liftings $\alpha' u'\colon fw'u'\Rightarrow gw'u'$ and $\alpha''u''\colon fw''u''\Rightarrow gw''u''$ and obtain arrows $s,t$ and a 2-cell $\varepsilon \colon w'u's\Rightarrow w''u''t$ showing compatibility of these liftings. This then gives us also the required arrows $u's$ and $u''t$ with the cell $\varepsilon$ to establish compatibility for the original liftings. 

Finally, $\widehat{\frakW}$ satisfies condition {\bf BF5} by construction.
\end{proof}

\begin{thm}\label{comparison-thm}
If $\frakW$ satisfies the conditions {\bf [WB1]}--{\bf[WB5]}, then there is an equivalence of bicategories $J\colon \calB(\frakW^{-1})\stackrel{\sim}{\longrightarrow}\calB(\widehat{\frakW}^{-1})$,
making the following triangle commute,
$$
\xymatrix{
&\calB(\frakW^{-1})\ar[dd]^{J}
\\
\calB\ar[ur]^{J_{\frakW}}\ar[dr]_{J_{\widehat{\frakW}}}
\\
&\calB(\widehat{\frakW}^{-1})
}
$$ 
where $\calB(\widehat{\frakW}^{-1})$ is the bicategory of fractions from \cite{Pr-comp} and $\calB(\frakW^{-1})$ is the bicategory of fractions defined in Section~\ref{newBF2}.   
\end{thm}

\begin{proof}
 We have shown that whenever a class of arrows $\frakW$ satisfies the stronger conditions {\bf BF1}-{\bf BF5}, the resulting bicategory of fractions is equivalent to
the traditional one from \cite{Pr-comp};  see
 Remarks~\ref{first-comparison}(\ref{fc1}) and  \ref{r:eq}.   So $\calB(\widehat{\frakW}^{-1})$ may be taken to be the classical bicategory of fractions  and Theorem~\ref{d:equivalence} now gives us the equivalence of the resulting bicategories of fractions.  
\end{proof}

\begin{cor}
When $\frakW$ satisfies the conditions {\bf [WB1]}--{\bf[WB5]}, the pseudo functor $J_{\frakW}\colon\calB\to\calB(\frakW^{-1})$ satisfies the universal property for the bicategory of fractions.
\end{cor}

\begin{proof}
A pseudo functor $\calB\to\calC$ sends the arrows in $\frakW$ to equivalences if and only if it sends the arrows in $\widehat{\frakW}$ to equivalences. 
\end{proof}

This  result also applies to  results for categories of fractions in the 1-category case given in \cite{GZ}.

\begin{cor}\label{classicalGZ}
A class of arrows $W$ in a category $\calC$ allows for the construction of a category of right fractions $\calC[W^{-1}]$ if it satisfies the following conditions:
\begin{enumerate}
\item 
$W$ contains all identities;
\item
For any pair of composable arrows $\xymatrix@1{B\ar[r]^v&C\ar[r]^w&D}$ in $W$ there is an arrow $\xymatrix@1{A\ar[r]^u&B}$ such that $\xymatrix@1{A\ar[r]^{wvu}&D}$ is in $W$;
\item
For any arrow $w\in W$ and any arrow $f$ which shares its codomain with $w$, there is an arrow $w'\in W$ and an arrow $f'$ such that the following square is defined and commutes:
$$
\xymatrix{
\ar[r]^{f'}\ar[d]_{w'}&\ar[d]^{w}
\\
\ar[r]_{f}&}
$$
\item
Given $w\in W$ and parallel arrows $f_1,f_2$ such that $wf_1=wf_2$, then there is an arrow $w'\in W$ such that $f_1w'=f_2w'$,
$$
\xymatrix{\ar[r]^{w'}&\ar@<.6ex>[r]^{f_1}\ar@<-.6ex>[r]_{f_2}&\ar[r]^w&}
$$
\end{enumerate} 
\end{cor}

\begin{egs}
\begin{enumerate}
\item
When one wants to add the inverse for an arrow $w$ in a monoid, the class $W$ in the traditional Gabriel-Zisman construction of \cite{GZ} would be required to contain all powers of $w$. In our case $W$ only needs to contain a cofinal set of powers of $w$.
\item
Consider the category of atlases and atlas maps for manifolds. 
In order to obtain the category containing all smooth maps between manifolds using the original conditions,  one needs to take  the category of fractions with respect to all atlas refinements.
With the new theory we may restrict ourselves to refinements in which no charts are repeated, or any other family of refinements that is weakly initial among all refinements.
\end{enumerate}
\end{egs}

\section{Simplifying 2-Cell Representatives}\label{simplifications}

As we have seen, the universal homomorphism $J_\frakW\colon\calB\to\calB(\frakW^{-1})$  is defined to be the identity on objects, and takes an arrow $f\colon  A \to B$ to the generalized arrow $\xymatrix{A&A\ar[l]_{1_A}\ar[r]^{f}&B}$ and a 2-cell $\alpha\colon  f \Rightarrow g$ to a 2-cell diagram of the form below.  
$$
\xymatrix@C=3em{
&\ar[dl]_{1_A}A\ar[dr]^f
\\
A\ar@{}[r]|{\iota_A\Downarrow}&A\ar[u]_{1_A}\ar[d]^{1_A}\ar@{}[r]|{\alpha1_A\Downarrow}&B
\\
&A\ar[ul]^{1_A}\ar[ur]_g
}
$$
As Tommasini observed in Remark 3.5 of \cite{Matteo1}, this homomorphism  is  neither 2-full nor 2-faithful in general.  
The map  $J_\frakW$ fails to be 2-full because not every 2-cell between $J_\frakW(f)$ and $J_\frakW(g)$ needs to have a representative of this particular form.
The map $J_\frakW$ fails to be 2-faithful because two  2-cell diagrams of this form, say with distinct right cells  $\beta$ and $\gamma$, represent the same 2-cell in the bicategory of fractions  when there is an arrow $t\in\frakW$ such that $\beta t=\gamma t$.
This leads us to consider the more general issue of the equivalence relation on the 2-cell diagrams.

In this section we discuss some variations of {\bf [WB4]} and consider when a 2-cell in the bicategory of fractions can be represented by a 2-cell diagram with a given left-hand side.  In the following section, we will look at choosing these left-hand sides to have nice additional properties that will simplify some of the composition constructions.  In some cases  representatives with a given left-hand side will even be unique.
We will prove in \cite{next_paper} that some of these properties hold for the case of essential equivalences between orbifold \'etale groupoids.  In fact they apply more generally to any fully faithful maps between \'etale topological groupoids.

Following the notation of \cite{AV} and \cite{Roberts-fractions} we say that an arrow $f\colon A\to B$ in a bicategory $\calB$ has a property $\mathcal P$  when the induced functor $f_*\colon \calB(X,A) \to \calB(X,B)$ has this property. Note that for full and faithful, these properties are closely related to Condition {\bf [WB4]}. In this section we will see that if the arrows in $\frakW$ have these properties and/or their duals, we are able to simplify our description of the 2-cells in the bicategory of fractions: each 2-cell will have a representative with a given left-hand side and we won't need equivalence classes if we have chosen representatives.

\begin{dfn}
An arrow $w$ in a bicategory $\calB$ is
\begin{itemize}
\item {\em full} if 
for any 2-cell $\alpha\colon wf\Rightarrow wg$ there is a 2-cell $\tilde\alpha\colon f\Rightarrow g$ such that $w\tilde\alpha=\alpha$.  
\item {\em fully faithful} or {\em ff} if 
for any 2-cell $\alpha\colon wf\Rightarrow wg$ there is a {\em unique} 2-cell $\tilde\alpha\colon f\Rightarrow g$ such that $w\tilde\alpha=\alpha$.  
\item {\em co-full} if  
for any 2-cell $\alpha\colon fw\Rightarrow gw$ there is a 2-cell $\alpha'\colon f\Rightarrow g$ such that $\alpha' w=\alpha$.  
\item {\em co-fully-faithful} or {\em co-ff} if  
for any 2-cell $\alpha\colon fw\Rightarrow gw$ there is a {\em unique} 2-cell $\alpha'\colon f\Rightarrow g$ such that $\alpha' w=\alpha$.  
\end{itemize}
\end{dfn}

Fractions condition {\bf [WB4]} connects some of these properties as follows:

\begin{lma}
If a class of arrows $\frakW$ satisfies condition {\bf[WB4]} and is co-fully-faithful it is also full.
\end{lma}

\begin{proof}
Consider a 2-cell $\alpha\colon wf\Rightarrow wg$ with $w\in \frakW$. Since $\frakW$ satisfies {\bf [WB4]} 
there is an arrow $v\in\frakW$ with a 2-cell $\beta\colon fv\Rightarrow gv$ such that $w\beta=\alpha v$.
Since $\frakW$ is co-full, there is a 2-cell $\tilde\alpha\colon f\Rightarrow g$ such that $\beta=\tilde\alpha v$. Hence, $w\tilde\alpha v=w\beta=\alpha v$. Since $\frakW$ is co-fully-faithful this implies that $w\tilde\alpha=\alpha$.
\end{proof}

\begin{lma} \label{L:ch2cell}
Let $\frakW$ be a class of co-full arrows in $\calB$ satisfying the conditions {\bf [WB1]}--{\bf [WB5]}.
Given any 2-cell diagram 
\begin{equation}\label{thislemma1}
\xymatrix@C=3em{
&\ar[dl]_{u_1}\ar[dr]^{f_1}
\\
\ar@{}[r]|{\alpha\Downarrow} & \ar[u]_{v_1}\ar[d]^{v_2}\ar@{}[r]|{\beta\Downarrow} & 
\\
&\ar[ul]^{u_2}\ar[ur]_{f_2}
}
\end{equation} 
in $\calB(\frakW^{-1})$
and any square
$$
\xymatrix{
\ar[r]^{t_1}\ar[d]_{t_2}\ar@{}[dr]|{\stackrel{\gamma}{\scriptstyle\stackrel{\sim}{\Leftarrow}}} &\ar[d]^{u_1}
\\
\ar[r]_{u_2}&
}
$$
in $\calB$ with $u_1t_1\in\frakW$,
there is a 2-cell $\delta$ such that the diagram
\begin{equation}\label{thislemma2}
\xymatrix{
&\ar[dl]_{u_1}\ar[dr]^{f_1}
\\
\ar@{}[r]|{\gamma\Downarrow} & \ar[u]_{t_1}\ar[d]^{t_2}\ar@{}[r]|{\delta\Downarrow} & 
\\
&\ar[ul]^{u_2}\ar[ur]_{f_2}
}
\end{equation} 
represents the same 2-cell in $\calB(\frakW^{-1})$ as (\ref{thislemma1}).
\end{lma}

\begin{proof}
By {\bf [WB3]} there is a square
$$
\xymatrix{
\ar@{}[ddrr]|{\theta\Downarrow}\ar[rr]^{\overline{t}_1}\ar[dd]_{\overline{v}_1} && \ar[d]^{v_1}
\\
&&\ar[d]^{u_1}
\\
\ar[r]_{t_1}&\ar[r]_{u_1}&
}
$$
with $\overline{v}_1\in\frakW$ and $\theta$ invertible.
By {\bf [WB4]} there is an arrow $\tilde{u}_1\in\frakW$ and an invertible 2-cell $\tilde\theta\colon (v_1\overline{t}_1)\tilde{u}_1\Rightarrow(t_1\overline{v}_1)\tilde{u}_1$.
Now consider the pasting of the diagram
\begin{equation}\label{pasting}
\xymatrix@C=4em{
&\ar[dr]_{v_1}\ar[r]^{v_2}\ar@{}[drr]|{\alpha^{-1}\Downarrow}&\ar[dr]^{u_2} 
\\
\ar[ur]^{\overline{t}_1\tilde{u}_1}\ar[dr]_{\overline{v}_1\tilde{u}_1}\ar@{}[rr]|{\tilde\theta\Downarrow} &&\ar[r]_{u_1}&
\\
&\ar[ur]^{t_1}\ar@{}[urr]|{\gamma\Downarrow} \ar[r]_{t_2}& \ar[ur]_{u_2}
}
\end{equation}
By {\bf [WB4]} there is an arrow $\tilde{u}_2\in\frakW$ with an invertible 2-cell $\zeta\colon (v_2(\overline{t}_1\tilde{u}_1))\tilde{u}_2\Rightarrow (t_2(\overline{v}_1\tilde{u}_1))\tilde{u}_2$ such that $u_2\zeta$ is equal to the pasting of the cells in (\ref{pasting}) composed with $\tilde{u}_2$. Finally, we need to ensure that certain compositions of arrows are in $\frakW$.
First consider the composition of arrows $\overline{v}_1\tilde{u}_1\tilde{u}_2$. Each of the three arrows in this composition is in $\frakW$, so by {\bf [WB2]} there is an arrow $s$ such that 
$\overline{v}_1\tilde{u}_1\tilde{u}_2s\in\frakW$.
Furthermore, $u_2t_2\in\frakW$ as well, so there is an arrow $r$ such that $(u_2t_2)(\overline{v}_1\tilde{u}_1\tilde{u}_2s)r\in\frakW$.
Then we have the following equality of pastings of 2-cells:
$$
\xymatrix@C=5em{
\ar[r]^{\overline{t}_1\tilde{u}_1\tilde{u}_2sr}\ar[d]_{\overline{v}_1\tilde{u}_1\tilde{u}_2sr}\ar@{}[dr]|{\tilde\theta\tilde{u}_2sr\Downarrow} & \ar[d]^{v_1}&&\ar[d]_{\overline{v}_1\tilde{u}_1\tilde{u}_2sr}\ar@{}[dr]|{\zeta sr\Downarrow}\ar[r]^{\overline{t}_1\tilde{u}_1\tilde{u}_2sr}&\ar[d]^{v_2}\ar[r]^{v_1}\ar@{}[dr]|{\alpha\Downarrow} & \ar[d]^{u_1}
\\
\ar[r]^{t_1}\ar[d]_{t_2}\ar@{}[dr]|{\gamma\Downarrow} & \ar[d]^{u_1}&\equiv & \ar[r]_{t_2}&\ar[r]_{u_2}&
\\
\ar[r]_{u_2}&
}
$$
We want to construct a cell $\delta$ such that $\beta$ and $\delta$ fit into a similar equality of 2-cell pastings.
So consider the following pasting diagram,
$$
\xymatrix@C=7em{
&\ar[r]^{t_1}\ar@{}[d]|{(\tilde\theta\tilde{u}_2s)^{-1}\Downarrow}&\ar[dr]^{f_1}
\\
\ar[ur]^{\overline{v}_1\tilde{u}_1\tilde{u}_2s}\ar[dr]_{\overline{v}_1\tilde{u}_1\tilde{u}_2s}\ar[r]|{\overline{t}_1\tilde{u}_1\tilde{u}_2s} & \ar@{}[d]|{\zeta s\Downarrow} \ar[ur]_{v_1}\ar[dr]^{v_2}\ar@{}[rr]|{\beta\Downarrow} &&
\\
&\ar[r]_{t_2}&\ar[ur]_{f_2}
}$$
Since the arrows in $\frakW$ are co-full, there is a 2-cell $\delta\colon f_1t_1\Rightarrow f_2t_2$ such that $\delta {\overline{v}_1\tilde{u}_1\tilde{u}_2s}$ is equal to the pasting of this diagram.
Then we get that 
$$
\xymatrix@C=5em{
\ar[r]^{\overline{t}_1\tilde{u}_1\tilde{u}_2sr}\ar[d]_{\overline{v}_1\tilde{u}_1\tilde{u}_2sr}\ar@{}[dr]|{\tilde\theta\tilde{u}_2sr\Downarrow} & \ar[d]^{v_1}&&\ar[d]_{\overline{v}_1\tilde{u}_1\tilde{u}_2sr}\ar@{}[dr]|{\zeta sr\Downarrow}\ar[r]^{\overline{t}_1\tilde{u}_1\tilde{u}_2sr}&\ar[d]^{v_2}\ar[r]^{v_1}\ar@{}[dr]|{\beta\Downarrow} & \ar[d]^{f_1}
\\
\ar[r]^{t_1}\ar[d]_{t_2}\ar@{}[dr]|{\delta\Downarrow} & \ar[d]^{f_1}&\equiv & \ar[r]_{t_2}&\ar[r]_{f_2}&
\\
\ar[r]_{f_2}&
}
$$
and hence we conclude that with $\delta$ thus defined, (\ref{thislemma2}) is equivalent to (\ref{thislemma1}).
\end{proof}

We now want to address the question about uniqueness of 2-cell representatives with a given left-hand side.
The following is the closest we can get to uniqueness for 2-cell diagrams with a given left-hand side without adding any further conditions on the class $\frakW$. This result is due to Matteo  Tommasini \cite{Matteo0}, who first pointed it out to us and proved it.  We include it here with his permission, with a different proof.  

\begin{prop}\label{Matteo3}
Let $\frakW$ be a class of arrows satisfying conditions {\bf[WB1]}--{\bf[WB5]}.
Let
\begin{equation}\label{equiv-diagrams}
\xymatrix{
&A_1\ar[dl]_{u_1}\ar[dr]^{f_1} &&&&A_1\ar[dl]_{u_1}\ar[dr]^{f_1}
\\
A\ar@{}[r]|{\Downarrow\alpha}&C\ar[u]^{v_1}\ar[d]_{v_2}\ar@{}[r]|{\beta\Downarrow}&B&\mbox{and}&A\ar@{}[r]|{\Downarrow\alpha}&C\ar[u]^{v_1}\ar[d]_{v_2}\ar@{}[r]|{\gamma\Downarrow}&B
\\
&A_2\ar[ul]^{u_2}\ar[ur]_{f_2} & && &A_2\ar[ul]^{u_2}\ar[ur]_{f_2}}
\end{equation}
be two equivalent 2-cell diagrams. Then there exists an arrow $w\colon D\to C$ such that $u_1v_1w\in \frakW$ and $\beta w=\gamma w$.
\end{prop}

\begin{proof}
Since the two 2-cell diagrams in (\ref{equiv-diagrams}) are equivalent there is a diagram with invertible 2-cells,
$$
\xymatrix{
&A_1
\\
C\ar[ur]^{v_1}\ar[dr]_{v_2} & E\ar[l]_s\ar[r]^t\ar@{}[u]|{\stackrel{\varepsilon_1}{\Rightarrow}}
\ar@{}[d]|{\stackrel{\varepsilon_2}{\Rightarrow}} & C\ar[ul]_{v_1}\ar[dl]^{v_2}
\\
&A_2}
$$
with $u_1v_1s\in \frakW$ such that 
\begin{equation}\label{left-eq}
\xymatrix{
\ar[d]_{t}\ar[r]^s\ar@{}[dr]|{\varepsilon_1\Downarrow} & \ar[d]^{v_1}\ar@{}[drr]|\equiv && \ar[d]_{t}\ar[r]^s\ar@{}[dr]|{\varepsilon_2\Downarrow} &\ar[d]|{v_2}\ar[r]^{v_1}\ar@{}[dr]|{\alpha\Downarrow} & \ar[d]^{u_1}
\\
\ar[d]_{v_2}\ar[r]|{v_1}\ar@{}[dr]|{\alpha\Downarrow} & \ar[d]^{u_1} && \ar[r]_{v_2}&\ar[r]_{u_2}&
\\
\ar[r]_{u_2} &}
\end{equation}
and
\begin{equation}\label{right-eq}
\xymatrix{
\ar[d]_{t}\ar[r]^s\ar@{}[dr]|{\varepsilon_1\Downarrow} & \ar[d]^{v_1}\ar@{}[drr]|\equiv && \ar[d]_{t}\ar[r]^s\ar@{}[dr]|{\varepsilon_2\Downarrow} &\ar[d]|{v_2}\ar[r]^{v_1}\ar@{}[dr]|{\beta\Downarrow} & \ar[d]^{f_1}
\\
\ar[d]_{v_2}\ar[r]|{v_1}\ar@{}[dr]|{\gamma\Downarrow} & \ar[d]^{f_1} && \ar[r]_{v_2}&\ar[r]_{f_2}&
\\
\ar[r]_{f_2} &}
\end{equation}
We want to use the first equation to derive a relationship between $\varepsilon_1$ and $\varepsilon_2$.
To make it possible to cancel $\alpha$ we first apply Proposition~\ref{isolifting}  to $u_1\varepsilon_1\colon u_1v_1s\Rightarrow u_1v_1t$ to obtain an arrow $\tilde{u}_1\colon E_1\to E$ in $\frakW$ and an invertible 2-cell $\tilde\varepsilon_1\colon s\tilde{u}_1\to t \tilde{u}_1$ such that $u_1v_1\tilde\varepsilon_1=u_1\varepsilon_1\tilde{u}_1$. Furthermore, by Lemma~\ref{Matteo1} there is an arrow $w_1\colon E_1'\to E_1$ in $\frakW$ such that $v_1\tilde\varepsilon_1w_1=\varepsilon_1\tilde{u}_1w_1$. Similarly, applying Proposition~\ref{isolifting} to $u_2\varepsilon_2\colon u_2v_2s\Rightarrow u_2v_2t$ gives us an arrow $\tilde{u}_2\colon E_2\to E$ in $\frakW$ with an invertible 2-cell $\tilde\varepsilon_2\colon s\tilde{u}_2\to t \tilde{u}_2$ such that $u_2v_2\tilde\varepsilon_2=u_2\varepsilon_2\tilde{u}_2$ and  there is an arrow $w_2\colon E_2'\to E_2$ in $\frakW$ such that $v_2\tilde\varepsilon_2w_2=\varepsilon_2\tilde{u}_2w_2$.  By condition {\bf [WB2]}, let $x_i\colon E_i''\to E_i'$ (for $i=1,2$) be arrows such that $\tilde{u}_iw_ix_i\in \frakW$ for $i=1,2$. Now apply condition {\bf [WB3]} to obtain an invertible  2-cell
$$
\xymatrix@C=4em{
F\ar[r]^{\overline{u}_2}\ar[d]_{\overline{u}_1}\ar@{}[dr]|{\stackrel{\delta}{\Leftarrow}} & E_1\ar[d]^{\tilde{u}_2w_2x_2}
\\
E_1\ar[r]_{\tilde{u}_1w_1x_1} & E}
$$
with $\overline{u}_1\in\frakW$. Now write $z_1:=\tilde{u}_1w_1x_1\overline{u}_1$ and $z_2:=\tilde{u}_2w_2x_2\overline{u}_2$.
Precomposing equation (\ref{left-eq}) horizontally by $z_1$ and then vertically by $u_1v_1s\delta$ gives the following equation:
\begin{equation}\label{left-eq1}
\xymatrix@C=3em@R=4em{
\ar[d]_{tz_1} \ar@/_2ex/[r]_{z_1} \ar@/^2ex/[r]^{z_2} \ar@{}[r]|{\Downarrow\delta }
\ar@{}[drr]|{\varepsilon_1z_1\Downarrow} & \ar[r]^s & \ar[d]^{v_1}\ar@{}[drr]|\equiv && \ar[d]_{tz_1}\ar@/_2ex/[r]_{z_1}\ar@{}[r]|{\Downarrow\delta}\ar@/^2ex/[r]^{z_2}\ar@{}[drr]|{\varepsilon_2z_1\Downarrow}&\ar[r]^s &\ar[d]|{v_2}\ar[r]^{v_1}\ar@{}[dr]|{\alpha\Downarrow} & \ar[d]^{u_1}
\\
\ar[d]_{v_2}\ar[rr]|{v_1}\ar@{}[drr]|{\alpha\Downarrow} && \ar[d]^{u_1} && \ar[rr]_{v_2}&&\ar[r]_{u_2}&
\\
\ar[rr]_{u_2} &&}
\end{equation}
Similarly, (\ref{right-eq}) induces the following equation:
\begin{equation}\label{right-eq1}
\xymatrix@C=3em@R=4em{
\ar[d]_{tz_1} \ar@/_2ex/[r]_{z_1} \ar@/^2ex/[r]^{z_2} \ar@{}[r]|{\Downarrow\delta}
\ar@{}[drr]|{\varepsilon_1z_1\Downarrow} & \ar[r]^s & \ar[d]^{v_1}\ar@{}[drr]|\equiv && \ar[d]_{tz_1}\ar@/_2ex/[r]_{z_1y}\ar@{}[r]|{\Downarrow\delta}\ar@/^2ex/[r]^{z_2}\ar@{}[drr]|{\varepsilon_2z_1\Downarrow}&\ar[r]^s &\ar[d]|{v_2}\ar[r]^{v_1}\ar@{}[dr]|{\beta\Downarrow} & \ar[d]^{f_1}
\\
\ar[d]_{v_2}\ar[rr]|{v_1}\ar@{}[drr]|{\gamma\Downarrow} && \ar[d]^{f_1} && \ar[rr]_{v_2}&&\ar[r]_{f_2}&
\\
\ar[rr]_{f_2} &&}
\end{equation}
Since $\varepsilon_1z_1=\varepsilon_1\tilde{u}_1w_1x_1\overline{u}_1=v_1\tilde\varepsilon_1w_1x_1\overline{u}_1$
We rewrite the left-hand side of (\ref{left-eq1}) as follows:
$$
\xymatrix@R=4em@C=4em{\ar[d]_{tz_1} \ar@/_2ex/[r]_{z_1} \ar@/^2ex/[r]^{z_2} \ar@{}[r]|{\Downarrow\delta}
\ar@{}[drr]|{\varepsilon_1z_1\Downarrow} & \ar[r]^s & \ar[d]^{v_1}\ar@{}[drr]|\equiv && \ar[d]_{tz_1} \ar@/_2ex/[r]_{z_1} \ar@/^2ex/[r]^{z_2} \ar@{}[r]|{\Downarrow\delta}
\ar@{}[dr]|{\stackrel{\tilde\varepsilon_1w_1x_1\overline{u}_1}{\Leftarrow}} & \ar@/^4ex/[dl]^(.3)s 
\\
\ar[d]_{v_2}\ar[rr]|{v_1}\ar@{}[drr]|{\alpha\Downarrow} && \ar[d]^{u_1} &&\ar[d]_{v_2}\ar[rr]|{v_1}\ar@{}[drr]|{\alpha\Downarrow} && \ar[d]^{u_1}
\\
\ar[rr]_{u_2} && &&\ar[rr]_{u_2}&&}$$
Similarly, we rewrite the right-hand side of (\ref{left-eq1}) as follows:
$$
\xymatrix@C=3em@R=4em{
\ar[d]_{tz_1}\ar@/_2ex/[r]_{z_1}\ar@{}[r]|{\Downarrow\delta}\ar@/^2ex/[r]^{z_2}\ar@{}[drr]|{\varepsilon_2z_1\Downarrow}&\ar[r]^s &\ar[d]|{v_2}\ar[r]^{v_1}\ar@{}[dr]|{\alpha\Downarrow} & \ar[d]^{u_1}\ar@{}[drr]|\equiv && \ar@/_2ex/[r]_{z_1}\ar@{}[r]|{\Downarrow\delta}\ar@/^2ex/[r]^{z_2}&\ar[r]^s\ar[d]_t\ar@{}[dr]|{\varepsilon_2\Downarrow} & \ar[d]|{v_2}\ar[r]^{v_2}\ar@{}[dr]|{\Downarrow\alpha} &\ar[d]^{u_1}
\\
\ar[rr]_{v_2}&&\ar[r]_{u_2}& &&  & \ar[r]_{v_2}&\ar[r]_{u_2}&
\\
&&&\ar@{}[drr]|\equiv && \ar@/_2ex/[d]_(.45){z_1}\ar@{}[d]|(.45){\stackrel{\delta}{\Leftarrow}}\ar@/^2ex/[d]^(.45){z_2}\ar[rr]^{sz_2}\ar@{}[drr]|(.55){\Downarrow\varepsilon_2z_2} & & \ar[d]|{v_2}\ar@{}[dr]|{\Downarrow \alpha}\ar[r]^{v_1} & \ar[d]^{u_1}
\\
&&&&&\ar[rr]_{v_2t}&&\ar[r]_{u_2} &
\\
&&&\ar@{}[drr]|\equiv && \ar@/_2ex/[d]_(.55){z_1}\ar@{}[d]|(.55){\stackrel{\delta }{\Leftarrow}}\ar@/^2ex/[d]^(.55){z_2}\ar[rr]^{sz_2}\ar@{}[drr]^(.3){\Downarrow\tilde\varepsilon_2w_2x_2\overline{u}_2} & & \ar[d]|{v_2}\ar@{}[dr]|{\Downarrow \alpha}\ar[r]^{v_1} & \ar[d]^{u_1}
\\
&&&&&\ar@/_2ex/[urr]_{t}&&\ar[r]_{u_2} &
}
$$
By composing with $\alpha^{-1}tz_1$ with the rewritten left and right-hand sides of (\ref{left-eq1}) we derive that
$$
\xymatrix@C=5.4em{
\ar@/^2ex/[dr]^s &&&\ar@/^2ex/[dr]^{sz_2} \ar@/_2ex/[dd]_(.35){z_1}\ar@{}[dd]|{\stackrel{\delta}{\Leftarrow}}\ar@/^2ex/[dd]^(.35){z_2}
\\
\ar@{}[r]|(.55){{\scriptscriptstyle\tilde{\varepsilon}_1w_1x_1\overline{u}_1\Downarrow}}&\ar[r]^{u_2v_2}&\ar@{}[r]|(.35)\equiv & \ar@{}[r]|(.55){{\scriptscriptstyle\tilde{\varepsilon}_2w_2x_2\overline{u}_2\Downarrow}}&\ar[r]^{u_2v_2}&
\\
\ar@/^2ex/[uu]^(.65){z_2}\ar@{}[uu]|{\stackrel{\delta}{\Rightarrow}}\ar@/_2ex/[uu]_(.65){z_1}\ar@/_2ex/[ur]_{tz_1}&&&\ar@/_2ex/[ur]_{t}
}
$$
By Lemma~\ref{Matteo1} there is an arrow $(r\colon G\to F)\in\frakW$ such that 
\begin{equation}\label{equality}
\xymatrix@C=5.4em{
&\ar@/^2ex/[dr]^s &&\ar[r]^r&\ar@/^2ex/[dr]^{sz_2} \ar@/_2ex/[dd]_(.35){z_1}\ar@{}[dd]|{\stackrel{\delta}{\Leftarrow}}\ar@/^2ex/[dd]^(.35){z_2}
\\
&\ar@{}[r]|(.55){{\scriptscriptstyle\tilde{\varepsilon}_1w_1x_1\overline{u}_1\Downarrow}}&\ar@{}[r]|\equiv && \ar@{}[r]|(.55){{\scriptscriptstyle\tilde{\varepsilon}_2w_2x_2\overline{u}_2\Downarrow}}&
\\
\ar[r]^r&\ar@/^2ex/[uu]^(.65){z_2}\ar@{}[uu]|{\stackrel{\delta}{\Rightarrow}}\ar@/_2ex/[uu]_(.65){z_1}\ar@/_2ex/[ur]_{tz_1}&&&\ar@/_2ex/[ur]_{t}
}
\end{equation}
Finally there is an arrow $r'\colon D\to G$ such that $u_1v_1sz_1rr'\in\frakW$.

We will now combine this result with (\ref{right-eq1}).
We first manipulate $\varepsilon_1$ and $\varepsilon_2$ with $\delta$ just as we have done above. Note that we did not need the presence of $u_1$ or $u_2$ for this, so the same calculations apply to the compositions with $\beta$ and $\gamma$. This gives us 
$$
\xymatrix@C=5.4em{
\ar@/^2ex/[dr]^s &\ar[dr]^{f_1}&&\ar@/^2ex/[dr]^{sz_2} \ar@/_2ex/[dd]_(.35){z_1}\ar@{}[dd]|{\stackrel{\delta}{\Leftarrow}}\ar@/^2ex/[dd]^(.35){z_2}&\ar[dr]^{f_1}
\\
\ar@{}[r]|(.55){{\scriptscriptstyle\tilde{\varepsilon}_1w_1x_1\overline{u}_1\Downarrow}}&\ar[u]^{v_1}\ar[d]_{v_2}\ar@{}[r]|{\gamma\Downarrow}&\ar@{}[r]|(.35)\equiv & \ar@{}[r]|(.55){{\scriptscriptstyle\tilde{\varepsilon}_2w_2x_2\overline{u}_2\Downarrow}}&\ar[u]^{v_1}\ar[d]_{v_2}\ar@{}[r]|{\beta\Downarrow}&
\\
\ar@/^2ex/[uu]^(.65){z_2}\ar@{}[uu]|{\stackrel{\delta}{\Rightarrow}}\ar@/_2ex/[uu]_(.65){z_1}\ar@/_2ex/[ur]_{tz_1}&\ar[ur]_{f_2}&&\ar@/_2ex/[ur]_{t}&\ar[ur]_{f_2}
}
$$
Now precomposing by $rr'$ and using the result from (\ref{equality}) gives us that $\beta$ and $\gamma$ become equal when precomposed by the same invertible cell. So we can conclude that
$\beta sz_2rr'=\gamma sz_2rr'$ and since $u_1v_1sz_2rr' \cong u_1v_1sz_1rr'\in\frakW$, we also have that $u_1v_1sz_2rr' \in\frakW$ by {\bf [WB5]}. So $w=sz_2rr'\colon D\to C$ has the required property. 
\end{proof}

We use this result together with the condition that the arrows in $\frakW$ be co-fully-faithful to obtain uniqueness of 2-cell representatives with a given left-hand side. The following lemma, proved by Matteo Tommasini \cite{Matteo0} and included here with his permission, gives us a key ingredient.  

\begin{lma}\label{Matteo2}
Let $\frakW$ be a class of arrows satisfying conditions {\bf[WB1]}--{\bf[WB5]} and let $a\colon B\to A$ and $b\colon C\to B$ be arrows such that both $a$ and $ab$ are in $\frakW$.
Then there is an arrow $c\colon D\to C$ such that $bc\in\frakW$.
\end{lma}

\begin{proof}
Since $ab\in\frakW$, condition {\bf [WB3]} gives us the existence of a square with an invertible 2-cell,
$$
\xymatrix{X\ar[r]^u\ar[d]_v\ar@{}[dr]|{\stackrel{\alpha}{\Leftarrow}}&C\ar[d]^{ab}
\\
B\ar[r]_a&A}$$
with $v\in\frakW$. Since $a\in\frakW$, we can apply Proposition~\ref{isolifting} to $\alpha\colon a(bu)\stackrel{\sim}{\Rightarrow} av$ to obtain an arrow $w\colon Y\to X$ in $\frakW$ and an invertible 2-cell $\tilde\alpha\colon buw\stackrel{\sim}{\Rightarrow}vw$. Since both $v$ and $w$ are in $\frakW$, there is an arrow $z\colon D\to Y$ such that $vwz\in\frakW$ by condition {\bf [WB2]}. Now $\tilde\alpha z \colon buwz\stackrel{\sim}{\Rightarrow}vwz$, so $buwz\in\frakW$ by condition {\bf [WB5]}. Hence $c= uwz\colon D\to C$ has the required property.
\end{proof}

\begin{thm}\label{main5}
Let $\frakW$ be a class of co-ff arrows in a bicategory $\calB$ satisfying conditions {\bf [WB1]}--{\bf[WB5]}.
Then each 2-cell in $\calB(\frakW^{-1})$ has at most one representative with a given left-hand 2-cell.
\end{thm}

\begin{proof}
Given two 2-cell diagrams with the same left-hand side as in (\ref{equiv-diagrams}), Proposition~\ref{Matteo3} gives us an arrow $w$ such that $u_1v_1w\in\frakW$ and $\beta w=\gamma w$.
Since $u_1v_1\in\frakW$ we can apply Lemma~\ref{Matteo2} to obtain an arrow $x\colon D'\to D$  such that $wx\in\frakW$. Now we have that $\beta wx=\gamma wx$ and since the arrows in $\frakW$ are co-ff we conclude that $\beta=\gamma$.
\end{proof}

\begin{cor}Let $\frakW$ be a class of co-ff arrows in a bicategory $\calB$ satisfying conditions {\bf [WB1]}--{\bf[WB5]}.
Then each 2-cell in $\calB(\frakW^{-1})$ has precisely one representative with a given left-hand 2-cell.
\end{cor}

\begin{proof}
This follows from Lemma~\ref{L:ch2cell} and Theorem~\ref{main5}.
\end{proof}
\begin{rmk}
 This provides further understanding in regard to the results provided in \cite{AV} and \cite{Roberts-fractions} where no equivalence relation is needed for the 2-cells in the localizations: Abbad and Vitale introduce a category of so called {\em faithful fractions} where the objects are arrows in $\frakW$ and hom-categories are hom-categories in the original bicategory between the domains of the objects. Roberts uses these conditions to obtain a decription of the 2-cells in his bicategory of fractions that can be viewed as the classical a 2-cell diagram with a strict pullback square as left-hand 2-cell. In the next section we will work out the case where one has pseudo pullbacks for arrows in $\frakW$. 
\end{rmk}
\begin{cor} \label{P:2-ff}
Suppose that  $\frakW$ be a class of co-ff arrows in a bicategory $\calB$ satisfying conditions {\bf [WB1]}--{\bf[WB5]}.  Then the universal homomorphism $J_\frakW\colon\calB\to\calB(\frakW^{-1})$ is 2-full and 2-faithful.
\end{cor}

\begin{proof}
To show that the homomorphism is 2-full, consider an arbitrary 2-cell between  $J_\frakW(f)$ and $J_\frakW(g)$.  This will have a representative of the form  $$
\xymatrix{
&\ar[dl]_{1_A}A\ar[dr]^f
\\
A\ar@{}[r]|{\alpha\Downarrow}&C\ar[u]_{s}\ar[d]^{t}\ar@{}[r]|{\beta\Downarrow}&B
\\
&A\ar[ul]^{1_A}\ar[ur]_g
}
$$  
Now consider the square
$$
\xymatrix@C=1.5em@R=1.5em{A\ar[r]^{1_A}\ar[d]_{1_A}\ar@{}[dr]|{\iota_A}&A\ar[d]^{1_A}
\\A\ar[r]_{1_A}& A}$$ and Lemma~\ref{L:ch2cell} says that we can represent the 2-cell  between  $J_\frakW(f)$ and $J_\frakW(g)$ using this square on the left side.  Thus, the 2-cell is the image of a 2-cell in $\calB$.  

To show that the map $J_{\frakW}$ is 2-faithful, suppose that we have two 2-cells 
$J_\frakW(\alpha)$ and $J_\frakW(\beta)$, represented by   
\begin{equation}\label{2-faithful}
\xymatrix{
&\ar[dl]_{1_A}A\ar[dr]^f
&&& &\ar[dl]_{1_A}A\ar[dr]^f\\
A\ar@{}[r]|{\iota_A}&A\ar[u]_{1_A}\ar[d]^{1_A}\ar@{}[r]|{\alpha1_A\Downarrow }&B
&\mbox{and}& A\ar@{}[r]|{\iota_A}&A\ar[u]_{1_A}\ar[d]^{1_A}\ar@{}[r]|{\beta1_A\Downarrow}&B
\\
&A\ar[ul]^{1_A}\ar[ur]_g &&& &A\ar[ul]^{1_A}\ar[ur]_g
}
\end{equation}
which represent the same 2-cell in $\calB(\frakW^{-1}) $.  Then there must be maps $r_1,r_2\colon E \rightrightarrows A$ with 2-cells $\varepsilon_1$, $\varepsilon_2$ as in
$$
\xymatrix{
&A
\\
A\ar@/^1ex/[ur]^{1_A}\ar@/_1ex/[dr]_{1_A}&\ar[l]_{r_1}\ar[r]^{r_2}E\ar@{}[u]|{\stackrel{\varepsilon_1}{\Rightarrow}}\ar@{}[d]|{\stackrel{\varepsilon_2}{\Rightarrow}}&A\ar@/_1ex/[ul]_{1_A}\ar@/^1ex/[dl]^{1_A}
\\
&A
}
$$
satisfying the equations to make the two diagrams in (\ref{2-faithful}) equivalent and such that $1_A1_Ar_1\in\frakW$. Write $\varepsilon'_1,\varepsilon'_2\colon r_1\Rightarrow r_2$ for the induced 2-cells.
Since the left-hand squares are just identities, this implies that $\varepsilon_1'=\varepsilon_2'\colon r_1\Rightarrow r_2$. The other equation then implies that
$\alpha\circ\varepsilon_1'=\beta\circ\varepsilon_1'$.
Since $\varepsilon_1'$ is invertible, this implies that $\alpha r_1=\beta r_1$.

Since $1_A1_Ar_1\in\frakW$,we conclude by {\bf [WB5]} that $r_1\in\frakW$.    Hence, since the arrows in $\frakW$ are co-ff, we get that there is a unique $\gamma\colon f\Rightarrow g$ such that $\gamma r_1=\alpha r_1$.   Hence, $\alpha=\beta$.
\end{proof}

\section{Bicategories with Pseudo Pullbacks}\label{pullbacks}

We now apply the ideas of Section~\ref{simplifications} to represent generalized 2-cells using pseudo pullbacks.
If a bicategory has all pseudo pullbacks of the form
$$
\xymatrix@C=3em{
P\ar[d]_{\overline{w}}\ar[r]^{\overline{f}}\ar@{}[dr]|{\stackrel{\scriptstyle\stackrel{\sim}{\Longleftarrow}}{\rho}} & \ar[d]^w
\\
\ar[r]_{f}&
}
$$
where $w\in\frakW$, and  the class $\frakW$ is stable under these pseudo pullbacks in the sense that $w\in\frakW$ implies that $\overline{w}\in\frakW$, it is possible to use the pseudo pullbacks as chosen squares as in {\bf[C2]} of Notation~\ref{choices} in the construction of $\calB(\frakW^{-1})$.  This makes the construction of this bicategory more canonical;  see 
 \cite{Matteo2} for instance.

We are interested in a different use of the pseudo pullbacks: as the left-hand sides of the generalized 2-cell diagrams. (The case with strict pullbacks was considered in \cite{Roberts-fractions}.)  This will allow us to simplify the horizontal composition operations.   It will require some additional assumptions on $\calB$ and $\frakW$, 
so we will develop conditions under which each  2-cell
has a representative diagram where $\alpha$ is a pseudo pullback.   The first condition is the following.  

\begin{defn} \label{D:pbc} We say that $\frakW$ is {\em pullback closed} if 
for any pseudo pullback
$$
\xymatrix{
P\ar[r]^{\overline{u}}\ar[d]_{\overline{v}}\ar@{}[dr]|{\stackrel{\scriptstyle\stackrel{\sim}{\Leftarrow}}{\rho}} &B\ar[d]^v
\\
A\ar[r]_u&C
}
$$ with arrows $u,v\in\frakW$, the composite $u\overline{v}$ is again in $\frakW$.      \end{defn}

Since $\rho$ is invertible, {\bf [WB5]} will imply that   $v\overline{u}\in\frakW$ as well. 

\begin{prop} \label{P:pb-rep}
If $\calB$ has all pseudo pullbacks for cospans in $\frakW$, and $\frakW$ satisfies conditions {\bf [WB1]}--{\bf[WB5]}, is pullback closed, and all arrows in $\frakW$ are co-full,
then each 2-cell in $\calB(\frakW^{-1})$ has a representative with the left-hand 2-cell a pseudo pullback.
\end{prop}

\begin{proof}
For any 2-cell diagram,
\begin{equation}\label{any-2-cell}
\xymatrix{
&A'\ar[dl]_{v}\ar[dr]^f\\
A\ar@{}[r]|{\alpha \Downarrow }&C\ar[u]_{u}\ar[d]^{u'}\ar@{}[r]|{\beta  \Downarrow} & B
\\
&A''\ar[ul]^{v'}\ar[ur]_{f'}
}
\end{equation}
the pseudo-pullback square
$$
\xymatrix{
P\ar[r]^{\overline{v}'}\ar@{}[dr]|{\stackrel{\scriptstyle\stackrel{\sim}{\Leftarrow}}{\rho_{v,v'}}}\ar[d]_{\overline{v}} & A'\ar[d]^{v}
\\
A''\ar[r]_{v'}& A
}
$$
exists and has the property that $v\overline{v}'\in\frakW$. Hence, by Lemma~\ref{L:ch2cell} there is a representative of (\ref{any-2-cell}) with this pseudo-pullback square as left-hand 2-cell.
\end{proof}

 Moreover, the argument from Theorem~\ref{main5} gives the following. 
 
\begin{prop}\label{P:upb-rep} If $\frakW$ satisfies conditions {\bf [WB1]}--{\bf[WB5]}, is pullback closed, and all arrows in $\frakW$ are co-ff, then  there is a canonical representation for each 2-cell which is unique up to equivalence of the central object.  \end{prop} 

\begin{proof}
The representation using the pseudo pullbacks is canonical and as unique as the choice of pseudo pullbacks.
\end{proof}

We finally show that if $\frakW$ is closed under pseudo pullbacks 
(rather than pullback closed), we can still use pseudo pullbacks to define the 2-cells:

\begin{prop}
 If $\frakW$ satisfies conditions {\bf [WB1]}--{\bf[WB5]}, is closed under pseudo pullbacks, and all arrows in $\frakW$ are co-ff, then the 2-cells in the bicategory of fractions can be uniquely represented by 2-cell diagrams with a chosen pseudo pullback as left-hand 2-cell. 
\end{prop}

\begin{proof}
Let $\widehat{\frakW}$ be the class of arrows generated from $\frakW$
under composition and closure under 2-isomorphisms.
Then $\widehat{\frakW}$ satisfies the stronger bicategory of fractions axioms, is pullback-closed and its arrows are still co-ff (this property is preserved by composition and closure under 2-isomorphisms).
So the result from Proposition \ref{P:upb-rep} applies to $\widehat{\frakW}$. Now note that $J\colon \calB(\frakW^{-1})\to\calB(\widehat{\frakW}^{-1})$ is an equivalence of bicategories and in particular, it is 2-full and 2-faithful.
Hence the 2-cells in $\calB(\widehat{\frakW}^{-1})$ between arrows in the image of $J$ are in 1-1 correspondence with 2-cells between the original arrows in $\calB(\frakW^{-1})$.
\end{proof}

Vertical composition of 2-cells  is not simplified by taking representatives with pseudo pullbacks. In fact it is slightly complicated, since  we need to  calculate the vertical composition of the 2-cell diagrams and then construct an equivalent 2-cell diagram that has the pseudo pullback on the left-hand side, using the lifting as in the proof of Lemma~\ref{L:ch2cell}.  
However, the horizontal whiskering operations can be significantly simplified by using pseudo pullbacks,  as we show in the following two subsections.

\subsection{Left Whiskering With Pseudo Pullbacks}  Throughout this subsection, we will assume that $\calB$ has all pseudo pullbacks of cospans in $\frakW$ and that  $\frakW$ satisfies all conditions of Proposition~\ref{P:pb-rep}:  its arrows are co-full, it satisfies conditions {\bf [WB1]}--{\bf[WB5]}, and is pullback closed. We will further require $\frakW$ to be full. (Note that if $\frakW$ is co-fully faithful, this is implied.) We furthermore choose a pseudo pullback
$$\xymatrix@C=3em{P_{u_1,u_2}\ar[r]^{\pi_1}\ar[d]_{\pi_2}\ar@{}[dr]|{\stackrel{\scriptstyle\stackrel{\sim}{\Longleftarrow}}{\rho_{u_1,u_2}}}&A'\ar[d]^{u_1}
\\
A''\ar[r]_{u_2} & A}$$
for each cospan $\xymatrix@1{A'\ar[r]^{u_1}&A&A''\ar[l]_{u_2}}$ in $\frakW$ and will now describe the left whiskering operation for 2-cell representatives with these chosen pseudo pullbacks as left-hand 2-cells.
So we consider whiskering of the form 
\begin{equation}\label{secondwhiskering}
\xymatrix@C=4em{
&A'\ar[dl]_{u_1}\ar[dr]^{f_1}
\\
A\ar@{}[r]|{\rho_{u_1,u_2}\Downarrow} & P_{u_1,u_2}\ar[u]_{\pi_1}\ar[d]^{\pi_2}\ar@{}[r]|{\beta\Downarrow} & B&B'\ar[l]_{v}\ar[r]^g & C
\\
&A''\ar[ul]^{u_2}\ar[ur]_{f_2}
}
\end{equation}
where $\rho_{u_1,u_2}$ is the chosen pseudo pullback.
We construct the composition of the 1-cells using chosen squares $\gamma_1$ and $\gamma_2$ as in Section~\ref{left_whiskering},
$$
\xymatrix{
D'\ar[d]_{\overline{v}_1}\ar[r]^{\overline{f}_1}\ar@{}[dr]|{\stackrel{\gamma_1}{\Leftarrow}} & B'\ar[d]^{v}\ar@{}[drr]|{\mbox{and}} && D''\ar[d]_{\overline{v}_2}\ar[r]^{\overline{f}_2}\ar@{}[dr]|{\stackrel{\gamma_2}{\Leftarrow}} & B'\ar[d]^{v}
\\
A'\ar[r]_{f_1} & B && A''\ar[r]_{f_2} &B
}
$$
such that $w_1:=u_1\overline{v}_1$ and $w_2:=u_2\overline{v}_2$ are in $\frakW$.
Let 
$$
\xymatrix@C=3em@R=3em{
P_{w_1,w_2}\ar@{}[dr]|{\stackrel{\scriptstyle\stackrel{\sim}{\Leftarrow}}{\rho_{w_1,w_2}}}\ar[r]^{\pi_1'}\ar[d]_{\pi_2'}&D' \ar[d]^{w_1=u_1\overline{v}_1}
\\
D''\ar[r]_{w_2=u_2\overline{v}_2}&A
}
$$ be the chosen pseudo pullback.
Then there is a unique arrow $h\colon P_{w_1,w_2}\to P_{u_1,u_2}$ such that $\pi_1h=\overline{v}_1\pi_1'$, $\pi_2h=\overline{v}_2\pi_2'$ and $\rho_{u_1,u_2}h=\rho_{w_1,w_2}$.
Finally, let $\tilde\beta\colon \overline{f}_1\pi_1'\Rightarrow\overline{f}_2\pi_2'$ be the lifting of the diagram,
$$
\xymatrix{
D'\ar@/^5ex/[2,3]^{\overline{f}_1}\ar[dr]^{\overline{v}_1}
\\
&A'\ar[dr]^{f_1}\ar@{}[r]|(.4){\stackrel{\gamma_1}{\Leftarrow}}&
\\
P_{w_1,w_2}\ar[r]^h\ar[uu]_{\pi_1'}\ar[dd]^{\pi_2'}\ar@{}[ur]|= \ar@{}[dr]|= & P_{u_1,u_2}\ar[u]_{\pi_1}\ar[d]^{\pi_2}\ar@{}[r]|{\beta\Downarrow}&&B'\ar[l]^v
\\
&A''\ar[ur]_{f_2}\ar@{}[r]|(.4){\stackrel{\gamma_2^{-1}}{\Rightarrow}}&
\\
D''\ar[ur]_{\overline{v}_2}\ar@/_5ex/[-2,3]_{\overline{f}_2}
}
$$ with respect to $v$ (this exists because we assume that $\frakW$ is full).
Then the result of whiskering as in (\ref{secondwhiskering}) is given by
\begin{equation}\label{result_secondwhiskering}
\xymatrix{
&A'\ar@/_1.5ex/[dl]_{u_1}&D'\ar[l]_{\overline{v}_1}\ar@/^2ex/[dr]^{\overline{f}_1}
\\
A\ar@{}[rr]|{\rho_{w_1,w_2}\Downarrow}&&P_{w_1,w_2}\ar[u]_{\pi_1'}\ar[d]^{\pi_2'}\ar@{}[r]|{\tilde\beta\Downarrow}&B'\ar[r]^g &C
\\
&A''\ar@/^1.5ex/[ul]^{u_2}&D''\ar[l]^{\overline{v}_2}\ar@/_2ex/[ur]_{\overline{f}_2}
}
\end{equation}

\begin{lma}\label{leftpbwhiskering}
Diagram {\em (\ref{result_secondwhiskering})} is equivalent to the diagram {\em (\ref{leftwhiskeringgeneral})} obtained for this type of whiskering in Section~\ref{left_whiskering}.
\end{lma}

\begin{proof}
 It was shown in \cite{Matteo1} that any pair of choices of the squares and liftings in the composition construction of Section~\ref{left_whiskering} give equivalent 2-cell diagrams as long as we use the composition squares from {\bf[C2]} of Notation~\ref{choices} for the composition of the 1-cells and  the squares have the right properties.  The only place where the chosen squares are essential is in the composition of the 1-cells, so with the exception of the cells $\gamma_1$ and $\gamma_2$ we can replace all cells used in the whiskering algorithm from Section~\ref{left_whiskering} with cells and squares we have just constructed above. So we will redo the construction from Section~\ref{left_whiskering} and use the universal properties of the pseudo pullbacks to adjust the squares to obtain a 
2-cell diagram that is clearly equivalent to (\ref{result_secondwhiskering}).

Recall that in Section~\ref{left_whiskering} we used chosen squares $\delta_1$, $\delta_2$ and $\delta_3$ to obtain diagrams
\begin{equation}\label{lefthand_old}
\xymatrix@C=3.9em{
&&\ar[dl]_{\overline{v}_1}& \ar[dr]|{\overline{v}_1}\ar@{}[ddr]|{\delta_1\Downarrow}\ar@/^2.5ex/[2,3]^{\overline{f}_1}
\\
&\ar[dl]_{u_1} \ar@{}[r]|{\delta_1\Downarrow}& \ar[u]_{{\pi}_1'\tilde{u}_1}\ar[dl]|{\tilde{v}_1}& \ar[u]^{{\pi}_1'\tilde{u}_1}\ar[dr]|{\tilde{v}_1} & \ar@{}[r]|{\stackrel{\gamma_1}{\Leftarrow}}\ar[dr]|{f_1} &
\\
\ar@{}[r]|{\rho_{u_1,u_2}\Downarrow} & P_{u_1,u_2}\ar[u]_{\pi_1}\ar[d]^{\pi_2}\ar@{}[r]|{\delta_3^{-1}\Downarrow} & \ar[u]_{t_1}\ar[d]^{t_2}T\ar@{}[r]|{\mbox{and}}& T\ar[u]^{t_1}\ar[d]_{t_2}\ar@{}[r]|{\delta_3^{-1}\Downarrow} &\ar[u]_{\pi_1}\ar[d]^{\pi_2}\ar@{}[r]|{\beta\Downarrow} &&\ar[l]^v\ar[r]_g &
\\
& \ar[ul]^{u_2}\ar@{}[r]|{\delta_2^{-1}\Downarrow} & \ar[ul]|{\tilde{v}_2} \ar[d]^{{\pi}'_2\tilde{u}_2} & \ar[ur]|{\tilde{v}_2}\ar@{}[r]|{\delta_2^{-1}\Downarrow}\ar[d]_{{\pi}_2'\tilde{u}_2} & \ar[ur]|{f_2}\ar@{}[r]|{\stackrel{\gamma_2^{-1}}{\Rightarrow}} &
\\
&&\ar[ul]^{\overline{v}_2} & \ar[ur]|{\overline{v}_2}\ar@/_2.5ex/[-2,3]_{\overline{f}_2}
}
\end{equation}

By the universal property of the pseudo pullback there is an arrow $\tilde{t}\colon T\to P_{u_1,u_2}$ such that the following diagram pastes to the same 2-cell as the first diagram in (\ref{lefthand_old}),
$$
\xymatrix@C=4em{
&&\ar[dl]_{\overline{v}_1}
\\
&\ar[dl]_{u_1} \ar@{}[r]|{=}& \ar[u]_{{\pi}_1'\tilde{u}_1}
\\
\ar@{}[r]|{\rho_{u_1,u_2}\Downarrow} & P_{u_1,u_2}\ar[u]_{\pi_1}\ar[d]^{\pi_2} & \ar[u]_{t_1}\ar[d]^{t_2}T\ar[l]_{\tilde{t}}
\\
& \ar[ul]^{u_2}\ar@{}[r]|{=} & \ar[d]^{{\pi}_2'\tilde{u}_2}
\\
&&\ar[ul]^{\overline{v}_2}
}
$$
We now replace the chosen squares $\delta_1$, $\delta_2$ by the new commuting squares in this diagram and let $\delta_3=\mbox{id}_{\tilde{t}}$. 
We obtain the following diagram,
$$
\xymatrix@C=4em{
&&\ar[dl]_{\overline{v}_1} \ar[dr]^{\overline{v}_1}\ar@{}[ddr]|{=}\ar@/^3.5ex/[2,3]^{\overline{f}_1}& 
\\
&\ar[dl]_{u_1} \ar@{}[r]|{=}& \ar[u]_{\overline{\pi}_1} &\ar[dr]^{f_1}\ar@{}[r]|{\stackrel{\gamma_1}{\Leftarrow}}&
\\
\ar@{}[r]|{\rho_{u_1,u_2}\Downarrow} & P_{u_1,u_2}\ar[u]_{\pi_1}\ar[d]^{\pi_2} & \ar[u]_{t_1}\ar[d]^{t_2}T\ar[l]^{\tilde{t}}\ar[r]_{\tilde{t}} & P_{u_1,u_2}\ar[u]_{\pi_1}\ar[d]^{\pi_2}\ar@{}[r]|\beta &&\ar[l]^v\ar[r]_g &
\\
& \ar[ul]^{u_2}\ar@{}[r]|{=} & \ar[d]^{\overline{\pi}_2}\ar@{}[r]|=&\ar[ur]_{f_2}\ar@{}[r]|{\stackrel{\gamma_2^{-1}}{\Rightarrow}}&
\\
&&\ar[ul]^{\overline{v}_2}\ar[ur]_{\overline{v}_2}\ar@/_3.5ex/[-2,3]_{\overline{f}_2}
}
$$
This is almost a 2-cell diagram: we just need to take a lifting $\tilde\beta'\colon \overline{f}_1\overline{\pi}_1t_1\Rightarrow\overline{f}_2\overline{\pi}_2t_2$ of the right-hand side with respect to $v$ (which is possible since $v$ is full)  

To show that the resulting 2-cell,
\begin{equation}\label{resulting}
\xymatrix@C=4em{
&\ar@/_1.5ex/[dl]_{u_1}&\ar[l]_{\overline{v}_1}\ar@/^2ex/[dr]^{\overline{f}_1}
\\
\ar@{}[r]|{\rho_{u_1,u_2}\Downarrow}&P_{u_1,u_2}\ar[u]_{\pi_1}\ar[d]^{\pi_2}&T\ar[l]^{\tilde{t}}\ar[u]_{\overline{\pi}_1t_1}\ar[d]^{\overline{\pi}_2t_2}\ar@{}[r]|{\tilde\beta'\Downarrow}&\ar[r]^g &
\\
&\ar@/^1.5ex/[ul]^{u_2}&\ar[l]^{\overline{v}_2}\ar@/_2ex/[ur]_{\overline{f}_2}
}
\end{equation} 
is equivalent to (\ref{result_secondwhiskering}),
note that there is a unique arrow $t'\colon T\to P_{w_1,w_2}$ such that 
$\rho_{w_1,w_2}t'=\rho_{u_1,u_2}\tilde{t}$. Now $\tilde\beta t'$ is another lifting of the 
right-hand side in (\ref{secondwhiskering}), so the diagrams with $\tilde\beta'$ and $\tilde\beta t'$ on the left-hand side are equivalent.
Hence, (\ref{result_secondwhiskering}) and (\ref{resulting}) are equivalent.
\end{proof}

\subsection{Right Whiskering With Pullbacks}  Throughout this section, we will assume all conditions of Proposition~\ref{P:pb-rep}: $\calB$ has all pseudo pullbacks of cospans in $\frakW$ (and we will use the chosen pseudo pullbacks as in the previous subsection), $\frakW$ satisfies  conditions {\bf [WB1]}--{\bf[WB5]}, is pullback closed, and its arrows are co-full. Furthermore, we will require $\frakW$ to be full as well.
We now consider right whiskering for 2-cell representatives where the left-hand 2-cell is a chosen pseudo pullback. So we start with the composition 
\begin{equation}\label{firstwhiskering}
\xymatrix@C=4em{
&&&B'\ar[dl]_{v_1}\ar[dr]^{g_1}
\\
A & A'\ar[l]_{u}\ar[r]^f & B\ar@{}[r]|{\rho_{v_1,v_2}\Downarrow} & P_{v_1,v_2} \ar@{}[r]|{\beta\Downarrow} \ar[u]_{\pi_1}\ar[d]^{\pi_2} & C
\\
&&&B''\ar[ul]^{v_2}\ar[ur]_{g_2}
}
\end{equation}
where $P_{v_1,v_2}$ is the chosen pseudo pullback of $v_1$ and $v_2$.
First we construct the composition of the 1-cells using chosen squares {\bf[C2]}
$$
\xymatrix{
D'\ar[d]_{\overline{v}_1}\ar[r]^{\overline{f}_1}\ar@{}[dr]|{\stackrel{\gamma_1}{\Leftarrow}} & B'\ar[d]^{v_1}\ar@{}[drr]|{\mbox{and}} && D''\ar[d]_{\overline{v}_2}\ar[r]^{\overline{f}_2}\ar@{}[dr]|{\stackrel{\gamma_2}{\Leftarrow}} & B''\ar[d]^{v_2}
\\
A'\ar[r]_f & B && A'\ar[r]_f &B
}
$$
such that $u_1:=u\overline{v}_1$ and $u_2:=u\overline{v}_2$ are in $\frakW$ as in Section~\ref{right_whiskering}.
Let 
$$\xymatrix@C=3.5em{P_{u_1,u_2}\ar[r]^{\overline\pi_1}\ar[d]_{\overline\pi_2}\ar@{}[dr]|{\stackrel{\scriptstyle\stackrel{\sim}{\Longleftarrow}}{\rho_{u_1,u_2}}} & D'\ar[d]^{u_1}\\D''\ar[r]_{u_2}&A}$$
be the chosen pseudo pullback of $u_1$ and $u_2$. Note that $\rho_{u_1,u_2}\colon u\overline{v_1}\overline\pi_1\Rightarrow u\overline{v}_2\overline\pi_2$. Since $u$ is full, there is a lifting $\tilde\rho_{u_1,u_2}\colon \overline{v}_1\overline\pi_1\Rightarrow\overline{v}_2\overline\pi_2$.
This cell can be pasted with $\gamma_1$ and $\gamma_2^{-1}$ to form
\begin{equation}\label{delta-pasting}
\xymatrix@C=4em{
& D'\ar[r]^{\overline{f}_1}\ar[d]_{\overline{v}_1}\ar@{}[dr]|{\stackrel{\gamma_1}{\Leftarrow}} & B'\ar[dr]^{v_1}
\\
P_{u_1,u_2} \ar[ur]^{\overline\pi_1}\ar[dr]_{\overline\pi_2}\ar@{}[r]|{\tilde\rho_{u_1,u_2}\Downarrow} & A'\ar[rr]_f && B
\\
& D''\ar@{}[ur]|{\stackrel{\gamma_2^{-1}}{\Rightarrow}}\ar[r]_{\overline{f}_2}\ar[u]^{\overline{v}_2} & B''\ar[ur]_{v_2}
}
\end{equation}
By the universal property of the pseudo pullback $P_{v_1,v_2}$, there is a unique arrow 
\begin{equation}\label{arrow-h}
h\colon P_{u_1,u_2}\to P_{v_1,v_2}\mbox{ such that }\pi_1h=\overline{f}_1\overline\pi_1\mbox{ and } \pi_2 h=\overline{f}_2\overline\pi_2
\end{equation}
and furthermore, $\rho_{v_1,v_2}h$ is equal to the pasting of (\ref{delta-pasting}). We claim that the following 2-cell diagram represents the result of whiskering (\ref{firstwhiskering}):
\begin{equation}\label{result_firstwhiskering}
\xymatrix@C=4em{
&\ar[dl]_{u_1} D'\ar[r]^{\overline{f}_1}\ar@{}[dr]|= & B'\ar[dr]^{g_1}
\\
A\ar@{}[r]|{\rho_{u_1,u_2}\Downarrow} & P_{u_1,u_2}\ar[r]_h\ar[u]_{\overline\pi_1}\ar[d]^{\overline\pi_2}\ar@{}[dr]|= & P_{v_1,v_2}\ar[u]_{\pi_1}\ar[d]^{\pi_2}\ar@{}[r]|{\beta\Downarrow} & C
\\
&D''\ar[ul]^{u_2}\ar[r]_{\overline{f}_2}&B''\ar[ur]_{g_2}
}
\end{equation}

\begin{lma}   
Diagram (\ref{result_firstwhiskering}) is equivalent to the diagram (\ref{rightwhiskeringgeneral}) obtained for this type of whiskering in Section~\ref{right_whiskering}.
\end{lma}

\begin{proof}
Again, we use the result from  \cite{Matteo1} that the equivalence classes of the resulting 2-cell diagrams in the whiskering constructions and vertical composition construction do not depend on the choice of the squares and liftings used as long as we use the chosen composition of 1-cells and the  appropriate arrows are in $\frakW$.  We will now go through the algorithm of Section~\ref{right_whiskering} and substitute the cells above.  We will show  that the result is precisely  (\ref{result_firstwhiskering}). 

In (\ref{deltas}), we take for $\delta_1$ and $\delta_2$  respectively,
$$
\xymatrix{
\ar[r]^h\ar[d]_{\overline{\pi}_1}\ar@{}[dr]|= & \ar[d]^{\pi_1}&&\ar[r]^h\ar[d]_{\overline{\pi}_2}\ar@{}[dr]|= & \ar[d]^{\pi_2}
\\
\ar[d]_{\overline{v}_1}\ar[r]|{\overline{f}_1}\ar@{}[dr]|{\stackrel{\gamma_1}{\Longleftarrow}} & \ar[d]^{v_1}&\mbox{and}& \ar[d]_{\overline{v}_2}\ar[r]|{\overline{f}_2}\ar@{}[dr]|{\stackrel{\gamma_2}{\Longleftarrow}} & \ar[d]^{v_2}
\\
\ar[r]_f&&&\ar[r]_f&
}
$$
This allows us to take $r_1$ and $r_2$ to be identity arrows and $t_i=\overline\pi_i$ for $i=1,2$. Furthermore, $\varphi_i$ is given by
$$
\xymatrix{
\ar[r]^h\ar[d]_{\overline{\pi}_i}\ar@{}[dr]|= & \ar[d]^{\pi_i}
\\
\ar[r]_{\overline{f}_i}&
}
$$
and $\varepsilon_i=\mbox{id}_{\overline{v}_i\overline\pi_i}$, for $i=1,2$.
The next step is then to compare the pastings,
$$
\xymatrix@C=3em{
\ar[d]_{\overline\pi_1} \ar[r]^h \ar@{}[dr]|=  &\ar[d]^{\pi_1}&&&\ar[dl]_{\overline\pi_2}\ar[r]^h\ar@{}[d]|=&\ar[dl]|{\pi_2}\ar[d]^{\pi_1}
\\
\ar[r]|{\overline{f}_1}\ar@{}[dr]|{\stackrel{\gamma_1}{\Longleftarrow}} \ar[d]_{\overline{v}_1}& \ar[d]^{v_1}&\mbox{and} &\ar[dr]_{\overline{v}_2}\ar[r]|{\overline{f}_2}\ar@{}[drr]|{\stackrel{\gamma_2}{\Longleftarrow}}& \ar[dr]|{v_2}\ar@{}[r]|{\stackrel{\rho_{v_1,v_2}}{\Longleftarrow}}&\ar[d]^{v_1}
\\
\ar[r]_f& && &\ar[r]_f&}
$$
Here we may choose $p$ and $q$ to be identity arrows, $\tau=\mbox{id}_h$
and $\tilde\alpha=\tilde\rho_{u_1,u_2}$, since
$$
\xymatrix@C=3em{
&\ar[dl]_{\overline\pi_2}\ar[d]|{\overline\pi_1}\ar[r]^h\ar@{}[dr]|= &\ar[d]^{\pi_1}&&&\ar[dl]_{\overline\pi_2}\ar[r]^h\ar@{}[d]|=&\ar[dl]|{\pi_2}\ar[d]|{\pi_1}
\\
\ar@{}[r]|{\stackrel{\tilde\rho_{u_1,u_2}}{\Longleftarrow}}\ar[dr]_{\overline{v}_2}& \ar[d]|{\overline{v}_1}\ar@{}[dr]|{\stackrel{\gamma_1}{\Longleftarrow}}\ar[r]|{\overline{f}_1}&\ar[d]^{v_1}&=&\ar[dr]_{\overline{v}_2}\ar[r]|{\overline{f}_2}\ar@{}[drr]|{\stackrel{\gamma_2}{\Longleftarrow}}& \ar[dr]|{v_2}\ar@{}[r]|{\stackrel{\rho_{v_1,v_2}}{\Longleftarrow}}&\ar[d]^{v_1}
\\
&\ar[r]_f & &&&\ar[r]_f&
}$$
by (\ref{delta-pasting}) and (\ref{arrow-h}).

Omitting the identity coherence cells, the resulting 2-cell diagram is
\begin{equation}\label{rwhiskeringwithpb}
\xymatrix@C=4em@R=5em{
&&&\ar@/_4ex/[2,-2]_{\overline{v}_1}\ar@/^0.5ex/[1,2]^{\overline{f}_1}&&
\\
&&\ar[dl]|{\overline{v}_1\overline\pi_1}&\ar[l]\ar@{}[ul]|{\mbox{\scriptsize id}} \ar[r]\ar@{}[ur]|{\mbox{\scriptsize id}}\ar[u]_{\overline\pi_1}&\ar[dr]^h&\ar@/^.5ex/[dr]^{g_1}
\\
&\ar[l]_u\ar@{}[r]|{\tilde\rho_{u_1,u_2}\Downarrow} & \ar[u]\ar[d] \ar@{}[ur]|{\mbox{\scriptsize id}}\ar@{}[dr]|{\mbox{\scriptsize id}} &\ar[l] \ar[u]\ar[d]\ar[r]\ar@{}[ur]|{\mbox{\scriptsize id}}\ar@{}[dr]|{\mbox{\scriptsize id}} &\ar[u]\ar[d] \ar@{}[r]|{\mbox{\scriptsize id}}&\ar[u]_{\pi_1}\ar[d]^{\pi_2}\ar@{}[r]|{\beta\Downarrow} &
\\
&&\ar[ul]|{\overline{v}_2\overline{\pi}_2} & \ar[d]^{\overline\pi_2}  \ar[r]\ar[l] \ar@{}[dr]|{\mbox{\scriptsize id}} \ar@{}[dl]|{\mbox{\scriptsize id}}
&\ar[ur]_h&\ar@/_0.5ex/[ur]_{g_2}
\\
&&&\ar@/^4ex/[-2,-2]^{\overline{v}_2}\ar@/_0.5ex/[urr]_{\overline{f}_2}&&
}
\end{equation}
where all unlabeled arrows are identity arrows.
Composing the cells in both the left-hand side and the right-hand side of this diagram gives us the 2-cell diagram in (\ref{result_firstwhiskering}) as required.
\end{proof}

\subsection{Horizontal Composition of 2-Cell Diagrams with Pseudo Pullbacks}\label{horcomppb}

Suppose that we have two 2-cells that we want to compose:
\begin{equation}\label{composable-2-cells}
\xymatrix@C=3.5em{
&A'\ar[dl]_{u_1}\ar[dr]^{f_1}&&B'\ar[dl]_{v_1}\ar[dr]^{g_1}
\\
A\ar@{}[r]|{\rho_{u_1,u_2}\Downarrow}&P_{u_1,u_2}\ar[u]_{\pi_{A'}}\ar[d]^{\pi_{A''}}\ar@{}[r]|{\beta\Downarrow}&B\ar@{}[r]|{\rho_{v_1,v_2}\Downarrow} &P_{v_1,v_2}\ar[u]_{\pi_{B'}}\ar[d]^{\pi_{B''}}\ar@{}[r]|{\gamma\Downarrow}&C
\\
&A''\ar[ul]^{u_2}\ar[ur]_{f_2} && B''\ar[ul]^{v_2}\ar[ur]_{g_2}
}
\end{equation}

The horizontal composition of these two general 2-cell diagrams is rather involved, being a combination of two whiskering operations and a vertical composition. However, for 2-cell diagrams with pseudo pullbacks as left-hand cells, the right-hand side of the horizontal composition can be calculated as a lifting with respect to $v_1$ of $\beta$ composed with suitable invertible 2-cells, whiskered with $g_1$ and then post-composed with $\gamma$. If furthermore, $\beta$ is invertible, the horizontal composition can be calculated by using two universal arrows obtained from the two pseudo-pullback squares in the initial diagram, whiskered with $\gamma$.  
We describe this here.

Let $\delta_1$ and $\delta_2$  be chosen squares (as in [{\bf C4}]) such that 
$u_1\overline{v}_1$ and $u_2\overline{v}_2$ are in $\frakW$, as in the following diagram.  
$$
\xymatrix@C=4em{
&& D\ar[dl]_{\overline{v}_1}\ar[dr]^{\overline{f}_1}\ar@{}[dd]|{\stackrel{\delta_1}{\Leftarrow}}
\\
&A'\ar[dl]_{u_1}\ar[dr]^{f_1}&&B'\ar[dl]_{v_1}\ar[dr]^{g_1}
\\
A\ar@{}[r]|{\rho_{u_1,u_2}\Downarrow}&P_{u_1,u_2}\ar[u]_{\pi_{A'}}\ar[d]^{\pi_{A''}}\ar@{}[r]|{\beta\Downarrow}&B\ar@{}[r]|{\rho_{v_1,v_2}\Downarrow} &P_{v_1,v_2}\ar[u]_{\pi_{B'}}\ar[d]^{\pi_{B''}}\ar@{}[r]|{\gamma\Downarrow}&C
\\
&A''\ar[ul]^{u_2}\ar[ur]_{f_2} && B''\ar[ul]^{v_2}\ar[ur]_{g_2}
\\
&&D'\ar[ul]^{\overline{v}_2}\ar[ur]_{\overline{f}_2}\ar@{}[uu]|{\stackrel{\delta_2}{\Leftarrow}}
}
$$
(Note that this diagram is not a pasting diagram.)
The left-hand side of the composed 2-cell diagram will be the chosen pseudo pullback $\rho_{u_1\overline{v}_1,u_2\overline{v}_2}$.   By the universal property of $\rho_{u_1,u_2}$, we obtain a unique arrow $$w_{u_1,u_2}\colon P_{u_1\overline{v}_1,u_2\overline{v}_2}\to P_{u_1,u_2}$$ such that $\rho_{u_1,u_2}w_{u_1,u_2}=\rho_{u_1\overline{v}_1,u_2\overline{v}_2}$, 
$$ \xymatrix@C=3.5em{
&& D\ar[dl]_{\overline{v}_1}&&&D \ar[dl]_{\overline{v}_1}
\\
&A'\ar[dl]_{u_1}\ar@{}[dr]|=& & & A' \ar[dl]_{u_1}
\\
A\ar@{}[r]|{\rho_{u_1,u_2}\Downarrow}&P_{u_1,u_2}\ar[u]_{\pi_{A'}}\ar[d]^{\pi_{A''}}&P_{u_1\overline{v}_1,u_2\overline{v}_2}\ar[uu]_{\pi_D}\ar[dd]^{\pi_{D'}}\ar[l]_{w_{u_1,u_2}}\ar@{}[r]|{\textstyle=}& A\ar@{}[rr]|{\rho_{u_1\overline{v}_1,u_2\overline{v}_2}\Downarrow}&&P_{u_1\overline{v}_1,u_2\overline{v}_2}\ar[uu]_{\pi_D}\ar[dd]^{\pi_{D'}}
\\
&A''\ar[ul]^{u_2}\ar@{}[ur]|= & & & A''\ar[ul]^{u_2}
\\
&&D'\ar[ul]^{\overline{v}_2} & && D'\ar[ul]^{\overline{v}_2}
}
$$
The arrow $w_{u_1,u_2}$ can be used to construct the following pasting diagram,
$$
\xymatrix{
D\ar[r]^{\overline{f}_1}\ar[dr]_{\overline{v}_1} & B'\ar@{}[d]|{\stackrel{\delta_1}{\Leftarrow}} \ar[ddr]^{v_1}
\\
& A'\ar[dr]|{f_1}
\\
P_{u_1\overline{v}_1,u_2\overline{v}_2}\ar@{}[ur]|=\ar@{}[dr]|=\ar[uu]^{\pi_D}\ar[dd]_{\pi_{D'}}\ar[r]^{w_{u_1,u_2}}&P_{u_1,u_2}\ar@{}[r]|{\beta\Downarrow} \ar[u]^{\pi_{A'}}\ar[d]_{\pi_{A''}}& B
\\
&A'' \ar[ur]|{f_2}\ar@{}[d]|{\stackrel{\delta_2^{-1}}{\Rightarrow}}
\\
D'\ar[ur]^{\overline{v}_2}\ar[r]_{\overline{f}_2} & B''\ar[uur]_{v_2}
}
$$
If $\beta$ is invertible, the universal property of the pseudo pullback $P_{v_1,v_2}$ gives rise to a unique arrow $$w_{v_1,v_2}\colon P_{u_1\overline{v}_1,u_2\overline{v}_2}\to P_{v_1,v_2}$$ 
such that 
$$
\xymatrix@C=3em{
D\ar[r]^{\overline{f}_1}\ar[dr]_{\overline{v}_1} & B'\ar@{}[d]|{\stackrel{\delta_1}{\Leftarrow}} \ar[ddr]^{v_1} & && D\ar[r]^{\overline{f}_1} & B'\ar[ddr]^{v_1}
\\
& A'\ar[dr]|{f_1}
\\
P_{u_1\overline{v}_1,u_2\overline{v}_2}\ar@{}[ur]|=\ar@{}[dr]|=\ar[uu]^{\pi_D}\ar[dd]_{\pi_{D'}}\ar[r]^{w_{u_1,u_2}}&P_{u_1,u_2}\ar@{}[r]|{\beta\Downarrow} \ar[u]^{\pi_{A'}}\ar[d]_{\pi_{A''}}& B&=& P_{u_1\overline{v}_1,u_2\overline{v}_2}\ar@{}[uur]|=\ar@{}[ddr]|=\ar[uu]^{\pi_D}\ar[dd]_{\pi_{D'}}\ar[r]^{w_{v_1,v_2}} & P_{v_1,v_2}\ar[uu]^{\pi_{B'}}\ar[dd]_{\pi_{B''}}\ar@{}[r]|{\rho_{v_1,v_2}\Downarrow} & B
\\
&A'' \ar[ur]|{f_2}\ar@{}[d]|{\stackrel{\delta_2^{-1}}{\Rightarrow}}
\\
D'\ar[ur]^{\overline{v}_2}\ar[r]_{\overline{f}_2} & B''\ar[uur]_{v_2} & && D'\ar[r]_{\overline{f}_2} & B''\ar[uur]_{v_2}
}
$$
Then the 2-cell diagram representing the horizontal composition of the 2-cell diagrams in (\ref{composable-2-cells}) is
\begin{equation}\label{horcomposition}
\xymatrix@C=5em{
&&D\ar[dll]_{u_1\overline{v}_1}\ar[r]^{\overline{f}_1} & B'\ar[dr]^{g_1}
\\
A\ar@{}[rr]|{\rho_{u_1\overline{v}_1,u_2\overline{v}_2}\Downarrow} && P_{u_1\overline{v}_1,u_2\overline{v}_2}\ar@{}[ur]|=\ar@{}[dr]|=\ar[u]_{\pi_{D}}\ar[d]^{\pi_{D'}} \ar[r]^{w_{v_1,v_2}} & P_{v_1,v_2}\ar[u]_{\pi_{B'}}\ar[d]^{\pi_{B''}}\ar@{}[r]|{\gamma\Downarrow} & C\\
&&D'\ar[ull]^{u_2\overline{v}_2}\ar[r]_{\overline{f}_2} & B''\ar[ur]_{g_2}
}
\end{equation}
The full details that  diagram (\ref{horcomposition}) is indeed the desired horizontal composition of the composable 2-cells in (\ref{composable-2-cells}) are given in Appendix \ref{horcomp}.

If $\beta$ is not invertible, we cannot use the universal property of the pseudo pullback $P_{u_1,u_2}$ as described above and we do not obtain such a nice reduction, but we will present the horizontal composition for that case in Appendix \ref{horcomp} as well.

\section{Future Directions: An Application to Orbifolds}\label{conclusions}
  In this section, we briefly sketch how the  results in this paper apply to the bicategory of orbigroupoids.  Details will be given in  \cite{next_paper}; here we only give an overview.  
  
One way to define orbifolds is by using the 2-category of orbigroupoids:  \'etale groupoids internal to a category of suitable topological spaces, such as topological manifolds or some more general category of spaces.   Then we consider the class of essential equivalences,  maps that are categorical equivalences internal to the topological category chosen:  they satisfy a suitably topologized version of  being essentially surjective and fully faithful. This bicategory has all pseudo pullbacks for cospans of essential equivalences. For more details, see \cite{WIT,MP}.   We define orbifolds as the bicategory of fractions of orbigroupoids with respect to the class of essential equivalences.  Essential equivalences are both ff and co-ff.
The class of essential equivalences is also pullback closed as in Definition~\ref{D:pbc}, and satisfies  the {\bf BF} conditions from \cite{Pr-comp}. Thus, we can apply the results of  Corollary~\ref{P:2-ff} and Proposition~\ref{P:upb-rep} to get the following: 

\begin{thm}
\begin{enumerate}
\item
The universal map from the 2-category of orbigroupoids to its bicategory of fractions with respect to the class $\frakW$ of essential equivalences,
$$
\xymatrix{
J_{\frakW}\colon \mbox{\bf OrbiGroupoids}\longrightarrow\mbox{\bf OrbiGroupoids}(\frakW^{-1})
}
$$ 
is 2-fully faithful. 
\item Each 2-cell in  $\mbox{\bf OrbiGroupoids}(\frakW^{-1})$ has a unique representation by a 2-cell diagram with any given left-hand side. 
\item Given a choice of pseudo pullbacks for cospans of essential equivalences the 2-cells in $\mbox{\bf OrbiGroupoids}(\frakW^{-1})$ can be uniquely represented by diagrams with these pseudo pullbacks as left-hand 2-cells and horizontal composition can be calculated as in Section~\ref{pullbacks}.
\end{enumerate}
\end{thm}

Furthermore, there is a subclass $\frakC\subset\frakW$ of {\em essential covering maps}, defined by,

\begin{dfn}
Let $\calG$ be an \'etale groupoid. An {\em essential covering map} $$\varphi^{\calU}\colon\calG^*(\calU)\to\calG$$ is determined by a (non-repeating) collection of open subsets $\calU\subseteq \calP(\calG_0)$   which  meets every orbit of $\calG$ (although it may not cover $\calG_0$).   Then $\calG^*(\calU) $ is the groupoid defined by  $\calG^*(\calU)_0=\coprod_{U\in\calU}U$, with  $\varphi^\calU_0\colon\calG(\calU)_0\to\calG_0$ defined by the inclusion maps. Furthermore, the space $\calG(\calU)_1$ and the remaining maps are determined by the pullback diagram
$$
\xymatrix@C=4em@R=1.5em{
\calG(\calU)_1\ar[d]_{(s,t)}\ar[r]^{\varphi^{\calU}_1}&\calG_1\ar[d]^{(s,t)}
\\
\calG(\calU)_0\times\calG(\calU)_0\ar[r]_{\varphi_0^\calU\times\varphi_0^\calU}&\calG_0\times\calG_0}
$$
\end{dfn}

The class $\frakC$ of essential covering maps is locally small and satisfies conditions {\bf[WB1]}--{\bf[WB5]}. As essential equivalences they are also ff and co-ff. So we get a bicategory $\mbox{\bf OrbiGroupoids}(\frakC^{-1})$ with small hom-categories, where $$J_{\frakC}\colon \mbox{\bf OrbiGroupoids}\longrightarrow\mbox{\bf OrbiGroupoids}(\frakC^{-1})$$ is 2-fully faithful. Furthermore, the essential covering maps are weakly initial in the essential equivalences in the sense described in  Definition~\ref{coveringclass}.   Hence,
there is an equivalence of bicategories,
$\mbox{\bf OrbiGroupoids}(\frakC^{-1})\simeq\mbox{\bf OrbiGroupoids}(\frakW^{-1})$. 

Now  $\frakC$ is not pullback-closed.  However, because of this equivalence of bicategories we can use the 2-cell diagrams from 
$\mbox{\bf OrbiGroupoids}(\frakW^{-1})$ as 2-cells between arrows in $\mbox{\bf OrbiGroupoids}(\frakC^{-1})$, and hence represent these by  2-cell diagrams with pseudo pullbacks as left-hand 2-cells;  these are not  necessarily in the shape required of 2-cell diagrams in $\mbox{\bf OrbiGroupoids}(\frakC^{-1})$  because certain composites will not be in $\frakC$, but they can be used as an alternate way to represent the 2-cells in this bicategory.  This allows us to  use the simplified composition  described in Section~\ref{pullbacks}. 
So we conclude:
 
\begin{thm}
\begin{enumerate}
\item The bicategory of fractions of orbigroupoids with respect to essential covering maps, $\mbox{\bf OrbiGroupoids}(\frakC^{-1})$ has small hom-categories. 
\item
The pseudo functor
$
\xymatrix@1{
J_{\frakW}\colon \mbox{\bf OrbiGroupoids}\longrightarrow\mbox{\bf OrbiGroupoids}(\frakC^{-1})
}
$ 
is 2-fully faithful. 
\item Each 2-cell in  $\mbox{\bf OrbiGroupoids}(\frakC^{-1})$ has a unique representation by a 2-cell diagram with any given left-hand side. 
\item Given a choice of pseudo-pullback squares the 2-cells in $\mbox{\bf OrbiGroupoids}(\frakC^{-1})$ can be uniquely represented by diagrams with pseudo pullbacks as left-hand 2-cells,  and horizontal composition can be calculated as in Section~\ref{pullbacks}.
\end{enumerate}
\end{thm}

For further details, proofs, and applications, see \cite{next_paper}.


\begin{appendices}
\section{Associativity Part I:  Associativity 2-cells} \label{assoc1}
The goal of these appendices is to study associativity coherence and well-definedness for composition in $\calB(\frakW^{-1})$.    In Appendix~\ref{assoc1} we will construct the associativity 2-cells,  based on an extension of  Proposition~\ref{anytwosquares}.  In Appendix~\ref{assoc2} we will show that these cells satisfy the coherence pentagon condition.  In Appendix~\ref{well-defined} we verify that all composition operations are well-defined on equivalence classes. In Appendix~\ref{horcomp} we give a proof for the presentation, given in Section \ref{horcomppb}, of the horizontal composition of two 2-cell diagrams with pull-back squares for left-hand 2-cells and where the left 2-cell diagram is invertible. Throughout the appendices,  we assume that $\calB$ is a bicategory and $\frakW$ is a class of arrows satisfying conditions {\bf [WB1]}--{\bf [WB5]}.

Consider the 2-cells $\beta$ and $\gamma$ in Proposition~\ref{anytwosquares}. They give rise to a generalized 2-cell
 in $\calB(\frakW^{-1})$,
$$
\xymatrix@C=4em{
&D_1\ar[dl]_{uv_1}\ar[dr]^{g_1}
\\
X\ar@{}[r]|{u\beta\Downarrow}&E\ar[u]_{s_1}\ar[d]^{s_2}\ar@{}[r]|{\gamma\Downarrow} &A
\\
&D_2\ar[ul]^{uv_2}\ar[ur]_{g_2}
}
$$
We  show that this is the unique cell with this property: if $\beta'$ and $\gamma'$ also satisfy the conditions of Proposition~\ref{anytwosquares}, then the 2-cell diagram defined by $\beta'$ and $\gamma'$ is equivalent to this one.

\begin{prop}\label{unique-2-cell}
For $v\colon C\to X$ and $w\colon A\to B$ both in $\frakW$ and $f\colon C\to B$ any arrow in $\calB$, 
and any two squares,
$$
\xymatrix{
D_1\ar[d]_{w_1}\ar[r]^{f_1} \ar@{}[dr]|{\alpha_1\stackrel{\sim}{\Leftarrow}} 
    & A\ar[d]^w & D_2\ar[d]_{w_2}\ar[r]^{f_2} \ar@{}[dr]|{\alpha_2\stackrel{\sim}{\Leftarrow}} & A\ar[d]^w
\\
C\ar[d]_v\ar[r]_f &B & C\ar[d]_v\ar[r]_f &B
\\
X&&X
}$$
with $vw_1,vw_2\in \frakW$,
there is a unique 2-cell in $\calB(\frakW^{-1})$
\begin{equation}\label{uniquecell}
\xymatrix{
&D_1\ar[dl]_{vw_1}\ar[dr]^{f_1}
\\
X\ar@{}[r]|{v\beta\Downarrow}&E\ar[u]_{s_1}\ar[d]^{s_2}\ar@{}[r]|{\gamma\Downarrow} &A
\\
&D_2\ar[ul]^{vw_2}\ar[ur]_{f_2}
}
\end{equation}
 such that the composites $(f\beta)\cdot(\alpha_1s_1)$ and $(\alpha_2 s_2)\cdot(w\gamma)$ are equal.
\end{prop}

\begin{proof}  Existence is a consequence of Proposition \ref{anytwosquares}, so we need only prove uniqueness.  
Let 
\begin{equation}\label{alternate}
\xymatrix{
&D_1\ar[dl]_{vw_1}\ar[dr]^{f_1}
\\
X\ar@{}[r]|{v\beta'\Downarrow}&E' \ar[u]_{t_1}\ar[d]^{t_2}\ar@{}[r]|{\gamma'\Downarrow} &A
\\
&D_2\ar[ul]^{vw_2}\ar[ur]_{f_2}
}
\end{equation}
be another 2-cell diagram with the property that the composites $(f\beta')\cdot(\alpha_1t_1)$ and $(\alpha_2 t_2)\cdot(w\gamma')$ are equal.
Let 
$$
\xymatrix{\ar[r]^{\overline{t}_1}\ar@{}[dr]|{\stackrel{\delta}{\Leftarrow}}\ar[d]_{\overline{s}_1}&\ar[d]^{vw_1s_1}
\\
\ar[r]_{vw_1t_1}&}
$$
be a square as in condition {\bf [WB3]} and let $\tilde{v}$ with
$$
\xymatrix{
\ar[r]^{\overline{t}_1\tilde{v}}\ar@{}[dr]|{\stackrel{\tilde{\delta}}{\Leftarrow}}\ar[d]_{\overline{s}_1\tilde{v}}&\ar[d]^{s_1}
\\
\ar[r]_{t_1}&}
$$
be a lifting as in {\bf [WB4]} for $\delta$ with respect to $vw_1$.
We use this cell in the following pasting,
$$
\xymatrix@C=4em@R=3em{
\ar[r]^{\overline{t}_1\tilde{v}}\ar@{}[dr]|{\stackrel{\tilde\delta}{\Leftarrow}}\ar[d]_{\overline{s}_1\tilde{v}}& \ar[d]^{s_1}\ar@{}[ddr]|{\stackrel{v\beta^{-1}}{\Leftarrow}}\ar[dr]^{s_2}
\\
\ar[dr]_{t_2}\ar[r]^{t_1}\ar@{}[drr]|{\stackrel{v\beta'}{\Leftarrow}}&\ar[dr]|{vw_1} & \ar[d]^{vw_2}
\\
&\ar[r]_{vw_2}&
}
$$
and then use condition {\bf[WB4]} to obtain an arrow $\overline{v}$ and a cell 
$$\xymatrix{\ar@{}[dr]|{\stackrel{\varepsilon}{\Leftarrow}}\ar[r]^{\overline{t}_1\tilde{v}\overline{v}}\ar[d]_{\overline{s}_1\tilde{v}\overline{v}}&\ar[d]^{s_2}
\\
\ar[r]_{t_2}&}$$ which form a  lifting for this pasting with respect to $vw_2$.
We would like to use the diagram
$$
\xymatrix@R=3em{
&\ar@{}[d]|{\stackrel{\tilde\delta\overline{v}}{\Leftarrow}}
\\
\ar@/^1ex/[ur]^{t_1}\ar@/_1ex/[dr]_{t_2}&\ar[l]_{\overline{s}_1\tilde{v}\overline{v}}\ar[r]^{\overline{t}_1\tilde{v}\overline{v}}\ar@{}[d]|{\stackrel{\varepsilon}{\Leftarrow}}&\ar@/_1ex/[ul]_{s_1}\ar@/^1ex/[dl]^{s_2}
\\
&
}
$$
to show that the two 2-cell diagrams are equivalent. However, we still need to make a couple of small adjustments.

By construction we have that the following pastings are equal:
$$
\xymatrix@C=3.5em{
\ar[r]^{\overline{t}_1\tilde{v}\overline{v}}\ar@{}[dr]|{\stackrel{\tilde\delta\overline{v}}{\Leftarrow}}\ar[d]_{\overline{s}_1\tilde{v}\overline{v}} &\ar[d]^{s_1} &&\ar[r]^{\overline{t}_1\tilde{v}\overline{v}}\ar@{}[dr]|{\stackrel{\varepsilon}{\Leftarrow}}\ar[d]_{\overline{s}_1\tilde{v}\overline{v}}&\ar[d]_{s_2}\ar@{}[dr]|{\stackrel{v\beta}{\Leftarrow}}\ar[r]^{s_1}&\ar[d]^{vw_1}
\\
\ar[r]_{t_1}\ar[d]_{t_2}\ar@{}[dr]|{\stackrel{v\beta'}{\Leftarrow}}&\ar[d]^{vw_1}&=&\ar[r]_{t_2}&\ar[r]_{vw_2}&
\\
\ar[r]_{vw_2}&
}
$$
By Lemma \ref{Matteo1} there is an arrow $v^*$ in $\frakW$ such that 
\begin{equation}\label{beta-equation}
\xymatrix@C=3.5em{
\ar[r]^{\overline{t}_1\tilde{v}\overline{v}v^*}\ar@{}[dr]|{\stackrel{\tilde\delta\overline{v}v^*}{\Leftarrow}}\ar[d]_{\overline{s}_1\tilde{v}\overline{v}v^*} &\ar[d]^{s_1} &&\ar[r]^{\overline{t}_1\tilde{v}\overline{v}v^*}\ar@{}[dr]|{\stackrel{\varepsilon v^*}{\Leftarrow}} \ar[d]_{\overline{s}_1\tilde{v}\overline{v}v^*}&\ar[d]_{s_2}\ar@{}[dr]|{\stackrel{\beta}{\Leftarrow}}\ar[r]^{s_1}&\ar[d]^{w_1}
\\
\ar[r]_{t_1}\ar[d]_{t_2}\ar@{}[dr]|{\stackrel{\beta'}{\Leftarrow}}&\ar[d]^{w_1}&=&\ar[r]_{t_2}&\ar[r]_{w_2}&
\\
\ar[r]_{w_2}&
}
\end{equation}

To obtain the corresponding result with $\gamma, \gamma'$ instead of $\beta, \beta'$,  we need to compose with the arrow $w$ so that the hypothesis of  {\bf[WB4]}  is satisfied.
We will also compose the pasting diagrams we are interested in with the cells $\alpha_2$ and
$\beta'^{-1}$.
This leads to the following calculation,
$$
\xymatrix@C=3em@R=1.8em{
&\ar[r]^{s_1} \ar@{}[d]|{\tilde\delta\overline{v}v^*\Downarrow} & \ar[r]^{f_1} \ar@{}[d]|{\gamma'\Downarrow}&\ar[dr]^w \ar@{}[dd]|{\alpha_2\Downarrow} &&& \ar[r]^{s_1} \ar@{}[d]|{\tilde\delta\overline{v}v^*\Downarrow} & \ar[r]^{f_1} &\ar[dr]^w \ar@{}[ddl]|{\alpha_1t_1\Downarrow} &
\\
\ar[ur]^{\overline{t}_1\tilde{v}\overline{v}v^*} \ar[r]_{\overline{s}_1\tilde{v}\overline{v}v^*} &\ar[ur]|{t_1} \ar[r]|{t_2}\ar[dr]_{t_1}&\ar@{}[d]|{\beta'^{-1}\Downarrow} \ar[dr]|{w_2}\ar[ur]|{f_2}&&\ar@{}[r]|=& \ar[ur]^{\overline{t}_1\tilde{v}\overline{v}v^*} \ar[r]_{\overline{s}_1\tilde{v}\overline{v}v^*} &\ar[ur]|{t_1}\ar[dr]_{t_1}&&&
\\
&&\ar[r]_{w_1}&\ar[ur]_f&&&&  \ar[r]_{w_1}&\ar[ur]_f 
\\
&&&& &&\ar[dr]^{s_1}&&\ar[dr]^w&
\\
&&&&\ar@{}[r]|= &\ar[ur]^{\overline{t}_1\tilde{v}\overline{v}v^*} \ar[dr]_{\overline{s}_1\tilde{v}\overline{v}v^*} \ar@{}[rr]|{\tilde\delta\overline{v}v^*\Downarrow} && \ar[ur]^{f_1}\ar[dr]_{w_1}\ar@{}[rr]|{\alpha_1\Downarrow}&&
\\
&&&& &&\ar[ur]_{t_1} &&\ar[ur]_f
\\
&&&& &&&\ar[r]^{f_1} \ar@{}[ddr]|{\alpha_1s_1\Downarrow} &\ar[dr]^w
\\
&&&&\ar@{}[r]|= &\ar[r]^{\overline{t}_1\tilde{v}\overline{v}v^*}\ar[dr]_{\overline{s}_1\tilde{v}\overline{v}v^*} &\ar@{}[d]|{\tilde\delta\overline{v}v^*\Downarrow}\ar[ur]^{s_1}\ar[dr]|{s_1}&&&
\\
&&&& &&\ar[r]_{t_1} & \ar[r]_{w_1}&\ar[ur]_{f}&
\\
&&&&& &&\ar[r]^{f_1} \ar@{}[d]|{\gamma\Downarrow} &\ar[dr]^{w}
\\
&&&&\ar@{}[r]|= &\ar[dr]_{\overline{s}_1\tilde{v}\overline{v}v^*} \ar[r]^{\overline{t}_1\tilde{v}\overline{v}v^*} & \ar[ur]^{s_1} \ar[r]|{s_2}\ar[dr]|{s_1} \ar@{}[d]|{\tilde\delta\overline{v}v^*\Downarrow} &\ar@{}[d]|{\beta^{-1}\Downarrow}\ar@{}[rr]|{\alpha_2\Downarrow}\ar[ur]|{f_2}\ar[dr]|{w_2} &&
\\
&&&&& &\ar[r]_{t_1} & \ar[r]_{w_1}& \ar[ur]_f
\\
&&&&& &&\ar[dr]^{f_1}\ar@{}[dd]|{\gamma\Downarrow}
\\
&&&&& &\ar[ur]^{s_1} \ar[dr]_{s_2} \ar@{}[dd]|{\varepsilon v^*\Downarrow} &&\ar[dr]^{w}\ar@{}[dd]|{\alpha_2\Downarrow}
\\
&&&&\ar@{}[r]|= &\ar[dr]_{\overline{s}_1\tilde{v}\overline{v}v^*} \ar[ur]^{\overline{t}_1\tilde{v}\overline{v}v^*} &&\ar[ur]^{f_2}\ar[dr]_{w_2}\ar@{}[dd]|{\beta'^{-1}\Downarrow} &&
\\
&&&&& &\ar[ur]^{t_2}\ar[dr]_{t_1} &&\ar[ur]_f
\\
&&&&& & & \ar[ur]_{w_1}
}
$$
where the last equality follows from (\ref{beta-equation}).
Since $\beta'$ and $\alpha_2$ are invertible 2-cells, we conclude that
$$
\xymatrix@C=3em{\ar[r]^{\overline{t}_1\tilde{v}\overline{v}v^*} \ar[d]_{\overline{s}_1\tilde{v}\overline{v}v^*} \ar@{}[dr]|{\tilde\delta\overline{v}v^*\Downarrow} & \ar[d]^{s_1}
\\
\ar[r]_{t_1}\ar[d]_{t_2}\ar@{}[dr]|{\gamma'\Downarrow} & \ar[d]^{f_1}&=& \ar[r]^{\overline{t}_1\tilde{v}\overline{v}v^*} \ar[d]_{\overline{s}_1\tilde{v}\overline{v}v^*} \ar@{}[dr]|{\varepsilon v^*\Downarrow} &\ar[d]|{s_2}\ar[r]^{s_1}\ar@{}[dr]|{\gamma\Downarrow} &\ar[d]^{f_1}
\\
\ar[r]_{f_2}&\ar[r]_{w}&&\ar[r]_{t_2}&\ar[r]_{f_2}& \ar[r]_w&}
$$
By Lemma~\ref{Matteo1} there is an arrow $\widetilde{w}\in\frakW$ such that
$$
\xymatrix@C=3em{\ar[r]^{\widetilde{w}}& \ar[r]^{\overline{t}_1\tilde{v}\overline{v}v^*} \ar[d]_{\overline{s}_1\tilde{v}\overline{v}v^*} \ar@{}[dr]|{\tilde\delta\overline{v}v^*\Downarrow} & \ar[d]^{s_1}
\\
&\ar[r]_{t_1}\ar[d]_{t_2}\ar@{}[dr]|{\gamma'\Downarrow} & \ar[d]^{f_1}&=&\ar[r]^{\widetilde{w}}& \ar[r]^{\overline{t}_1\tilde{v}\overline{v}v^*} \ar[d]_{\overline{s}_1\tilde{v}\overline{v}v^*} \ar@{}[dr]|{\varepsilon v^*\Downarrow} &\ar[d]|{s_2}\ar[r]^{s_1}\ar@{}[dr]|{\gamma} &\ar[d]^{f_1}
\\
&\ar[r]_{f_2}&&&&\ar[r]_{t_2}&\ar[r]_{f_2}& }
$$
Finally, let $r$ be an arrow such that the composition $vw_1s_1\overline{t}_1\tilde{v}\overline{v}v^*\tilde{w}r\in\frakW$.
Then the cells 
$$
\xymatrix@R=3em@C=4.5em{
&\ar@{}[d]|{\stackrel {\tilde\delta\overline{v}v^*\tilde{w}r}{\Leftarrow}}
\\
\ar@/^1.5ex/[ur]^{t_1}\ar@/_1.5ex/[dr]_{t_2}&\ar[l]_{\overline{s}_1\tilde{v}\overline{v}v^*\tilde{w}r}\ar[r]^{\overline{t}_1\tilde{v}\overline{v}v^*\tilde{w}r}\ar@{}[d]|{\stackrel{\varepsilon v^*\tilde{w}r} {\Leftarrow}} &\ar@/_1.5ex/[ul]_{s_1}\ar@/^1.5ex/[dl]^{s_2}
\\
&
}
$$
satisfy the equations to establish the fact that (\ref{uniquecell}) and (\ref{alternate}) 
are equivalent 2-cell diagrams, as claimed.
\end{proof}

\begin{no} \label{connect}
We will say that the 2-cell 
$$
\xymatrix@R=1.9em@C=3em{
&D_1 \ar[dl]_{vw_1}\ar[dr]^{f_1}
\\
X \ar@{}[r]|{v\beta \Downarrow}& E \ar[u]_{s_1}\ar[d]^{s_2} \ar@{}[r]|{\gamma \Downarrow} &A
\\
&D_2\ar[ul]^{vw_2}\ar[ur]_{f_2}
}
$$
above {\em connects} the squares $\alpha_1$ and $\alpha_2$.
\end{no}

\begin{lma}\label{multiplicativity}
Let $v\colon C\to X$ and $w\colon A\to B$ both be in $\frakW$ and $f\colon C\to B$ any arrow in $\calB$,  and let 
$$
\xymatrix{
D_i\ar[r]^{\overline{f}_i}\ar[d]_{\overline{w}_i}\ar@{}[dr]|{\stackrel{\alpha_i}{\Leftarrow}}&A\ar[d]^w
\\
C\ar[d]_{v}\ar[r]_f & B
\\
X}
$$
be invertible 2-cells with $v\overline{w}_i\in\frakW$ for $i=1,2,3$.
For each pair $i,j$, let $(v\beta_{ij},\gamma_{ij})$ be the canonical 2-cell connecting the squares $\alpha_i$ and $\alpha_j$ as given in Proposition~\ref{unique-2-cell},
$$
\xymatrix@C=4em{
&D_i\ar[dl]_{vw_2}\ar[dr]^{f_i}
\\
X\ar@{}[r]|{v\beta_{ij}\Downarrow}&E_{ij}\ar[u]_{s_{ij}}\ar[d]^{t_{ij}}\ar@{}[r]|{\gamma_{ij}\Downarrow} &A
\\
&D_j\ar[ul]^{vw_j}\ar[ur]_{f_j}
}
$$
Then $(v\beta_{ii},\gamma_{ii})$ is the identity 2-cell on the span $(v\overline{w}_i,f)$ and these two cells are closed under vertical composition: $(v\beta_{jk},\gamma_{jk})\cdot(v\beta_{ij},\gamma_{ij})=(v\beta_{ik},\gamma_{ik})$.
\end{lma}

\begin{proof}
Straight forward calculation.
\end{proof}
\begin{prop} \label{assoc-2-cells}  For any path of composable spans:
\begin{equation}\label{composablepath}
\xymatrix{
&\ar[dl]_{w_1}\ar[dr]_{f_1}&&\ar[dl]_{w_2}\ar[dr]_{f_2}&&\ar[dl]_{w_3}\ar[dr]^{f_3}
\\
&&&&&&
}
\end{equation}
there is an associativity 2-cell 
$$\alpha_{(w_3,f_3),(w_2,f_2),(w_1,f_1)}\colon (w_3,f_3)\circ((w_2,f_2)\circ(w_1,f_1))\Rightarrow((w_3,f_3)\circ(w_2,f_2))\circ(w_1,f_1)$$ between the composites as constructed in Section~\ref{newBF2}.  
\end{prop}

\begin{proof} 
If we first compose the left-hand pair and use the choices as described in the construction of $\calB(\frakW^{-1})$, we obtain $(w_3,f_3)\circ((w_2,f_2)\circ(w_1,f_1))$ as the following span,
\begin{equation}\label{leftfirst}
\xymatrix{
&&&\ar[dl]_{\widetilde{w}_3}\ar[ddrr]^{\widetilde{f}_2}\ar@{}[dddr]|{\stackrel{\beta_1}{\Leftarrow}}
\\
&&\ar[dl]_{\overline{w}_2}\ar@{}[dd]|{\stackrel{\alpha_1}{\Leftarrow}}\ar[dr]^{\overline{f}_1}&&
\\
&\ar[dl]_{w_1}\ar[dr]_{f_1}&&\ar[dl]^{w_2}\ar[dr]_{f_2}&&\ar[dl]^{w_3}\ar[dr]^{f_3}
\\
&&&&&&
}
\end{equation}
Note that $w_1\overline{w}_2\widetilde{w}_3\in\frakW$.
If we first compose the right-hand pair we get $((w_3,f_3)\circ(w_2,f_2))\circ(w_1,f_1)$ as the span,
\begin{equation}\label{rightfirst}
\xymatrix{
&&&\ar[ddll]_{\widetilde{w}_2}\ar[dr]^{\widetilde{f}_1}\ar@{}[dddl]|{\stackrel{\beta_2}{\Leftarrow}}
\\
&&&&\ar[dl]_{\overline{w}_3}\ar@{}[dd]|{\stackrel{\alpha_2}{\Leftarrow}}\ar[dr]^{\overline{f}_2}
\\
&\ar[dl]_{w_1}\ar[dr]_{f_1}&&\ar[dl]^{w_2}\ar[dr]_{f_2}&&\ar[dl]^{w_3}\ar[dr]^{f_3}
\\
&&&&&&
}
\end{equation}
where $w_1\widetilde{w}_2\in\frakW$ and $w_2\overline{w}_3\in\frakW$.
The associativity 2-cell will be a vertical composite of two 2-cells going through the intermediate:  
\begin{equation}\label{middle}
\xymatrix{
&&&\ar[dl]_{\overline{\overline{w}}_3}\ar[dr]^{\overline{\overline{f}}_1}\ar@{}[dd]|{\stackrel{\alpha_3}{\Leftarrow}}
\\
&&\ar[dl]_{\overline{w}_2}\ar@{}[dd]|{\stackrel{\alpha_1}{\Leftarrow}}\ar[dr]|{\overline{f}_1}&&\ar[dl]|{\overline{w}_3}\ar@{}[dd]|{\stackrel{\alpha_2}{\Leftarrow}}\ar[dr]^{\overline{f}_2}
\\
&\ar[dl]_{w_1}\ar[dr]_{f_1}&&\ar[dl]^{w_2}\ar[dr]_{f_2}&&\ar[dl]^{w_3}\ar[dr]^{f_3}
\\
&&&&&&
}
\end{equation}
where $\alpha_3$ is chosen as in {\bf  [C4]} with $w_1\overline{w}_2\overline{\overline{w}}_3\in\frakW$; also $w_2\overline{w}_3\in\frakW$ by the choice of $\alpha_2$ as in {\bf  [C4]}.
We construct the associativity 2-cell as a vertical composition  of two 2-cells:   (\ref{leftfirst}) $\Rightarrow$  (\ref{middle})  and 
 (\ref{middle}) $\Rightarrow$ (\ref{rightfirst}).  (Note that by Lemma~\ref{multiplicativity} the resulting associativity 2-cell does not depend on the choice of the square $\alpha_3$.)

 (\ref{leftfirst}) $\Rightarrow$  (\ref{middle}):    the diagrams in  (\ref{leftfirst}) and (\ref{middle}) only differ in the following chosen squares:
$$
\xymatrix@C=4em{
\ar[d]_{\widetilde{w}_3}\ar[rr]^{\widetilde{f}_2}\ar@{}[drr]|{\beta_1\Downarrow}&&\ar[d]^{w_3}\ar@{}[drr]|{\mbox{and}}&&\ar[d]_{\overline{\overline{w}}_3}\ar@{}[dr]|{\alpha_3\Downarrow}\ar[r]^{\overline{\overline{f}}_1}&\ar[d]_{\overline{w}_3}\ar@{}[dr]|{\alpha_2\Downarrow}\ar[r]^{\overline{f}_2} & \ar[d]^{w_3}
\\
\ar[d]_{w_1\overline{w}_2}\ar[r]_{\overline{f}_1}&\ar[r]_{f_2}&  && \ar[d]_{w_1\overline{w}_2} \ar[r]_{\overline{f}_1}&\ar[r]_{f_2}&
\\
&&&&&}
$$
By Proposition~\ref{unique-2-cell} there is a unique 2-cell in $\calB(\frakW^{-1})$ connecting these two squares. Let 
$$
\xymatrix@C=6em{
&\ar[dl]_{(w_1\overline{w}_2)\widetilde{w}_3}\ar[dr]^{\widetilde{f}_2}
\\
\ar@{}[r]|{(w_1\overline{w}_2)\varepsilon_1\Downarrow}&\ar[u]^{s_1}\ar[d]_{t_1}\ar@{}[r]|{\delta_1\Downarrow}&
\\
&\ar[ul]^{(w_1\overline{w}_2)\overline{\overline{w}}_3}\ar[ur]_{\overline{f}_2\overline{\overline{f}}_1}
}
$$
be a diagram representing this 2-cell.
Composing it with $f_3$ gives,
\begin{equation}\label{cell-introduction}
\xymatrix@C=6em{
&\ar[dl]_{(w_1\overline{w}_2)\widetilde{w}_3}\ar[dr]^{f_3\widetilde{f}_2}
\\
\ar@{}[r]|{(w_1\overline{w}_2)\varepsilon_1\Downarrow}&\ar[u]^{s_1}\ar[d]_{t_1}\ar@{}[r]|{f_3\delta_1\Downarrow}&
\\
&\ar[ul]^{(w_1\overline{w}_2)\overline{\overline{w}}_3}\ar[ur]_{f_3\overline{f}_2\overline{\overline{f}}_1}
}
\end{equation}

 (\ref{middle}) $\Rightarrow$ (\ref{rightfirst}):  the diagrams in (\ref{middle}) and (\ref{rightfirst}) only differ by the following two squares:
$$
\xymatrix@C=3.5em{
\ar[d]_{\overline{\overline{w}}_3}\ar[r]^{\overline{\overline{f}}_1}\ar@{}[dr]|{\stackrel{\alpha_3}{\Leftarrow}}&\ar[d]^{\overline{w}_3}&&\ar[dd]_{\widetilde{w}_2}\ar@{}[ddr]|{\stackrel{\beta_2}{\Leftarrow}}\ar[r]^{\widetilde{f}_1}&\ar[d]^{\overline{w}_3}
\\
\ar[d]_{\overline{w}_2}\ar@{}[dr]|{\stackrel{\alpha_1}{\Leftarrow}}\ar[r]|{\overline{f}_1}&\ar[d]^{w_2}&\mbox{and}&&\ar[d]^{w_2}
\\
\ar[d]_{w_1}\ar[r]_{f_1}&&&\ar[d]_{w_1}\ar[r]_{f_1}&
\\
&&&
}
$$
By Proposition~\ref{unique-2-cell} there is a unique 2-cell in 
$\calB(\frakW^{-1})$ connecting these two squares. Let 
$$
\xymatrix@C=6em{
&\ar[dl]_{w_1\overline{w}_2\overline{\overline{w}}_3}\ar[dr]^{\overline{\overline{f}}_1}
\\
\ar@{}[r]|{w_1\overline{w}_2\varepsilon_2\Downarrow}&\ar[u]^{s_2}\ar[d]_{t_2}\ar@{}[r]|{\delta_2\Downarrow}&
\\
&\ar[ul]^{w_1\widetilde{w}_2}\ar[ur]_{\widetilde{f}_1}
}
$$
be a diagram representing this 2-cell. Composing with $f_3\overline{f}_2$ gives,
\begin{equation}\label{cell-elimination}
\xymatrix@C=6em{
&\ar[dl]_{w_1\overline{w}_2\overline{\overline{w}}_3}\ar[dr]^{f_3\overline{f}_2\overline{\overline{f}}_1}
\\
\ar@{}[r]|{w_1\overline{w}_2\varepsilon_2\Downarrow}&\ar[u]^{s_2}\ar[d]_{t_2}\ar@{}[r]|{f_3\overline{f}_2\delta_2\Downarrow}&
\\
&\ar[ul]^{w_1\widetilde{w}_2}\ar[ur]_{f_3\overline{f}_2\widetilde{f}_1}
}
\end{equation}

The associativity 2-cell for the composable path given in (\ref{composablepath}) is the vertical composition of (\ref{cell-introduction}) and (\ref{cell-elimination}).  To calculate this composition (as described in Section~\ref{newBF2}), we use the choices of {\bf [C5]} and {\bf [C6]} of Notation~\ref{choices} to obtain a square
$$
\xymatrix{
\ar[r]^{\overline{s}_2}\ar[d]_{\overline{t}_1}\ar@{}[dr]|{\stackrel{\varphi}{\Leftarrow}}& \ar[d]^{t_1}
\\
\ar[r]_{s_2}&
}
$$
with $\varphi$ invertible and $w_1\overline{w}_2\overline{\overline{w}}_3s_2\overline{t}_1\in\frakW$.
Then the associativity 2-cell $\alpha_{(w_3,f_3),(w_2,f_2),(w_1,f_1)}$ is represented by
$$
\xymatrix@C=7em{
&&\ar@/_.8ex/[ddll]_{w_1\overline{w}_2\widetilde{w}_3}\ar@/^.8ex/[ddrr]^{f_3\widetilde{f}_2}&&
\\
&&\ar[u]_{s_1}\ar[dl]_{t_1}\ar[dr]^{t_1}\ar@{}[dll]|{w_1\overline{w}_2\varepsilon_1\Downarrow}\ar@{}[drr]|{f_3\delta_1\Downarrow}&&
\\
&\ar[l]_(.45){w_1\overline{w}_2\overline{\overline{w}}_3}\ar@{}[r]|{\varphi\Downarrow}&\ar[u]_{\overline{s}_2}\ar[d]^{\overline{t}_1}\ar@{}[r]|{\varphi\Downarrow}&\ar[r]^(.45){f_3\overline{f}_2\overline{\overline{f}}_1}&
\\
&& \ar[ul]^{s_2}\ar[ur]_{s_2}\ar[d]^{t_2}\ar@{}[ull]|{w_1\overline{w}_2\varepsilon_2\Downarrow} \ar@{}[urr]|{f_3\overline{f}_2\delta_2\Downarrow}&&
\\
&&\ar@/^.8ex/[uull]^{w_1\widetilde{w}_2}\ar@/_.8ex/[uurr]_{f_3\overline{f}_2\widetilde{f}_1}
}
$$
\end{proof}

\begin{prop}\label{vert_associative}
Vertical composition of 2-cells is strictly associative.
\end{prop}

\begin{proof}
Consider three vertically composable 2-cell diagrams,
$$
\xymatrix@C=3.2em{
&\ar[dl]_{u_1}\ar[dr]^{f_1} & && &  \ar[dl]_{u_2}\ar[dr]^{f_2}& && &  \ar[dl]_{u_3}\ar[dr]^{f_3} &
\\
\ar@{}[r]|{\Downarrow\alpha_1} & \ar[u]^{s_1}\ar[d]_{t_1}\ar@{}[r]|{\Downarrow\beta_1} &\ar@{}[rr]|{ \mbox{and  }}
		&& \ar@{}[r]|{\Downarrow\alpha_2} & \ar[u]^{s_2}\ar[d]_{t_2}\ar@{}[r]|{\Downarrow\beta_2} &\ar@{}[rr]|{ \mbox{and  }}
		&& \ar@{}[r]|{\Downarrow\alpha_3} & \ar[u]^{s_3}\ar[d]_{t_3}\ar@{}[r]|{\Downarrow\beta_3} & 
\\
&\ar[ul]^{u_2}\ar[ur]_{f_2} &&&&\ar[ul]^{u_3}\ar[ur]_{f_3} &&&&\ar[ul]^{u_4}\ar[ur]_{f_4} &.
}
$$
Our proof that the two ways of composing these cells vertically are equivalent will mimick the construction of the associativity isomorphism in the proof of the previous proposition.
The constructed cells will in this case become the cells that witness the equivalence. 
However, since we are only interested in the equivalence rather that the cells witnessing it, we will do this in two steps without composing the cells obtained in the two steps.
 
The two possible vertical compositions correspond to choices of squares $\delta_i$ and $\varepsilon_i$ with $i=1,2$ as in
\begin{equation}\label{two_composites}
\xymatrix@C=1.8em@R=1.8em{
&&&\ar[dl]_{\widetilde{s}_3}\ar[ddrr]^{\widetilde{t}_2}\ar@{}[3,1]|{\stackrel{\varepsilon_1}{\Leftarrow}}&&&&&&&\ar[ddll]_{\widetilde{s}_2}\ar[dr]^{\widetilde{t}_1}\ar@{}[3,-1]|{\stackrel{\varepsilon_2}{\Leftarrow}}&&&
\\
&&\ar[dl]_{\overline{s}_2}\ar[dr]^{\overline{t}_1}\ar@{}[dd]|{\stackrel{\delta_1}{\Leftarrow}}&&&&&&&&&\ar[dl]_{\overline{s}_3}\ar[dr]^{\overline{t}_2}\ar@{}[dd]|{\stackrel{\delta_2}{\Leftarrow}}
\\
&\ar[dl]_{s_1}\ar[dr]_{t_1}&&\ar[dl]^{s_2}\ar[dr]_{t_2}&&\ar[dl]^{s_3}\ar[dr]^{t_3}&\ar@{}[r]|{\mbox{and}}&&\ar[dl]_{s_1}\ar[dr]_{t_1}&&\ar[dl]^{s_2}\ar[dr]_{t_2}&&\ar[dl]^{s_3}\ar[dr]^{t_3}
\\
&&&&&&&&&&&&&
}
\end{equation}
with $u_2s_2\overline{t}_1\widetilde{s}_3\in\frakW$ and $u_2s_2\overline{s}_3\widetilde{t}_1\in\frakW$.
We will also consider the following diagram:
\begin{equation}\label{third_composite}
\xymatrix{
&&&\ar[dl]_{\hat{s}_3}\ar[dr]^{\hat{t}_1}
\\
&&\ar[dl]_{\overline{s}_2}\ar[dr]|{\overline{t}_1}\ar@{}[rr]|{\stackrel{\delta_3}{\Leftarrow}}&&\ar[dl]|{\overline{s}_3}\ar[dr]^{\overline{t}_2}
\\
&\ar[dl]_{s_1}\ar[dr]_{t_1}\ar@{}[rr]|{\stackrel{\delta_1}{\Leftarrow}}&&\ar[dl]^{s_2}\ar[dr]_{t_2}\ar@{}[rr]|{\stackrel{\delta_2}{\Leftarrow}}&&\ar[dl]^{s_3}\ar[dr]^{t_3}
\\
&&&&&&}
\end{equation}
where $\delta_3$ is an invertible 2-cell such that $u_1s_1\overline{s}_2\hat{s}_3$ is in $\frakW$.
Note that none of these are pasting diagrams yet, but they can be made into pasting diagrams by  adding the cells $\alpha_i$ or the cells $\beta_i$ as a bottom row to the diagrams.
With the $\alpha_i$ cells we obtain the left-hand 2-cells of our composite 2-cell diagrams and with the $\beta_i$ cells we obtain the right-hand 2-cells of our composite diagrams. As we want to argue about both at the same time, we will give the argument for variable $\gamma_1$, $\gamma_2$ and $\gamma_3$. We begin by comparing the diagrams
$$
\xymatrix@C=1.8em{
&&&&\ar[dl]_{\widetilde{s}_3} \ar[ddrr]^{\widetilde{t}_2} \ar@{}[dddr]|{\stackrel{\varepsilon_1}{\Leftarrow}} &&&&&&& \ar[dl]_{\hat{s}_3}\ar[dr]^{\hat{t}_1}
\\
&&& \ar[dl]_{\overline{s}_2} \ar[dr]|{\overline{t}_1} \ar@{}[dd]|{\stackrel{\delta_1}{\Leftarrow}} &&&&&&& \ar[dl]_{\overline{s}_2}\ar[dr]|{\overline{t}_1} \ar@{}[rr]|{\stackrel{\delta_3}{\Leftarrow}}&& \ar[dl]|{\overline{s}_3}\ar[dr]^{\overline{t}_2}
\\
&&\ar[dl]_{s_1} \ar[dr]|{t_1}&&\ar[dl]|{s_2} \ar[dr]|{t_2} &&\ar[dl]|{s_3} \ar[dr]^{t_3}&\ar@{}[r]|{\mbox{and}}&&\ar[dl]_{s_1}\ar[dr]|{t_1}\ar@{}[rr]|{\stackrel{\delta_1}{\Leftarrow}}&&\ar[dl]|{s_2}\ar[dr]|{t_2}\ar@{}[rr]|{\stackrel{\delta_2}{\Leftarrow}}&&\ar[dr]^{t_3}\ar[dl]|{s_3}
\\
&\ar[dl]_{u_1}\ar@/_1.5ex/[drrr]_{x_1}\ar@{}[rr]|{\stackrel{\gamma_1}{\Leftarrow}} && 
\ar[dr]|{x_2}\ar@{}[rr]|{\stackrel{\gamma_2}{\Leftarrow}}&&\ar[dl]|{x_3} \ar@{}[rr]|{\stackrel{\gamma_3}{\Leftarrow}} && \ar@/^1.5ex/[1,-3]^{x_4}&\ar[dl]_{u_1}\ar@/_1.5ex/[1,3]_{x_1} \ar@{}[rr]|{\stackrel{\gamma_1}{\Leftarrow}} && 
\ar[dr]|{x_2} \ar@{}[rr]|{\stackrel{\gamma_2}{\Leftarrow}}&&\ar[dl]|{x_3}\ar@{}[rr]|{\stackrel{\gamma_3}{\Leftarrow}}&&\ar@/^1.5ex/[1,-3]^{x_4}
\\
&&&&&&&&&&&&&
}
$$ 
These two diagrams only differ in the rectangle with $\varepsilon_1$ versus the composition of 
$\delta_3$ and $\delta_2$. As both $u_1s_1\overline{s}_2\widetilde{s}_3$ and $u_1s_1\overline{s}_2\hat{s}_3$ are in $\frakW$, we can apply Proposition~\ref{anytwosquares} to these two rectangles and obtain arrows and 2-cells as in the following diagram,
$$
\xymatrix{
&\ar@/_1.5ex/[dl]_{\widetilde{s}_3}\ar@/^1.5ex/[dr]^{\widetilde{t}_2}
\\
\ar@{}[r]|{\Downarrow\sigma_1}&\ar[u]|{y_1}\ar[d]|{y_2}\ar@{}[r]|{\Downarrow\tau_1}&
\\
&\ar@/^1.5ex/[ul]^{\hat{s}_3}\ar@/_1.5ex/[ur]_{\overline{t}_2\hat{t}_1}
}
$$
with the property that
$$
\xymatrix@C=1.8em{
&&&&\ar@/_1.5ex/[dddl]_{\hat{s}_3} &&&&&&&\ar@/^1.5ex/[4,2]^{\widetilde{t}_2}
\\
&&&&\ar[u]_{y_2}\ar[d]^{y_1}\ar@{}[ddl]|(.45){\stackrel{\sigma_1}{\Leftarrow}}&&&&&&&\ar[u]^{y_1}\ar[d]_{y_2}\ar@{}[dr]|{\stackrel{\tau_1}{\Leftarrow}}
\\
&&&&\ar[dl]|{\widetilde{s}_3} \ar[ddrr]^{\widetilde{t}_2} \ar@{}[dddr]|{\stackrel{\varepsilon_1}{\Leftarrow}} &&&&&&& \ar[dl]_{\hat{s}_3}\ar[dr]|{\hat{t}_1}&&
\\
&&& \ar[dl]_{\overline{s}_2} \ar[dr]|{\overline{t}_1} \ar@{}[dd]|{\stackrel{\delta_1}{\Leftarrow}} &&&&&&& \ar[dl]_{\overline{s}_2}\ar[dr]|{\overline{t}_1} \ar@{}[rr]|{\stackrel{\delta_3}{\Leftarrow}}&& \ar[dl]|{\overline{s}_3}\ar[dr]|{\overline{t}_2}
\\
&&\ar[dl]_{s_1} \ar[dr]|{t_1}&&\ar[dl]|{s_2} \ar[dr]|{t_2} &&\ar[dl]|{s_3} \ar[dr]^{t_3}&\ar@{}[r]|{\equiv}&&\ar[dl]_{s_1}\ar[dr]|{t_1}\ar@{}[rr]|{\stackrel{\delta_1}{\Leftarrow}}&&\ar[dl]|{s_2}\ar[dr]|{t_2}\ar@{}[rr]|{\stackrel{\delta_2}{\Leftarrow}}&&\ar[dr]^{t_3}\ar[dl]|{s_3}
\\
&\ar[dl]_{u_1}\ar@/_1.5ex/[drrr]_{x_1}\ar@{}[rr]|{\stackrel{\gamma_1}{\Leftarrow}} && 
\ar[dr]|{x_2}\ar@{}[rr]|{\stackrel{\gamma_2}{\Leftarrow}}&&\ar[dl]|{x_3} \ar@{}[rr]|{\stackrel{\gamma_3}{\Leftarrow}} && \ar@/^1.5ex/[1,-3]^{x_4}&\ar[dl]_{u_1}\ar@/_1.5ex/[1,3]_{x_1} \ar@{}[rr]|{\stackrel{\gamma_1}{\Leftarrow}} && 
\ar[dr]|{x_2} \ar@{}[rr]|{\stackrel{\gamma_2}{\Leftarrow}}&&\ar[dl]|{x_3}\ar@{}[rr]|{\stackrel{\gamma_3}{\Leftarrow}}&&\ar@/^1.5ex/[1,-3]^{x_4}
\\
&&&&&&&&&&&&&
}
$$ 
By substituting the $\alpha_i$ for the $\gamma_i$ and by subtituting the $\beta_i$ for the $\gamma_i$ we see that if the vertical composition had been constructed with the cells $\delta_1$, $\delta_2$ and $\delta_3$ it would have been equivalent to the composition obtained by composing the first two 2-cells first. By a similar argument we see that the new composition is also equivalent to the composition obtained by composing the last two diagrams first. So we conclude that the two compositions considered are equivalent and hence vertical composition is strictly associative.
\end{proof}

\section{Associativity Part II:  Coherence} \label{assoc2}

We will only sketch the proof for the associativity pentagon. The other coherence diagrams are straight forward.
We will  view the diagram  (\ref{middle}) as a kind of common subdivision of (\ref{leftfirst}) and (\ref{rightfirst}), and break up the coherence into transitions given by Proposition~\ref{unique-2-cell}, and transitions with two layers of cells. There are two versions of this  two layer case. They seem dual to each other, but their proofs are not, as the arrows in $\frakW$ play very different roles. The two cases are covered in Propositions~\ref{double-layer1} and~\ref{double-layer2} below.  
\begin{prop}\label{double-layer1}
Suppose we have two diagrams in $\calB$,
\begin{equation}\label{leftconfigurations}
\xymatrix{
\ar[rr]^{\overline{f}_2}\ar@{}[drr]|{\alpha_2\Downarrow}\ar[d]_{\overline{w}_3}&&\ar[d]^{w_3}&\mbox{and}&\ar[rr]^{\widetilde{f}_2}\ar@{}[drr]|{\beta_2\Downarrow}\ar[d]_{\widetilde{w}_3}&&\ar[d]^{w_3}
\\
\ar[d]_{\overline{w}_2}\ar[r]^{\overline{f}_1}\ar@{}[dr]|{\alpha_1\Downarrow}&\ar[d]^{w_2}\ar[r]_{f_2}&&&\ar[d]_{\widetilde{w}_2}\ar[r]^{\widetilde{f}_1}\ar@{}[dr]|{\beta_1\Downarrow}&\ar[d]^{w_2}\ar[r]_{f_2}&
\\
\ar[d]_{w_1}\ar[r]_{f_1}& &&& \ar[d]_{w_1}\ar[r]_{f_1}&
\\
&&&&&
}
\end{equation}
with $\alpha_1,\alpha_2,\beta_1$ and $\beta_2$ invertible and all of $w_1, w_1\overline{w}_2,w_1\overline{w}_2\overline{w}_3, w_1\widetilde{w}_2$, and $w_1\widetilde{w}_2\widetilde{w}_3$ in $\frakW$. Furthermore, suppose that we have
two  2-cell diagrams 
$$
\xymatrix@C=3.5em{
&\ar[dl]_{w_1\overline{w}_2}\ar[dr]^{\overline{f}_1}&
\\
\ar@{}[r]|{w_1\varepsilon_i\Downarrow}&\ar[u]_{s_i}\ar[d]^{t_i}\ar@{}[r]|{\delta_i\Downarrow}& &\mbox{ for }i=1,2,
\\
&\ar[ul]^{w_1\widetilde{w}_2}\ar[ur]_{\widetilde{f}_1}
}
$$
that both connect $\alpha_1$ and $\beta_1$ in the sense of Notation~\ref{connect}.
And, suppose that there are 2-cells $\sigma_i$, $\tau_i$ and $\theta_i$ for $i=1,2$
as in
$$
\xymatrix@C=4em{
&\ar[dl]_{\overline{w}_2}\ar@{}[dr]|{\sigma_i\Downarrow}&\ar[l]_{\overline{w}_3}\ar[dr]^{\overline{f}_2}
\\
\ar@{}[r]|{\varepsilon_i\Downarrow}&\ar[u]^{s_i}\ar[d]_{t_i}\ar@{}[dr]|{\tau_i\Downarrow} & \ar[l]_{v_{3,i}}\ar[u]_{\overline{s}_i}\ar[d]^{\overline{t}_i}\ar@{}[r]|{\theta_i\Downarrow}&
\\
&\ar[ul]^{\widetilde{w}_2} &\ar[l]^{\widetilde{w}_3}\ar[ur]_{\widetilde{f}_2}
}
$$
such that $w_1\overline{w}_2\overline{w}_3\overline{s}_i\in\frakW$ and 
$$
\xymatrix@C=3em{
&\ar[drr]^{\overline{f}_2}\ar@{}[dd]|{\theta_i\Downarrow} &&&&&\ar[d]^{\overline{w}_3}\ar[drr]^{\overline{f}_2}\ar@{}[ddrr]|{\alpha_2\Downarrow}
\\
\ar[ur]^{\overline{s}_i}\ar[dr]^{\overline{t}_i}\ar[d]_{v_{3,i}}\ar@{}[ddr]|{\tau_i\Downarrow} &&&\ar[d]^{w_3}\ar@{}[drr]|=&&\ar[d]_{v_{3,i}}\ar[ur]^{\overline{s}_i}\ar@{}[r]|{\sigma_i\Downarrow}&\ar[dr]^{\overline{f}_1}&&\ar[d]^{w_3}
\\
\ar[dr]_{t_i}&\ar[d]^{\widetilde{w}_3}\ar[urr]^{\widetilde{f}_2}\ar@{}[rr]|{\beta_2\Downarrow}&&&&\ar[ur]_{s_i}\ar[dr]_{t_i}\ar@{}[rr]|{\delta_i\Downarrow}&& \ar[r]_{f_2}&
\\
&\ar[urr]_{f_2\widetilde{f}_1}&&&&&\ar[ur]_{\widetilde{f}_1}
}
$$
for $i=1,2$. Then the 2-cell diagrams,
\begin{equation}\label{equivcells}
\xymatrix@C=4em{
&\ar[dl]_{w_1\overline{w}_2}\ar@{}[dr]|{\sigma_1\Downarrow}&\ar[l]_{\overline{w}_3}\ar[dr]^{\overline{f}_2} &&& \ar[dl]_{w_1\overline{w}_2}\ar@{}[dr]|{\sigma_2\Downarrow}&\ar[l]_{\overline{w}_3}\ar[dr]^{\overline{f}_2} 
\\
\ar@{}[r]|{w_1\varepsilon_1\Downarrow}&\ar[u]^{s_1}\ar[d]_{t_1}\ar@{}[dr]|{\tau_1\Downarrow} & \ar[l]|{v_{3,1}}\ar[u]_{\overline{s}_1}\ar[d]^{\overline{t}_1}\ar@{}[r]|{\theta_1\Downarrow}&\ar@{}[r]|{\mbox{and}} & \ar@{}[r]|{w_1\varepsilon_2\Downarrow}&\ar[u]^{s_2}\ar[d]_{t_2}\ar@{}[dr]|{\tau_2\Downarrow} & \ar[l]|{v_{3,2}}\ar[u]_{\overline{s}_2}\ar[d]^{\overline{t}_2}\ar@{}[r]|{\theta_2\Downarrow}&
\\
&\ar[ul]^{w_1\widetilde{w}_2} &\ar[l]^{\widetilde{w}_3}\ar[ur]_{\widetilde{f}_2}&&&\ar[ul]^{w_1\widetilde{w}_2} &\ar[l]^{\widetilde{w}_3}\ar[ur]_{\widetilde{f}_2}
}
\end{equation}
are equivalent.
\end{prop}

\begin{proof}
By Proposition~\ref{unique-2-cell} we know that
$$
\xymatrix@C=3.5em{
&\ar[dl]_{w_1\overline{w}_2}\ar[dr]^{\overline{f}_1}& && &\ar[dl]_{w_1\overline{w}_2}\ar[dr]^{\overline{f}_1}&
\\
\ar@{}[r]|{w_1\varepsilon_1\Downarrow}&\ar[u]_{s_1}\ar[d]^{t_1}\ar@{}[r]|{\delta_1\Downarrow}& &\mbox{and}&\ar@{}[r]|{w_1\varepsilon_2\Downarrow}&\ar[u]_{s_2}\ar[d]^{t_2}\ar@{}[r]|{\delta_2\Downarrow}&
\\
&\ar[ul]^{w_1\widetilde{w}_2}\ar[ur]_{\widetilde{f}_1}& && & \ar[ul]^{w_1\widetilde{w}_2}\ar[ur]_{\widetilde{f}_1}}
$$
are equivalent 2-cell diagrams as they both connect $\alpha_1$ and $\beta_1$.
So there are 2-cells
$$
\xymatrix{
&
\\
\ar@/^1ex/[ur]^{s_1}\ar@/_1ex/[dr]_{t_1}&\ar[l]_{r_1}\ar[r]^{r_2}\ar@{}[u]|{\stackrel{\varphi}{\Rightarrow}} \ar@{}[d]|{\stackrel{\psi}{\Rightarrow}} &\ar@/_1ex/[ul]_{s_2}\ar@/^1ex/[dl]^{t_2}
\\
&
}
$$
such that $w_1\overline{w}_2s_1r_1\in\frakW$ and
$$
\xymatrix@C=2.8em{
\ar[r]^{r_1}\ar[d]_{r_2}\ar@{}[dr]|{\varphi\Downarrow} &\ar[d]^{s_1}\ar@{}[dr]|= & \ar[r]^{r_1}\ar[d]_{r_2}\ar@{}[dr]|{\psi\Downarrow}&\ar[d]|{t_1}\ar[r]^{s_1}\ar@{}[dr]|{w_1\varepsilon_1\Downarrow}& \ar[d]^{w_1\overline{w}_2} \ar@{}[drr]|{\mbox{and}} && \ar[r]^{r_1}\ar[d]_{r_2}\ar@{}[dr]|{\varphi\Downarrow} &\ar[d]^{s_1}\ar@{}[dr]|= & \ar[r]^{r_1}\ar[d]_{r_2}\ar@{}[dr]|{\psi\Downarrow}&\ar[d]|{t_1}\ar[r]^{s_1}\ar@{}[dr]|{\delta_1\Downarrow}& \ar[d]^{\overline{f}_1}
\\
\ar[r]_{s_2}\ar[d]_{t_2}\ar@{}[dr]|{w_1\varepsilon_2\Downarrow} & \ar[d]^{w_1\overline{w}_2} & \ar[r]_{t_2}&\ar[r]_{w_1\widetilde{w}_2} & && \ar[r]_{s_2}\ar[d]_{t_2}\ar@{}[dr]|{\delta_2\Downarrow} & \ar[d]^{\overline{f}_1} & \ar[r]_{t_2}&\ar[r]_{\widetilde{f}_1} & 
\\
\ar[r]_{w_1\widetilde{w}_2} & &&&&&\ar[r]_{\widetilde{f}_1}&&&&
}
$$
Now consider the cospan $\xymatrix@1{\ar[r]^{v_{3,i}}&&\ar[l]_{r_i}}$.
Since both $w_1\overline{w}_2s_iv_{3,i}$ and $w_1\overline{w}_2s_ir_i$ are in $\frakW$ we can use conditions {\bf [WB3]}, {\bf [WB4]} and {\bf [WB2]} to obtain a square with an invertible 2-cell,
$$
\xymatrix{
\ar[r]^{r'_i}\ar[d]_{v'_{3,i}}\ar@{}[dr]|{\stackrel{\rho'_i}{\Leftarrow}} & \ar[d]^{v_{3,i}}
\\
\ar[r]_{r_i}&
}
$$
with $w_1\overline{w}_2s_ir_i\overline{v}_{3,i}\in\frakW$.
We apply the same conditions then to $w_1\overline{w}_2s_1r_1{v}'_{3,1}$
and $w_1\overline{w}_2s_2r_2{v}'_{3,2}$ to obtain a square with an invertible 2-cell,
$$
\xymatrix{
\ar[r]^{u_2}
\ar[d]_{u_1}\ar@{}[dr]|{\stackrel{\omega}{\Leftarrow}} & \ar[d]^{{v}'_{3,2}}
\\
\ar[r]_{{v}'_{3,1}}&
}
$$
such that $w_1\overline{w}_2s_1r_1{v}'_{3,1}u_1\in\frakW$.
Now write $\rho_1:=\rho_1'u_1$, $\overline{r}_1:={r}'_1u_1$, 
$\overline{v}_3:={v}'_{3,1}u_1$, and $\overline{r}_2:=r_2'u_2$. Finally, write
$\rho_2$ for the pasting of
$$
\xymatrix{
\ar[d]_{u_1}\ar[r]^{u_2} &\ar[ddl]^{{v}'_{3,2}}\ar[r]^{r_2'}\ar@{}[dl]|{\stackrel{\omega}{\Leftarrow}}\ar@{}[ddr]|{\stackrel{\rho_2'}{\Leftarrow}} &\ar[dd]^{v_{3,2}}
\\
\ar[d]_{v'_{3,1}}
\\
\ar[rr]_{r_2}&&
}
$$
Then we obtain the following diagram,
$$
\xymatrix@C3.5em{
\ar[d]_{v_{3,1}}\ar@{}[dr]|{\stackrel{\rho_1}{\Rightarrow}} &\ar[l]_{\overline{r}_1}\ar[r]^{\overline{r}_2}\ar[d]_{\overline{v}_3}\ar@{}[dr]|{\stackrel{\rho_2}{\Leftarrow}} &\ar[d]^{v_{3,2}}
\\
&\ar[l]^{r_1}\ar[r]_{r_2} &
}
$$
Now consider the following two pasting diagrams,
$$
\xymatrix@C=3.5em{
& \ar[dl]_{\overline{s}_1} \ar[d]|{v_{3,1}} 
& \ar[l]_{\overline{r}_1} \ar@{}[dl]|{\stackrel{\rho_1}{\Rightarrow}}\ar[d]|{\overline{v}_3} \ar@{}[dr]|{\stackrel{\rho_2^{-1}}{\Rightarrow}} \ar[r]^{\overline{r}_2} & \ar[d]|{v_{3,2}} \ar[dr]^{\overline{s}_2} 
&&& \ar[dl]_{\overline{t}_1} \ar[d]|{v_{3,1}} & \ar[l]_{\overline{r}_1} \ar@{}[dl]|{\stackrel{\rho_1}{\Rightarrow}}\ar[d]|{\overline{v}_3} \ar@{}[dr]|{\stackrel{\rho_2^{-1}}{\Rightarrow}}\ar[r]^{\overline{r}_2} & \ar[d]|{v_{3,2}}\ar[dr]^{\overline{t}_2}
\\
\ar@{}[r]|{\stackrel{\sigma_1}{\Rightarrow}} \ar@/_1.5ex/[drr]_{\overline{w}_3} & \ar@/_.75ex/[dr]|{s_1} 
& \ar[l]^{r_1} \ar[r]_{r_2} \ar@{}[d]|{\stackrel{\varphi}{\Rightarrow}} 
& \ar@/^.75ex/[dl]|{s_2} \ar@{}[r]|{\stackrel{\sigma_1^{-1}}{\Rightarrow}} &\ar@/^1.5ex/[dll]^{\overline{w}_3} 
& \ar@{}[r]|{\stackrel{\tau_1}{\Rightarrow}} \ar@/_1.5ex/[drr]_{\widetilde{w}_3} & \ar@/_.75ex/[dr]|{t_1} & \ar[l]^{r_1}\ar[r]_{r_2}\ar@{}[d]|{\stackrel{\psi}{\Rightarrow}} & 
\ar@/^.75ex/[dl]|{t_2} \ar@{}[r]|{\stackrel{\tau_2^{-1}}{\Rightarrow}} &\ar@/^1.5ex/[dll]^{\widetilde{w}_3}
\\
&& \ar[d]_{w_1\overline{w}_2} && & && \ar[d]_{w_1\widetilde{w}_2}
\\
&&&&&&&
}
$$
Use condition {\bf [WB4]} to lift the first pasting with respect to $w_1\overline{w}_2\overline{w}_3$ to obtain $\varphi'\colon \overline{s}_1\overline{r}_1u\Rightarrow\overline{s}_2\overline{r}_2u$; similarly, apply condition {\bf[WB4]} to the pasting of the second diagram composed with $u$ and lift with respect to $w_1\widetilde{w}_2\widetilde{w}_3$ to obtain $\widetilde\psi\colon\overline{t}_1\overline{r}_1uu'\Rightarrow\overline{t}_2\overline{r}_2uu'$.
Now write $\widetilde{r}_1=\overline{r}_1uu'$, $\widetilde{r}_2=\overline{r}_2uu'$, and $\widetilde\varphi=\varphi'u'$. 
Then the reader may check that the 2-cells 
$$
\xymatrix{
&
\\
\ar@/^1ex/[ur]^{\overline{s}_1}\ar@/_1ex/[dr]_{\overline{t}_1}&\ar[l]_{\widetilde{r}_1}\ar[r]^{\widetilde{r}_2}\ar@{}[u]|{\stackrel{\widetilde{\varphi}}{\Rightarrow}}\ar@{}[d]|{\stackrel{\widetilde{\psi}}{\Rightarrow}} &\ar@/_1ex/[ul]_{\overline{s}_2}\ar@/^1ex/[dl]^{\overline{t}_2}
\\
&
}
$$
witness to the 2-cell diagrams in (\ref{equivcells}) being equivalent.
\end{proof}

The following proposition is the dual to the previous one; however, the proof is quite different,   due to the special role played by arrows in $\frakW$.

\begin{prop}\label{double-layer2}
Suppose we have two diagrams in $\calB$,
\begin{equation}\label{rightconfigurations}
\xymatrix{
&\ar[l]_{w_1}\ar@{}[drr]|{\alpha_1\Downarrow}\ar[d]_{{f}_1}&&\ar[ll]_{\overline{w}_2}\ar[d]^{\overline{f}_1}&\mbox{and}&&\ar[l]_{w_1}\ar@{}[drr]|{\beta_1\Downarrow}\ar[d]_{f_1}&&\ar[ll]_{\widetilde{w}_2}\ar[d]^{\widetilde{f}_1}
\\
&&\ar[l]^{w_2}\ar[d]_{f_2}\ar@{}[dr]|{\alpha_2\Downarrow}&\ar[l]_{\overline{w}_3}\ar[d]^{\overline{f}_2}&&&&\ar[l]^{w_2}\ar[d]_{f_2}\ar@{}[dr]|{\beta_2\Downarrow}&\ar[l]_{\widetilde{w}_3}\ar[d]^{\widetilde{f}_2}
\\
&&&\ar[l]^{w_3} && &&& \ar[l]^{w_3}
}
\end{equation} with all 2-cells invertible and all of $w_3, w_2\overline{w}_3, w_2\widetilde{w}_3, w_1\overline{w}_2$, and $w_1\widetilde{w}_2$ in $\frakW$.
Suppose further that we have two 2-cell diagrams 
$$
\xymatrix@C=3.5em{
&\ar[dl]_{w_2\overline{w}_3}\ar[dr]^{\overline{f}_2}&
\\
\ar@{}[r]|{w_2\varepsilon_i\Downarrow}&\ar[u]_{s_i}\ar[d]^{t_i}\ar@{}[r]|{\delta_i\Downarrow}& & \mbox{ for $i=1,2$},
\\
&\ar[ul]^{w_2\widetilde{w}_3}\ar[ur]_{\widetilde{f}_2}
}
$$
that both connect $\alpha_2$ and $\beta_2$.
Suppose that there are 2-cells $\sigma_i$, $\tau_i$ and $\zeta_i$ for $i=1,2$
as in,
$$
\xymatrix@C=4em{
&\ar[dl]_{\overline{w}_2}\ar@{}[dr]|{\sigma_i\Downarrow}\ar[r]^{\overline{f}_1}&\ar[dr]^{\overline{f}_2}
\\
\ar@{}[r]|{\zeta_i\Downarrow}&\ar[u]^{\overline{s}_i}\ar[d]_{\overline{t}_i}\ar@{}[dr]|{\tau_i\Downarrow} \ar[r]_{g_{1,i}}& \ar[u]_{{s}_i}\ar[d]^{{t}_i}\ar@{}[r]|{\delta_i\Downarrow}&
\\
&\ar[ul]^{\widetilde{w}_2} \ar[r]_{\widetilde{f}_1}&\ar[ur]_{\widetilde{f}_2}
}
$$
such that $w_1\overline{w}_2\overline{s}_i\in\frakW$ for $i=1,2$, and
\begin{equation}\label{compatibilityeqn}
\xymatrix@C=3.5em{
&&\ar[dll]_{\overline{w}_2}&&&&\ar[dll]_{\overline{w}_2}\ar[d]_{\overline{f}_1}
\\
\ar@{}[rrr]|{\zeta_i\Downarrow} \ar[d]_{f_1} &&&\ar[ul]_{\overline{s}_i}\ar[dl]|{\overline{t}_i} \ar[d]^{g_{1,i}} \ar@{}[dr]|= & \ar[d]_{f_1}\ar@{}[rr]|{\alpha_1\Downarrow} && \ar[dl]|{\overline{w}_3}\ar@{}[r]|{\sigma_i\Downarrow}&\ar[d]^{g_{1,i}}\ar[ul]_{\overline{s}_i}
\\
\ar@{}[rr]|{\beta_1\Downarrow} &&\ar[d]^{\widetilde{f}_1}\ar[ull]|{\widetilde{w}_2}\ar@{}[r]|{\tau_i^{-1}\Downarrow} & \ar[dl]^{t_i} & & \ar[l]^{w_2} \ar@{}[rr]|{\varepsilon_i\Downarrow} && \ar[ul]|{s_i}\ar[dl]^{t_i}
\\
&&\ar[ull]^{w_2\widetilde{w}_3} &&&&\ar[ul]^{\widetilde{w}_3}&
}
\end{equation}
for $i=1,2$. Then the 2-cell diagrams,
\begin{equation}\label{equivcells2}
\xymatrix@C=4em{
&\ar[dl]_{w_1\overline{w}_2}\ar@{}[dr]|{\sigma_1\Downarrow}\ar[r]^{\overline{f}_1}&\ar[dr]^{\overline{f}_2} &&& \ar[dl]_{w_1\overline{w}_2}\ar@{}[dr]|{\sigma_2\Downarrow}\ar[r]^{\overline{f}_1}&\ar[dr]^{\overline{f}_2} 
\\
\ar@{}[r]|{w_1\zeta_1\Downarrow}&\ar[u]^{\overline{s}_1}\ar[d]_{\overline{t}_1}\ar@{}[dr]|{\tau_1\Downarrow} \ar[r]_{g_{1,1}}& \ar[u]_{s_1}\ar[d]^{t_1}\ar@{}[r]|{\delta_1\Downarrow}&\ar@{}[r]|{\mbox{and}} & \ar@{}[r]|{w_1\zeta_2\Downarrow}&\ar[u]^{\overline{s}_2}\ar[d]_{\overline{t}_2}\ar@{}[dr]|{\tau_2\Downarrow} \ar[r]_{g_{1,2}} & \ar[u]_{s_2}\ar[d]^{t_2}\ar@{}[r]|{\delta_2\Downarrow}&
\\
&\ar[ul]^{w_1\widetilde{w}_2}  \ar[r]_{\widetilde{f}_1}&\ar[ur]_{\widetilde{f}_2}&&&\ar[ul]^{w_1\widetilde{w}_2} \ar[r]_{\widetilde{f}_1}&\ar[ur]_{\widetilde{f}_2}
}
\end{equation}
are equivalent.
\end{prop}

\begin{proof}
By Proposition~\ref{unique-2-cell} we know that
$$
\xymatrix@C=3.5em{
&\ar[dl]_{w_2\overline{w}_3}\ar[dr]^{\overline{f}_2}& && &\ar[dl]_{w_2\overline{w}_3}\ar[dr]^{\overline{f}_2}&
\\
\ar@{}[r]|{w_2\varepsilon_1\Downarrow}&\ar[u]_{s_1}\ar[d]^{t_1}\ar@{}[r]|{\delta_1\Downarrow}& &\mbox{and}&\ar@{}[r]|{w_2\varepsilon_2\Downarrow}&\ar[u]_{s_2}\ar[d]^{t_2}\ar@{}[r]|{\delta_2\Downarrow}&
\\
&\ar[ul]^{w_2\widetilde{w}_3}\ar[ur]_{\widetilde{f}_2}& && & \ar[ul]^{w_2\widetilde{w}_3}\ar[ur]_{\widetilde{f}_2}}
$$
are equivalent 2-cell diagrams as they both connect $\alpha_2$ and $\beta_2$.
So there are 2-cells
\begin{equation}\label{phipsi}
\xymatrix{
&
\\
\ar@/^1ex/[ur]^{s_1}\ar@/_1ex/[dr]_{t_1}&\ar[l]_{r_1}\ar[r]^{r_2}\ar@{}[u]|{\stackrel{\varphi}{\Rightarrow}} \ar@{}[d]|{\stackrel{\psi}{\Rightarrow}} &\ar@/_1ex/[ul]_{s_2}\ar@/^1ex/[dl]^{t_2}
\\
&
}
\end{equation}
such that $w_2\overline{w}_3s_1r_1\in\frakW$ and
\begin{equation}\label{eqveqns}
\xymatrix@C=2.75em{
\ar[r]^{r_1}\ar[d]_{r_2}\ar@{}[dr]|{\varphi\Downarrow} &\ar[d]^{s_1}\ar@{}[dr]|= & \ar[r]^{r_1}\ar[d]_{r_2}\ar@{}[dr]|{\psi\Downarrow}&\ar[d]|{t_1}\ar[r]^{s_1}\ar@{}[dr]|{w_2\varepsilon_1\Downarrow}& \ar[d]^{w_2\overline{w}_3} \ar@{}[drr]|{\mbox{and}} && \ar[r]^{r_1}\ar[d]_{r_2}\ar@{}[dr]|{\varphi\Downarrow} &\ar[d]^{s_1}\ar@{}[dr]|= & \ar[r]^{r_1}\ar[d]_{r_2}\ar@{}[dr]|{\psi\Downarrow}&\ar[d]|{t_1}\ar[r]^{s_1}\ar@{}[dr]|{\delta_1\Downarrow}& \ar[d]^{\overline{f}_2}
\\
\ar[r]_{s_2}\ar[d]_{t_2}\ar@{}[dr]|{w_2\varepsilon_2\Downarrow} & \ar[d]^{w_2\overline{w}_3} & \ar[r]_{t_2}&\ar[r]_{w_2\widetilde{w}_3} & && \ar[r]_{s_2}\ar[d]_{t_2}\ar@{}[dr]|{\delta_2\Downarrow} & \ar[d]^{\overline{f}_2} & \ar[r]_{t_2}&\ar[r]_{\widetilde{f}_2} & 
\\
\ar[r]_{w_2\widetilde{w}_3} & &&&&&\ar[r]_{\widetilde{f}_2}&&&&
}
\end{equation}
Since the composites $w_1\overline{w}_2\overline{s}_i\in\frakW$ for $i=1,2$, we can use conditions {\bf[WB3]}, {\bf[WB4]} and {\bf[WB2]} to obtain  an invertible 2-cell $\varphi'$ as in
$$
\xymatrix@C=3em{
&&\ar[dl]_{\overline{s}_1}
\\
&\ar[l]_{w_1\overline{w}_2}\ar@{}[r]|(.55){{\varphi}'\Downarrow}
&\ar[u]_{{r}'_1}\ar[d]^{r'_2}
\\
&&\ar[ul]^{\overline{s}_2}
}
$$
with $w_1\overline{w}_2\overline{s}_1r'_1\in\frakW$.
We want to define a corresponding cell $\psi'$. So
consider the diagram,
\begin{equation}\label{tildepsi}
\xymatrix@C=4.5em{
\ar[r]^{\overline{t}_1} \ar[dr]|{\overline{s}_1} \ar@{}[drr]|{w_1\zeta_1^{-1}\Downarrow} & \ar[dr]^{w_1\widetilde{w}_2}
\\
\ar[u]^{r'_1} \ar@{}[r]|{\varphi'\Downarrow} \ar[d]_{r'_2} & \ar[r]|{w_1\overline{w}_2} &
\\
\ar[ur]|{\overline{s}_2} \ar@{}[urr]|{w_1\zeta_2\Downarrow} \ar[r]_{\overline{t}_2}&\ar[ur]_{w_1\widetilde{w}_2}
}
\end{equation}
Since $w_1\widetilde{w}_2\in\frakW$, we apply conditions {\bf[WB4]} and {\bf[WB2]}
to lift the pasting of this diagram with respect to $w_1\widetilde{w}_2$ to obtain
${\psi}'\colon \overline{t}_1r'_1w'\Rightarrow\overline{t}_2r'_2w'$.
Now note that $\widetilde{w}_2\psi'$ and the composite of
$$
\xymatrix@C=4em{
\ar[r]^{\overline{t}_1} \ar[dr]|{\overline{s}_1} \ar@{}[drr]|{\zeta_1^{-1}\Downarrow} & \ar[dr]^{\widetilde{w}_2}
\\
\ar[u]^{r'_1w'} \ar@{}[r]|{\varphi'w'\Downarrow} \ar[d]_{r'_2w'} & \ar[r]^{\overline{w}_2} &
\\
\ar[ur]|{\overline{s}_2} \ar@{}[urr]|{\zeta_2\Downarrow} \ar[r]_{\overline{t}_2}&\ar[ur]_{\widetilde{w}_2}
}
$$
are both liftings of the pasting of (\ref{tildepsi}) with respect to $w_1$.
So by condition {\bf [WB4]} there is an arrow $w''$ such that $\psi'w''$ is equal to the composition of this last pasting with $w''$. We will need this in our calculations, so we write $\overline{r}_i=r'_iw'w''$, $\widetilde\varphi=\varphi'w'w''$, and $\widetilde\psi=\psi'w''$. This gives us the following diagram
\begin{equation}\label{tildephipsi}
\xymatrix{
&
\\
\ar@/^1ex/[ur]^{\overline{s}_1}\ar@/_1ex/[dr]_{\overline{t}_1}&\ar[l]_{\overline{r}_1}\ar[r]^{\overline{r}_2}\ar@{}[u]|{\stackrel{\widetilde\varphi}{\Rightarrow}} \ar@{}[d]|{\stackrel{\widetilde\psi}{\Rightarrow}} &\ar@/_1ex/[ul]_{\overline{s}_2}\ar@/^1ex/[dl]^{\overline{t}_2}
\\
&
}
\end{equation}
These cells satisfy the required equation with the $\zeta_i$ by construction:
$$
\xymatrix@C=4em@R=2.6em{
\ar[r]^{\overline{r}_1}\ar[d]_{\overline{r}_2}\ar@{}[dr]|{\widetilde\varphi\Downarrow}&\ar[d]^{\overline{s}_1}&\ar@{}[d]|{=}&\ar[r]^{\overline{r}_1}\ar[d]_{\overline{r}_2}\ar@{}[dr]|{\widetilde\psi\Downarrow}&\ar[r]^{\overline{s}_1}\ar[d]_{\overline{t}_1}\ar@{}[dr]|{w_1\zeta_1\Downarrow}&\ar[d]^{w_1\overline{w}_2}
\\
\ar[r]^{\overline{s}_2}\ar[d]_{\overline{t}_2}\ar@{}[dr]|{w_1\zeta_2\Downarrow}&\ar[d]^{w_1\overline{w}_2}&&\ar[r]_{\overline{t}_2}&\ar[r]_{w_1\widetilde{w}_2}&
\\
\ar[r]_{w_1\widetilde{w}_2}&
}
$$
We will next see that after precomposing with an appropriate arrow they will also satisfy the equation for the composites of the right-hand sides of (\ref{equivcells2}).
Since the cells $\varphi$ and $\psi$ satisfy the equation with the $\delta_i$ as stated in (\ref{eqveqns}), we will 
focus on the cylinder with the diagram (\ref{phipsi}) as bottom and (\ref{tildephipsi})
as top. The sides of this cylinder are given by 
$$\xymatrix@C=3em{
\ar[d]_{\overline{f}_1}&\ar[l]_{\overline{s}_1}\ar[r]^{\overline{t}_1}\ar[d]|{g_{1,1}}\ar@{}[dl]|{\stackrel{\sigma_1}{\Rightarrow}}\ar@{}[dr]|{\stackrel{\tau_1}{\Rightarrow}}&\ar[d]^{\widetilde{f}_1}\ar@{}[drr]|{\mbox{and}}&&\ar[d]_{\overline{f}_1}&\ar[l]_{\overline{s}_2}\ar[r]^{\overline{t}_2}\ar[d]|{g_{1,2}}\ar@{}[dl]|{\stackrel{\sigma_2}{\Rightarrow}}\ar@{}[dr]|{\stackrel{\tau_2}{\Rightarrow}}&\ar[d]^{\widetilde{f}_1}
\\
&\ar[l]^{s_1}\ar[r]_{t_1}&&&&\ar[l]^{s_2}\ar[r]_{t_2}&}
$$
Before we can discuss the commutativity of this cylinder, we need to build cells to fill in the following frame,
$$
\xymatrix{
\ar[d]_{g_{1,1}}&\ar[l]_{\overline{r}_1}\ar[r]^{\overline{r}_2}&\ar[d]^{g_{1,2}}
\\
&\ar[l]_{r_1}\ar[r]^{r_2}&}
$$
Since $w_1\overline{w}_2s_1r_1\in\frakW$, we can use conditions {\bf[WB3]}, {\bf[WB4]} and {\bf[WB2]} to construct an invertible 2-cell $\rho_1$ as in
$$
\xymatrix{
\ar[d]_{h_1}\ar[r]^u\ar@{}[drr]|{\rho_1\Downarrow} & \ar[r]^{\overline{r}_1}&\ar[d]^{g_{1,1}}
\\
\ar[rr]_{r_1}&&
}
$$
where $w_1\overline{w}_2\overline{s}_1\overline{r}_1u\in\frakW$. Use this to construct a left-hand square in the frame.
To obtain a cell to fill the remaining right-hand square, we consider the following pasting diagram,
$$
\xymatrix@C=2.8em@R=2.8em{
\ar[dd]_{h_1}\ar@{}[ddr]|{\rho_1\Downarrow}\ar[r]^u&\ar[d]|{\overline{r}_1}\ar@{}[dr]|{\widetilde{\psi}^{-1}\Downarrow}\ar[r]^{\overline{r}_2}&\ar[d]^{\overline{t}_2}\ar[dr]^{g_{1,2}}\ar@{}[ddr]|{\tau_2\Downarrow}
\\
&\ar[d]|{g_{1,1}}\ar[r]^{\overline{t}_1}\ar@{}[drr]|{\tau_1^{-1}\Downarrow}&\ar[dr]|{\widetilde{f}_1}&\ar[d]^{t_2}
\\
\ar[r]^{r_1}\ar[dr]_{r_2}&\ar@{}[d]|{\psi\Downarrow}\ar[rr]^{t_1}&&\ar[r]_{w_2\widetilde{w}_3}&
\\
&\ar[urr]_{t_2}
}
$$
Now lift with respect to $w_2\widetilde{w}_3t_2$ to obtain $\rho_2\colon g_{1,2}\overline{r}_2u\widetilde{t}\Rightarrow r_2h_1\widetilde{t}$.
So the middle frame gets filled as follows:
$$
\xymatrix@C=3.5em{
\ar[d]_{g_{1,1}}& \ar[l]_{\overline{r}_1u\widetilde{t}} \ar[d]|{h_1\widetilde{t}} 
\ar@{}[dl]|{\rho_1\widetilde{t}\Downarrow} \ar@{}[dr]|{\rho_2\Downarrow} \ar[r]^{\overline{r}_2u\widetilde{t}} &\ar[d]^{g_{1,2}}
\\
&\ar[l]^{r_1}\ar[r]_{r_2}&}
$$
Furthermore, we have adjusted the top of the cylinder to become
$$
\xymatrix@C=4em{
&
\\
\ar@/^1ex/[ur]^{\overline{s}_1}\ar@/_1ex/[dr]_{\overline{t}_1}&\ar[l]_{\overline{r}_1u\widetilde{t}}\ar[r]^{\overline{r}_2u\widetilde{t}}\ar@{}[u]|{\stackrel{\widetilde\varphi u\widetilde{t}}{\Rightarrow}} \ar@{}[d]|{\stackrel{\widetilde\psi u\widetilde{t}}{\Rightarrow}} &\ar@/_1ex/[ul]_{\overline{s}_2}\ar@/^1ex/[dl]^{\overline{t}_2}
\\
&
}
$$
We have defined $\rho_2$ in such a way that if the half of the cylinder that contains the $\psi,\widetilde\psi, \tau_1$ and $\tau_2$ gets composed with $w_2\widetilde{w}_3$ it commutes. Condition {\bf[WB4]} now gives that there is an arrow $x$ such that if we precompose the top of the cylinder and the middle frame both with $x$, this half of the cylinder commutes.
So now the top and the middle frame are respectively,
$$
\xymatrix@C=4em{
&&\ar@{}[drr]|{\mbox{and}}&&\ar[d]_{g_{1,1}}& \ar[l]_{\overline{r}_1u\widetilde{t}x} \ar[d]|{h_1\widetilde{t}x} 
\ar@{}[dl]|{\rho_1\widetilde{t}x\Downarrow} \ar@{}[dr]|{\rho_2x\Downarrow} \ar[r]^{\overline{r}_2u\widetilde{t}x} &\ar[d]^{g_{1,2}}
\\
\ar@/^1ex/[ur]^{\overline{s}_1}\ar@/_1ex/[dr]_{\overline{t}_1}&\ar[l]_{\overline{r}_1u\widetilde{t}x}\ar[r]^{\overline{r}_2u\widetilde{t}x}\ar@{}[u]|{\stackrel{\widetilde\varphi u\widetilde{t}x}{\Rightarrow}} \ar@{}[d]|{\stackrel{\widetilde\psi u\widetilde{t}x}{\Rightarrow}} &\ar@/_1ex/[ul]_{\overline{s}_2}\ar@/^1ex/[dl]^{\overline{t}_2}& &&\ar[l]^{r_1}\ar[r]_{r_2}&
\\
&}
$$
To investigate the commutativity of the other half of the cylinder, we will show that 
\begin{equation}\label{bigclaim}
\xymatrix@C=4em@R=3em{
&&&&&\ar[dr]^{\overline{s}_1} \ar@{}[d]|{\widetilde\varphi u\widetilde{t}x\Downarrow} &
\\
\ar[r]^{\overline{r}_1u\widetilde{t}x} \ar[d]_{h_1\widetilde{t}x}
\ar@{}[dr]|{\rho_1\widetilde{t}x\Downarrow} & \ar[d]|{g_{1,1}} \ar[r]^{\overline{s}_1}
\ar@{}[dr]|{\sigma_1\Downarrow} & \ar[d]^{\overline{f}_1} \ar@{}[drr]|{\textstyle =} &&\ar[ur]^{\overline{r}_1u\widetilde{t}x} \ar[r]_{\overline{r}_2u\widetilde{t}x}\ar[d]_{h_1\widetilde{t}x} \ar@{}[dr]|{\rho_2x\Downarrow} &\ar[d]|{g_{1,2}}\ar[r]_{\overline{s}_2}\ar@{}[dr]|{\sigma_2\Downarrow} &\ar[d]^{\overline{f}_1}
\\
\ar[dr]_{r_2}\ar[r]_{r_1} & \ar@{}[d]|{\varphi\Downarrow} \ar[r]_{s_1}&\ar[dr]^{w_2\overline{w}_3} &&\ar[r]_{r_2}&\ar[r]_{s_2} 
\ar[dr]_{t_2}\ar@{}[drr]|{w_2\varepsilon_2\Downarrow}&\ar[dr]^{w_2\overline{w}_3}
\\
&\ar[ur]_{s_2}\ar@{}[rr]|{w_2\varepsilon_2\Downarrow} \ar[dr]_{t_2}&&&&&\ar[r]_{w_2\widetilde{w}_3} &&
\\
&&\ar[ur]_{w_2\widetilde{w}_3} &&&&&
}
\end{equation}
We begin by rewriting the left-hand side.
By (\ref{eqveqns}) this pasting is equal to the pasting of
$$
\xymatrix@C=5em@R=3em{
\ar[r]^{\overline{r}_1u\widetilde{t}x} \ar[d]_{h_1\widetilde{t}x} \ar@{}[dr]|{\rho_1\widetilde{t}x\Downarrow} & \ar[r]^{\overline{s}_1} \ar[d]|{g_{1,1}}\ar@{}[dr]|{\sigma_1\Downarrow} &\ar[d]^{\overline{f}_1}
\\
\ar[r]_{r_1}\ar[d]_{r_2}\ar@{}[dr]|{\psi\Downarrow} & \ar[r]_{s_1}\ar[d]|{t_1}\ar@{}[dr]|{w_2\varepsilon_1\Downarrow} & \ar[d]^{w_2\overline{w}_3}
\\
\ar[r]_{t_2}&\ar[r]_{w_2\widetilde{w}_3}&
}
$$
We use (\ref{eqveqns}) to rewrite the right two 2-cells in this diagram to get
$$
\xymatrix@C=4.5em{
\ar[r]^{\overline{r}_1u\widetilde{t}x} \ar[d]_{h_1\widetilde{t}x} \ar@{}[dr]|{\rho_1\widetilde{t}x\Downarrow} &\ar[d]|{g_{1,1}}\ar@{}[ddr]|{\tau_1\Downarrow}\ar[dr]^{\overline{t}_1}\ar@{}[drr]|{\zeta_1\Downarrow}\ar[r]^{\overline{s}_1}& \ar[dr]^{\overline{w}_2}  \ar[rr]^{\overline{f}_1} &\ar@{}[dr]|{\alpha_1^{-1}\Downarrow}&\ar[dd]^{w_2\overline{w}_3}
\\
\ar[d]_{r_2} \ar@{}[dr]|{\psi\Downarrow} \ar[r]_{r_1} & \ar[dr]_{t_1}&\ar[d]^{\widetilde{f}_1}\ar@{}[dr]|{\beta_1\Downarrow} \ar[r]_{\widetilde{w}_2}&\ar[dr]^{f_1}&
\\
\ar[rr]_{t_2}&&\ar[rr]_{w_2\widetilde{w}_3}&&
}
$$
Now note that we have constructed $\widetilde\varphi$ and $\widetilde\psi$ such that 
$$
\xymatrix@C=3.5em{
\ar[r]^{\overline{s}_1\overline{r}_1}\ar[d]_{\overline{t}_1\overline{r}_1}\ar@{}[dr]|{\zeta_1\overline{r}_1\Downarrow}&\ar[d]^{\overline{w}_2}\ar@{}[drr]|= && \ar[d]_{\overline{r}_1}\ar[dr]|{\overline{r}_2}\ar[r]^{\overline{r}_1}\ar@{}[drr]|{\widetilde\varphi\Downarrow}\ar@{}[ddr]|{\widetilde{\psi}^{-1}\Downarrow} & \ar[dr]^{\overline{s}_1}
\\
\ar[r]_{\widetilde{w}_2}&&& \ar[dr]_{\overline{t}_1}&\ar[d]^{\overline{t}_2}\ar[r]^{\overline{s}_2}\ar@{}[dr]|{\zeta_2\Downarrow}&\ar[d]^{\overline{w}_2}
\\
&&&&\ar[r]_{\widetilde{w}_2}&
}
$$
so we make this substitution in the diagram above to obtain,
$$
\xymatrix@C=4.5em{
\ar[r]^{u\widetilde{t}x} \ar[dd]_{h_1\widetilde{t}x}
\ar@{}[ddr]|{\rho_1\widetilde{t}x\Downarrow} & \ar[d]_{\overline{r}_1}
\ar@{}[ddr]|{\widetilde{\psi}^{-1}\Downarrow} \ar[dr]|{\overline{r}_2}\ar@{}[drr]|{\widetilde{\varphi}\Downarrow}\ar[r]^{\overline{r}_1}&\ar[dr]^{\overline{s}_1}
\\
&\ar[d]_{g_{1,1}}\ar[dr]|{\overline{t}_1}\ar@{}[ddr]|{\tau_1\Downarrow} & \ar[d]_{\overline{t}_2}\ar[r]_{\overline{s}_2}\ar@{}[dr]|{\zeta_2\Downarrow} & \ar[d]^{\overline{w}_2}\ar[r]^{\overline{f}_1}\ar@{}[dr]|{\alpha_1^{-1}\Downarrow} &\ar[dd]^{w_2\overline{w}_3}
\\
\ar[r]^{r_1}\ar[d]_{r_2}\ar@{}[dr]|{\psi\Downarrow} & \ar[dr]_{t_1} &\ar[d]_{\widetilde{f}_1} \ar[r]_{\widetilde{w}_2}\ar@{}[dr]|{\beta_1\Downarrow}&\ar[dr]|{f_1} &
\\
\ar[rr]_{t_2}&&\ar[rr]_{w_2\widetilde{w}_3}&&
}
$$
We use (\ref{eqveqns}) again; this time to rewrite the bottom right-hand corner of the diagram:
$$
\xymatrix@C=4.5em{
\ar[r]^{u\widetilde{t}x} \ar[dd]_{h_1\widetilde{t}x}
\ar@{}[ddr]|{\rho_1\widetilde{t}x\Downarrow} & \ar[d]_{\overline{r}_1}
\ar@{}[ddr]|{\widetilde{\psi}^{-1}\Downarrow} \ar[dr]|{\overline{r}_2}\ar@{}[drr]|{\widetilde{\varphi}\Downarrow}\ar[r]^{\overline{r}_1}&\ar[dr]^{\overline{s}_1}
\\
&\ar[d]_{g_{1,1}}\ar[dr]|{\overline{t}_1}\ar@{}[ddr]|{\tau_1\Downarrow} & \ar[d]_{\overline{t}_2} \ar[r]_{\overline{s}_2}\ar[dr]|{g_{1,2}} & \ar@{}[d]|{\sigma_2\Downarrow}\ar[r]^{\overline{f}_1} &\ar[dd]^{w_2\overline{w}_3}
\\
\ar[r]^{r_1}\ar[d]_{r_2}\ar@{}[dr]|{\psi\Downarrow} & \ar[dr]_{t_1} &\ar[d]^{\widetilde{f}_1} \ar@{}[r]|{\stackrel{\tau_2}{\Leftarrow}}&\ar@{}[dr]|{w_2\varepsilon_2\Downarrow}\ar[dl]|{t_2}\ar[ur]_{s_2} &
\\
\ar[rr]_{t_2}&&\ar[rr]_{w_2\widetilde{w}_3}&&
}
$$
and by the definition of $\rho_2$, this is equal to
$$
\xymatrix@C=4.5em{
\ar[r]^{u\widetilde{t}x} \ar[dd]_{h_1\widetilde{t}x}
\ar@{}[3,2]|{\rho_2x\Downarrow} & 
\ar[dr]|{\overline{r}_2}\ar@{}[drr]|{\widetilde{\varphi}\Downarrow}\ar[r]^{\overline{r}_1}&\ar[dr]^{\overline{s}_1}
\\
& &  \ar[r]_{\overline{s}_2}\ar[dr]|{g_{1,2}} & \ar@{}[d]|{\sigma_2\Downarrow}\ar[r]^{\overline{f}_1} &\ar[dd]^{w_2\overline{w}_3}
\\
\ar[d]_{r_2} &  & &\ar@{}[dr]|{w_2\varepsilon_2\Downarrow}\ar[dl]|{t_2}\ar[ur]_{s_2} &
\\
\ar[rr]_{t_2}&&\ar[rr]_{w_2\widetilde{w}_3}&&
}
$$
This completes our proof of equation (\ref{bigclaim}).
Since $\varepsilon_2$ is invertible we can compose both sides of (\ref{bigclaim}) by $w_2\varepsilon_2^{-1}r_2h_1\widetilde{t}x$ and it follows that
$$
\xymatrix@C=4em@R=3em{
&&&&&\ar[dr]^{\overline{s}_1} \ar@{}[d]|{\widetilde\varphi u\widetilde{t}x\Downarrow} &
\\
\ar[r]^{\overline{r}_1u\widetilde{t}x} \ar[d]_{h_1\widetilde{t}x}
\ar@{}[dr]|{\rho_1\widetilde{t}x\Downarrow} & \ar[d]|{g_{1,1}} \ar[r]^{\overline{s}_1}
\ar@{}[dr]|{\sigma_1\Downarrow} & \ar[d]^{\overline{f}_1} \ar@{}[drr]|{\textstyle =} &&\ar[ur]^{\overline{r}_1u\widetilde{t}x} \ar[r]_{\overline{r}_2u\widetilde{t}x}\ar[d]_{h_1\widetilde{t}x} \ar@{}[dr]|{\rho_2x\Downarrow} &\ar[d]|{g_{1,2}}\ar[r]_{\overline{s}_2}\ar@{}[dr]|{\sigma_2\Downarrow} &\ar[d]^{\overline{f}_1}
\\
\ar[dr]_{r_2}\ar[r]_{r_1} & \ar@{}[d]|{\varphi\Downarrow} \ar[r]_{s_1}&\ar[dr]^{w_2\overline{w}_3} &&\ar[r]_{r_2}&\ar[r]_{s_2} 
&\ar[dr]^{w_2\overline{w}_3}
\\
&\ar[ur]_{s_2}&&&&&&&
}
$$
It follows from condition {\bf[WB4]} that there is an arrow $y$ such that
$$
\xymatrix@C=4em@R=3em{
&&&&&\ar[dr]^{\overline{s}_1} \ar@{}[d]|{\widetilde\varphi u\widetilde{t}x\Downarrow} &
\\
\ar[r]^{\overline{r}_1u\widetilde{t}xy} \ar[d]_{h_1\widetilde{t}xy}
\ar@{}[dr]|{\rho_1\widetilde{t}xy\Downarrow} & \ar[d]|{g_{1,1}} \ar[r]^{\overline{s}_1}
\ar@{}[dr]|{\sigma_1\Downarrow} & \ar[d]^{\overline{f}_1} \ar@{}[drr]|{\textstyle =} &&\ar[ur]^{\overline{r}_1u\widetilde{t}xy} \ar[r]_{\overline{r}_2u\widetilde{t}xy}\ar[d]_{h_1\widetilde{t}xy} \ar@{}[dr]|{\rho_2xy\Downarrow} &\ar[d]|{g_{1,2}}\ar[r]_{\overline{s}_2}\ar@{}[dr]|{\sigma_2\Downarrow} &\ar[d]^{\overline{f}_1}
\\
\ar[dr]_{r_2}\ar[r]_{r_1} & \ar@{}[d]|{\varphi\Downarrow} \ar[r]_{s_1}&&&\ar[r]_{r_2}&\ar[r]_{s_2} 
&
\\
&\ar[ur]_{s_2}&&&&&&
}
$$
Hence, it follows from the arguments above that the cells
$$
\xymatrix@C=4em{
&
\\
\ar@/^1ex/[ur]^{\overline{s}_1}\ar@/_1ex/[dr]_{\overline{t}_1}&\ar[l]_{\overline{r}_1u\widetilde{t}xy}\ar[r]^{\overline{r}_2u\widetilde{t}xy}\ar@{}[u]|{\stackrel{\widetilde\varphi u\widetilde{t}xy}{\Rightarrow}} \ar@{}[d]|{\stackrel{\widetilde\psi u\widetilde{t}xy}{\Rightarrow}} &\ar@/_1ex/[ul]_{\overline{s}_2}\ar@/^1ex/[dl]^{\overline{t}_2}
\\
&
}
$$
witness to the equivalence of the 2-cell diagrams in (\ref{equivcells2}).
\end{proof}

\begin{rmk}
Analogous to the situation in Proposition~\ref{unique-2-cell}, we say that the 2-cell diagrams in (\ref{equivcells}) (respectively in (\ref{equivcells2})) {\em connect} the 2-cell configurations in (\ref{leftconfigurations}) (respectively (\ref{rightconfigurations})). Propositions~\ref{double-layer1} and \ref{double-layer2}  only state uniqueness results, but it is not hard to prove existence as well. Since we will only need uniqueness in the proof of associativity coherence, we will not include the proofs of existence.  
\end{rmk}

\begin{prop} \label{P:coherence}
For any composable path of four spans,
$$\xymatrix{
&\ar[dl]_{w_1}\ar[dr]_{f_1}&&\ar[dl]_{w_2}\ar[dr]_{f_2}&&\ar[dl]_{w_3}\ar[dr]_{f_3}&&\ar[dl]_{w_4}\ar[dr]^{f_4}\\
&&&&&&&&
}
$$
the associativity 2-cells defined in Propostion~\ref{assoc-2-cells} make the associativity coherence pentagon commute.
\end{prop}

\begin{proof}

The following diagram shows the associativity coherence pentagon.  
\[
\scalebox{.58}{\includegraphics{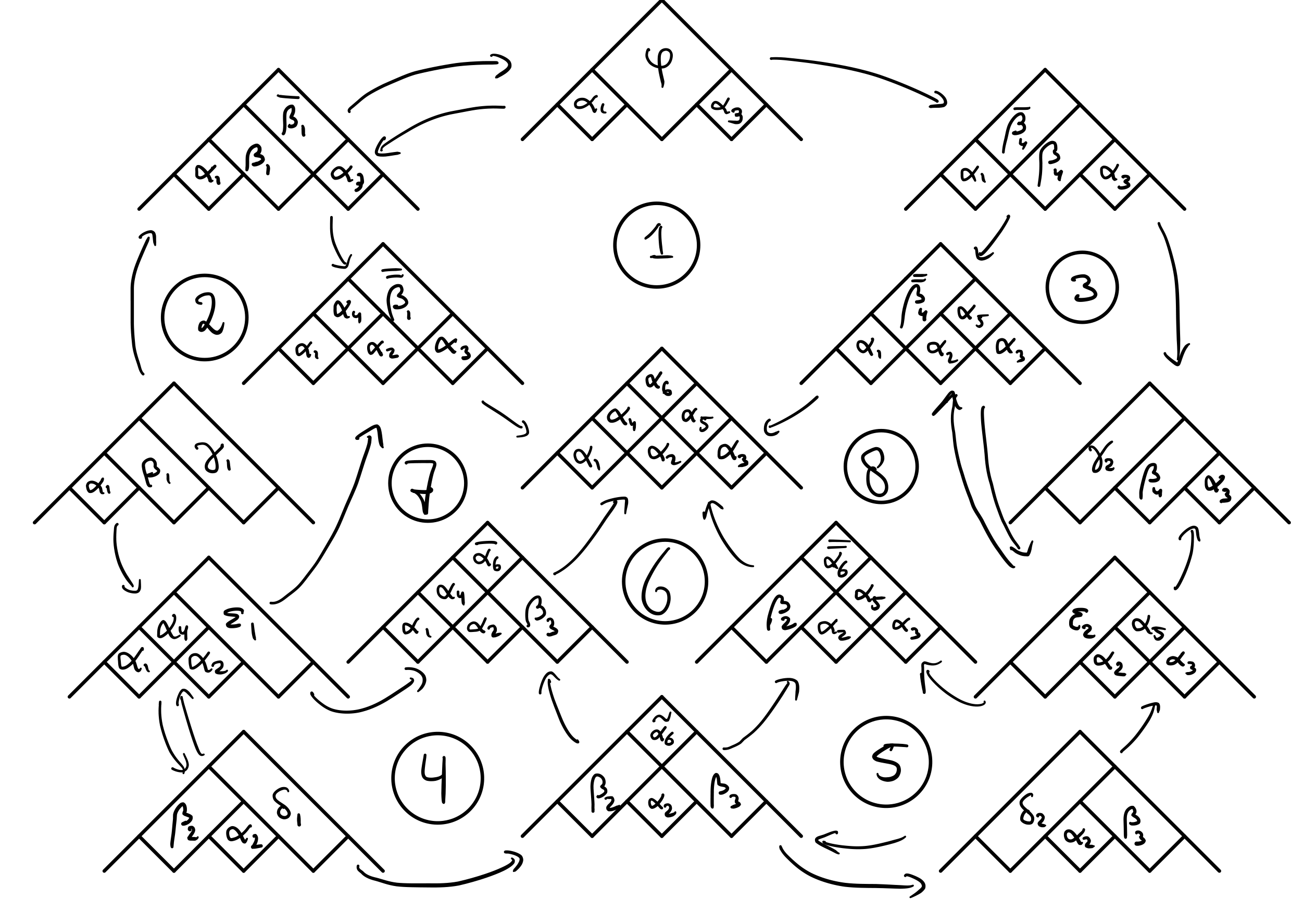}}
\]
 We  have divided the pentagon into regions corresponding to various subdivisions, and we will show that   each region commutes by one of the three results in Propositions~\ref{unique-2-cell},  \ref{double-layer1} and \ref{double-layer2}.  We sketch the argument for each region, leaving the details for the reader.

For region \raisebox{.5pt}{\textcircled{\raisebox{-.9pt} {1}}}  both composites provide a whiskering of a 2-cell that connects the squares
$$
\xymatrix@C=1.5em@R=1.5em
{
\ar[dd]\ar[rr]\ar@{}[ddrr]|{\stackrel{\varphi}{\Leftarrow}}&&\ar[dd] &&\ar[d]\ar[r]\ar@{}[dr]|{\stackrel{\alpha_6}{\Leftarrow}}&\ar[d]\ar[r]\ar@{}[dr]|{\stackrel{\alpha_5}{\Leftarrow}}&\ar[d]
\\
&&&\mbox{and}&\ar[d]\ar[r]\ar@{}[dr]|{\stackrel{\alpha_4}{\Leftarrow}}&\ar[d]\ar[r]\ar@{}[dr]|{\stackrel{\alpha_2}{\Leftarrow}}&\ar[d]
\\
\ar[d]\ar[rr]&&&&\ar[d]\ar[r]&\ar[r]&
\\
&&&&
}
$$
Since there is only one such 2-cell by Proposition~\ref{unique-2-cell}, this region commutes.

For region \raisebox{.5pt}{\textcircled{\raisebox{-.9pt} {2}}}  the two compositions connect the diagrams
$$
\xymatrix@R=1.5em@C=1.5em{
\ar[d]\ar@{}[1,3]|{\stackrel{\gamma_1}{\Leftarrow}}\ar[rrr]&&&\ar[d] && \ar[d]
\ar@{}[1,2]|{\stackrel{\overline{\overline{\beta}}_1}{\Leftarrow}}\ar[rr]&&\ar[d]\ar[r]
\ar@{}[dr]|{\stackrel{\alpha_3}{\Leftarrow}}&\ar[d]
\\
\ar[d]\ar[rr]\ar@{}[drr]|{\stackrel{\beta_1}{\Leftarrow}}&&\ar[d]\ar[r]& &\mbox{and} & \ar[d]\ar[r]\ar@{}[dr]|{\stackrel{\alpha_4}{\Leftarrow}}&\ar[d]\ar[r]\ar@{}[dr]|{\stackrel{\alpha_2}{\Leftarrow}} &\ar[d]\ar[r]&
\\
\ar[d]\ar[rr]&& &&&\ar[d]\ar[r]&\ar[r]&
\\
&&&&&
}
$$
as in Proposition~\ref{double-layer1}. 

Region \raisebox{.5pt}{\textcircled{\raisebox{-.9pt} {3}}}  is the dual of region \raisebox{.5pt}{\textcircled{\raisebox{-.9pt} {2}}} and follows from  Proposition~\ref{double-layer2}.

For region \raisebox{.5pt}{\textcircled{\raisebox{-.9pt} {4}}}  commutativity is obtained from Proposition~\ref{double-layer1} applied to 
$$
\xymatrix@R=1.5em@C=1.5em{
\ar[d]\ar[rrr]\ar@{}[1,3]|{\stackrel{\delta_1}{\Leftarrow}}&&&\ar[d] &&\ar[r]\ar[d]\ar@{}[dr]|{\stackrel{\overline{\alpha}_6}{\Leftarrow}}&\ar[d]\ar[rr]\ar@{}[drr]|{\stackrel{\beta_3}{\Leftarrow}}&&\ar[d]
\\
\ar[dd]\ar[r]\ar@{}[ddr]|{\stackrel{\beta_2}{\Leftarrow}}&\ar[rr]\ar[dd]&& &\mbox{and}&\ar[d]\ar[r]\ar@{}[dr]|{\stackrel{\alpha_4}{\Leftarrow}}&\ar[d]\ar[rr]&&
\\
&&&&&\ar[d]\ar[r]\ar@{}[dr]|{\stackrel{\alpha_1}{\Leftarrow}}&\ar[d]
\\
\ar[r]\ar[d]&&&&&\ar[r]\ar[d]&
\\
&&&&&
}
$$
where we view the pasting of $\alpha_1$ and $\alpha_4$ as a single cell.

Region \raisebox{.5pt}{\textcircled{\raisebox{-.9pt} {5}}}  is the dual  of region \raisebox{.5pt}{\textcircled{\raisebox{-.9pt} {4}}}  and commutativity can be obtained by applying Proposition~\ref{double-layer2} to
$$
\xymatrix@R=1.5em@C=1.5em{
&\ar[d]\ar[l]&&&\ar[lll]\ar@{}[dlll]|{\stackrel{\delta_2}{\Leftarrow}}\ar[d]&&&\ar[d]\ar[l]&&\ar[ll]\ar@{}[dll]|{\stackrel{\beta_2}{\Leftarrow}}\ar[d]&\ar[l]\ar[d]\ar@{}[dl]|{\stackrel{\overline{\overline{\alpha}}_6}{\Leftarrow}}
\\
&&&\ar[dd]\ar[ll]&\ar[l]\ar@{}[ddl]|{\stackrel{\beta_3}{\Leftarrow}}\ar[dd]&\mbox{and} &&&&\ar[ll]\ar[d]&\ar[l]\ar[d]\ar@{}[dl]|{\stackrel{\alpha_5}{\Leftarrow}}
\\
&&&& && &&&\ar[d]&\ar[d]\ar[l]\ar@{}[dl]|{\alpha_3}
\\
&&&&\ar[l]&& &&&&\ar[l]
}
$$
where we view the pasting of $\alpha_5$ and $\alpha_3$ as a single cell and the pasting of $\overline{\overline{\alpha}}_6$ and $\beta_2$ as a single cell.

Region \raisebox{.5pt}{\textcircled{\raisebox{-.9pt} {6}}}  could be done with an application of either Proposition~\ref{double-layer1} or Proposition~\ref{double-layer2}. If we use Proposition~\ref{double-layer1}, we focus on the diagrams,
$$
\xymatrix@C=1.5em@R=1.5em{
\ar[r]\ar@{}[dr]|{\stackrel{\widetilde{\alpha}_6}{\Leftarrow}}\ar[d]&\ar[d]\ar[rr]\ar@{}[drr]|{\stackrel{\beta_3}{\Leftarrow}}&&\ar[d]&&\ar[r]\ar@{}[dr]|{\stackrel{\alpha_6}{\Leftarrow}}\ar[d]&\ar[d]\ar[r]\ar@{}[dr]|{\stackrel{\alpha_5}{\Leftarrow}}&\ar[d]\ar[r]\ar@{}[dr]|{\stackrel{\alpha_3}{\Leftarrow}}&\ar[d]
\\
\ar[d]\ar[r]\ar@{}[dr]|{\stackrel{\alpha_4}{\Leftarrow}}&\ar[d] \ar[rr]&&&\mbox{and}&\ar[dd]\ar[r]\ar@{}[ddr]|{\stackrel{\beta_2}{\Leftarrow}}&\ar[dd]\ar[r]&\ar[r]&
\\
\ar[d]\ar[r]\ar@{}[dr]|{\stackrel{\alpha_1}{\Leftarrow}}&\ar[d]&&&&&
\\
\ar[r]&&&&&\ar[r]&
}
$$
Here we consider the pasting of $\alpha_4$ and $\alpha_1$ as a single cell, the pasting of $\widetilde{\alpha}_6$ and $\beta_3$ as a single cell, and the pasting of $\alpha_6$, $\alpha_5$ and $\alpha_3$ as a single cell.

For region \raisebox{.5pt}{\textcircled{\raisebox{-.9pt} {7}}}  the two ways of composing provide to 2-cells that connect the rectangles,
$$
\xymatrix@C=1.5em@R=1.5em
{\ar[rrr]\ar[d]\ar@{}[drrr]|{\stackrel{\varepsilon_1}{\Leftarrow}}&&&\ar[d]&\ar@{}[d]|{\mbox{and}}&\ar[d]\ar[r]\ar@{}[dr]|{\stackrel{\alpha_6}{\Leftarrow}}&\ar[d]\ar[r]\ar@{}[dr]|{\stackrel{\alpha_5}{\Leftarrow}}&\ar[d]\ar[r]\ar@{}[dr]|{\stackrel{\alpha_3}{\Leftarrow}}&\ar[d]
\\
\ar[rrr]\ar[d]&&& && \ar[d]\ar[r]&\ar[r]&\ar[r]&\\
&&&&&
}
$$
and there is only one such cell by Proposition~\ref{unique-2-cell}, so this region commutes.

Region \raisebox{.5pt}{\textcircled{\raisebox{-.9pt} {8}}}  is the dual of region \raisebox{.5pt}{\textcircled{\raisebox{-.9pt} {7}}}  whose  two compositions give the 2-cell connecting the rectangles,
$$
\xymatrix{
&\ar[l]\ar[d]&&&\ar[lll]\ar[d]\ar@{}[dlll]|{\stackrel{\varepsilon_2}{\Leftarrow}}&\ar@{}[d]|{\mbox{and}}&
&\ar[l]\ar[d]&\ar[d]\ar[l]\ar@{}[dl]|{\stackrel{\alpha_1}{\Leftarrow}}&\ar[d]\ar[l]\ar@{}[dl]|{\stackrel{\alpha_4}{\Leftarrow}}&\ar[l]\ar[d]\ar@{}[dl]|{\stackrel{\alpha_6}{\Leftarrow}}
\\
&&&&\ar[lll]&&&&\ar[l]&\ar[l]&\ar[l]
}
$$
\end{proof}
\section{Well-Definedness of Composition} \label{well-defined}
In this appendix we show that vertical composition and horizontal whiskering are well-defined on equivalence classes of 2-cell diagrams.
 We start by observing that the equivalence relation on 2-cell diagrams is generated by the following non-symmetric relation:
$$
\xymatrix{
& C_1\ar[dl]_{u_1}\ar[dr]^{f_1}& && &C_1\ar[dl]_{u_1}\ar[dr]^{f_1}
\\
A\ar@{}[r]|{\alpha'\Downarrow}&D'\ar[u]_{v_1'} \ar[d]^{v_2'}\ar@{}[r]|{\beta'\Downarrow} &B &\preceq&A\ar@{}[r]|{\alpha\Downarrow}&D\ar[u]_{v_1} \ar[d]^{v_2}\ar@{}[r]|{\beta\Downarrow} &B
\\
& C_2\ar[ul]^{v_2}\ar[ur]_{f_2}&&&&C_2\ar[ul]^{v_2}\ar[ur]_{f_2}}
$$
if there are invertible 2-cells $\gamma_1,\,\gamma_2$ such that 
$$
\xymatrix@C=3em@R=3em{
&C_1\ar[dr]^{u_1} & &&C_1\ar[dr]^{u_1} 
\\
D'\ar[r]^t \ar[ur]^{v_1'}\ar[dr]_{v_2'}& D\ar@{}[ul]|(.25){\stackrel{\gamma_1}{\Rightarrow}}\ar@{}[dl]|(.25){\stackrel{\gamma_2}{\Leftarrow}}\ar[u]_{v_1}\ar[d]^{v_2}\ar@{}[r]|{\alpha\Downarrow}&A &=&D' \ar[u]^{v_1'}\ar[d]_{v_2'}\ar@{}[r]|{\alpha'\Downarrow}&A
\\
&C_2\ar[ur]_{u_2}&&&C_2\ar[ur]_{u_2}
}
$$
and 
$$
\xymatrix@C=3em@R=3em{
&C_1\ar[dr]^{f_1} & &&C_1\ar[dr]^{f_1} 
\\
D'\ar[r]^t \ar[ur]^{v_1'}\ar[dr]_{v_2'}& D\ar@{}[ul]|(.25){\stackrel{\gamma_1}{\Rightarrow}}\ar@{}[dl]|(.25){\stackrel{\gamma_2}{\Leftarrow}}\ar[u]_{v_1}\ar[d]^{v_2}\ar@{}[r]|{\beta\Downarrow}&B &=&D' \ar[u]^{v_1'}\ar[d]_{v_2'}\ar@{}[r]|{\beta'\Downarrow}&B 
\\
&C_2\ar[ur]_{f_2}&&&C_2\ar[ur]_{f_2}
}
$$ where $u_1t \in \calW$ (equivalently, $u_1' \in \calW$). 
So it is sufficient to check well-definedness with respect to this relation.
The main tool we will use for this is Proposition~\ref{anytwosquares}.
We will repeatedly create squares that can be compared using this proposition and the cells produced that way will show that the 2-cell diagrams resulting from composing or whiskering equivalent 2-cell diagrams are again equivalent.

\begin{prop}\label{vertical-well-defined}
Vertical composition of 2-cell diagrams is well-defined on equivalence classes.
\end{prop}

\begin{proof}
Consider two 2-cell diagrams
\begin{equation}\label{firsttwo}
\xymatrix@C=4em{
&\ar[dl]_{u_1}\ar[dr]^{f_1} & && &  \ar[dl]_{u_2}\ar[dr]^{f_2}&
\\
\ar@{}[r]|{\Downarrow\alpha_1} & \ar[u]^{v_1}\ar[d]_{v_2}\ar@{}[r]|{\Downarrow\beta_1} &\ar@{}[rr]|{ \mbox{and  }}
		&& \ar@{}[r]|{\Downarrow\alpha_2} & \ar[u]^{v_3}\ar[d]_{v_4}\ar@{}[r]|{\Downarrow\beta_2} & 
\\
&\ar[ul]^{u_2}\ar[ur]_{f_2} &&&&\ar[ul]^{u_3}\ar[ur]_{f_3} &
}
\end{equation}
as in Section~\ref{newBF2} and two 2-cell diagrams
\begin{equation}\label{secondtwo}
\xymatrix@C=4em{
&\ar[dl]_{u_1}\ar[dr]^{f_1} & && &  \ar[dl]_{u_2}\ar[dr]^{f_2}&
\\
\ar@{}[r]|{\Downarrow\alpha_1'} & \ar[u]^{w_1}\ar[d]_{w_2}\ar@{}[r]|{\Downarrow\beta_1'} &\ar@{}[rr]|{ \mbox{and  }}
		&& \ar@{}[r]|{\Downarrow\alpha_2'} & \ar[u]^{w_3}\ar[d]_{w_4}\ar@{}[r]|{\Downarrow\beta_2'} & 
\\
&\ar[ul]^{u_2}\ar[ur]_{f_2} &&&&\ar[ul]^{u_3}\ar[ur]_{f_3} &
}
\end{equation}
with 2-cells $\gamma_1,\gamma_2,\gamma_3$ and $\gamma_4$ such that
\begin{equation}\label{eqrel1}
\xymatrix@C=3em@R=3em{
&C_1\ar[dr]^{u_1} & &&C_1\ar[dr]^{u_1} 
\\
D'\ar[r]^{t_1} \ar[ur]^{w_1}\ar[dr]_{w_2}& D\ar@{}[ul]|(.25){\stackrel{\gamma_1}{\Rightarrow}}\ar@{}[dl]|(.25){\stackrel{\gamma_2}{\Leftarrow}}\ar[u]_{v_1}\ar[d]^{v_2}\ar@{}[r]|{\alpha_1\Downarrow}&A &=&D' \ar[u]^{w_1}\ar[d]_{w_2}\ar@{}[r]|{\alpha_1'\Downarrow}&A
\\
&C_2\ar[ur]_{u_2}&&&C_2\ar[ur]_{u_2}
}
\end{equation}
\begin{equation}\label{eqrel2}
\xymatrix@C=3em@R=3em{
&C_1\ar[dr]^{f_1} & &&C_1\ar[dr]^{f_1} 
\\
D'\ar[r]^{t_1} \ar[ur]^{w_1}\ar[dr]_{w_2}& D\ar@{}[ul]|(.25){\stackrel{\gamma_1}{\Rightarrow}}\ar@{}[dl]|(.25){\stackrel{\gamma_2}{\Leftarrow}}\ar[u]_{v_1}\ar[d]^{v_2}\ar@{}[r]|{\beta_1\Downarrow}&B &=&D' \ar[u]^{w_1}\ar[d]_{w_2}\ar@{}[r]|{\beta_1'\Downarrow}&B 
\\
&C_2\ar[ur]_{f_2}&&&C_2\ar[ur]_{f_2}
}
\end{equation}
\begin{equation}\label{eqrel3}
\xymatrix@C=3em@R=3em{
&C_1\ar[dr]^{u_1} & &&C_1\ar[dr]^{u_1} 
\\
D'\ar[r]^{t_2} \ar[ur]^{w_4}\ar[dr]_{w_3}& D\ar@{}[ul]|(.25){\stackrel{\gamma_3}{\Rightarrow}}\ar@{}[dl]|(.25){\stackrel{\gamma_{4}}{\Leftarrow}}\ar[u]_{v_3}\ar[d]^{v_4}\ar@{}[r]|{\alpha_2\Downarrow}&A &=&D' \ar[u]^{w_3}\ar[d]_{w_4}\ar@{}[r]|{\alpha_2'\Downarrow}&A
\\
&C_2\ar[ur]_{u_2}&&&C_2\ar[ur]_{u_2}
}
\end{equation}
and
\begin{equation}\label{eqrel4}
\xymatrix@C=3em@R=3em{
&C_1\ar[dr]^{f_1} & &&C_1\ar[dr]^{f_1} 
\\
D'\ar[r]^{t_2} \ar[ur]^{w_3}\ar[dr]_{w_4}& D\ar@{}[ul]|(.25){\stackrel{\gamma_3}{\Rightarrow}}\ar@{}[dl]|(.25){\stackrel{\gamma_4}{\Leftarrow}}\ar[u]_{v_3}\ar[d]^{v_4}\ar@{}[r]|{\beta_2\Downarrow}&B &=&D' \ar[u]^{w_3}\ar[d]_{w_4}\ar@{}[r]|{\beta_2'\Downarrow}&B 
\\
&C_2\ar[ur]_{f_2}&&&C_2\ar[ur]_{f_2}
}
\end{equation}
Vertical composition of the two 2-cell diagrams in (\ref{firsttwo})
is given by
\begin{equation}\label{verticalcomp1}
\xymatrix@C=5em@R=3em{
&&\ar[ddll]_{u_1}\ar[ddrr]^{f_1}&&
\\
&&\ar[u]_{v_1} \ar@{}[dll]|{\alpha_1\Downarrow} \ar@{}[drr]|{\beta_1\Downarrow} \ar[dl]_{v_2}\ar[dr]^{v_2} &&
\\
&\ar[l]_{u_2}\ar@{}[r]|{\delta_1\Downarrow} & \ar[u]|{x_1}\ar@{}[r]|{\delta_1\Downarrow} \ar[d]|{x_2} & \ar[r]^{f_2} &
\\
&& \ar[ul]^{v_3}\ar@{}[ull]|{\alpha_2\Downarrow} \ar[d]^{v_4}\ar@{}[urr]|{\beta_2\Downarrow} \ar[ur]_{v_3} &&
\\
&&\ar[uull]^{u_3}\ar[uurr]_{f_3}
}
\end{equation}
and vertical composition of the two 2-cell diagrams in (\ref{secondtwo})
is given by:
\begin{equation}\label{verticalcomp2}
\xymatrix@C=5em@R=3em{
&&\ar[ddll]_{u_1}\ar[ddrr]^{f_1}&&
\\
&&\ar[u]_{w_1} \ar@{}[dll]|{\alpha'_1\Downarrow} \ar@{}[drr]|{\beta'_1\Downarrow} \ar[dl]_{w_2}\ar[dr]^{w_2} &&
\\
&\ar[l]_{u_2}\ar@{}[r]|{\delta_2\Downarrow} & \ar[u]|{y_1}\ar@{}[r]|{\delta_2\Downarrow} \ar[d]|{y_2} & \ar[r]^{f_2} &
\\
&& \ar[ul]^{w_3}\ar@{}[ull]|{\alpha'_2\Downarrow} \ar[d]^{w_4}\ar@{}[urr]|{\beta'_2\Downarrow} \ar[ur]_{w_3} &&
\\
&&\ar[uull]^{u_3}\ar[uurr]_{f_3}
}
\end{equation}
for suitable arrows $x_1,x_2,y_1,y_2$ such that $u_1v_1x_1$ and $u_1v_1y_1$ are in $\frakW$ and suitable invertible 2-cells  $\delta_1$ and $\delta_2$.
By equations (\ref{eqrel1})-(\ref{eqrel4}) the 2-cell diagram (\ref{verticalcomp2}) can be rewritten as:
\begin{equation}\label{secondtwoA}
\xymatrix@C=5em@R=3em{
&&& \ar@/_2ex/[2,-3]_{u_1}\ar@/^2ex/[2,3]^{f_1}&&
\\
&& \ar[ur]^{v_1} \ar[d]_{v_2} \ar@{}[dr]|(.25){\stackrel{\gamma_2}{\Rightarrow}} & \ar[u]_{w_1} \ar@{}[ul]|(.25){\stackrel{\gamma_1}{\Leftarrow}} \ar@{}[ur]|(.25){\stackrel{\gamma_1}{\Rightarrow}}   \ar[dl]_{w_2} \ar[dr]^{w_2}\ar[r]|{t_1} \ar[l]|{t_1}& \ar[ul]_{v_1} \ar[d]^{v_2} \ar@{}[dl]|(.25){\stackrel{\gamma_2}{\Leftarrow}}&&
\\
\ar@{}[urr]|{\alpha_1\Downarrow} \ar@{}[drr]|{\alpha_2\Downarrow}&&\ar[ll]_{u_2}\ar@{}[r]|{\delta_2\Downarrow} & \ar[u]|{y_1}\ar@{}[r]|{\delta_2\Downarrow} \ar[d]|{y_2} & \ar[rr]^{f_2} \ar@{}[urr]|{\beta_1\Downarrow}\ar@{}[drr]|{\beta_2\Downarrow}&&
\\
&&\ar[u]^{v_3} \ar[dr]_{v_4} \ar@{}[ur]|(.25){\stackrel{\gamma_3}{\Leftarrow}}& \ar[ul]^{w_3} \ar@{}[dl]|(.25){\stackrel{\gamma_4}{\Rightarrow}} \ar@{}[dr]|(.25){\stackrel{\gamma_4}{\Leftarrow}}\ar[d]^{w_4} \ar[ur]_{w_3}\ar[r]|{t_2} \ar[l]|{t_2}  &\ar[u]_{v_3} \ar@{}[ul]|(.25){\stackrel{\gamma_3}{\Rightarrow}} \ar[dl]^{v_4} &&
\\
&&&\ar@/^2ex/[-2,-3]^{u_3} \ar@/_2ex/[-2,3]_{f_3}&&&
}
\end{equation}
We can now apply Proposition~\ref{anytwosquares} to
$$
\xymatrix@C=4em{
\ar[r]^{x_1}\ar[d]_{x_2}\ar@{}[dr]|{u_2\delta_2\Downarrow} &\ar[d]^{u_2v_2} \ar@{}[drr]|{\mbox{and}} &&\ar[r]^{y_1}\ar@{}[dr]|{\delta_2\Downarrow}\ar[d]_{y_2} & \ar[r]^{t_1}\ar@{}[ddr]|{u_2\gamma_2\Downarrow}\ar[d]^{w_2} &\ar[dd]^{u_2v_2}
\\
\ar[r]_{u_2v_3}\ar[d]_{u_2v_3} &&&\ar[d]_{t_2}\ar[r]_{w_3}\ar@{}[dr]|{u_2\gamma_3\Downarrow} & \ar[dr]_{u_2} &
\\
&&&\ar[rr]_{u_2v_3} \ar[d]_{u_2v_3}&&
\\
&&&
}
$$
This gives us invertible 2-cells $\varepsilon_1$ and $\varepsilon_2$ as in the following diagram,
$$
\xymatrix{
&&
\\
\ar@/^1.5ex/[ur]^{x_1}\ar@/_1.5ex/[dr]_{x_2} & \ar[l]|{r_1}\ar[r]|{r_2}\ar@{}[u]|{\stackrel{\varepsilon_1}{\Rightarrow}}\ar@{}[d]|{\stackrel{\varepsilon_2}{\Rightarrow}} & \ar@/_1.5ex/[ul]_{t_1y_1}\ar@/^1.5ex/[dl]^{t_2y_2}
\\
&&
}
$$
where $u_2v_3x_2r_1\in\frakW$ and such that
$$
\xymatrix@C=3em{
\ar[rr]^{r_1}\ar[d]_{r_2}\ar@{}[drr]|{\varepsilon_1\Downarrow} && \ar[d]^{x_1}
\\
\ar[d]_{y_2}\ar[r]^{y_1}\ar@{}[dr]|{\delta_2\Downarrow}&\ar[d]^{w_2}\ar[r]^{t_1}\ar@{}[ddr]|{u_2\gamma_2\Downarrow} & \ar[dd]^{u_2v_2}\ar@{}[drr]|= &&\ar[d]_{r_2}\ar[r]^{r_1}\ar@{}[dr]|{\varepsilon_2\Downarrow} &\ar[r]^{x_1} \ar[d]|{x_2}\ar@{}[dr]|{u_2\delta_2\Downarrow} & \ar[d]^{u_2v_2}
\\
\ar[d]_{t_2}\ar@{}[drr]|{u_2\gamma_3\Downarrow}\ar[r]_{w_3} & \ar[dr]|{u_2} & && \ar[r]_{t_2y_2}&\ar[r]_{u_2v_3}&
\\
\ar[rr]_{u_2v_3} &&
}
$$
Now the reader may check that the following diagram can be used to show that the 2-cells diagrams (\ref{verticalcomp1}) and (\ref{verticalcomp2}) are equivalent:
$$
\xymatrix{
&\ar@{}[d]|{\stackrel{\gamma_1^{-1}}{\Rightarrow}}
\\
\ar[ur]^{v_1} && \ar[ll]|{t_1} \ar[ul]_{w_1}
\\
\ar[u]^{x_1}\ar@{}[urr]|{\stackrel{\varepsilon_1}{\Rightarrow}}\ar[d]_{x_2} & \ar[l]|{r_1}\ar[r]|{r_2} & \ar[u]_{y_1}\ar[d]^{y_2}
\\
\ar@{}[urr]|{\stackrel{\varepsilon_2}{\Rightarrow}} \ar[dr]_{v_4}&\ar@{}[d]|{\stackrel{\gamma_4}{\Rightarrow}}& \ar[ll]|{t_2}\ar[dl]^{w_4}
\\
&&
}
$$
\end{proof}

\begin{prop}\label{wd-left-whiskering}
Left whiskering of a 2-cell diagram and an arrow in the bicategory of fractions is well-defined on 
equivalence classes of 2-cell diagrams.
\end{prop}

\begin{proof}
We will again consider a generator of the equivalence relation:

$$
\xymatrix@C=4em{
&\ar[dl]_{u_1}\ar[dr]^{f_1} & && &&\ar[dll]_{u_1}\ar[drr]^{f_1}&
\\
\ar@{}[r]|{\alpha\Downarrow}&\ar[u]|{s_1}\ar[d]|{s_2}\ar@{}[r]|{\beta\Downarrow}&&\preceq&\ar@{}[r]|{\alpha\Downarrow}&\ar[ur]|{s_1}\ar[dr]|{s_2}&\ar[l]|{r}\ar[u]|{t_1}\ar[d]|{t_2}\ar@{}[ul]|(.25){\stackrel{\varepsilon_1}{\Leftarrow}}\ar@{}[ur]|(.25){\stackrel{\varepsilon_1}{\Rightarrow}}\ar@{}[dr]|(.25){\stackrel{\varepsilon_2}{\Leftarrow}}\ar@{}[dl]|(.25){\stackrel{\varepsilon_2}{\Rightarrow}}\ar[r]|r &\ar[ul]|{s_1}\ar[dl]|{s_2}\ar@{}[r]|{\beta\Downarrow} &
\\
&\ar[ul]^{u_2}\ar[ur]_{f_2} &&&&&\ar[ull]^{u_2}\ar[urr]_{f_2}&
}
$$
Whiskering these 2-cell diagrams with $\xymatrix@1{&\ar[l]_{v}\ar[r]^g &}$ gives us the following 2-cell diagrams:
\begin{equation}\label{two_diagrams}
\xymatrix@C=3em{
&&\ar[dl]_{\overline{v}_1}  \ar@{}[ddl]|{\Downarrow\sigma_1}  \ar[2,2]^{g\overline{f}_1} &&&&&& \ar[dl]_{\overline{v}_1}  \ar@{}[ddl]|{\Downarrow\tau_1}  \ar[2,2]^{g\overline{f}_1}
\\
&\ar[dl]_{u_1} & \ar[dl]|{\tilde{v}_1}\ar[u]|{\overline{s}_1} && & && \ar[dll]_{u_1} &\ar[u]|{\overline{t}_1}\ar[dl]|{\hat{v}_1}&&&
\\
\ar@{}[r]|{\Downarrow\alpha} & \ar[u]|{s_1} \ar[d]|{s_2} \ar@{}[r]|{\Downarrow\sigma_3\tilde{v}} & \ar[u]|{x_1\tilde{v}} \ar[d]|{x_2\tilde{v}} \ar@{}[rr]|{\Downarrow g\tilde\beta_1}&& \ar@{}[r]|{\mbox{and}}& \ar@{}[r]|{\alpha\Downarrow} & \ar[ur]|{s_1}\ar[dr]|{s_2} & \ar[l]|r \ar@{}[ul]|(.25){\stackrel{\varepsilon_1}{\Leftarrow}}\ar@{}[dl]|(.25){\stackrel{\varepsilon_2}{\Rightarrow}}\ar[u]|{t_1}\ar[d]|{t_2}\ar@{}[r]|{\tau_3\hat{v}\Downarrow} & \ar[u]_{\hat{x}_1\hat{v}}\ar[d]^{\hat{x}_2\hat{v}}\ar@{}[rr]|{\Downarrow g\tilde\beta_2} & &
\\
&\ar[ul]^{u_2} \ar@{}[r]|{\Downarrow\sigma_2} &\ar[ul]|{\tilde{v}_2}\ar[d]|{\overline{s}_2}&& &  &&\ar[ull]^{u_2}\ar@{}[r]|{\tau_2\Downarrow} & \ar[d]|{\overline{t}_2}\ar[ul]|{\hat{v}_2}&&
\\
&&\ar[ul]^{\overline{v}_2}\ar[uurr]_{g\overline{f}_2} && & &&&\ar[ul]^{\overline{v}_2}\ar[uurr]_{g\overline{f}_2} &&
}
\end{equation}
where $\tilde\beta_1$ is the lifting with respect to $v$ of
$$
\xymatrix@C=3em{
\ar[dr]|{\overline{v}_1}\ar@/^2ex/[2,3]^{\overline{f}_1} \ar@{}[ddr]|{\Downarrow\sigma_1}&&
\\
\ar[u]|{\overline{s}_1} \ar[dr]|{\tilde{v}_1} & \ar[dr]|{f_1}&&&
\\
\ar[u]|{x_1} \ar[d]|{x_2} \ar@{}[r]|{\Downarrow\sigma_3} & \ar[u]|{s_1} \ar[d]|{s_2}      \ar@{}[r]|{\Downarrow\beta} &\ar@{}[u]|{\gamma_1\Downarrow}\ar@{}[d]|{\gamma_2^{-1}\Downarrow} &\ar[l]_v
\\
\ar[d]|{\overline{s}_2}\ar[ur]|{\tilde{v}_2}\ar@{}[r]|{\Downarrow\sigma_2}&\ar[ur]|{f_2}&&
\\
\ar[ur]|{\overline{v}_2}\ar@/_2ex/[-2,3]_{\overline{f}_2}&&
}
$$
and $\tilde\beta_2$ is the lifting with respect to $v$ of:
$$
\xymatrix@C=3em{
\ar@/^2ex/[2,4]^{\overline{f}_1} \ar[dr]|{\overline{v}_1} &&&&
\\
\ar[u]|{\hat{s}_1}\ar@{}[r]|{\tau_1\Downarrow}\ar[dr]|{\hat{v}_1} & \ar@/^1ex/[drr]^{f_1} &&&
\\
\ar[u]|{\hat{x}_1}\ar[d]|{\hat{x}_2}\ar@{}[r]|{\tau_3\Downarrow} & \ar[r]|r\ar@{}[ur]|(.25){\stackrel{\varepsilon_1}{\Rightarrow}} \ar@{}[dr]|(.25){\stackrel{\varepsilon_2}{\Leftarrow}} \ar[u]|{t_1}\ar[d]|{t_2}&\ar[ul]|{s_1}\ar[dl]|{s_2} \ar@{}[r]|{\beta\Downarrow} & \ar@{}[u]|{\Downarrow\gamma_1} \ar@{}[d]|{\Downarrow\gamma_2^{-1}} &\ar[l]|v
\\
\ar[ur]|{\hat{v}_2}\ar@{}[r]|{\tau_2\Downarrow}\ar[d]|{\overline{t}_2} & \ar@/_1ex/[urr]_{f_2} &&&
\\
\ar[ur]|{\overline{v}_2}\ar@/_2ex/[-2,4]_{\overline{f}_2} &&&&
}
$$
To show that 2-cell diagrams in (\ref{two_diagrams}) are equivalent, we begin by applying Proposition~\ref{anytwosquares} to the following
two diagrams:
$$
\xymatrix{
&\ar[dl]_{x_2\tilde{v}}\ar[dr]^{x_1\tilde{v}}\ar[rr]^{\overline{s}_1x_1\tilde{v}}\ar@{}[dd]|{\stackrel{\sigma_3\tilde{v}}{\Leftarrow}} &&\ar@{}[ddl]|{\stackrel{\sigma_1 x_1\tilde{v}}{\Leftarrow}}\ar[dd]^{\overline{v}_1}&& &\ar[dl]_{\hat{x}_2\hat{v}}\ar[dr]^{\hat{x}_1\hat{v}}\ar[rr]^{\overline{t}_1\hat{x}_1\hat{v}}\ar@{}[dd]|{\stackrel{\tau_3\hat{v}}{\Leftarrow}} &&\ar@{}[ddl]|{\stackrel{\tau_1\hat{x}_1\hat{v}}{\Leftarrow}}\ar[ddd]^{\overline{v}_1}
\\
\ar[dr]_{\tilde{v}_2}&&\ar[dl]^{\tilde{v}_1} &&\mbox{and}& \ar[dr]_{\hat{v}_2}&&\ar[dl]^{\hat{v}_1} &
\\
&\ar[d]_{u_1s_1}\ar[rr]_{s_1} && &&& \ar[d]_r\ar@{}[dr]|{\stackrel{\varepsilon_1}{\Leftarrow}}\ar[drr]^{t_1}&&
\\
&&&&&&\ar[rr]_{s_1} \ar[d]_{u_1s_1}&&
\\
&&&&&&&
}
$$
This gives us arrows $y_1$ and $y_2$ and cells $\rho_1$ and $\rho_2$ as in
$$
\xymatrix@C=5em{
&\ar@/_1.5ex/[dl]_{r\hat{v}_2\hat{x}_2\hat{v}}\ar@/^1.5ex/[dr]^{\overline{t}_1\hat{x}_1\hat{v}}
\\
\ar@{}[r]|{\rho_1\Downarrow}&\ar[u]|{y_1}\ar[d]|{y_2}\ar@{}[r]|{\rho_2\Downarrow}&
\\
&\ar@/^1.5ex/[ul]^{\tilde{v}_2x_2\tilde{v}}\ar@/_1.5ex/[ur]_{\overline{s}_1x_1\tilde{v}} 
}
$$
with the property that
\begin{equation}\label{first_eqns}
\xymatrix{
\ar[ddd]_{y_2}\ar[rr]^{y_1}\ar@{}[dddr]|{\stackrel{\rho_1}{\Leftarrow}} &&\ar[dl]_{\hat{x}_2\hat{v}}\ar@{}[dd]|{\stackrel{\tau_3\hat{v}}{\Leftarrow}}\ar[dr]^{\hat{x}_1\hat{v}}\ar[rr]^{\overline{t}_1\hat{x}_1\hat{v}}&& \ar@{}[ddl]|{\stackrel{\tau_1\hat{x}_1\hat{v}}{\Leftarrow}}\ar[ddd]^{\overline{v}_1} &&&\ar[d]_{y_2}\ar[rr]^{y_1}\ar@{}[drr]|{\stackrel{\rho_2}{\Leftarrow}} && \ar[d]^{\overline{t}_1\hat{x}_1\hat{v}}
\\
&\ar[dr]_{\hat{v}_2}&&\ar[dl]^{\hat{v}_1}&\ar@{}[drr]|{{\textstyle\equiv}}&&&\ar[dl]_{x_2\tilde{v}}\ar@{}[dd]|{\stackrel{\sigma_3\tilde{v}}{\Leftarrow}}\ar[dr]^{x_1\tilde{v}}\ar[rr]^{\overline{s}_1x_1\tilde{v}} &&\ar@{}[ddl]|{\stackrel{\sigma_1 x_1\tilde{v}_1}{\Leftarrow}}\ar[dd]^{\overline{v}_1}
\\
&&\ar[d]_r\ar[drr]^{t_1}\ar@{}[dr]|{\stackrel{\varepsilon_1}{\Leftarrow}}&&&&\ar[dr]_{\tilde{v}_2}&&\ar[dl]^{\tilde{v}_1} &
\\
\ar[rr]_{\tilde{v}_2x_2\tilde{v}} && \ar[rr]_{s_1}&& &&&\ar[rr]_{s_1}&&
}
\end{equation}
Now we apply Proposition~\ref{anytwosquares} to the following two diagrams:
$$
\xymatrix@C=4em{
\ar[r]^{y_1}\ar[d]_{y_2}\ar@{}[3,1]|{\stackrel{\rho_1}{\Leftarrow}} & \ar[d]^{\hat{v}}&&\ar[r]^{y_1}\ar[d]_{y_1} & \ar[d]^{\hat{x}_2}
\\
\ar[d]_{\tilde{v}}&\ar[d]^{\hat{x}_2}&&\ar[d]_{\hat{v}}&\ar[d]^{\hat{v}_2}
\\
\ar[d]_{x_2}&\ar[d]^{r\hat{v}_2}&\mbox{and}&\ar[d]_{\hat{x}_2}\ar@{}[r]|{\stackrel{\tau_2\hat{x}_2\hat{v}y_1}{\Leftarrow}} & \ar[dd]|{t_2}\ar[dr]^{r}
\\
\ar[d]_{\overline{s}_2}\ar@{}[dr]|{\stackrel{\sigma_2}{\Leftarrow}}\ar[r]^{\tilde{v}_2}&\ar[d]^{s_2}&&\ar[d]_{\overline{t}_2}&\ar@{}[r]|{\stackrel{\varepsilon_2}{\Leftarrow}}&\ar[dl]^{s_2}
\\
\ar[r]_{\overline{v}_2}\ar[d]_{u_2\overline{v}_2}&&&\ar[d]_{u_2\overline{v}_2}\ar[r]_{\overline{v}_2}&
\\
&&&
}
$$
this gives us arrows $z_1$ and $z_2$ and cells $\omega_1$ and $\omega_2$ as  in the following diagram
$$
\xymatrix@C=5em{
&\ar@/_1.5ex/[dl]_{\overline{t}_2\hat{x}_2\hat{v}y_1}\ar@/^1.5ex/[dr]^{y_1}
\\
\ar@{}[r]|{\omega_1\Downarrow}&\ar[u]|{z_1}\ar[d]|{z_2}\ar@{}[r]|{\omega_2\Downarrow}&
\\
&\ar@/^1.5ex/[ul]^{\overline{s}_2x_2\tilde{v}y_2}\ar@/_1.5ex/[ur]_{y_1} 
}
$$
with the property that
\begin{equation}\label{omega_eqn}
\xymatrix@C=3.5em{
&&&&&\ar[d]_{z_2}\ar[r]^{z_1}\ar@{}[dr]|{\stackrel{\omega_2}{\Leftarrow}}&\ar[d]^{y_1}
\\
\ar[r]^{z_1}\ar[d]_{z_2}&\ar[d]|{y_1}\ar[r]^{y_1}&\ar[d]^{\hat{x}_2\hat{v}}&&&\ar[d]_{y_2}\ar[r]_{y_1}&\ar[d]^{\hat{v}}
\\
\ar[d]_{y_2}&\ar[d]|{\hat{v}}&\ar[d]^{\hat{v}_2}&&&\ar[d]_{\tilde{v}}&\ar[d]^{\hat{x}_2}
\\
\ar[d]_{\tilde{v}}\ar@{}[r]|{\stackrel{\omega_1}{\Leftarrow}}&\ar[d]|{\hat{x}_2}\ar@{}[r]|{\stackrel{\tau_2\hat{x}_2\hat{v}y_1}{\Leftarrow}} &\ar[dd]_{t_2}\ar[dr]^{r} &&\equiv&\ar[d]_{x_2}\ar@{}[ur]|{\stackrel{\rho_1}{\Leftarrow}}&\ar[d]^{r\hat{v}_2}
\\
\ar[d]_{x_2}&\ar[d]|{\overline{t}_2}&\ar@{}[r]|{\stackrel{\varepsilon_2}{\Leftarrow}}&\ar[dl]^{s_2}&&\ar[d]_{\overline{s}_2}\ar[r]^{\tilde{v}_2}\ar@{}[dr]|{\stackrel{\sigma_2}{\Leftarrow}}&\ar[d]^{s_2}
\\
\ar[r]_{\overline{s}_2}&\ar[r]_{\overline{v}_2}&&&&\ar[r]_{\overline{v}_2}&
}
\end{equation}

The cells we have constructed so far allow us to perform the following calculation of pasting diagrams for any cell $\delta\colon d_1s_1\Rightarrow d_2s_2$:
$$
\xymatrix@C=2.9em{
&&\ar[dr]^{\overline{v}_1}&& & &&& \ar[dr]^{\overline{v}_1} &&
\\
&&\ar[u]^{\overline{s}_1}\ar@{}[r]|{\Downarrow\sigma_1}\ar[dr]|{\tilde{v}_1} &\ar[dr]^{d_1} & & && \ar[ur]^{\overline{t}_1}\ar@{}[rr]|{\Downarrow\tau_1}\ar[dr]|{\hat{v}_1}\ar@{}[d]|{\tau_3\hat{v}\Downarrow} &\ar@{}[dr]|(.7){\varepsilon_1\Downarrow}&\ar[dr]^{d_1}&
\\
\ar@/^2.5ex/[uurr]^{\overline{t}_1\hat{x}_1\hat{v}} & \ar@/_1.5ex/[l]_{y_1z_2}\ar@/^1.5ex/[l]^{y_1z_1}\ar@{}[l]|{\Uparrow\omega_2}\ar[r]_{y_2z_2}\ar@{}[uur]^{\stackrel{\rho_2z_2}{\Rightarrow}} & \ar[u]^{x_1\tilde{v}}\ar[d]_{x_2\tilde{v}}\ar@{}[r]|{\sigma_3\tilde{v}\Downarrow}& \ar[u]_{s_1}\ar@{}[r]|{\delta\Downarrow}\ar[d]^{s_2} &\ar@{}[r]|{\equiv} & & \ar[ur]^{\hat{x}_1\hat{v}}\ar[r]_{\hat{v}} &\ar@{}[dd]|{\Downarrow\rho_1z_2}\ar[r]_{\hat{v}_2} & \ar[r]|r \ar[ur]|{t_1}&  \ar[u]_{s_1}\ar@{}[r]|{\delta\Downarrow}\ar[d]^{s_2} &\ar@{}[r]|(.7){\mbox{by (\ref{first_eqns})}}&
\\
&&\ar[d]_{\overline{s}_2}\ar@{}[r]|{\sigma_2\Downarrow}\ar[ur]|{\tilde{v}_2} & \ar[ur]_{d_2} & & \ar@/^1.5ex/[ur]^{y_1z_1} \ar@{}[ur]|{\Downarrow\omega_2}\ar@/_1.5ex/[ur]_{y_1z_2}\ar[dr]_{y_2z_2} &&&\ar[ur]|{\tilde{v}_2} \ar@{}[r]|{\sigma_2\Downarrow}\ar[d]|{\overline{s}_2} &\ar[ur]_{d_2} &
\\
&& \ar[ur]_{\overline{v}_2} && & &\ar[r]_{\tilde{v}} & \ar[ur]|{x_2} &\ar[ur]|{\overline{v}_2} && 
\\
&&&& & &&&\ar[dr]^{\overline{v}_1}&&
\\
&&&& & &&\ar[ur]^{\overline{t}_1}\ar@{}[rr]|{\Downarrow\tau_1} \ar[dr]|{\hat{v}_1} && \ar[dr]^{d_1}
\\
&&&& \ar@{}[r]|\equiv & & \ar[l]_{y_2z_2}\ar[r]^{y_1z_1}\ar@{}[d]|{\stackrel{\omega_1}{\Leftarrow}} & \ar[u]^{\hat{x}_1\hat{v}}\ar@{}[r]|{\Downarrow\tau_3\hat{v}}\ar[d]_{\hat{x}_2\hat{v}} &\ar[ur]|{t_1}\ar[dr]|{t_2}\ar[r]|r & \ar@{}[ul]|(.3){\Downarrow\varepsilon_1}\ar@{}[dl]|(.3){\Downarrow\varepsilon_2}\ar[u]_{s_2}\ar[d]^{s_2}\ar@{}[r]|{\Downarrow\delta} & \ar@{}[r]|(.7){\mbox{by (\ref{omega_eqn})}}&
\\
&&&& & &&\ar[ur]|{\hat{v}_2}\ar[dr]|{\overline{t}_2}\ar@{}[rr]|{\Downarrow\tau_2} && \ar[ur]_{d_2} & 
\\
&&&& & &&&\ar@/^2ex/[-2,-3]^{\overline{s}_2x_2\tilde{v}}\ar[ur]_{\overline{v}_2}&&
}
$$
Applying this result with $\beta$ instead of $\delta$ implies that
$$
\xymatrix@C=3em{
&&\ar@/^1.5ex/[dr]^{\overline{f}_1}&&&&&\ar@/^1.5ex/[dr]^{\overline{f}_1}&&
\\
\ar@/^3ex/[urr]^{\overline{t}_1\hat{x}_1\hat{v}}\ar@{}[urr]|{\stackrel{\rho_2z_2}{\Rightarrow}} & \ar@/_1.4ex/[l]_{y_1z_2}\ar@{}[l]|{\Uparrow\omega_2}\ar@/^1.4ex/[l]^{y_1z_1}\ar[r]_{y_2z_2}&\ar[u]|{\overline{s}_1x_1\tilde{v}}\ar[d]|{\overline{s}_2x_2\tilde{v}}\ar@{}[r]|{\Downarrow\tilde\beta_1} & \ar[r]^v &
\ar@{}[r]|\equiv & \ar@/_2ex/[drr]_{\overline{s}_2{x}_2\tilde{v}}\ar@{}[drr]|{\stackrel{\omega_1}{\Leftarrow}} & \ar[l]_{y_2z_2}\ar[r]^{y_1z_1}& \ar[u]|{\overline{t}_1\hat{x}_1\hat{v}}\ar@{}[r]|{\tilde\beta_2}\ar[d]|{\overline{t}_2\hat{x}_2\hat{v}} & \ar[r]^v&
\\
&&\ar@/_1.5ex/[ur]_{\overline{f}_2} &&&&&\ar@/_1.5ex/[ur]_{\overline{f}_2}&&
}
$$
So  by Lemma~\ref{Matteo1} and {\bf WB2} we get  an arrow $q$ such that $u_2s_2r\hat{v}_2\hat{x}_2\hat{v}y_1z_1q\in \frakW$ and 
$$
\xymatrix@C=4em{
&&\ar@/^1.5ex/[dr]^{\overline{f}_1}&&&&\ar@/^1.5ex/[dr]^{\overline{f}_1}&&
\\
\ar@/^3ex/[urr]^{\overline{t}_1\hat{x}_1\hat{v}}\ar@{}[urr]|{\stackrel{\rho_2z_2q}{\Rightarrow}} & \ar@/_1.4ex/[l]_{y_1z_2q}\ar@{}[l]|{\Uparrow\omega_2q}\ar@/^1.4ex/[l]^{y_1z_1q}\ar[r]_{y_2z_2q}&\ar[u]|{\overline{s}_1x_1\tilde{v}}\ar[d]|{\overline{s}_2x_2\tilde{v}}\ar@{}[r]|{\Downarrow\tilde\beta_1} &
\ar@{}[r]|\equiv & \ar@/_2ex/[drr]_{\overline{s}_2{x}_2\tilde{v}}\ar@{}[drr]|{\stackrel{\omega_1q}{\Leftarrow}} & \ar[l]_{y_2z_2q}\ar[r]^{y_1z_1q}& \ar[u]|{\overline{t}_1\hat{x}_1\hat{v}}\ar@{}[r]|{\tilde\beta_2}\ar[d]|{\overline{t}_2\hat{x}_2\hat{v}} & 
\\
&&\ar@/_1.5ex/[ur]_{\overline{f}_2} &&&&\ar@/_1.5ex/[ur]_{\overline{f}_2}&
}
$$
Applying the calculation above with $\alpha$ instead of $\delta$ gives us the remaining result to conclude that the arrows and cells in
$$
\xymatrix@R=4em@C=3.5em{
&&
\\
\ar@/^3ex/[ur]^{\overline{s}_1x_1\tilde{v}}\ar@/_3ex/[dr]_{\overline{s}_2x_2\tilde{v}} & \ar[l]^{y_2z_2q}\ar@{}[u]|{\stackrel{\rho_2z_2q}{\Leftarrow}}\ar[r]_{y_1z_1q}\ar@{}@<.8ex>[r]|{\Uparrow\omega_2q}\ar@/^2.3ex/[r]^{y_1z_2q}\ar@{}[d]|{\stackrel{\omega_1q}{\Leftarrow}} & \ar@/_3ex/[ul]_{\overline{t}_1\hat{x}_1\hat{v}}\ar@/^3ex/[dl]^{\overline{t}_2\hat{x}_2\hat{v}}
\\
&&
}$$
witness to the fact that the two cell diagrams in (\ref{two_diagrams}) are equivalent.
We conclude that left-whiskering is well-defined on equivalence classes of 2-cell diagrams.
\end{proof}

\begin{prop}\label{wd-right-whiskering}
Right whiskering of a 2-cell diagram and an arrow in the bicategory of fractions is well-defined on 
equivalence classes of 2-cell diagrams.
\end{prop}

\begin{proof}
We will sketch the proof of this result as the details get rather involved and don't necessarily illuminate the idea behind the proof.
Any interested reader is welcome to contact the authors for further details.

Consider the following whiskering diagrams:
\begin{equation}\label{two_whiskerings}
\xymatrix@R=3em{
&&&\ar[dl]_{v_1}\ar[dr]^{g_1}&&&&&&\ar[dll]_{v_1}\ar[drr]^{g_1}&&
\\
&\ar[l]_{u}\ar[r]^{f}&\ar@{}[r]|{\alpha\Downarrow}&\ar[u]|{s_1}\ar[d]|{s_2}\ar@{}[r]|{\beta\Downarrow}&\ar@{}[r]|{\mbox{and}}& &\ar[l]_{u}\ar[r]^{f}&\ar@{}[r]|{\alpha\Downarrow}
&\ar[ur]|{s_1}\ar[dr]|{s_2}
&\ar[l]|a\ar@{}[ul]|(.3){\stackrel{\theta_1}{\Leftarrow}}
\ar@{}[ur]|(.3){\stackrel{\theta_1}{\Rightarrow}}\ar@{}[dl]|(.3){\stackrel{\theta_2}{\Rightarrow}}\ar@{}[dr]|(.3){\stackrel{\theta_2}{\Leftarrow}}\ar[u]|{b_1}\ar[d]|{b_2}\ar[r]|a&\ar[ul]|{s_1}\ar[dl]|{s_2}\ar@{}[r]|{\beta\Downarrow} &
\\
&&&\ar[ul]^{v_2}\ar[ur]_{g_2}&&&&&&\ar[ull]^{v_2}\ar[urr]_{g_2}&&
}
\end{equation}
We want to show that the 2-cell diagrams that result after whiskering are equivalent. These two diagrams are
\begin{equation}\label{first_whisker}
\xymatrix@C=4em{
&&&\ar@/_2ex/[3,-2]_{\overline{v}_1}
\ar@/^2ex/[2,2]^{\overline{f}_1}&&&
\\
&&&\ar@{}[l]|{\varepsilon_1^{-1}\Downarrow}\ar[dl]|{r_1}\ar[dr]|{r_1}\ar[u]|{t_1}\ar@{}[r]|{\varphi_1^{-1}\Downarrow}&&&
\\
&&\ar[dl]|{s_1'}\ar@{}[r]|{\rho_1\Downarrow}&\ar[u]|{\overline{p}}\ar@{}[r]|{\rho_1\Downarrow} \ar[dl]|{\overline{r}_1}\ar[dr]|{\overline{r}_1} & \ar[dr]|{f_1'}&\ar[dr]^{g_1}
\\
&\ar[l]_{u}\ar@{}[r]|{\tilde\alpha^{-1}\Downarrow} & \ar[u]|p\ar[d]|q\ar@{}[r]|{\rho_3^{-1}x\Downarrow} & \ar[u]|{\tilde{r}_2x}\ar[d]|{\tilde{r}_1x}\ar@{}[r]|{\rho_3^{-1}x\Downarrow} & \ar[u]|{p}\ar[d]|q \ar@{}[r]|{\tau^{-1}\Downarrow}&\ar@{}[r]|{\beta\Downarrow}\ar[u]|{s_1}\ar[d]|{s_2} &
\\
&&\ar[ul]|{s_2'}\ar@{}[r]|{\rho_2^{-1}\Downarrow} &\ar[ul]|{\overline{r}_2}\ar[ur]|{\overline{r}_2}\ar[d]|{\overline{q}}\ar@{}[r]|{\rho_2^{-1}\Downarrow} & \ar[ur]|{f_2'} & \ar[ur]_{g_2}&
\\
&&&\ar[ul]|{r_2}\ar[ur]|{r_2}\ar[d]|{t_2}\ar@{}[l]|{\varepsilon_2\Downarrow}\ar@{}[r]|{\varphi_2\Downarrow}&&&
\\
&&&\ar@/^1.5ex/[-3,-2]^{\overline{v}_2}\ar@/_1.5ex/[-2,2]_{\overline{f}_2} &&&
}
\end{equation}
and
\begin{equation}\label{second_whisker}
\xymatrix@C=4em{
&&&\ar@/_2ex/[3,-2]_{\overline{v}_1}
\ar@/^2ex/[2,3]^{\overline{f}_1}&&&
\\
&&&\ar@{}[l]|{\hat\varepsilon_1^{-1}\Downarrow}\ar[dl]|{\hat{r}_1}\ar[dr]|{\hat{r}_1}\ar[u]|{\hat{t}_1}\ar@{}[drrr]|{\hat\varphi_1^{-1}\Downarrow}&&&
\\
&&\ar[dl]|{b_1'}\ar@{}[r]|{\hat\rho_1\Downarrow}&\ar[u]|{\overline{\hat{p}}}\ar@{}[r]|{\hat\rho_1\Downarrow} \ar[dl]|{\overline{\hat{r}}_1}\ar[dr]|{\overline{\hat{r}}_1} & \ar[dr]|{\hat{f}_1'}&&\ar[dr]^{g_1}
\\
&\ar[l]_{u}\ar@{}[r]|{\hat\alpha^{-1}\Downarrow} & \ar[u]|{\hat{p}}\ar[d]|{\hat{q}}\ar@{}[r]|{\hat\rho_3^{-1}\hat{x}\Downarrow} & \ar[u]|{\tilde{\hat{r}}_2\hat{x}}\ar[d]|{\tilde{\hat{r}}_1\hat{x}}\ar@{}[r]|{\hat\rho_3^{-1}\hat{x}\Downarrow} & \ar[u]|{\hat{p}}\ar[d]|{\hat{q}} \ar@{}[r]|{\hat\tau^{-1}\Downarrow}&\ar[ur]^{b_1}\ar[dr]_{b_2}\ar[r]|{a}&\ar@{}[ul]|(.25){\stackrel{\theta_1}{\Rightarrow}}\ar@{}[dl]|(.25){\stackrel{\theta_2}{\Leftarrow}}\ar@{}[r]|{\beta\Downarrow}\ar[u]|{s_1}\ar[d]|{s_2} &
\\
&&\ar[ul]|{b_2'}\ar@{}[r]|{\hat\rho_2^{-1}\Downarrow} &\ar[ul]|{\overline{\hat{r}}_2}\ar[ur]|{\overline{\hat{r}}_2}\ar[d]|{\overline{\hat{q}}}\ar@{}[r]|{\hat\rho_2^{-1}\Downarrow} & \ar[ur]|{\hat{f}_2'} && \ar[ur]_{g_2}&
\\
&&&\ar[ul]|{\hat{r}_2}\ar[ur]|{\hat{r}_2}\ar[d]|{\hat{t}_2}\ar@{}[l]|{\hat\varepsilon_2\Downarrow}\ar@{}[urrr]|{\hat\varphi_2\Downarrow}&&&
\\
&&&\ar@/^1.5ex/[-3,-2]^{\overline{v}_2}\ar@/_1.5ex/[-2,3]_{\overline{f}_2} &&&
}
\end{equation}
respectively.

We will produce the cells that witness that these diagrams are equivalent.  To do this,  consider 2-cell diagrams comparing the following four squares:
\begin{equation}\label{four_squares}
\xymatrix{
\ar[r]^{\tilde{r}_2x}\ar[d]_{\tilde{r}_2x}\ar@{}[dr]|(.55){\rho_1^{-1}\tilde{r}_2x\Downarrow} & \ar[r]^{p\overline{r}_1} & \ar[r]^{f_1'} &\ar[dddl]|{s_1}\ar[ddd]^{s_2}&&\ar[d]_{\tilde{\hat{r}}_2\hat{x}}\ar[r]^{\tilde{\hat{r}}_2\hat{x}}\ar@{}[dr]|(.55){\hat{\rho}_1^{-1}\tilde{\hat{r}}_2\hat{x}\Downarrow} & \ar[r]^{\hat{p}\overline{\hat{r}}_1} & \ar[r]^{\hat{f}_1'} & \ar[ddd]_{b_1}\ar[r]^{a}\ar@{}[dr]|(.35){\stackrel{\theta_1^{-1}}{\Leftarrow}} &\ar[dddl]|{s_1}\ar[ddd]^{s_2}
\\
\ar[d]_{\overline{p}}&&&&&\ar[d]_{\overline{\hat{p}}} &&&&
\\
\ar[uurr]_{r_1}\ar[d]_{t_1}\ar@{}[rr]|{\varphi_1\Downarrow}&&&&\Rightarrow& \ar[d]_{\hat{t}_1}\ar[uurr]_{\hat{r}_1}\ar@{}[rrr]|{\hat\varphi_1\Downarrow} &&&
\\
\ar[d]_{\overline{v}_1}\ar[rr]^{\overline{f}_1}\ar@{}[drr]|{\gamma_1\Downarrow} &&\ar[dr]_{v_1}\ar@{}[r]|{\stackrel{\alpha^{-1}}{\Leftarrow}} & \ar[d]^{v_2}&&\ar[d]_{\overline{v}_1}\ar[rrr]^{\overline{f}_1}\ar@{}[drrr]|{\gamma_1\Downarrow} &&&\ar[dr]_{v_1}\ar@{}[r]|{\stackrel{\alpha^{-1}}{\Leftarrow}} & \ar[d]^{v_2}
\\
\ar[rrr]_{f} \ar[d]_u&&&&&\ar[rrrr]_f\ar[d]_u&&&&
\\
&&&&&&&&&
\\
&\ar@{}[r]|{\textstyle\Downarrow} &&&&&& \Downarrow &&
\\
\ar[r]^{\tilde{r}_1x}\ar[d]_{\tilde{r}_1x}\ar@{}[dr]|(.55){\rho_2^{-1}\tilde{r}_1x\Downarrow} & \ar[r]^{q\overline{r}_2} & \ar[r]^{f_2'}\ar@{}[ddd]|{\varphi_2\Downarrow} &\ar[ddd]^{s_2} && \ar[d]_{\tilde{\hat{r}}_1\hat{x}} \ar[r]^{\tilde{\hat{r}}_1\hat{x}} \ar@{}[dr]|(.55){\hat{\rho}_2^{-1}\tilde{\hat{r}}_1\hat{x}\Downarrow} & \ar[r]^{\hat{q}\overline{\hat{r}}_2} & \ar[r]^{\hat{f}_2'}\ar@{}[dddr]|{\hat{\varphi}_2\Downarrow} & \ar[r]^a\ar@{}[dr]|{\stackrel{\theta_2}{\Leftarrow}}\ar[dddr]_{b_2} & \ar[ddd]^{s_2}
\\
\ar[d]_{\overline{q}}&&&&&\ar[d]_{\overline{\hat{q}}}&&&&
\\
\ar[d]_{t_2}\ar[uurr]_{r_2} &&&&\Rightarrow & \ar[d]_{\hat{t}_2} \ar[uurr]_{\hat{r}_2} &&&&
\\
\ar[rrr]^{\overline{f}_2}\ar[d]_{\overline{v}_2}\ar@{}[drrr]|{\gamma_2\Downarrow} &&&\ar[d]^{v_2}&& \ar[rrrr]^{\overline{f}_2}\ar[d]_{\overline{v}_2}\ar@{}[1,4]|{\gamma_2\Downarrow} &&&&\ar[d]^{v_2}
\\
\ar[d]_u\ar[rrr]_f&&& && \ar[d]_u\ar[rrrr]_f&&&&
\\
&&&&&&&&&
}
\end{equation}
By composing these 2-cell diagrams vertically, we obtain two 2-cell diagrams comparing the top left and bottom right square. By Proposition~\ref{unique-2-cell} these 2-cell diagrams are equivalent. This will provide us two additional cells which paste with cells from the 2-cell diagrams to provide us the cells that witness the equivalence of (\ref{first_whisker}) and (\ref{second_whisker}).

We start with the 2-cell diagram comparing the two squares in the top row.
However, we will ignore the cells $\gamma_1$ and $\alpha^{-1}$. So by applying Proposition~\ref{anytwosquares}, we obtain arrows $c,\hat{c}$ and cells $\xi_1$ and $\xi_2$ as in 
\begin{equation}\label{top_cells}
\xymatrix@C=4em{
&&\ar@/_2ex/[dl]_{t_1\overline{p}\tilde{r}_2x} \ar@/^2ex/[dr]^{f_1'p\overline{r}_1\tilde{r}_2x}
\\
&\ar[l]_{u\overline{v}_1}\ar@{}[r]|{\xi_1\Downarrow}&\ar[u]_{c}\ar[d]^{\hat{c}}\ar@{}[r]|{\xi_2\Downarrow} &
\\
&&\ar@/^2ex/[ul]^{\hat{t}_1\overline{\hat{p}}\tilde{\hat{r}}_2\hat{x}}\ar@/_2ex/[ur]_{a\hat{f}_1'\hat{p}\overline{\hat{r}}_1\tilde{\hat{r}}_2\hat{x}}
}
\end{equation}
such that 
\begin{equation}\label{top-equation}
\xymatrix{
\ar[r]^{\hat{c}}\ar[d]_{c}\ar@{}[dddr]|{\stackrel{\xi_1^{-1}}{\Leftarrow}} & \ar[d]_{\tilde{\hat{r}}_2\hat{x}}\ar[r]^{\tilde{\hat{r}}_2\hat{x}}
\ar@{}[dr]|(.55){\hat{\rho}_1^{-1}\tilde{\hat{r}}_2\hat{x}\Downarrow} & \ar[r]^{\hat{p}\overline{\hat{r}}_1} & \ar[r]^{\hat{f}_1'} \ar@{}[dddr]|{\hat\varphi_1\Downarrow}& \ar[dddr]_{b_1} \ar[r]^a&\ar@{}[dl]|{\stackrel{\theta_1^{-1}}{\Leftarrow}}\ar[ddd]^{s_1}&& \ar[ddd]_{c}\ar[r]^{\hat{c}}\ar@{}[3,3]|{\xi_2^{-1}\Downarrow} & \ar[r]^{\tilde{\hat{r}}_2\hat{x}} & \ar[r]^{\overline{\hat{r}}_1} & \ar[d]^{\hat{p}}
\\
\ar[d]_{\tilde{r}_2x}&\ar[d]_{\overline{\hat{p}}} &&&&\ar@{}[drr]|{\textstyle\equiv} && &&&\ar[d]^{\hat{f}_1'}
\\
\ar[d]_{\overline{p}} &\ar[d]_{\hat{t}_1}\ar[uurr]_{\hat{r}_1}&&&& && &&&\ar[d]^{a}
\\
\ar[r]_{t_1}&\ar[rrrr]_{\overline{f}_1} &&&& && \ar[d]_{\tilde{r}_2x}\ar[r]^{\tilde{r}_2x}\ar@{}[dr]|(.55){\rho_1\tilde{r}_2x\Downarrow} & \ar[r]^{p\overline{r}_1} & \ar[r]^{f_1'}&\ar[ddd]^{s_1}
\\
&&&&&&&\ar[d]_{\overline{p}}&&&
\\
&&&&&&&\ar[uurr]_{r_1}\ar[d]_{t_1}\ar@{}[rrr]|{\varphi_1\Downarrow} &&&
\\
&&&&&&&\ar[rrr]_{\overline{f}_1} &&&
}
\end{equation}

The 2-cell diagram to compare the two squares on the right-hand side of (\ref{four_squares})
can be built from cells we have already. The two arrows in the middle can be taken as identity arrows, and we will omit them to avoid adding unitor cells. So the reader may verify that the following 2-cell diagram compares the two squares on the right:
\begin{equation}\label{right_cells}
\xymatrix@C=3.5em{
&&\ar@{}[d]|{\hat{\varepsilon}^{-1}_1\Downarrow} \ar[ddl]_{\overline{v}_1} & \ar[l]_{\hat{t}_1}\ar[dl]|{\hat{r}_1}
\\
&&\ar[dl]|{b_1'}\ar@{}[r]|{\hat\rho_1\Downarrow} &\ar[u]_{\overline{\hat{p}}}\ar[dl]|{\overline{\hat{r}}_1} && \ar[dr]^{\overline{\hat{r}}_1}&&\ar[dr]^{\hat{f}_1'}&&
\\
&\ar[l]_u\ar@{}[r]|{\hat\alpha^{-1}\Downarrow} & \ar[u]|{\hat{p}}\ar[d]|{\hat{q}}\ar@{}[rr]|{\hat\rho_3^{-1}\Downarrow} &&\ar[ul]_{\tilde{\hat{r}}_2\hat{x}}\ar[dl]^{\tilde{\hat{r}}_1\hat{x}}\ar[ur]^{\tilde{\hat{r}}_2\hat{x}}\ar[dr]_{\tilde{\hat{r}}_1\hat{x}}\ar@{}[rr]|{\hat\rho_3^{-1}\Downarrow} && \ar[ur]^{\hat{p}}\ar[dr]_{\hat{q}}\ar@{}[rr]|{\hat\tau^{-1}\Downarrow} && \ar[r]^a&
\\
&&\ar[ul]|{b_2'}\ar@{}[d]|{\hat\varepsilon_2\Downarrow}\ar@{}[r]|{\hat\rho_2^{-1}\Downarrow}&\ar[d]^{\overline{\hat{q}}} \ar[ul]|{\overline{\hat{r}}_2}&& \ar[ur]_{\overline{\hat{r}}_2} && \ar[ur]_{\hat{f}_2'} &&
\\
&&\ar[uul]^{\overline{v}_2} &\ar[l]^{\hat{t}_2} \ar[ul]|{\hat{r}_2}&&&&&&
}
\end{equation}
Composing (\ref{top_cells}) with (\ref{right_cells}) gives us:
\begin{equation}\label{top_right_composite}
    \xymatrix@C=2.9em{
    &&&&&\ar@/_2ex/[3,-4]_{\overline{v}_1t_1\overline{p}\tilde{r}_2x}\ar@{}[3,-2]|{\overline{v}_1\xi_1\Downarrow}\ar@{}[3,2]|{\xi_2\Downarrow}\ar@/^2ex/[3,5]^{f_1'p\overline{r}_1\tilde{r}_2x}
    \\
    &&&&&&&&&
    \\
    &&&&&&&&&
    \\
    &\ar[l]_u&\ar[l]_{\overline{v}_1}\ar@{}[d]|{\hat\varepsilon_1^{-1}\Downarrow}&&&\ar[uuu]_{c}\ar[ddd]^{\hat{c}}&&&&&
    \\
    &&\ar[ul]|{b_1'}\ar@{}[d]|{\hat\alpha^{-1}\Downarrow}&\ar[l]|{\hat{r}_1}\ar[ul]_{\hat{t}_1}\ar@{}[d]|{\hat\rho_1^{-1}\Downarrow}&&&&&\ar[r]^{\hat{f}_1'}\ar@{}[d]|{\hat{\tau}_1^{-1}\Downarrow}&\ar[ur]_a&
    \\
    &\ar[uu]^{\overline{v}_2}\ar@{}[r]|{\hat\varepsilon_2\Downarrow}&\ar[uul]|{b_2'}&\ar[l]|{\hat{q}}\ar[ul]|{\hat{p}}\ar@{}[d]|{\hat\rho_2^{-1}\Downarrow} & \ar[l]|{\overline{\hat{r}}_1}\ar[ul]_{\overline{\hat{q}}}\ar@{}[d]|{\hat\rho_3^{-1}\Downarrow} &&\ar[r]^{\overline{\hat{r}}_1}\ar@{}[d]|{\hat\rho_3^{-1}\Downarrow} & \ar[ur]^{\hat{p}}\ar[r]_{\hat{q}} & \ar[ur]_{\hat{f}_2'}
    \\
    &&&\ar[ull]^{\hat{t}_2}\ar[ul]|{\hat{r}_2} & \ar[l]^{\overline{\hat{q}}}\ar[ul]|{\overline{\hat{r}}_2} &\ar[l]^{\tilde{\hat{r}}_1\hat{x}}\ar[ul]_{\tilde{\hat{r}}_2\hat{x}}\ar[ur]^{\tilde{\hat{r}}_2\hat{x}}\ar[r]_{\tilde{\hat{r}}_1\hat{x}} & \ar[ur]_{\overline{\hat{r}}_2}
    }
\end{equation}
 
Similar to the situation for the right two squares, the 2-cell diagram comparing the two squares on the left of (\ref{four_squares}) can also be constructed from cells we have constructed already. Again collapsing all identity arrows, the following 2-cell diagram is what is needed to compare the left two squares:
\begin{equation}\label{left_cells}
\xymatrix@C=2.8em{
&&&&\ar[dll]_{t_1} \ar[dl]|{r_1}
\ar@{}[dd]|{\rho_1\Downarrow}&&&&&&
\\
&&\ar[dl]_{\overline{v}_1}\ar@{}[r]|{\varepsilon_1^{-1}\Downarrow} & \ar[dll]|{s_1'}
\ar@{}[dd]|{\tilde\alpha^{-1}\Downarrow} &&\ar[dl]|{\overline{r}_1} \ar[ul]_{\overline{p}}\ar@{}[dd]|{\rho_3^{-1}\Downarrow} && \ar@{}[dd]|{\rho_3^{-1}\Downarrow} \ar[dr]^{\overline{r}_1} &&\ar[dr]^{f_1'}\ar@{}[dd]|{\tau^{-1}\Downarrow} &
\\
&\ar[l]_u &&&\ar[ul]|p \ar[dl]|q \ar@{}[dd]|{\rho_2^{-1}\Downarrow} &&\ar[ul]|{\tilde{r}_2x} \ar[dl]|{\tilde{r}_1x} \ar[ur]|{\tilde{r}_2x} \ar[dr]|{\tilde{r}_1x} &&\ar[ur]^p\ar[dr]_q &&
\\
&&\ar[ul]^{\overline{v}_2}\ar@{}[r]|{\varepsilon_2\Downarrow} & \ar[ull]|{s_2'} &&\ar[ul]|{\overline{r}_2} \ar[dl]^{\overline{q}} &&\ar[ur]_{\overline{r}_2} &&\ar[ur]_{f_2'} &
\\
&&&&\ar[ull]^{t_2}\ar[ul]|{r_2}&&&&&&
}
\end{equation}
To compare the bottom two squares in (\ref{four_squares}), we apply Proposition~\ref{anytwosquares} to 
$$
\xymatrix{
\ar[d]_{\tilde{r}_1x}\ar[r]^{\tilde{r}_1x} \ar@{}[dr]|(.55){\rho_2\tilde{r}_1x\Downarrow}& \ar[r]^{q\overline{r}_2} & \ar[r]^{f_2'} &\ar[ddd]^{s_2} &&\ar[d]_{\tilde{\hat{r}}_1\hat{x}}\ar[r]^{\tilde{\hat{r}}_1\hat{x}}\ar@{}[dr]|(.55){\hat\rho_2\tilde{\hat{r}}_1\hat{x}\Downarrow} &\ar[r]^{\hat{q}\overline{\hat{r}}_2} & \ar[r]^{\hat{f}_2'} & \ar[r]^a\ar[dddr]_{b_2} &\ar@{}[dl]|{\stackrel{\theta_2}{\Leftarrow}}\ar[ddd]^{s_2}
\\
\ar[d]_{\overline{q}}&&&&&\ar[d]_{\overline{\hat{q}}}&&&&
\\
\ar[d]_{t_2}\ar[uurr]_{r_2}\ar@{}[rrr]|{\varphi_2\Downarrow} &&&&\mbox{and} & \ar[d]_{\hat{t}_2}\ar[uurr]_{\hat{r}_2}\ar@{}[rrr]|{\hat\varphi_2\Downarrow}&&&&
\\
\ar[d]_{u\overline{v}_2}\ar[rrr]_{\overline{f}_2} &&& && \ar[d]_{u\overline{v}_2}\ar[rrrr]_{\overline{f}_2} &&&&
\\
&&&&&&&&&}
$$
This gives us cells as in
\begin{equation}\label{bottom_cells}
\xymatrix@C=4em{
&&\ar@/_2ex/[dl]_{t_2\overline{q}\tilde{r}_1x}\ar@/^2ex/[dr]^{f_2'q\overline{r}_2\tilde{r}_1x}
\\
&\ar[l]_{u\overline{v}_2} \ar@{}[r]|{\omega_1\Downarrow}&\ar[u]_d\ar[d]^{\hat{d}} \ar@{}[r]|{\omega_2\Downarrow}&
\\
&&\ar@/^2ex/[ul]^{\hat{t}_2\overline{\hat{q}}\tilde{\hat{r}}_1\hat{x}}\ar@/_2ex/[ur]_{a\hat{f}_2'\hat{q}\overline{\hat{r}}_2\tilde{\hat{r}}_1\hat{x}}
}
\end{equation}
such that
\begin{equation}\label{bottom_equation}
\xymatrix{
\ar[r]^d \ar[d]_{\hat{d}}\ar@{}[dddr]|{\omega_1} & \ar[d]_{\tilde{r}_1x}\ar[r]^{\tilde{r}_1x}\ar@{}[dr]|{\rho_2\tilde{r}_1x\Downarrow} & \ar[r]^{q\overline{r}_2} &\ar[r]^{f_2'}\ar@{}[ddd]|{\varphi_2\Downarrow} & \ar[ddd]^{s_2} && \ar[r]^d \ar[d]_{\hat{d}} \ar@{}[1,4]|{\omega_2\Downarrow}& \ar[r]^{\tilde{r}_1x} & \ar[r]^{\overline{r}_2} & \ar[r]^q &\ar[d]^{f_2'}
\\
\ar[d]_{\tilde{\hat{r}}_1\hat{x}} & \ar[d]_{\overline{q}}&&& \ar@{}[drr]|{\textstyle\equiv} && \ar[d]_{\tilde{\hat{r}}_1\hat{x}} \ar[r]^{\tilde{\hat{r}}_1\hat{x}} \ar@{}[dr]|{\hat\rho_2\tilde{\hat{r}}_1\hat{x}\Downarrow }&\ar[r]^{\hat{q}\overline{\hat{r}}_2}&\ar[r]^{\hat{f}_2'}\ar@{}[dddr]|{\hat\varphi_2\Downarrow}&\ar[dddr]_{b_2}\ar[r]^a&\ar@{}[dl]|{\stackrel{\theta_2}{\Leftarrow}}\ar[ddd]^{s_2}
\\
\ar[d]_{\overline{\hat{q}}}&\ar[d]_{t_2}\ar[uurr]_{r_2}&&& &&\ar[d]_{\overline{\hat{q}}} &&&&
\\
\ar[r]_{\hat{t}_2} & \ar[rrr]_{\overline{f}_2}&&& &&\ar[d]_{\hat{t}_2}\ar[uurr]_{\hat{r}_2} &&&&
\\
&&&& && \ar[0,4]_{\overline{f}_2} &&&&
}
\end{equation}
Composing (\ref{left_cells}) with (\ref{bottom_cells}) gives us:
\begin{equation}\label{left_bottom_composite}
    \xymatrix@C=3em{
    &&&&\ar@{}[d]|{\rho_1\Downarrow}\ar[dll]_{t_1}\ar[dl]|{r_1} & \ar[l]_{\overline{p}}\ar[dl]|{\overline{r}_1}\ar@{}[d]|{\rho_3^{-1}\Downarrow} & \ar[l]_{\tilde{r}_2x}\ar[dl]|{\tilde{r}_1x}\ar[r]^{\tilde{r}_2x}\ar[dr]|{\tilde{r}_1x} & \ar@{}[d]|{\rho_3^{-1}\Downarrow}\ar[dr]^{\overline{r}_1}
\\
&&\ar[dl]_{\overline{v}_1}\ar@{}[r]|{\varepsilon_1^{-1}\Downarrow}&\ar[dll]|{s_1'}\ar@{}[d]|{\tilde\alpha^{-1}\Downarrow} & \ar[l]|p \ar[dl]|q \ar@{}[d]|{\overline\rho^{-1}_2} &\ar[l]|{\overline{r}_2}\ar[dl]|{\overline{q}}&&\ar[r]_{\overline{r}_2} & \ar[r]^p\ar[dr]_q &\ar@{}[d]|{\tau^{-1}\Downarrow}\ar[dr]^{f_1'} &
\\
&\ar[l]_u&\ar@{}[d]|{\varepsilon_2\Downarrow}&\ar[ll]|{s_2'} &\ar[l]|{r_2}\ar[dll]^{t_2}\ar@{}[ddrr]|{\omega_1\Downarrow} && \ar[uu]|{d}\ar[dd]|{\hat{d}}\ar@{}[ddrr]|{\omega_2\Downarrow} &&&\ar[r]_{f_2'} &
\\
&&\ar[ul]^{\overline{v}_2} &&&&&&&&
\\
&&&&&&\ar@/^2ex/[-1,-4]^{\hat{t}_2\overline{\hat{q}}\tilde{\hat{r}}_1\hat{x}}\ar@/_2ex/[-2,4]_{a\hat{f}_2'\hat{q}\overline{\hat{r}}_2\tilde{\hat{r}}_1\hat{x}}&&&&}
\end{equation}

As we noted at the beginning, the 2-cell diagrams (\ref{top_right_composite}) and (\ref{left_bottom_composite}) are equivalent, so there are arrows and 2-cells as in 
$$
\xymatrix{
&\ar@{}[d]|{\stackrel{\chi}{\Rightarrow}}
\\
\ar@/^1.2ex/[ur]^{c}\ar@/_1.2ex/[dr]_{\hat{c}}&\ar[l]|{e_1}\ar[r]|{e_2}\ar@{}[d]|{\stackrel{\hat{\chi}}{\Rightarrow}} & \ar@/_1.2ex/[ul]_{d}\ar@/^1.2ex/[dl]^{\hat{d}}
\\
&
}
$$
to witness this equivalence; i.e., such that 
\begin{equation}\label{equiv_eqn1}
\xymatrix{
&&\ar[r]^{\tilde{r}_2x}\ar@/_1.5ex/[dr]|{\tilde{r}_1x} &\ar@{}[d]|{\rho_3^{-1}\Downarrow}\ar[dr]^{\overline{r}_1} &&& & &&\ar@/^2.5ex/[1,5]^{f_1'p\overline{r}_1\tilde{r}_2x} &\ar@{}[ddd]|{\xi_2\Downarrow}
\\
&&&\ar[r]\ar@{}[ddd]|{\omega_2\Downarrow} & \ar[dr]_q\ar[r]^p & \ar@{}[d]|{\tau^{-1}\Downarrow}\ar[dr]^{f_1'}& & && &&&&&
\\
\ar@/^2ex/[uurr]^c & \ar[l]^{e_1}\ar[r]_{e_2}\ar@{}[u]|{\stackrel{\chi}{\Rightarrow}} &\ar[uu]|{d}\ar[dd]|{\hat{d}} &&&\ar[r]_{f_2'} & \ar@{}[r]|{\textstyle\equiv} & \ar@/_2ex/[ddrr]_{\hat{d}}&\ar[l]_{e_2}\ar[r]^{e_1}\ar@{}[d]|{\stackrel{\hat{\chi}}{\Leftarrow}} & \ar[dd]|{\hat{c}}\ar[uu]^{c} &&&\ar[r]^{\hat{f}_1'}\ar@{}[d]|{\hat\tau^{-1}\Downarrow} & \ar[ur]_a
\\
&&&&&& & &&&\ar[r]^{\overline{\hat{r}}_1}\ar@{}[d]|{\hat\rho_3^{-1}\Downarrow}&\ar[ur]^{\hat{p}}\ar[r]_{\hat{q}} & \ar[ur]_{\hat{f}_2'}
\\
&&\ar@/_2.5ex/[-2,4]_{a\hat{f}_2'\hat{q}\overline{\hat{r}}_2\tilde{\hat{r}}_1\hat{x}}&&&& & &&\ar@/^1.5ex/[ur]^(.7){\tilde{\hat{r}}_2\hat{x}}\ar[r]_{\tilde{\hat{r}}_2\hat{x}}&\ar[ur]_{\overline{\hat{r}}_2}&&&&
}
\end{equation}
and
\begin{equation}\label{equiv_eqn2}
\xymatrix@C=3.1em{
&&&&\ar@{}[d]|{\rho_1\Downarrow}\ar[dll]_{t_1}\ar[dl]|{r_1} & \ar[l]_{\overline{p}}\ar[dl]|{\overline{r}_1}\ar@{}[d]|{\rho_3^{-1}\Downarrow} & \ar[l]_{\tilde{r}_2x}\ar@/^1.5ex/[dl]|{\tilde{r}_1x}&&
\\
&&\ar[dl]_{\overline{v}_1}\ar@{}[r]|{\varepsilon_1^{-1}\Downarrow}&\ar[dll]|{s_1'}\ar@{}[d]|{\tilde\alpha^{-1}\Downarrow} & \ar[l]|p \ar[dl]|q \ar@{}[d]|{\overline\rho^{-1}_2} &\ar[l]|{\overline{r}_2}\ar[dl]|{\overline{q}}&&&
\\
&\ar[l]_u&\ar@{}[d]|{\varepsilon_2\Downarrow}&\ar[ll]|{s_2'} &\ar[l]|{r_2}\ar[dll]^{t_2}\ar@{}[ddrr]|{\omega_1\Downarrow} && \ar[uu]|{d}\ar[dd]|{\hat{d}} &\ar[l]^{e_2}\ar[r]_{e_1}\ar@{}[u]|{\stackrel{\chi}{\Leftarrow}}&\ar@/_2.5ex/[uull]_{c}\ar@{}[r]|{\textstyle\equiv}&
\\
&&\ar[ul]^{\overline{v}_2} &&&&
\\
&&&&&&\ar@/^2ex/[-1,-4]^{\hat{t}_2\overline{\hat{q}}\tilde{\hat{r}}_1\hat{x}}
\\
&&& &&&&&\ar@/_2ex/[3,-4]_{\overline{v}_1t_1\overline{p}\tilde{r}_2x}\ar@{}[3,-2]|{\overline{v}_1\xi_1\Downarrow}
\\
&&& &&&&&
\\
&&& &&&&&
\\
&&& &\ar[l]_u&\ar[l]_{\overline{v}_1}\ar@{}[d]|{\hat\varepsilon_1^{-1}\Downarrow}&&&\ar[uuu]_{c}\ar[ddd]^{\hat{c}}&\ar[l]_{e_1}\ar[r]^{e_2}\ar@{}[dd]|{\stackrel{\hat\chi}{\Rightarrow}} & \ar@/^2.5ex/[3,-2]^{\hat{d}}
\\
&&& &&\ar[ul]|{b_1'}\ar@{}[d]|{\hat\alpha^{-1}\Downarrow}&\ar[l]|{\hat{r}_1}\ar[ul]_{\hat{t}_1}\ar@{}[d]|{\hat\rho_1^{-1}\Downarrow}&& &&
\\
&&& &\ar[uu]^{\overline{v}_2}\ar@{}[r]|{\hat\varepsilon_2\Downarrow}&\ar[uul]|{b_2'}&\ar[l]|{\hat{q}}\ar[ul]|{\hat{p}}\ar@{}[d]|{\hat\rho_2^{-1}\Downarrow} & \ar[l]|{\overline{\hat{r}}_1}\ar[ul]_{\overline{\hat{q}}}\ar@{}[d]|{\hat\rho_3^{-1}\Downarrow} & &&
\\
&&&  &&&\ar[ull]^{\hat{t}_2}\ar[ul]|{\hat{r}_2} & \ar[l]^{\overline{\hat{q}}}\ar[ul]|{\overline{\hat{r}}_2} &\ar[l]^{\tilde{\hat{r}}_1\hat{x}}\ar[ul]_{\tilde{\hat{r}}_2\hat{x}}
}
\end{equation}
It can be checked by a long but straightforward calculation using all the equations set up in this proof that the following cells witness the equivalence of (\ref{first_whisker}) and (\ref{second_whisker}):
$$
\xymatrix{
&&\ar@{}[ddd]|{\stackrel{\xi_1}{\Rightarrow}}&&
\\
&\ar[ur]^{t_1}&&\ar[ul]_{\hat{t}_1}
\\
&&&&
\\
\ar[uur]^{\overline{p}}&&&&\ar[uul]_{\overline{\hat{p}}}
\\
\ar[u]^{\tilde{r}_2x}\ar[d]_{\tilde{r}_1x} & \ar[l]|c\ar@/^5ex/[rrr]^{\hat{c}}\ar@{}@<2.5ex>[rrr]|{\hat\chi\Downarrow}&\ar[l]|{e_1}\ar[r]|{e_2} & \ar[r]|{\hat{d}}\ar@/^5ex/[lll]^d \ar@{}@<2.5ex>[lll]|{\stackrel{\chi}{\Rightarrow}} & \ar[u]_{\tilde{\hat{r}}_2\hat{x}}\ar[d]^{\tilde{\hat{r}}_2\hat{x}}
\\
\ar[ddr]_{\overline{q}}&&\ar@{}[ddd]|{\stackrel{\omega_1}{\Rightarrow}} && \ar[ddl]^{\overline{\hat{q}}}
\\
&&&&
\\
&\ar[dr]_{t_2} &&\ar[dl]^{\hat{t}_2} &
\\
&&&&
}
$$
\end{proof}
\section{Horizontal Composition of 2-Cell Diagrams}\label{horcomp}

In this appendix we provide a proof for the following result, described in Section \ref{horcomppb}:

\begin{prop} Let $\calB$ be a bicategory and let
$\frakW$ be  a class of arrows in $\calB$ that is pullback-closed, satisfies the fractions axioms and is full and co-full. If the cell $\beta$ in the following diagram of composable 2-cell diagrams  is invertible,
\begin{equation}\label{composable-2-cells2}
\xymatrix@C=3.5em{
&A'\ar[dl]_{u_1}\ar[dr]^{f_1}&&B'\ar[dl]_{v_1}\ar[dr]^{g_1}
\\
A\ar@{}[r]|{\rho_{u_1,u_2}\Downarrow}&P_{u_1,u_2}\ar[u]_{\pi_{A'}}\ar[d]^{\pi_{A''}}\ar@{}[r]|{\beta\Downarrow}&B\ar@{}[r]|{\rho_{v_1,v_2}\Downarrow} &P_{v_1,v_2}\ar[u]_{\pi_{B'}}\ar[d]^{\pi_{B''}}\ar@{}[r]|{\gamma\Downarrow}&C
\\
&A''\ar[ul]^{u_2}\ar[ur]_{f_2} && B''\ar[ul]^{v_2}\ar[ur]_{g_2}
}
\end{equation}  then the horizontal composition of these 2-cells in $\calB[\frakW^{-1}]$ can be represented by the 2-cell diagram 

\begin{equation}\label{result}\xymatrix@C=5em{
&&D\ar[dll]_{u_1\overline{v}_1}\ar[r]^{\overline{f}_1} & B'\ar[dr]^{g_1}
\\
A\ar@{}[rr]|{\rho_{u_1\overline{v}_1,u_2\overline{v}_2}\Downarrow} && P_{u_1\overline{v}_1,u_2\overline{v}_2}\ar@{}[ur]|=\ar@{}[dr]|=\ar[u]_{\pi_{D}}\ar[d]^{\pi_{D'}} \ar[r]^{w_{v_1,v_2}} & P_{v_1,v_2}\ar[u]_{\pi_{B'}}\ar[d]^{\pi_{B''}}\ar@{}[r]|{\gamma\Downarrow} & C\\
&&D'\ar[ull]^{u_2\overline{v}_2}\ar[r]_{\overline{f}_2} & B''\ar[ur]_{g_2}
} \end{equation}

as described in Section \ref{horcomppb}.
\end{prop}

\begin{proof}
We construct the horizontal composition of the 2-cell diagrams of (\ref{composable-2-cells2})     using whiskering and vertical composition:
$$([\rho_{v_1,v_2},\gamma](f_2,u_2))\cdot((v_1,g_1)[\rho_{u_1,u_2},\beta])$$
We start by considering the whiskering $(v_1,g_1)[\rho_{u_1,u_2},\beta]$.  To construct this, we need the chosen square:  
$$
\xymatrix@R=2.5em@C=2.5em{
\ar[r]^{v_1^*}\ar[d]_{f_2^*} \ar@{}[dr]|{\Downarrow\varepsilon_{1,2}}& \ar[d]^{f_2}
\\
\ar[r]_{v_1}&
}
$$
This lets us construct the composition of the spans of arrows as in the following diagram (which is not a pasting diagram):
\[\begin{tikzcd}[column sep=small, row sep = small]
	&& {A'} && D \\
	&& {} && {\Downarrow\delta_1} \\
	A & {\rho_{u_1,u_2}\Downarrow} & {P_{u_1,u_2}} & \beta\Downarrow & B && {B'} && C \\
	&& {} && {\Downarrow\varepsilon_{1,2}^{-1}} \\
	&& {A''} && {D^*}
	\arrow["{v_1}"{description}, from=3-7, to=3-5]
	\arrow["{g_1}", from=3-7, to=3-9]
	\arrow["{\pi_{A''}}"{description}, from=3-3, to=5-3]
	\arrow["{\pi_{A'}}"{description}, from=3-3, to=1-3]
	\arrow["{u_1}"', from=1-3, to=3-1]
	\arrow["{u_2}", from=5-3, to=3-1]
	\arrow["{v_1^*}", from=5-5, to=5-3]
	\arrow["{f_2^*}"{description}, from=5-5, to=3-7]
	\arrow["{\overline{f}_1}"{description}, from=1-5, to=3-7]
	\arrow["{\overline{v}_1}"', from=1-5, to=1-3]
	\arrow["{f_1}"{description}, from=1-3, to=3-5]
	\arrow["{f_2}"{description}, from=5-3, to=3-5]
\end{tikzcd}\]
The left-hand 2-cell for the 2-cell diagram representing the  whiskering $(v_1,g_1) [\rho_{u_1,u_2},\beta]$ is the pseudo-pullback square
$$
\xymatrix@C=4.5em{
P_{u_1\overline{v}_1,u_2v_1^*} \ar[d]_{\pi'_{D^*}} \ar[r]^{\pi'_D} \ar@{}[dr]|{\wr\Downarrow\rho_{u_1\overline{v}_1,u_2v_1^*}} & D \ar[d]^{u_1\overline{v}_1}
\\
D^*\ar[r]_{u_2v_1^*} &A}
$$
Let $w^*_{u_1,u_2}\colon P_{u_1\overline{v}_1,u_2v_1^*} \rightarrow P_{u_1,u_2}$ be the unique arrow such that
$\rho_{u_1,u_2}w^*_{u_1,u_2}=\rho_{u_1\overline{v}_1,u_2v_1^*}$.
Then the right-hand 2-cell in the diagram representing the whiskering of $[\rho_{u_1,u_2},\beta]$ with $(v_1,g_1)$ can be obtained by considering diagram below and then taking a lifting with respect to $v_1$ for the right-hand pasting diagram (using fullness of $\frakW$):
\begin{equation}\label{first-2-cell}\begin{tikzcd}[ampersand replacement=\&]
	\&\& D \\
	\&\&\& {A'} \\
	A \&\& {P_{u_1\overline{v}_1,u_2v_1^*}} \& {P_{u_1,u_2}} \& B \& {B'} \& C \\
	\&\&\& {A''} \\
	\&\& {D^*}
	\arrow["{\pi_{A'}}", from=3-4, to=2-4]
	\arrow["{\pi_{A''}}"', from=3-4, to=4-4]
	\arrow[""{name=0, anchor=center, inner sep=0}, "{v_1^*}"{description}, from=5-3, to=4-4]
	\arrow["{\pi'_{D^*}}"', from=3-3, to=5-3]
	\arrow[""{name=1, anchor=center, inner sep=0}, "{f_1}"{description}, from=2-4, to=3-5]
	\arrow[""{name=2, anchor=center, inner sep=0}, "{f_2}"{description}, from=4-4, to=3-5]
	\arrow["{v_1}", from=3-6, to=3-5]
	\arrow[""{name=3, anchor=center, inner sep=0}, "{f_2^*}"', curve={height=30pt}, from=5-3, to=3-6]
	\arrow[""{name=4, anchor=center, inner sep=0}, "{u_2v_1^*}", curve={height=-6pt}, from=5-3, to=3-1]
	\arrow["{g_1}"', from=3-6, to=3-7]
	\arrow["{\pi'_D}", from=3-3, to=1-3]
	\arrow[""{name=5, anchor=center, inner sep=0}, "{\overline{v}_1}"{description}, from=1-3, to=2-4]
	\arrow[""{name=6, anchor=center, inner sep=0}, "{\overline{f}_1}", curve={height=-30pt}, from=1-3, to=3-6]
	\arrow[""{name=7, anchor=center, inner sep=0}, "{u_1\overline{v}_1}"', curve={height=6pt}, from=1-3, to=3-1]
	\arrow[""{name=8, anchor=center, inner sep=0}, "{w^*_{u_1,u_2}}", from=3-3, to=3-4]
	\arrow["{\varepsilon_{12}^{-1}}", shorten <=2pt, shorten >=4pt, Rightarrow, from=4-4, to=3]
	\arrow["{\rho_{u_1\overline{v}_1,u_2v_1^*}}"{description}, shorten <=22pt, shorten >=22pt, Rightarrow, from=7, to=4]
	\arrow["\beta", shorten <=12pt, shorten >=12pt, Rightarrow, from=1, to=2]
	\arrow["{\delta_1}", shorten <=4pt, shorten >=2pt, Rightarrow, from=6, to=2-4]
	\arrow[shorten <=22pt, shorten >=22pt, Rightarrow, no head, from=8, to=0]
	\arrow[shorten <=22pt, shorten >=22pt, Rightarrow, no head, from=5, to=8]
\end{tikzcd}\end{equation}
We write  $\tilde\beta\colon \overline{f}_1\pi_D'\Rightarrow f_2^*\pi'_{D^*}$ for the lifted cell.  We obtain then the following 2-cell diagram in the bicategory of fractions:
\[\begin{tikzcd}[ampersand replacement=\&]
	\&\& D \\
	A \&\& {P_{u_1\overline{v}_1,u_2v_1^*}} \&\& {B'} \&\& C \\
	\&\& {D^*}
	\arrow[""{name=0, anchor=center, inner sep=0}, "{u_2v_1^*}", from=3-3, to=2-1]
	\arrow[""{name=1, anchor=center, inner sep=0}, "{u_1\overline{v}_1}"', from=1-3, to=2-1]
	\arrow["{\pi'_D}", from=2-3, to=1-3]
	\arrow["{\pi'_{D^*}}"', from=2-3, to=3-3]
	\arrow["{g_1}", from=2-5, to=2-7]
	\arrow[""{name=2, anchor=center, inner sep=0}, "{\overline{f}_1}", from=1-3, to=2-5]
	\arrow[""{name=3, anchor=center, inner sep=0}, "{f_2^*}"', from=3-3, to=2-5]
	\arrow["{\rho_{u_1\overline{v}_1,u_2v_1^*}}"{description}, shorten <=4pt, shorten >=4pt, Rightarrow, from=1, to=0]
	\arrow["\tilde\beta"', shorten <=8pt, shorten >=8pt, Rightarrow, from=2, to=3]
\end{tikzcd}\]
Now we consider the other half of the composition, the whiskering
$[\rho_{v_1,v_2},\gamma](u_2,f_2)$.
The domain and codomain spans of arrows for the whiskering are constructed in the following diagram (not a pasting diagram):
\[\begin{tikzcd}[ampersand replacement=\&]
	\&\&\&\& {D^*} \&\& {B'} \\
	A \&\& {A''} \&\& B \&\& {P_{v_1,v_2}} \&\& C \\
	\&\&\&\& {D'} \&\& {B''}
	\arrow["{\pi_{B'}}"', from=2-7, to=1-7]
	\arrow["{\pi_{B''}}", from=2-7, to=3-7]
	\arrow[""{name=0, anchor=center, inner sep=0}, "{v_1}"{description}, from=1-7, to=2-5]
	\arrow[""{name=1, anchor=center, inner sep=0}, "{v_2}"{description}, from=3-7, to=2-5]
	\arrow[""{name=2, anchor=center, inner sep=0}, "{g_1}", from=1-7, to=2-9]
	\arrow[""{name=3, anchor=center, inner sep=0}, "{g_2}"', from=3-7, to=2-9]
	\arrow[""{name=4, anchor=center, inner sep=0}, "{f_2}"{description}, from=2-3, to=2-5]
	\arrow["{u_2}"{description}, from=2-3, to=2-1]
	\arrow[""{name=5, anchor=center, inner sep=0}, "{f_2^*}"{description}, from=1-5, to=1-7]
	\arrow["{v_1^*}"{description}, from=1-5, to=2-3]
	\arrow[""{name=6, anchor=center, inner sep=0}, "{\overline{f}_2}"{description}, from=3-5, to=3-7]
	\arrow["{\overline{v}_2}"{description}, from=3-5, to=2-3]
	\arrow["{\rho_{v_1,v_2}}", shorten <=10pt, shorten >=10pt, Rightarrow, from=0, to=1]
	\arrow["\gamma"', shorten <=10pt, shorten >=10pt, Rightarrow, from=2, to=3]
	\arrow["{\varepsilon_{1,2}}"{description}, shorten <=28pt, shorten >=28pt, Rightarrow, from=4, to=5]
	\arrow["{\delta_2}"{description}, shorten <=25pt, shorten >=25pt, Rightarrow, from=6, to=4]
\end{tikzcd}\]
To find a 2-cell diagram representing this whiskering,  we start with the pseudo pullback,
$$
\xymatrix{P_{u_2v_1^*,u_2\overline{v}_2}\ar[rr]^{\pi_{D^*}''}\ar[dd]_{\pi_{D'}''}\ar@{}[ddrr]|{\wr\Downarrow\rho_{u_2v_1^*,u_2\overline{v}_2}} &&D^*\ar[d]^{v_1^*}
\\
&&A''\ar[d]^{u_2}
\\
D'\ar[r]_{\overline{v}_2}&A''\ar[r]_{u_2} &A
}
$$
Using fullness of $\frakW$, let 
$$\xymatrix@C=3.5em@R=3.5em{P_{u_2v_1^*,u_2\overline{v}_2} \ar[r]^{\pi_{D^*}''}\ar[d]_{\pi_{D'}''}\ar@{}[dr]|{\wr\Downarrow\tilde\rho_{u_2v_1^*,u_2\overline{v}_2}} &D^*\ar[d]^{v_1^*}
\\
D'\ar[r]_{\overline{v}_2}&A''
}
$$
be the lifting of this diagram with respect to $u_2$, and let $x_{v_1,v_2}\colon P_{u_2v_1^*,u_2\overline{v}_2}\to P_{v_1,v_2}$ be the unique arrow such that the following equality of  pasting diagrams holds:
\[\begin{tikzcd}[ampersand replacement=\&, column sep = small]
	\&\& {B'} \&\&\&\& {D^*} \&\& {B'} \\
	\\
	{P_{u_2v_1^*,u_2\overline{v}_2}} \&\& {P_{v_1,v_2}} \&\& B \& {P_{u_2v_1^*,u_2\overline{v}_2}} \&\& {A''} \&\& B \\
	\\
	\&\& {B''} \&\&\&\& {D'} \&\& {B''}
	\arrow[""{name=0, anchor=center, inner sep=0}, "{x_{v_1,v_2}}", from=3-1, to=3-3]
	\arrow["{\pi_{B'}}", from=3-3, to=1-3]
	\arrow[""{name=1, anchor=center, inner sep=0}, "{f_2^*\pi_{D^*}''}", curve={height=-12pt}, from=3-1, to=1-3]
	\arrow["{\pi_{B''}}"', from=3-3, to=5-3]
	\arrow[""{name=2, anchor=center, inner sep=0}, "{\overline{f}_2\pi_{D'}''}"', curve={height=12pt}, from=3-1, to=5-3]
	\arrow[""{name=3, anchor=center, inner sep=0}, "{v_1}", curve={height=-12pt}, from=1-3, to=3-5]
	\arrow[""{name=4, anchor=center, inner sep=0}, "{v_2}"', curve={height=12pt}, from=5-3, to=3-5]
	\arrow[""{name=5, anchor=center, inner sep=0}, "{\pi_{D^*}''}", curve={height=-6pt}, from=3-6, to=1-7]
	\arrow[""{name=6, anchor=center, inner sep=0}, "{\pi_{D'}''}"', curve={height=6pt}, from=3-6, to=5-7]
	\arrow["{v_1^*}"', from=1-7, to=3-8]
	\arrow["{\overline{v}_2}", from=5-7, to=3-8]
	\arrow[""{name=7, anchor=center, inner sep=0}, "{f_2}", from=3-8, to=3-10]
	\arrow[""{name=8, anchor=center, inner sep=0}, "{f_2^*}", from=1-7, to=1-9]
	\arrow[""{name=9, anchor=center, inner sep=0}, "{\overline{f}_2}"', from=5-7, to=5-9]
	\arrow["{v_2}"', from=5-9, to=3-10]
	\arrow["{v_1}", from=1-9, to=3-10]
	\arrow["\equiv"{description}, draw=none, from=3-5, to=3-6]
	\arrow[shorten <=12pt, shorten >=15pt, Rightarrow, no head, from=1, to=0]
	\arrow[shorten <=12pt, shorten >=15pt, Rightarrow, no head, from=2, to=0]
	\arrow["{\rho_{v_1,v_2}}"{description}, shorten <=18pt, shorten >=18pt, Rightarrow, from=3, to=4]
	\arrow["{\tilde{\rho}_{u_2v_1^*,u_2\overline{v}_2}}", shift left=5, shorten <=18pt, shorten >=18pt, Rightarrow, from=5, to=6]
	\arrow["{\varepsilon_{1,2}}", shorten <=20pt, shorten >=20pt, Rightarrow, from=8, to=7]
	\arrow["{\delta_2^{-1}}", shorten <=20pt, shorten >=20pt, Rightarrow, from=7, to=9]
\end{tikzcd}\]
The whiskering $[\rho_{v_1,v_2},\gamma] (u_2,f_2)$ can now be represented by the diagram
\begin{equation}\label{sec-2-cell}
\xymatrix@C=6em{&D^*\ar[r]^{f_2^*}\ar@{}[dr]|= \ar[dl]_{u_2v_1^*}& B'\ar[dr]^{g_1}
\\
A\ar@{}[r]|{\Downarrow\rho_{u_2v_1^*,u_2\overline{v}_2}}&P_{u_2v_1^*,u_2\overline{v}_2} \ar[u]_{\pi''_{D^*}}\ar[d]^{\pi''_{D'}}\ar[r]^{x_{v_1,v_2}}\ar@{}[dr]|= & P_{v_1,v_2}\ar[u]^{\pi_{B'}}\ar[d]_{\pi_{B''}}\ar@{}[r]|{\Downarrow\gamma}&C
\\
&D'\ar[ul]^{u_2\overline{v}_2}\ar[r]_{\overline{f}_2} & B''\ar[ur]_{g_2}
}
\end{equation}

We now want to construct the vertical composition of the whiskerings  $[\rho_{v_1,v_2},\gamma](f_2,u_2)$ and $(v_1,g_1)[\rho_{u_1,u_2},\beta]$ as presented in  (\ref{first-2-cell}) and (\ref{sec-2-cell}). For this we need the following pseudo pullback (or any square that commutes up to an invertible 2-cell):
$$
\xymatrix{R\ar[r]^{\pi_1}\ar[d]_{\pi_2}\ar@{}[dr]|{\rho\Downarrow\wr} &P_{u_1\overline{v}_1,\overline{v}_2v_1^*}\ar[d]^{\pi'_{D^*}}\\P_{u_2v_1^*,u_2\overline{v}_2}\ar[r]_{\pi''_{D^*}} & D^*}
$$
Furthermore, let $r\colon R\to P_{u_1\overline{v}_1,u_2\overline{v}_2}$ be the unique arrow such that the following equality of pasting diagrams holds,
\[\begin{tikzcd}[ampersand replacement=\&,column sep = 2em]
	\&\& D \&\&\&\& D \\
	\&\&\&\&\&\& {P_{u_1\overline{v}_1,\overline{v}_2v_1^*}} \\
	R \&\& {P_{u_1\overline{v_1},u_2\overline{v}_2}} \&\& A \& R \&\& {D^*} \&\& A \\
	\&\&\&\&\&\& {P_{u_2v_1^*,u_2\overline{v}_2}} \\
	\&\& {D'} \&\&\&\& {D'}
	\arrow[""{name=0, anchor=center, inner sep=0}, "r"{description}, from=3-1, to=3-3]
	\arrow["{\pi_D}", from=3-3, to=1-3]
	\arrow["{\pi_{D'}}"', from=3-3, to=5-3]
	\arrow[""{name=1, anchor=center, inner sep=0}, "{\pi_{D'}r}"', curve={height=6pt}, from=3-1, to=5-3]
	\arrow[""{name=2, anchor=center, inner sep=0}, "{\pi_Dr}", curve={height=-6pt}, from=3-1, to=1-3]
	\arrow[""{name=3, anchor=center, inner sep=0}, "{u_1\overline{v}_1}", curve={height=-15pt}, from=1-3, to=3-5]
	\arrow[""{name=4, anchor=center, inner sep=0}, "{u_2\overline{v}_2}"', curve={height=15pt}, from=5-3, to=3-5]
	\arrow["\equiv"{description}, draw=none, from=3-5, to=3-6]
	\arrow[""{name=5, anchor=center, inner sep=0}, "{\pi_1}", from=3-6, to=2-7]
	\arrow[""{name=6, anchor=center, inner sep=0}, "{\pi_2}"', from=3-6, to=4-7]
	\arrow["{\pi_D'}", from=2-7, to=1-7]
	\arrow["{\pi_{D^*}'}"', from=2-7, to=3-8]
	\arrow["{\pi''_{D^*}}", from=4-7, to=3-8]
	\arrow[""{name=7, anchor=center, inner sep=0}, "{u_2v_1^*}"{description}, from=3-8, to=3-10]
	\arrow[""{name=8, anchor=center, inner sep=0}, "{u_1\overline{v}_1}", curve={height=-15pt}, from=1-7, to=3-10]
	\arrow["{\pi''_{D'}}"', from=4-7, to=5-7]
	\arrow[""{name=9, anchor=center, inner sep=0}, "{u_2\overline{v}_2}"', curve={height=15pt}, from=5-7, to=3-10]
	\arrow[shorten <=12pt, shorten >=12pt, Rightarrow, no head, from=2, to=0]
	\arrow[shorten <=12pt, shorten >=12pt, Rightarrow, no head, from=0, to=1]
	\arrow["{\rho_{u_1\overline{v}_1,u_2\overline{v}_2}}"{description}, shorten <=25pt, shorten >=25pt, Rightarrow, from=3, to=4]
	\arrow["\rho", shift left=5, shorten <=8pt, shorten >=8pt, Rightarrow, from=5, to=6]
	\arrow["{\rho_{u_1\overline{v}_1,\overline{v}_2v_1^*}}"{description}, shift right=4,shorten <=12pt, shorten >=12pt, Rightarrow, from=8, to=7]
	\arrow["{\rho_{u_2v_1^*,u_2,\overline{v}_2}}"{description},shift right=4, shorten <=12pt, shorten >=12pt, Rightarrow, from=7, to=9]
\end{tikzcd}\]
Now let 
$$
\xymatrix{
&D\ar[dr]^{g_1\overline{f}_1}\ar@{}[dd]|{\widetilde{\beta\gamma}\Downarrow}
\\
P_{u_1\overline{v}_1,u_2\overline{v}_2} \ar[ur]^{\pi_D} \ar[dr]_{\pi_{D'}} && C
\\
&D'\ar[ur]_{g_2\overline{f}_2}
}
$$
be a lifting with respect to $r$ of the pasting of the following diagram,
\begin{equation}\begin{tikzcd}[ampersand replacement=\&]\label{initial_diagram}
	{P_{u_1\overline{v}_1,u_2\overline{v}_2}} \&\& D \\
	\& {P_{u_1\overline{v}_1,u_2v_1^*}} \\
	R \&\& {D^*} \&\& {B'} \\
	\& {P_{u_2v_1^*,u_2\overline{v}_2}} \&\& {P_{v_1,v_2}} \&\& C \\
	{P_{u_1\overline{v}_1,u_2\overline{v}_2}} \&\& {D'} \&\& {B''}
	\arrow[""{name=0, anchor=center, inner sep=0}, "r"', from=3-1, to=1-1]
	\arrow[""{name=1, anchor=center, inner sep=0}, "{\pi_1}", from=3-1, to=2-2]
	\arrow[""{name=2, anchor=center, inner sep=0}, "{\pi_2}"', from=3-1, to=4-2]
	\arrow["{\pi'_{D^*}}", from=2-2, to=3-3]
	\arrow["{\pi''_{D^*}}"', from=4-2, to=3-3]
	\arrow[""{name=3, anchor=center, inner sep=0}, "{\pi'_D}", from=2-2, to=1-3]
	\arrow[""{name=4, anchor=center, inner sep=0}, "{\pi''_{D'}}"', from=4-2, to=5-3]
	\arrow[""{name=5, anchor=center, inner sep=0}, "{f_2^*}", from=3-3, to=3-5]
	\arrow[""{name=6, anchor=center, inner sep=0}, "{\overline{f}_1}", from=1-3, to=3-5]
	\arrow["{x_{v_1,v_2}}"{description}, from=4-2, to=4-4]
	\arrow[""{name=7, anchor=center, inner sep=0}, "{\pi_{B'}}"{description}, from=4-4, to=3-5]
	\arrow[""{name=8, anchor=center, inner sep=0}, "{\pi_{B''}}"{description}, from=4-4, to=5-5]
	\arrow[""{name=9, anchor=center, inner sep=0}, "{\overline{f}_2}"', from=5-3, to=5-5]
	\arrow["{g_1}", from=3-5, to=4-6]
	\arrow["{g_2}"', from=5-5, to=4-6]
	\arrow["{\pi_D}", from=1-1, to=1-3]
	\arrow[""{name=10, anchor=center, inner sep=0}, "r"', from=3-1, to=5-1]
	\arrow["{\pi_{D'}}"', from=5-1, to=5-3]
	\arrow["\rho", shift left=5, shorten <=4pt, shorten >=4pt, Rightarrow, from=1, to=2]
	\arrow[shorten <=10pt, shorten >=8pt, Rightarrow, no head, from=5, to=7]
	\arrow[shift right=5, shorten <=10pt, shorten >=8pt, Rightarrow, no head, from=8, to=9]
	\arrow["\gamma", shift left=5, shorten <=7pt, shorten >=7pt, Rightarrow, from=7, to=8]
	\arrow["\tilde\beta"', shift right=5, shorten <=4pt, shorten >=6pt, Rightarrow, from=6, to=5]
	\arrow[shorten <=26pt, shorten >=26pt, Rightarrow, no head, from=0, to=3]
	\arrow[shorten <=26pt, shorten >=26pt, Rightarrow, no head, from=10, to=4]
\end{tikzcd}\end{equation}

We need to show that the 2-cell diagram (\ref{result}) that we constructed in Section  \ref{horcomppb}, 
$$\xymatrix@C=5em{
&&D\ar[dll]_{u_1\overline{v}_1}\ar[r]^{\overline{f}_1} & B'\ar[dr]^{g_1}
\\
A\ar@{}[rr]|{\rho_{u_1\overline{v}_1,u_2\overline{v}_2}\Downarrow} && P_{u_1\overline{v}_1,u_2\overline{v}_2}\ar@{}[ur]|=\ar@{}[dr]|=\ar[u]_{\pi_{D}}\ar[d]^{\pi_{D'}} \ar[r]^{w_{v_1,v_2}} & P_{v_1,v_2}\ar[u]_{\pi_{B'}}\ar[d]^{\pi_{B''}}\ar@{}[r]|{\gamma\Downarrow} & C\\
&&D'\ar[ull]^{u_2\overline{v}_2}\ar[r]_{\overline{f}_2} & B''\ar[ur]_{g_2}
}
$$
is equivalent to the following diagram, whose right side is the lift of (\ref{initial_diagram}) with respect to $r$:  
\begin{equation}\label{firstrepr}
\xymatrix@C=5em{&&D\ar[dll]_{u_1\overline{v}_1}\ar[drr]^{g_1\overline{f}_1} &&
\\
A\ar@{}[rr]|{\rho_{u_1\overline{v}_1,u_2\overline{v}_2}\Downarrow} &&
P_{u_1\overline{v}_1,u_2\overline{v}_2}\ar[u]_{\pi_{D}}\ar[d]^{\pi_{D'}}\ar@{}[rr]|{\widetilde{\beta\gamma}}&& C
\\
&&D'\ar[ull]^{u_2\overline{v}_2}\ar[urr]_{g_2\overline{f}_2} &&
}
\end{equation}

To do this, we precompose $\widetilde{\beta\gamma}$ by $r$, which allows us to expand $\widetilde{\beta\gamma}$, replacing it with (\ref{initial_diagram}).
Let diagram (\ref{diagram_A}) be the following sub-diagram of the result:  
\begin{equation}\label{diagram_A} \tag{\bf I}
\begin{tikzcd}[ampersand replacement=\&]
	\&\& {P_{u_1\overline{v}_1,u_2\overline{v}_2}} \&\& D \\
	R \&\& {P_{u_1\overline{v}_1,u_2v_1^*}} \&\& {D^*} \&\& {B'} \\
	\&\& {P_{u_2v_1^*,u_2\overline{v}_2}} \&\& {P_{v_1,v_2}}
	\arrow["{\pi_1}"', from=2-1, to=2-3]
	\arrow["r", from=2-1, to=1-3]
	\arrow["{\pi_2}"', from=2-1, to=3-3]
	\arrow[""{name=0, anchor=center, inner sep=0}, "{\pi'_{D^*}}"{description}, from=2-3, to=2-5]
	\arrow[""{name=1, anchor=center, inner sep=0}, "{\pi''_{D^*}}"', from=3-3, to=2-5]
	\arrow[""{name=2, anchor=center, inner sep=0}, "{\pi_D}", from=1-3, to=1-5]
	\arrow[""{name=3, anchor=center, inner sep=0}, "{\pi_{D'}}"{description}, from=2-3, to=1-5]
	\arrow[""{name=4, anchor=center, inner sep=0}, "{\overline{f}_1}", from=1-5, to=2-7]
	\arrow[""{name=5, anchor=center, inner sep=0}, "{f_2^*}"{description}, from=2-5, to=2-7]
	\arrow["{x_{v_1,v_2}}"', from=3-3, to=3-5]
	\arrow[""{name=6, anchor=center, inner sep=0}, "{\pi_{B'}}"{description}, from=3-5, to=2-7]
	\arrow[shorten <=5pt, shorten >=5pt, Rightarrow, no head, from=1-3, to=2-3]
	\arrow["\rho"', shorten <=4pt, shorten >=4pt, Rightarrow, from=2-3, to=3-3]
	\arrow["\tilde\beta"', shorten <=5pt, shorten >=5pt, Rightarrow, from=1-5, to=2-5]
	\arrow[shorten <=6pt,shorten >=6pt, Rightarrow, no head, from=2-5, to=3-5]
\end{tikzcd}
\end{equation}
So diagram  (\ref{initial_diagram}) is obtained from  diagram (\ref{diagram_A}) by postcomposing it with $g_1$ and then with $\gamma$. 
We now take diagram (\ref{diagram_A}) and postcompose with $v_1$. 
\[\begin{tikzcd}[ampersand replacement=\&]
	\&\& {P_{u_1\overline{v}_1,u_2\overline{v}_2}} \&\& D \\
	R \&\& {P_{u_1\overline{v}_1,u_2v_1^*}} \&\& {D^*} \&\& {B'} \& B \\
	\&\& {P_{u_2v_1^*,u_2\overline{v}_2}} \&\& {P_{v_1,v_2}}
	\arrow["{\pi_1}"', from=2-1, to=2-3]
	\arrow["r", from=2-1, to=1-3]
	\arrow["{\pi_2}"', from=2-1, to=3-3]
	\arrow[""{name=0, anchor=center, inner sep=0}, "{\pi'_{D^*}}"{description}, from=2-3, to=2-5]
	\arrow[""{name=1, anchor=center, inner sep=0}, "{\pi''_{D^*}}"', from=3-3, to=2-5]
	\arrow[""{name=2, anchor=center, inner sep=0}, "{\pi_D}", from=1-3, to=1-5]
	\arrow[""{name=3, anchor=center, inner sep=0}, "{\pi_{D'}}"{description}, from=2-3, to=1-5]
	\arrow[""{name=4, anchor=center, inner sep=0}, "{\overline{f}_1}", from=1-5, to=2-7]
	\arrow[""{name=5, anchor=center, inner sep=0}, "{f_2^*}"{description}, from=2-5, to=2-7]
	\arrow["{x_{v_1,v_2}}"', from=3-3, to=3-5]
	\arrow[""{name=6, anchor=center, inner sep=0}, "{\pi_{B'}}"{description}, from=3-5, to=2-7]
	\arrow["{v_1}", from=2-7, to=2-8]
	\arrow[shorten <=4pt, shorten >=4pt, Rightarrow, no head, from=1-3, to=2-3]
	\arrow["\rho"', shorten <=4pt, shorten >=4pt, Rightarrow, from=2-3, to=3-3]
	\arrow["\tilde\beta"', shorten <=5pt, shorten >=5pt, Rightarrow, from=1-5, to=2-5]
	\arrow[shorten <=5pt, shorten >=5pt, Rightarrow, no head, from=2-5, to=3-5]
\end{tikzcd}\]
Since $\tilde{\beta}$ was originally defined as a lift with respect to $v_1$, this  allows us to expand $\tilde\beta$:
\[\begin{tikzcd}[ampersand replacement=\&]
	{P_{u_1\overline{v}_1,u_2\overline{v}_2}} \& D \&\& {B'} \\
	\&\& {A'} \\
	R \& {P_{u_1\overline{v}_1,u_2v_1^*}} \& {P_{u_1,u_2}} \&\& B \\
	\&\& {A''} \\
	{P_{u_2v_1^*,u_2\overline{v}_2}} \& {D^*} \&\& {B'}
	\arrow["r", from=3-1, to=1-1]
	\arrow["{\pi_2}"', from=3-1, to=5-1]
	\arrow[""{name=0, anchor=center, inner sep=0}, "{\pi_1}"', from=3-1, to=3-2]
	\arrow[""{name=1, anchor=center, inner sep=0}, "{\pi''_{D^*}}"', from=5-1, to=5-2]
	\arrow["{\pi'_{D^*}}"{description}, from=3-2, to=5-2]
	\arrow["{\pi'_D}"{description}, from=3-2, to=1-2]
	\arrow[""{name=2, anchor=center, inner sep=0}, "{\pi_D}", from=1-1, to=1-2]
	\arrow[""{name=3, anchor=center, inner sep=0}, "{\overline{v}_1}"{description}, from=1-2, to=2-3]
	\arrow[""{name=4, anchor=center, inner sep=0}, "{f_1}"{description}, from=2-3, to=3-5]
	\arrow[""{name=5, anchor=center, inner sep=0}, "{f_2}"{description}, from=4-3, to=3-5]
	\arrow["{\pi_{A'}}"', from=3-3, to=2-3]
	\arrow["{\pi_{A''}}", from=3-3, to=4-3]
	\arrow[""{name=6, anchor=center, inner sep=0}, "{\overline{f}_1}", from=1-2, to=1-4]
	\arrow["{v_1}", from=1-4, to=3-5]
	\arrow[""{name=7, anchor=center, inner sep=0}, "{w^*_{u_1,u_2}}"', from=3-2, to=3-3]
	\arrow[""{name=8, anchor=center, inner sep=0}, "{v_1^*}"{description}, from=5-2, to=4-3]
	\arrow[""{name=9, anchor=center, inner sep=0}, "{f_2^*}"', from=5-2, to=5-4]
	\arrow["{v_1}"', from=5-4, to=3-5]
	\arrow["\rho"', shorten <=36pt, shorten >=36pt, Rightarrow, from=0, to=1]
	\arrow[shorten <=36pt, shorten >=36pt, Rightarrow, no head, from=2, to=0]
	\arrow["\beta",shift right=4, shorten <=12pt, shorten >=12pt, Rightarrow, from=4, to=5]
	\arrow["{\delta_1}", shorten <=35pt, shorten >=35pt, Rightarrow, from=6, to=4]
	\arrow[shorten <=25pt, shorten >=25pt, Rightarrow, no head, from=3, to=7]
	\arrow[shorten <=25pt, shorten >=25pt, Rightarrow, no head, from=8, to=7]
	\arrow["{\varepsilon_{1,2}^{-1}}", shorten <=35pt, shorten >=35pt, Rightarrow, from=5, to=9]
\end{tikzcd}\]
We now postcompose by $\rho_{v_1,v_2}$.
\[\begin{tikzcd}[ampersand replacement=\&]
	{P_{u_1\overline{v}_1,u_2\overline{v}_2}} \& D \\
	\&\& {A'} \\
	R \& {P_{u_1\overline{v}_1,u_2v_1^*}} \& {P_{u_1,u_2}} \&\&\& {B'} \\
	\&\& {A''} \\
	\& {D^*} \&\& {B'} \&\& B \\
	{P_{u_2v_1^*,u_2\overline{v}_2}} \& {P_{v_1,v_2}} \&\& {B''} \\
	\& {D'}
	\arrow["r", from=3-1, to=1-1]
	\arrow["{\pi_2}"', from=3-1, to=6-1]
	\arrow[""{name=0, anchor=center, inner sep=0}, "{\pi_1}"', from=3-1, to=3-2]
	\arrow[""{name=1, anchor=center, inner sep=0}, "{\pi''_{D^*}}"', from=6-1, to=5-2]
	\arrow["{\pi'_{D^*}}"{description}, from=3-2, to=5-2]
	\arrow["{\pi'_D}"{description}, from=3-2, to=1-2]
	\arrow[""{name=2, anchor=center, inner sep=0}, "{\pi_D}", from=1-1, to=1-2]
	\arrow[""{name=3, anchor=center, inner sep=0}, "{\overline{v}_1}"{description}, from=1-2, to=2-3]
	\arrow[""{name=4, anchor=center, inner sep=0}, "{f_1}"{description}, from=2-3, to=5-6]
	\arrow[""{name=5, anchor=center, inner sep=0}, "{f_2}"{description}, from=4-3, to=5-6]
	\arrow["{\pi_{A'}}"', from=3-3, to=2-3]
	\arrow["{\pi_{A''}}", from=3-3, to=4-3]
	\arrow[""{name=6, anchor=center, inner sep=0}, "{\overline{f}_1}", from=1-2, to=3-6]
	\arrow["{v_1}", from=3-6, to=5-6]
	\arrow[""{name=7, anchor=center, inner sep=0}, "{w^*_{u_1,u_2}}"', from=3-2, to=3-3]
	\arrow[""{name=8, anchor=center, inner sep=0}, "{v_1^*}"{description}, from=5-2, to=4-3]
	\arrow[""{name=9, anchor=center, inner sep=0}, "{f_2^*}"{description}, from=5-2, to=5-4]
	\arrow[""{name=10, anchor=center, inner sep=0}, "{v_1}"{description}, from=5-4, to=5-6]
	\arrow["{x_{v_1,v_2}}"', from=6-1, to=6-2]
	\arrow[""{name=11, anchor=center, inner sep=0}, "{\pi_{B'}}"{description}, from=6-2, to=5-4]
	\arrow[""{name=12, anchor=center, inner sep=0}, "{\pi_{B''}}"{description}, from=6-2, to=6-4]
	\arrow[""{name=13, anchor=center, inner sep=0}, "{v_2}"', from=6-4, to=5-6]
	\arrow["{\pi''_{D'}}"', from=6-1, to=7-2]
	\arrow[""{name=14, anchor=center, inner sep=0}, "{\overline{f}_2}"', from=7-2, to=6-4]
	\arrow["\rho"', shorten <=44pt, shorten >=44pt, Rightarrow, from=0, to=1]
	\arrow[shorten <=35pt, shorten >=35pt, Rightarrow, no head, from=2, to=0]
	\arrow["\beta"', shorten <=10pt, shorten >=10pt, Rightarrow, from=4, to=5]
	\arrow["{\delta_1}", shorten <=25pt, shorten >=25pt, Rightarrow, from=6, to=4]
	\arrow[shorten <=35pt, shorten >=35pt, Rightarrow, no head, from=3, to=7]
	\arrow[shorten <=35pt, shorten >=35pt, Rightarrow, no head, from=8, to=7]
	\arrow[shorten <=2pt, shorten >=2pt, Rightarrow, no head, from=5-2, to=6-2]
	\arrow["{\rho_{v_1,v_2}}"', shorten <=2pt, shorten >=2pt, Rightarrow, from=5-4, to=6-4]
	\arrow["{\varepsilon_{1,2}^{-1}}"', shorten <=10pt, shorten >=10pt, Rightarrow, from=4-3, to=5-4]
	\arrow[shorten <=5pt, shorten >=2pt, Rightarrow, no head, from=6-2, to=7-2]
\end{tikzcd}\]
By the definition of $x_{v_1,v_2}$ this is equal to
\[\begin{tikzcd}[ampersand replacement=\&]
	\&\&\& D \&\& {B'} \\
	\& {P_{u_1\overline{v}_1,u_2\overline{v}_2}} \&\&\& {A'} \\
	\& {P_{u_1\overline{v}_1,u_2v_1^*}} \&\& {P_{u_1,u_2}} \& {A''} \&\& B \\
	R \&\&\& {D^*} \&\& {B'} \\
	\& {P_{u_2v_1^*,u_2\overline{v}_2}} \&\&\&\& {A''} \& {B''} \\
	\&\&\&\&\& {D'}
	\arrow["{v_1}", from=1-6, to=3-7]
	\arrow[""{name=0, anchor=center, inner sep=0}, "{f_1}"{description}, from=2-5, to=3-7]
	\arrow[""{name=1, anchor=center, inner sep=0}, "{f_2}"{description}, from=3-5, to=3-7]
	\arrow["{v_1}"{description}, from=4-6, to=3-7]
	\arrow[""{name=2, anchor=center, inner sep=0}, "{f_2}"{description}, from=5-6, to=3-7]
	\arrow["{v_2}"{description}, from=5-7, to=3-7]
	\arrow["{\overline{v}_1}"', from=1-4, to=2-5]
	\arrow[""{name=3, anchor=center, inner sep=0}, "{\overline{f}_1}", from=1-4, to=1-6]
	\arrow[""{name=4, anchor=center, inner sep=0}, "{\pi_{A'}}"{description}, from=3-4, to=2-5]
	\arrow["{\pi_{A''}}"', from=3-4, to=3-5]
	\arrow[""{name=5, anchor=center, inner sep=0}, "{w^*_{u_1,u_2}}"{description}, from=3-2, to=3-4]
	\arrow[""{name=6, anchor=center, inner sep=0}, "{\pi_D'}"{description}, from=3-2, to=1-4]
	\arrow["{\pi_D}", from=2-2, to=1-4]
	\arrow[""{name=7, anchor=center, inner sep=0}, "r", from=4-1, to=2-2]
	\arrow["{\pi_1}"{description}, from=4-1, to=3-2]
	\arrow[""{name=8, anchor=center, inner sep=0}, "{f_2^*}"{description}, from=4-4, to=4-6]
	\arrow["{v_1^*}"{description}, from=4-4, to=3-5]
	\arrow[""{name=9, anchor=center, inner sep=0}, "{v_1^*}"{description}, from=4-4, to=5-6]
	\arrow[""{name=10, anchor=center, inner sep=0}, "{\overline{f}_2}"', from=6-6, to=5-7]
	\arrow["{\overline{v}_2}"', from=6-6, to=5-6]
	\arrow["{\pi''_{D^*}}"{description}, from=5-2, to=4-4]
	\arrow[""{name=11, anchor=center, inner sep=0}, "{\pi''_{D'}}"', from=5-2, to=6-6]
	\arrow[""{name=12, anchor=center, inner sep=0}, "{\pi'_{D^*}}"', from=3-2, to=4-4]
	\arrow[""{name=13, anchor=center, inner sep=0}, "{\pi_2}"', from=4-1, to=5-2]
	\arrow["{\delta_1}", shorten <=12pt, shorten >=12pt, Rightarrow, from=1-6, to=2-5]
	\arrow["\beta", shorten <=6pt, shorten >=6pt, Rightarrow, from=2-5, to=3-5]
	\arrow[shorten <=25pt, shorten >=25pt, Rightarrow, no head, from=1-4, to=3-4]
	\arrow[shorten <=15pt, shorten >=15pt, Rightarrow, no head, from=2-2, to=3-2]
	\arrow["{\varepsilon_{1,2}^{-1}}"', shorten <=15pt, shorten >=15pt, Rightarrow, from=3-5, to=4-6]
	\arrow["{\delta_2^{-1}}", shorten <=23pt, shorten >=20pt, Rightarrow, from=2, to=10]
	\arrow[shorten <=8pt, shorten >=8pt, Rightarrow, no head, from=3-4, to=4-4]
	\arrow["\rho", shorten <=25pt, shorten >=20pt, Rightarrow, from=3-2, to=5-2]
	\arrow["{\varepsilon_{1,2}}"', shorten <=2pt, shorten >=2pt, Rightarrow, from=4-6, to=5-6]
	\arrow["{\tilde{\rho}_{u_2v_1^*,u_2\overline{v}_2}}"',shorten <=30pt, shorten >=40pt, Rightarrow, from=4-4, to=6-6]
\end{tikzcd}\]
We can now cancel $\varepsilon_{1,2}^{-1}$ and $\varepsilon_{1,2}$:
\begin{equation}\label{number7}
\begin{tikzcd}[ampersand replacement=\&,column sep = small]
	\& {P_{u_1\overline{v}_1,u_2\overline{v}_2}} \&\&\& D \&\& {B'} \\
	R \&\& {P_{u_1\overline{v}_1,u_2\overline{v}_2}} \&\& {P_{u_1,u_2}} \&\& {A'} \&\& B \\
	\& {P_{u_2v_1^*,u_2\overline{v}_2}} \&\& {D^*} \&\&\& {A''} \\
	\&\& {D'} \&\&\&\& {B''}
	\arrow[""{name=0, anchor=center, inner sep=0}, "{\pi_{A'}}"{description}, from=2-5, to=2-7]
	\arrow[""{name=1, anchor=center, inner sep=0}, "{\pi_{A''}}"{description}, from=2-5, to=3-7]
	\arrow["{w^*_{u_1,u_2}}"', from=2-3, to=2-5]
	\arrow[""{name=2, anchor=center, inner sep=0}, "{\pi_1}"', from=2-1, to=2-3]
	\arrow["{\pi'_D}"{description}, from=2-3, to=1-5]
	\arrow[""{name=3, anchor=center, inner sep=0}, "{\overline{v}_1}", from=1-5, to=2-7]
	\arrow[""{name=4, anchor=center, inner sep=0}, "{\overline{f}_1}", from=1-5, to=1-7]
	\arrow[""{name=5, anchor=center, inner sep=0}, "{\pi_D}", from=1-2, to=1-5]
	\arrow["r", from=2-1, to=1-2]
	\arrow["{v_1}", from=1-7, to=2-9]
	\arrow[""{name=6, anchor=center, inner sep=0}, "{f_1}"{description}, from=2-7, to=2-9]
	\arrow["{f_2}"{description}, from=3-7, to=2-9]
	\arrow[""{name=7, anchor=center, inner sep=0}, "{\pi_2}"', from=2-1, to=3-2]
	\arrow[""{name=8, anchor=center, inner sep=0}, "{\pi''_{D^*}}"', from=3-2, to=3-4]
	\arrow[""{name=9, anchor=center, inner sep=0}, "{\pi'_{D^*}}"{description}, from=2-3, to=3-4]
	\arrow["{v_1^*}"{description}, from=3-4, to=3-7]
	\arrow["{\pi''_{D'}}"', from=3-2, to=4-3]
	\arrow[""{name=10, anchor=center, inner sep=0}, "{\overline{v}_2}"{description}, from=4-3, to=3-7]
	\arrow[""{name=11, anchor=center, inner sep=0}, "{\overline{f}_2}"', from=4-3, to=4-7]
	\arrow["{v_2}"', from=4-7, to=2-9]
	\arrow[shorten <=20pt, shorten >=20pt, Rightarrow, no head, from=1-2, to=2-3]
	\arrow["{\delta_1}", shorten <=7pt, shorten >=7pt, Rightarrow, from=1-7, to=2-7]
	\arrow[shorten <=5pt, shorten >=5pt, Rightarrow, no head, from=1-5, to=2-5]
	\arrow[shorten <=10pt, shorten >=10pt, Rightarrow, no head, from=2-5, to=3-4]
	\arrow["\rho", shorten <=12pt, shorten >=12pt, Rightarrow, from=2-3, to=3-2]
	\arrow["{\tilde\rho_{u_2v^*_1,u_2\overline{v}_2}}"', shorten <=30pt, shorten >=20pt, Rightarrow, from=8, to=10]
	\arrow["{\delta_2^{-1}}", shift left=8, shorten <=2pt, shorten >=2pt, Rightarrow, from=10, to=11]
	\arrow["\beta", shorten <=5pt, shorten >=5pt, Rightarrow, from=2-7, to=3-7]
\end{tikzcd}\end{equation}
We again decide to focus on just a part of this diagram -- we call this part ({\bf II}).
\[\begin{tikzcd}[ampersand replacement=\&]
	\&\&\& {P_{u_q\overline{v}_1,u_2v_1^*}} \&\& {P_{u_1,u_2}} \\
	{\bf (II)} \& \equiv \& R \&\& {D^*} \&\& {A''} \\
	\&\&\& {P_{u_2v_1^*,u_2\overline{v}_2}} \&\& {D'}
	\arrow[""{name=0, anchor=center, inner sep=0}, "{\pi_1}", from=2-3, to=1-4]
	\arrow[""{name=1, anchor=center, inner sep=0}, "{\pi_2}"', from=2-3, to=3-4]
	\arrow["{\pi_{D^*}'}", from=1-4, to=2-5]
	\arrow["{\pi''_{D^*}}"', from=3-4, to=2-5]
	\arrow[""{name=2, anchor=center, inner sep=0}, "{v_1^*}", from=2-5, to=2-7]
	\arrow[""{name=3, anchor=center, inner sep=0}, "{w^*_{u_1,u_2}}", from=1-4, to=1-6]
	\arrow["{\pi_{A''}}", from=1-6, to=2-7]
	\arrow[""{name=4, anchor=center, inner sep=0}, "{\pi_{D'}''}"', from=3-4, to=3-6]
	\arrow["{\overline{v}_2}"', from=3-6, to=2-7]
	\arrow["\rho", shift left=10, shorten <=10pt, shorten >=10pt, Rightarrow, from=0, to=1]
	\arrow["{\tilde\rho_{u_2v_1^*,u_2\overline{v}_2}}"{description}, shorten <=10pt, shorten >=10pt, Rightarrow, from=2, to=4]
	\arrow[shorten <=20pt, shorten >=20pt, Rightarrow, no head, from=3, to=2]
\end{tikzcd}\]
We will now show that if we post-compose diagram {\bf (II)} with $u_2$, we get   an identity 2-cell. To show this, we post-compose diagram {\bf (II)} with $u_2$ and then pre-compose with the invertible 2-cell $\rho_{u_1,u_2}$:
\[\begin{tikzcd}[ampersand replacement=\&, column sep = small]
	\&\& D \&\&\&\&\&\& D \& {A'} \\
	\& {P_{u_1\overline{v}_1,u_2\overline{v}_2}} \& {P_{u_1,u_2}} \& {A'} \&\&\&\& {P_{u_1\overline{v}_1,u_2\overline{v}_2}} \\
	R \&\& {D^*} \& {A''} \& A \& \equiv \& R \&\& {D^*} \& {A''} \& A \\
	\& {P_{u_2v_1^*,u_2\overline{v}_2}} \& {D'} \&\&\&\&\& {P_{u_2v_1^*,u_2\overline{v}_2}} \& {D'} \& {A''} \\
	\&\&\&\&\&\&\& D \\
	\&\&\&\&\& \equiv \& R \& {P_{u_1\overline{v}_1,u_2\overline{v}_2}} \&\& A \\
	\&\&\&\&\&\&\& {D'}
	\arrow[""{name=0, anchor=center, inner sep=0}, "{\pi_2}"', from=3-1, to=4-2]
	\arrow[""{name=1, anchor=center, inner sep=0}, "{\pi_1}", from=3-1, to=2-2]
	\arrow["{\pi'_{D^*}}"', from=2-2, to=3-3]
	\arrow["{\pi''_{D^*}}", from=4-2, to=3-3]
	\arrow[""{name=2, anchor=center, inner sep=0}, "{v_1^*}", from=3-3, to=3-4]
	\arrow[""{name=3, anchor=center, inner sep=0}, "{w_{u_1,u_2}^*}", from=2-2, to=2-3]
	\arrow["{\pi_{A''}}"{description}, from=2-3, to=3-4]
	\arrow[""{name=4, anchor=center, inner sep=0}, "{u_2}", from=3-4, to=3-5]
	\arrow[""{name=5, anchor=center, inner sep=0}, "{\pi_{A'}}", from=2-3, to=2-4]
	\arrow["{u_1}", from=2-4, to=3-5]
	\arrow["{\pi'_D}", from=2-2, to=1-3]
	\arrow[""{name=6, anchor=center, inner sep=0}, "{\overline{v}_1}", from=1-3, to=2-4]
	\arrow["{\overline{v}_2}"', from=4-3, to=3-4]
	\arrow[""{name=7, anchor=center, inner sep=0}, "{\pi''_{D'}}"', from=4-2, to=4-3]
	\arrow[""{name=8, anchor=center, inner sep=0}, "{\pi_1}", from=3-7, to=2-8]
	\arrow[""{name=9, anchor=center, inner sep=0}, "{\pi_2}"', from=3-7, to=4-8]
	\arrow["{\pi''_{D^*}}", from=4-8, to=3-9]
	\arrow["{\pi'_{D^*}}"', from=2-8, to=3-9]
	\arrow["{\pi'_D}", from=2-8, to=1-9]
	\arrow[""{name=10, anchor=center, inner sep=0}, "{v_1^*}", from=3-9, to=3-10]
	\arrow[""{name=11, anchor=center, inner sep=0}, "{\overline{v}_1}", from=1-9, to=1-10]
	\arrow["{u_2}", from=3-10, to=3-11]
	\arrow["{u_1}", from=1-10, to=3-11]
	\arrow["{u_2}"', from=4-10, to=3-11]
	\arrow[""{name=12, anchor=center, inner sep=0}, "{\overline{v}_2}"', from=4-9, to=4-10]
	\arrow["{\pi''_{D'}}"', from=4-8, to=4-9]
	\arrow[""{name=13, anchor=center, inner sep=0}, "r", from=6-7, to=6-8]
	\arrow["{\pi_D}", from=6-8, to=5-8]
	\arrow["{\pi_{D'}}", from=6-8, to=7-8]
	\arrow[""{name=14, anchor=center, inner sep=0}, "{\pi'_{D'}\pi_2}"', curve={height=6pt}, from=6-7, to=7-8]
	\arrow[""{name=15, anchor=center, inner sep=0}, "{\pi'_D\pi_1}", curve={height=-6pt}, from=6-7, to=5-8]
	\arrow[""{name=16, anchor=center, inner sep=0}, "{u_1\overline{v}_1}", curve={height=-12pt}, from=5-8, to=6-10]
	\arrow[""{name=17, anchor=center, inner sep=0}, "{u_2\overline{v}_2}"', curve={height=12pt}, from=7-8, to=6-10]
	\arrow["\rho", shorten <=19pt, shorten >=21pt, Rightarrow, from=2-2, to=4-2]
	\arrow[shorten <=5pt, shorten >=5pt, Rightarrow, no head, from=1-3, to=2-3]	
	\arrow[shorten <=5pt, shorten >=5pt, Rightarrow, no head, from=2-3, to=3-3]
	\arrow["{\rho_{u_1,u_2}}"{description}, shorten <=2pt, shorten >=2pt, Rightarrow, from=2-4, to=3-4]
	\arrow["{\tilde\rho_{u_2v_1^*,u_2\overline{v}_2}}"{description}, shorten <=7pt, shorten >=7pt, Rightarrow, from=3-3, to=4-3]
	\arrow["\rho", shorten <=19pt, shorten >=21pt, Rightarrow, from=2-8, to=4-8]
	\arrow["{\rho_{u_1\overline{v}_1,u_2v_1^*}}",shift right=4, shorten <=27pt, shorten >=27pt, Rightarrow, from=11, to=10]
	\arrow["{\rho_{u_2v_1^*,u_2\overline{v}_2}}",shift right=4, shorten <=10pt, shorten >=10pt, Rightarrow, from=10, to=12]
	\arrow[shorten <=3pt, shorten >=8pt, Rightarrow, no head, from=15, to=6-8]
	\arrow[shorten <=8pt, shorten >=3pt, Rightarrow, no head, from=6-8, to=14]
	\arrow["{\rho_{u_1\overline{v}_1,u_2\overline{v}_2}}"{description}, shorten <=15pt, shorten >=15pt, Rightarrow, from=16, to=17]
\end{tikzcd}\]
The first equality above follows from the universal property of the arrow $w^*_{u_1,u_2}$, and the second equality follows from the definition of the arrow $r$. The definition of $w_{u_1,u_2}$ now implies that this pasting is equal to
\[\begin{tikzcd}[ampersand replacement=\&]
	{P_{u_1\overline{v}_1,u_2v_1^*}} \& D \& {A'} \\
	R \& {P_{u_1,\overline{v}_1,u_2\overline{v}_2}} \& {P_{u_1,u_2}} \& {A''} \& A \\
	{P_{u_2v_1^*,u_2\overline{v}_2}} \& {D'}
	\arrow["{u_2}"{description}, from=2-4, to=2-5]
	\arrow["{\pi_{A''}}", from=2-3, to=2-4]
	\arrow[""{name=0, anchor=center, inner sep=0}, "{w_{u_1,u_2}}", from=2-2, to=2-3]
	\arrow["{\pi_{A'}}"', from=2-3, to=1-3]
	\arrow[""{name=1, anchor=center, inner sep=0}, "{u_1}", curve={height=-12pt}, from=1-3, to=2-5]
	\arrow[""{name=2, anchor=center, inner sep=0}, "{\overline{v}_2}"', curve={height=12pt}, from=3-2, to=2-4]
	\arrow["{\pi_{D'}}", from=2-2, to=3-2]
	\arrow["{\pi_D}"', from=2-2, to=1-2]
	\arrow[""{name=3, anchor=center, inner sep=0}, "{\overline{v}_1}", from=1-2, to=1-3]
	\arrow[""{name=4, anchor=center, inner sep=0}, "r"', from=2-1, to=2-2]
	\arrow["{\pi_1}", from=2-1, to=1-1]
	\arrow["{\pi_2}"', from=2-1, to=3-1]
	\arrow[""{name=5, anchor=center, inner sep=0}, "{\pi''_{D'}}"', from=3-1, to=3-2]
	\arrow[""{name=6, anchor=center, inner sep=0}, "{\pi'_D}", from=1-1, to=1-2]
	\arrow["{\rho_{u_1,u_2}}"', shorten <=3pt, shorten >=4pt, Rightarrow, from=1, to=2-4]
	\arrow[shorten <=13pt, shorten >=16pt, Rightarrow, no head, from=1-1, to=2-2]
	\arrow[shorten <=14pt, shorten >=17pt, Rightarrow, no head, from=3-1, to=2-2]
	\arrow[shorten <=15pt, shorten >=18pt, Rightarrow, no head, from=1-2, to=2-3]
	\arrow[shorten <=14pt, shorten >=19pt, Rightarrow, no head, from=2-3, to=3-2]
\end{tikzcd}\]
Now composing with an appropriate whiskering of $\rho_{u_1,u_2}^{-1}$ gives the promised identity 2-cell.
We conclude that there is an arrow $\tilde{u}_2\colon R'\to R$ in $\frakW$ such that 
diagram {\bf (II)} pre-composed with $\tilde{u}_2$ is an identity 2-cell as claimed.   When we substitute this into (\ref{number7}), we get the following pasting diagram:
\[\begin{tikzcd}[ampersand replacement=\&]
	\&\&\&\& {B'} \\
	\&\&\& D \& {A'} \\
	{R'} \& R \& {P_{u_1\overline{v}_1,u_2\overline{v}_2}} \& {P_{u_1,u_2}} \&\&\& B \\
	\&\& {P_{u_2v_1^*,u_2\overline{v}_2}} \& {D'} \& {A''} \\
	\&\&\&\& {B''}
	\arrow["{\tilde{u}_2}", from=3-1, to=3-2]
	\arrow["r", from=3-2, to=3-3]
	\arrow["{\pi_D}", from=3-3, to=2-4]
	\arrow[""{name=0, anchor=center, inner sep=0}, "{\pi_2}"{description}, from=3-2, to=4-3]
	\arrow["{\pi''_{D'}}"', from=4-3, to=4-4]
	\arrow[""{name=1, anchor=center, inner sep=0}, "{\pi_{D'}}"{description}, from=3-3, to=4-4]
	\arrow[""{name=2, anchor=center, inner sep=0}, "{\overline{v}_1}", from=2-4, to=2-5]
	\arrow[""{name=3, anchor=center, inner sep=0}, "{w_{u_1,u_2}}", from=3-3, to=3-4]
	\arrow["{\pi_{A'}}"{description}, from=3-4, to=2-5]
	\arrow[""{name=4, anchor=center, inner sep=0}, "{\pi_{A''}}"{description}, from=3-4, to=4-5]
	\arrow[""{name=5, anchor=center, inner sep=0}, "{f_1}"{description}, from=2-5, to=3-7]
	\arrow["{\overline{f}_1}", from=2-4, to=1-5]
	\arrow[""{name=6, anchor=center, inner sep=0}, "{v_1}", curve={height=-6pt}, from=1-5, to=3-7]
	\arrow["{\overline{f}_2}"', from=4-4, to=5-5]
	\arrow[""{name=7, anchor=center, inner sep=0}, "{\overline{v}_2}"{description}, from=4-4, to=4-5]
	\arrow[""{name=8, anchor=center, inner sep=0}, "{f_2}"{description}, from=4-5, to=3-7]
	\arrow[""{name=9, anchor=center, inner sep=0}, "{v_2}"', curve={height=6pt}, from=5-5, to=3-7]
	\arrow[shorten <=15pt, shorten >=15pt, Rightarrow, no head, from=2-4, to=3-4]
	\arrow[shorten <=15pt, shorten >=15pt, Rightarrow, no head, from=3-3, to=4-3]
	\arrow[shorten <=15pt, shorten >=15pt, Rightarrow, no head, from=3-4, to=4-4]
	\arrow["{\delta_1}", shorten <=5 pt, shorten >=5pt, Rightarrow, from=1-5, to=2-5]
	\arrow["\beta", shorten <=24pt, shorten >=24pt, Rightarrow, from=2-5, to=4-5]
	\arrow["{\delta_2^{-1}}", shorten <=5pt, shorten >=5pt, Rightarrow, from=4-5, to=5-5]
\end{tikzcd}\]
We can rewrite this as
\[\begin{tikzcd}[ampersand replacement=\&]
	\&\& D \&\& {B'} \\
	\&\&\& {A'} \\
	{R'} \& {P_{u_1\overline{v}_1,u_2\overline{v}_2}} \& {P_{u_1,u_2}} \&\& B \\
	\&\&\& {A''} \\
	\&\& {D'} \&\& {B''}
	\arrow["{r\tilde{u}_2}", from=3-1, to=3-2]
	\arrow[""{name=0, anchor=center, inner sep=0}, "{\pi_D}", from=3-2, to=1-3]
	\arrow[""{name=1, anchor=center, inner sep=0}, "{\pi_{D'}}"', from=3-2, to=5-3]
	\arrow["{w_{u_1,u_2}}"', from=3-2, to=3-3]
	\arrow["{\pi_{A'}}"{description}, from=3-3, to=2-4]
	\arrow["{\pi_{A''}}"{description}, from=3-3, to=4-4]
	\arrow[""{name=2, anchor=center, inner sep=0}, "{\overline{v}_1}"{description}, from=1-3, to=2-4]
	\arrow[""{name=3, anchor=center, inner sep=0}, "{\overline{v}_2}"{description}, from=5-3, to=4-4]
	\arrow[""{name=4, anchor=center, inner sep=0}, "{f_2}"{description}, from=4-4, to=3-5]
	\arrow[""{name=5, anchor=center, inner sep=0}, "{f_1}"{description}, from=2-4, to=3-5]
	\arrow[""{name=6, anchor=center, inner sep=0}, "{\overline{f}_2}"', from=5-3, to=5-5]
	\arrow[""{name=7, anchor=center, inner sep=0}, "{\overline{f}_1}", from=1-3, to=1-5]
	\arrow["{v_1}", from=1-5, to=3-5]
	\arrow["{v_2}"', from=5-5, to=3-5]
	\arrow["\beta"', shorten <=20pt, shorten >=20pt, Rightarrow, from=2-4, to=4-4]
	\arrow[shorten <=21pt, shorten >=21pt, Rightarrow, no head, from=1-3, to=3-3]
	\arrow[shorten <=21pt, shorten >=21pt, Rightarrow, no head, from=3-3, to=5-3]
	\arrow["{\delta_1}"{description}, shorten <=17pt, shorten >=17pt, Rightarrow, from=7, to=5]
	\arrow["{\delta_2^{-1}}"{description}, shorten <=17pt, shorten >=17pt, Rightarrow, from=4, to=6]
\end{tikzcd}\]
If we assume that  $\beta$ is invertible, this pasting is equal to the following by definition of $w_{v_1,v_2}$ given in Section \ref{horcomppb}:
\[\begin{tikzcd}[ampersand replacement=\&]
	\& D \& {B'} \\
	{R'} \& {P_{u_1\overline{v}_1,u_2\overline{v}_2}} \& {P_{v_1,v_2}} \& B \\
	\& {D'} \& {B''}
	\arrow["{r\tilde{u}_2}", from=2-1, to=2-2]
	\arrow[""{name=0, anchor=center, inner sep=0}, "{\pi_D}", from=2-2, to=1-2]
	\arrow[""{name=1, anchor=center, inner sep=0}, "{\pi_{D'}}"', from=2-2, to=3-2]
	\arrow["{\overline{f}_1}", from=1-2, to=1-3]
	\arrow["{\overline{f}_2}"', from=3-2, to=3-3]
	\arrow["{w_{v_1,v_2}}"', from=2-2, to=2-3]
	\arrow[""{name=2, anchor=center, inner sep=0}, "{\pi_{B'}}"', from=2-3, to=1-3]
	\arrow[""{name=3, anchor=center, inner sep=0}, "{\pi_{B''}}", from=2-3, to=3-3]
	\arrow[""{name=4, anchor=center, inner sep=0}, "{v_1}", curve={height=-6pt}, from=1-3, to=2-4]
	\arrow[""{name=5, anchor=center, inner sep=0}, "{v_2}"', curve={height=6pt}, from=3-3, to=2-4]
	\arrow[shorten <=18pt, shorten >=18pt, Rightarrow, no head, from=1-2, to=2-3]
	\arrow[shorten <=18pt, shorten >=18pt, Rightarrow, no head, from=2-3, to=3-2]
	\arrow["{\rho_{v_1,v_2}}"{description}, shorten <=8pt, shorten >=8pt, Rightarrow, from=4, to=5]
\end{tikzcd}\]
When we post-compose this with $\rho_{v_1,v_2}^{-1}$ we find that  when the pasting of diagram {\bf (I)} is pre-composed with $\tilde{u}_2$ and post-composed with $v_1$, the result is equal to the identity 2-cell on $v_1\pi_{B'}w_{v_1,v_2}r\tilde{u}_2$. 
So there is an arrow $\tilde{v}_1\colon R''\to R$ in $\frakW$ such that pasting (\ref{diagram_A}) pre-composed by $r\tilde{u}_2\tilde{v}_1$ is the identity 2-cell on $\pi_{B'}w_{v_1,v_2}r\tilde{u}_2\tilde{v}_1$.
We finally post-compose with $\gamma$ to find that $\widetilde{\gamma\beta}$ is precomposed with $r\tilde{u}_2\tilde{v}_1$ is equal to
\[\begin{tikzcd}[ampersand replacement=\&]
	\&\&\& D \& {B'} \\
	{R''} \&\& {P_{u_1\overline{v}_1,u_2\overline{v}_2}} \& {P_{u_1,u_2}} \&\& C \\
	\&\&\& {D'} \& {B''}
	\arrow["{r\tilde{u}_2\tilde{v}_1}", from=2-1, to=2-3]
	\arrow[""{name=0, anchor=center, inner sep=0}, "{\pi_D}", from=2-3, to=1-4]
	\arrow[""{name=1, anchor=center, inner sep=0}, "{\pi_{D'}}"', from=2-3, to=3-4]
	\arrow["{w_{u_1,u_2}}"', from=2-3, to=2-4]
	\arrow[""{name=2, anchor=center, inner sep=0}, "{\pi_{B'}}"{description}, from=2-4, to=1-5]
	\arrow[""{name=3, anchor=center, inner sep=0}, "{\pi_{B''}}"{description}, from=2-4, to=3-5]
	\arrow["{\overline{f}_1}"{description}, from=1-4, to=1-5]
	\arrow["{\overline{f}_2}"{description}, from=3-4, to=3-5]
	\arrow[""{name=4, anchor=center, inner sep=0}, "{g_1}", from=1-5, to=2-6]
	\arrow[""{name=5, anchor=center, inner sep=0}, "{g_2}"', from=3-5, to=2-6]
	\arrow["\gamma", shorten <=25pt, shorten >=25pt, Rightarrow, from=1-5, to=3-5]
	\arrow[shorten <=16pt, shorten >=16pt, Rightarrow, no head, from=1-4, to=2-4]
	\arrow[shorten <=16pt, shorten >=16pt, Rightarrow, no head, from=2-4, to=3-4]
\end{tikzcd}\]
We conclude that diagram (\ref{horcomposition}) given in Section \ref{horcomppb} and the diagram constructed from the vertical composition of whiskerings are equivalent as claimed. 
\end{proof}

\begin{rmk}
If $\beta$ is not invertible, the 2-cell diagram (\ref{firstrepr}) within the proof above gives a representation of the horizontal composition. Unfortunately, without further assumptions there is no obvious way to simplify this representation.
\end{rmk}
\end{appendices}
\end{document}